\documentclass{article}
\usepackage{amsmath, amssymb, amsthm, amsfonts}
\usepackage{mathrsfs, bm, relsize, upgreek}
\usepackage{graphicx, tikz, caption}
\usetikzlibrary{calc}
\usepackage{mathtools}
\usepackage[shortlabels]{enumitem}
\usepackage{svrsymbols}
\usepackage{pgfplots}
\usepackage{verbatim}
\usepackage{thmtools}
\usepackage{bbm}
\usepackage[colorlinks=true, linkcolor=blue, citecolor=blue, urlcolor=blue]{hyperref}
\usepackage[left=2.65cm,right=2.65cm,top=2.7cm,bottom=2.7cm]{geometry}

\mathtoolsset{showonlyrefs}

\pgfplotsset{compat=1.18}

\DeclareFontFamily{OT1}{pzc}{}
\DeclareFontShape{OT1}{pzc}{m}{it}{<-> s * [1.10] pzcmi7t}{}
\DeclareMathAlphabet{\mathpzc}{OT1}{pzc}{m}{it}

\DeclareMathAlphabet{\dcal}{U}{dutchcal}{m}{n}
\SetMathAlphabet{\dcal}{bold}{U}{dutchcal}{b}{n}
\DeclareMathAlphabet{\dbcal}{U}{dutchcal}{b}{n}

\newcommand{\col}[2]{\begin{pmatrix} #1 \\ #2 \end{pmatrix}}

\newcommand{\spaced}[2][\quad]{#1\mbox{#2}#1}

\newcommand{\Bsp}{\begin{aligned}{\vphantom{\int}}}
\newcommand{\Msp}{{\vphantom{\int}}}
\newcommand{\Esp}{{\vphantom{\int}}\end{aligned}}
\newcommand{\Lcm}{{{\,\vphantom{|}}_,}}

\newcommand{\R}{\mathbb{R}}
\newcommand{\C}{\mathbb{C}}
\newcommand{\Z}{\mathbb{Z}}

\newcommand{\X}{\mathbb{X}}

\newcommand{\cA}{\mathcal{A}}

\newcommand{\cD}{\mathcal{D}}
\newcommand{\cE}{\mathcal{E}}
\newcommand{\cF}{\mathcal{F}}
\newcommand{\cG}{\mathcal{G}}
\newcommand{\cH}{\mathcal{H}}
\newcommand{\cI}{\mathcal{I}}

\newcommand{\cK}{\mathcal{K}}

\newcommand{\cN}{\mathcal{N}}
\newcommand{\cO}{\mathcal{O}}
\newcommand{\cP}{\mathcal{P}}
\newcommand{\cQ}{\mathcal{Q}}
\newcommand{\cR}{\mathcal{R}}
\newcommand{\cS}{\mathcal{S}}
\newcommand{\cT}{\mathcal{T}}
\newcommand{\cU}{\mathcal{U}}
\newcommand{\cV}{\mathcal{V}}
\newcommand{\cW}{\mathcal{W}}

\newcommand{\cY}{\mathcal{Y}}

\newcommand{\sB}{\mathscr{B}}
\newcommand{\sC}{\mathscr{C}}

\newcommand{\sH}{\mathscr{H}}

\newcommand{\sJ}{\mathscr{J}}

\newcommand{\sM}{\mathscr{M}}
\newcommand{\sN}{\mathscr{N}}

\newcommand{\sP}{\mathscr{P}}

\newcommand{\sR}{\mathscr{R}}
\newcommand{\sS}{\mathscr{S}}

\newcommand{\sU}{\mathscr{U}}
\newcommand{\sV}{\mathscr{V}}

\newcommand{\sZ}{\mathscr{Z}}

\newcommand{\slp}{\underline{\dcal{S}}}

\newcommand{\Slp}{\dcal{S}}
\newcommand{\Dlp}{\dcal{D}}

\newcommand{\lV}{\lVert}
\newcommand{\rV}{\rVert}

\newcommand{\del}{\partial}
\newcommand{\half}{\frac{1}{2}}

\newcommand{\grad}{\nabla}
\newcommand{\pv}{\operatorname{pv}}
\newcommand{\supp}{\operatorname{supp}}
\newcommand{\cof}{\operatorname{cof}}

\newcommand{\init}{{\scriptscriptstyle{\mathit{init}}}}
\newcommand{\gp}{{\mathrm{g\kern 0.1ex p}}}
\newcommand{\ca}{{\mathrm{c\kern 0.1ex a}}}

\newcommand{\res}{{\operatorname{res}}}
\newcommand{\nres}{\bign_{\operatorname{res}}}
\newcommand{\Nres}{\mathcal{N}_{\operatorname{res}}}

\newcommand{\dcF}{\dcal{F}}

\newcommand{\dcX}{\dcal{X}}
\newcommand{\splash}{{\scriptscriptstyle{\mathrm{S}}}}
\newcommand{\uX}{\underline{X}}

\newcommand{\uxi}{\underline{\xi}}

\newcommand{\utheta}{\underline{\theta}}

\newcommand{\play}{{\hspace{.07em}\triangleright}}
\newcommand{\revplay}{{\hspace{.05em}\triangleleft}}
\newcommand{\tosplash}{{\scriptstyle{\mathrm{ts}}}}
\newcommand{\skipto}{{\play}}
\newcommand{\wbox}{{\scalebox{0.5}{$\,\square$}}}
\newcommand{\bbox}{{\scalebox{0.5}{$\,\blacksquare$}}}

\makeatletter
\newcommand{\ol}{\@ifnextchar\bgroup{\overline}{\overline}}
\makeatother

\newcommand{\cXbar}{{
  \mathord{
    \ooalign{
      $\dcX$\cr
      \hidewidth\raisebox{0.7ex}{\rule[0pt]{0.4em}{0.3pt}}\hidewidth
    }
  }}_\gp
}

\DeclareSymbolFont{extraup}{U}{zavm}{m}{n}
\DeclareMathSymbol{\vardiamond}{\mathalpha}{extraup}{83}

\newcommand{\bigh}{\displaystyle{\mathlarger{\dcal{h}}}}
\newcommand{\bign}{\dcal{N}}
\newcommand{\bigtau}{\mathrm{T}}
\newcommand{\bX}{\bm{X}}
\newcommand{\bU}{\bm{U}}
\newcommand{\bB}{\bm{B}}
\newcommand{\bP}{\bm{P}}
\newcommand{\bM}{\mathbf{M}}

\newcommand{\bD}{{\mathrm{D}}}
\newcommand{\bE}{{\mathrm{E}}}
\newcommand{\bI}{{\mathrm{I}}}
\newcommand{\bJ}{\mathrm{J}}
\newcommand{\bj}{\bm J}
\newcommand{\bR}{\mathrm{R}}
\newcommand{\bV}{\bm{V}}
\newcommand{\bW}{\bm{W}}
\newcommand{\bom}{\displaystyle{\mathlarger{\bm{\omega}}}}

\newcommand{\bvarphi}{\mathlarger{\bm{\upvarphi}}}

\newcommand{\bn}{{n_*}}
\newcommand{\HTW}{{H^2(S^1)}}

\newcommand{\sHTO}{{\sH^1}}
\newcommand{\sHTT}{{\sH^2}}
\newcommand{\npo}{{n+1}}
\newcommand{\nmo}{{n-1}}

\newtheorem{theorem}{Theorem}[section]
\newtheorem{proposition}[theorem]{Proposition}
\newtheorem{lemma}[theorem]{Lemma}
\newtheorem{corollary}[theorem]{Corollary}

\theoremstyle{definition}
\newtheorem{definition}[theorem]{Definition}

\theoremstyle{remark}
\newtheorem{remark}[theorem]{Remark}

\newtheorem{innercustomthm}{Proposition}
\newenvironment{myprop}[1]
  {\renewcommand\theinnercustomthm{#1}\innercustomthm}
  {\endinnercustomthm}

\numberwithin{equation}{section}

\author{Diego C\'{o}rdoba, Alberto Enciso, and Matthew Hernandez}

\title{Splash--squeeze singularities and analytic breakdown\\in ideal incompressible MHD}

\begin{document}

\maketitle

\begin{abstract}
We construct \emph{splash--squeeze} singularities for the free boundary ideal incompressible plasma--vacuum system, in which two arcs of the plasma boundary come together to form a smooth, glancing self-intersection. As the interface self-intersects, Sobolev norms remain bounded, although analyticity is necessarily lost. This contrasts classical splash singularities, in which solutions remain analytic up to the time of self-intersection.

The narrowing gap bounded by these arcs is not occupied by plasma, as squeezing the plasma itself would cause blow-up in Sobolev norms. Instead, the gap represents the region outside the plasma, a vacuum carrying a nontrivial magnetic field. The plasma on either side pinches the field as the gap closes, and, in response, the field vanishes to infinite order at the intersection point (and nowhere else), thereby forming an analytic singularity. This gives the first example of analytic breakdown without Sobolev blow-up in a locally well-posed free-boundary incompressible fluid system, and can be viewed as the first rigorous construction of a \emph{squeeze}-type singularity, in which we study and quantify the precise behavior of an active, incompressible vector field as it is completely pinched off by a free-boundary in finite time.

The proof combines a magnetically-aligned Lagrangian formulation of ideal MHD together with weighted elliptic estimates in the vacuum that remain uniform as the width of the gap tends to zero. Our framework may provide a starting point for the analysis of squeeze-type singularities in other incompressible fluid models.
\end{abstract}

\setcounter{tocdepth}{2}
\tableofcontents

\section{Introduction}

The only known scenario of singularity formation for incompressible fluids with a smooth free boundary is the {\em splash singularity},\footnote{We would be remiss not to mention the ``splat'' variant, in which a (non-analytic) interface intersects itself along an arc, rather than a single point.} in which the boundary develops a self-intersection at some positive time~$t_\splash$, as depicted in Figures~\ref{pre_intersect}--\ref{post_intersect}. Classical splash singularities, which have been constructed for models such as Euler, Navier--Stokes, Muskat, and viscous MHD, are characterized by two essential features. First, the fluid preserves its regularity: when the initial data is Sobolev (say, in $H^k$) or analytic, the solution remains bounded in $H^k$ or analytic norms up to the time~$t_\splash$. Second, nothing exists in the region which is pinched in two as the arcs of the interface come together. An inverted scenario, in which an incompressible fluid becomes completely pinched by a pair of smooth arcs, is typically impossible.\footnote{See our discussion of literature ruling out the formation of smooth fluid pinch-offs in Section \ref{litreviewsection}.} In particular, splash singularities cannot occur between two fluids, and in the single-fluid case, they may describe the collision of two smooth waves but not the formation of a water droplet.

Our objective is to introduce a new type of singularity for incompressible fluids by constructing what we call \emph{splash--squeeze} singularities. In these, a splash forms in such a way that the fluid squeezes a nontrivial, active, incompressible vector field living in the complementary region. The model we consider is the free-boundary ideal incompressible MHD system, describing the evolution of a plasma (understood as an electrically conducting incompressible fluid) and its surrounding magnetic field.

%

We are able to demonstrate the complete pinching of an active, incompressible vector field because it is only the magnetic field in the vacuum that we squeeze, as opposed to a substance made up of material particles, like a fluid. While the vacuum magnetic field is incompressible, it carries no particles, and for this reason one expects the magnetic field in the narrowing gap to resist the formation of a ``nice'' squeeze in some sense less than a fluid (but more than an empty vacuum) would. We make this intuition precise by proving the following: our splash--squeeze constructions remain bounded in $H^k$ up to the time of self-intersection, but, in contrast with the classical splash singularity, they encounter a fundamental obstruction to analyticity. Due to this, we are able to construct splash--squeeze solutions which blow up in any analytic norm in finite time, starting from analytic initial data.

In the rest of the introduction, we shall provide a more detailed explanation of the structure of these splash--squeeze singularities in ideal MHD and of the associated phenomenon of breakdown of analyticity

\begin{figure}[h]
\begin{center}
\begin{minipage}{0.45\textwidth}
    \centering
    \includegraphics[width=\textwidth]{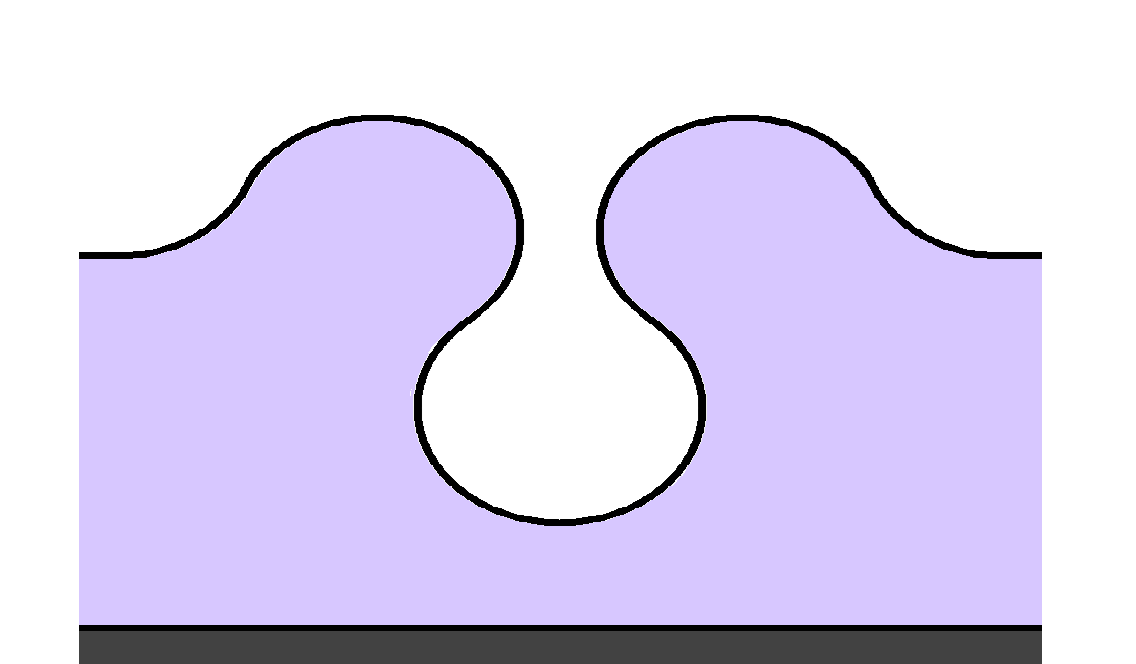}
    \begin{tikzpicture}[overlay, remember picture]
    \end{tikzpicture}
    \captionof{figure}{Before a splash}\label{pre_intersect}
\end{minipage}
\hfill 
\begin{minipage}{0.45\textwidth}
    \centering
    \includegraphics[width=\textwidth]{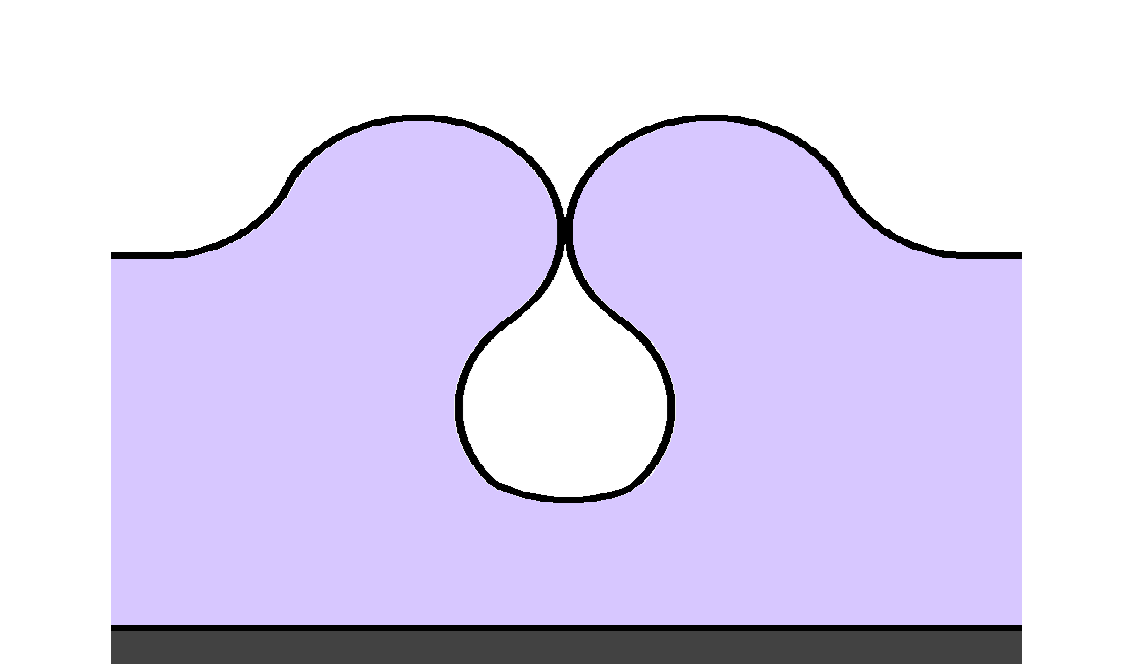}
    \begin{tikzpicture}[overlay, remember picture]
    \end{tikzpicture}
    \captionof{figure}{Splash}\label{post_intersect}
\end{minipage}
\end{center}
\end{figure}

\subsection{The ideal MHD system}

The 2D ideal incompressible MHD equations model the macroscopic behavior of incompressible plasmas, such as those found in astrophysical environments or fusion devices, when dissipative effects such as viscosity and resistivity are negligible. 

In the geometric configuration considered in this paper, the physical system (assumed periodic in $x_1$) is confined to the planar plasma chamber $\mathcal{C}$. In each period, the boundary $\del\mathcal{C}$ consists of three fixed components: the plasma floor $\{x_2 = 0\}$, an upper vacuum wall $\mathcal W_1$, 
and a circular wall $\mathcal W_2$.\footnote{We may take $\cW_1$ to be $\{x_2 = 10\}$ and $\cW_2$ to be $\{|x-(0,1)|= 10^{-1}\}$, for example.} We denote by $\Omega(t)$ the region of the chamber occupied by the plasma at time $t$, and by $\mathcal{V}(t) = \mathcal{C} \setminus \overline{\Omega(t)}$ the remaining vacuum region. At least for small times, the boundary of the plasma region, $\partial\Omega(t)$, has two connected components: the floor $\{x_2=0\}$ and the interface $\Gamma(t)= \overline{\Omega(t)} \cap \overline{\mathcal{V}(t)}$, which does not touch the vacuum wall $\mathcal{W}=\mathcal W_1\cup\mathcal W_2$.
\begin{figure}[h]
\begin{center}
\begin{minipage}{0.45\textwidth}
    \centering
        \includegraphics[width=\textwidth]{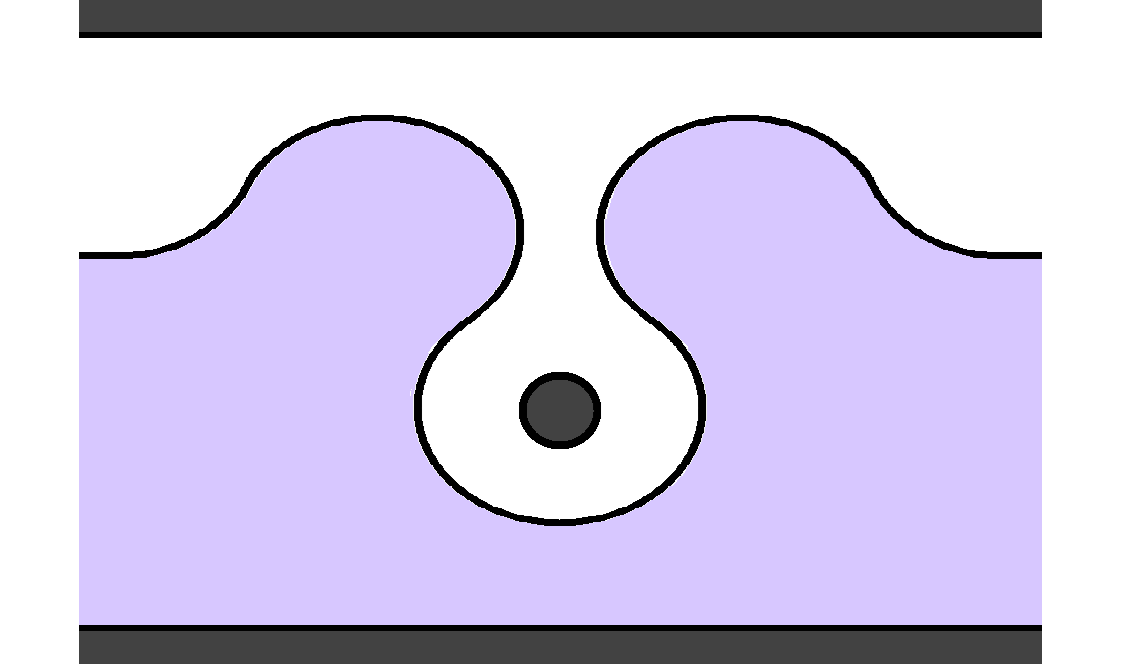}

    \begin{tikzpicture}[overlay, remember picture]

        \node at (0.48,2.05) {$\mathcal{W}_2$};
        \node at (3.4,4.45) {$\mathcal{W}_1$};
        \node at (2.1,2) {$\Omega(t)$};
        \node at (0,4) {$\mathcal{V}(t)$};
        \node at (-3.5,3) {$\Gamma(t)$};
        \node at (3.8,0.65) {$\{x_2=0\}$};
    \end{tikzpicture}
    
    \caption{Plasma chamber before splash}\label{OpenedSplash1}
\end{minipage}
\hfill 
\begin{minipage}{0.45\textwidth}
    \centering
            \includegraphics[width=\textwidth]{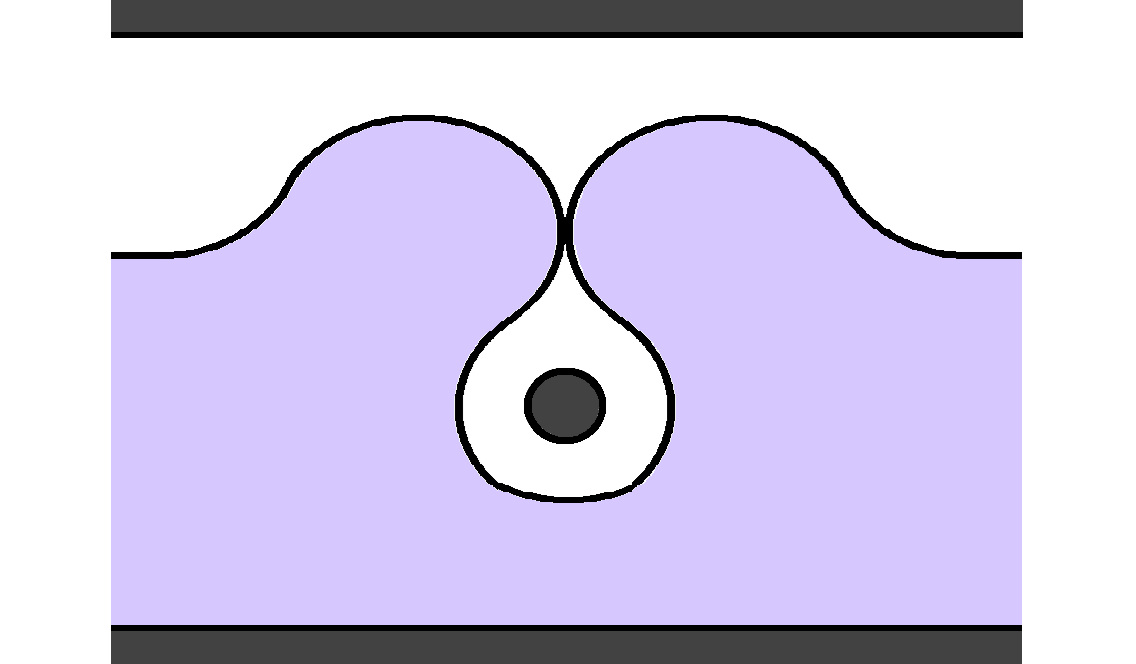}

    \begin{tikzpicture}[overlay, remember picture]

        \node at (2.1,2) {$\Omega_0$};
        \node at (-3.2,3) {$\Gamma_0$};
    \end{tikzpicture}
    \captionof{figure}{Plasma chamber at splash time}\label{ClosedSplash1}
\end{minipage}
\end{center}
\end{figure}

\begin{figure}[!h]
    \centering

\end{figure}

For plasma velocity field $u(t,x)$, magnetic field in the plasma $b(t,x)$, and magnetic field in the vacuum $h(t,x)$, the ideal MHD systems in the plasma and the vacuum are given by
\begin{align}
(x \in \Omega(t)) \qquad
& \left\{
    \begin{aligned}
    & \left.
        \begin{aligned}
        \del_t u + u \cdot \grad u &= b \cdot \grad b - \grad p \\
        \del_t b + u \cdot \grad b &= u \cdot \grad b \, ,
        \end{aligned}
        \right. \\[1ex]
    &  \left.
    \begin{aligned}
        \grad \cdot u &= 0 \ \\
        \grad \cdot b &= 0 \, ,
    \end{aligned}
    \right.
    \end{aligned}
    \right.
\label{IdealMHD1intro} \\[1ex]
(x \in \cV(t) ) \qquad
&   \left\{
        \left.
        \begin{aligned}
            \grad \cdot h &= 0 \ \\
            \grad^\perp \cdot h & = 0 \, .
        \end{aligned}
        \right.
    \right.
\label{IdealMHD2intro} 
\end{align}
To specify the dynamics of the plasma with a given initial datum $(u_0(x),b_0(x),h_0(x),\Gamma_0)$, the system must be supplemented with appropriate boundary conditions. The moving plasma--vacuum boundary $\Gamma(t)$ is transported by the fluid, magnetic flux does not cross the interface, and the pressure is continuous across $\Gamma(t)$. The boundary of the chamber $\del \mathcal{C} = \{x_2=0\}\cup \mathcal W$ is perfectly conductive, so that both $u$ and $b$ must be tangent to the floor $\{x_2=0\}$ and $h$ must be tangent to the vacuum wall $\cW$. We may also prescribe the circulation of $h$ around each wall $\mathcal{W}_i$, given by a constant~$\eta_i$. The boundary conditions for the MHD system are thus\footnote{Our constructions can also be adapted to more general circulation $\eta_i(t)$ and normal prescriptions $n\cdot h=f(t,x)$ for $h$ at the vacuum wall $\cW$, such as continuously differentiable $\eta_i(t)$ and $f(t,x)$.}
\begin{align}
(x \in \Gamma(t)) \qquad
& \left\{
    \begin{aligned}
        n \cdot b &= 0 \\
        n \cdot h &= 0 \\
        p = &\half |h|^2 \, ,
    \end{aligned}
\right.
\label{IdealMHD3intro} \\[1ex]
(x_2 = 0) \qquad
& \left\{
    \begin{aligned}
        n \cdot u &= 0 \\
        n \cdot b &= 0\, ,
    \end{aligned}
    \right.\label{IdealMHD4intro}\\[1ex]
(x \in \mathcal{W}) \qquad
& \left\{
    \begin{aligned}
        n \cdot h &= 0\, ,
    \end{aligned}
\right. \\
(i=1,2) \qquad
&  \int_{\mathcal W_i} h \cdot d \vec{r} = \eta_i \, .
\label{IdealMHD5intro}
\end{align}

Under the assumption that the interface is initially free of self-intersections, the free-boundary MHD system is locally well-posed in Sobolev spaces~\cite{SunWangZhang2019}. If, additionally, the initial data are analytic, classical methods prove~\cite{Nirenberg72,safonov} that the system is also locally well-posed in analytic function spaces, as we show in Proposition~\ref{analyticlwp} in the main text.

As the singularity forms, the interface develops a glancing self-intersection at exactly one point, transitioning from the picture in Figure~\ref{OpenedSplash1} to that in Figure~\ref{ClosedSplash1}. This disconnects the vacuum region, pinching off a drop-shaped component. The inclusion of the two fixed walls $\mathcal W_1,\mathcal W_2$ in our geometric setting, with well-chosen boundary conditions, ensures that $h$ remains nontrivial at the moment our splash--squeeze forms.

To gain some intuition, let us consider the moments leading up to a splash. As the two nearly glancing arcs of $\Gamma(t)$ approach one another, the tangency condition $n\cdot h = 0$ requires the magnetic field lines to become compressed in the shrinking neck. One might expect this to require (or generate) large forces as the external field undergoes a topological change when the critical moment of pinch-off occurs. 

Despite this, we discover that the formation of the glancing arcs at the moment of splash generically creates a rapid decay of the magnetic field near the self-intersection, causing it to vanish to infinite order precisely at the singular point and nowhere else. Thus, the somewhat surprising conclusion of our analysis is that the system develops splash--squeeze singularities without losing smoothness, but it cannot develop a splash--squeeze without losing analyticity.

\subsection{Main result and strategy of the proof}

Our main result can be (somewhat informally) stated as follows. For a detailed statement, see 
Theorem~\ref{maintheorem2} in the main text.

\begin{theorem}\label{T.intro}
    Fix any positive integer~$k$. There exists analytic initial data $(u_\init,b_\init,h_\init,\Gamma_\init)$ 
    such that the (unique) corresponding solution $(u,b,h,\Gamma)$ to the system \eqref{IdealMHD1intro}--\eqref{IdealMHD5intro} has the following properties:
    \begin{enumerate}
        \item {\em Persistence of smoothness:} There is some $t_\splash>0$ for which solution is well-defined for time~$t\in[0,t_\splash]$, and remains bounded in $H^k$ up to time~$t_\splash$.
        \item {\em Breakdown of analyticity:} There is some $t_\star \in(0,t_\splash]$ such that the solution $(u,b, h,\Gamma)$ is analytic for $t\in[0,t_\star)$, but not up to time~$t_\star$.
        \item {\em Singularity formation:} The interface develops a single glancing self-intersection at time~$t_\splash$. For $t\in[0,t_\splash)$, the interface curve $\Gamma(t)$ does not intersect itself.
        \item {\em Magnetic squeezing:} At time~$t_\splash$, the magnetic field~$h$ is non-vanishing in each pinched-off region of the vacuum.
    \end{enumerate}
\end{theorem}

Let us now discuss the main ideas behind the proof of this theorem. Given any solution to ideal MHD which forms a splash--squeeze singularity, there is a conceptual reason that analyticity must fail at the self-intersection point: the key is that the glancing-arc shape of a splash curve can never be realized as the zero-set of a nontrivial harmonic function.\footnote{
Details are given in Section~7.} What is much more involved is to \emph{construct} solutions to the MHD equations with the properties we are seeking. 

Our first idea is to restrict our attention to initial data for which the magnetic lines of~$b_0$ nicely foliate the plasma region. This allows us to introduce magnetically-aligned Lagrangian coordinates, resulting in a semilinear wave-equation form in the bulk. Although this simplifies the interior evolution significantly, the equations are not yet ready for analysis because they involve a loss of derivatives. To compensate for this, we introduce good unknowns ($\dot U^*$, $\dot B^*$) at the surface, thereby reducing the entire free-boundary MHD evolution to a closed, coupled system of wave equations.

The remaining difficulty then lies not in the local existence structure itself, but in bounding the vacuum-side quantities that directly ``feel the pinch'', notably Dirichlet-to-Neumann maps that blow up as the pinch-width $\delta$ tends to zero.
To overcome this, we establish $\delta$-weighted elliptic estimates with constants independent of $\delta$ to quantify and exploit the decay of $h$ near the pinch. Additionally, we prove a weighted analogue of a vital Dirichlet-to-Neumann cancellation discovered in previous studies (see \cite{currentvortexsheet,mhdexistence}), but only understood in the non-splash setting. Together, these yield uniform-in-$\delta$ control of the combined pinch-sensitive forcing terms, allowing us to prove local existence in standard Sobolev spaces with a basic iteration method.

With this result in hand, we may readily show with a time-reversal argument\footnote{This type of argument was first introduced in \cite{splash2} to construct splash singularities for water waves.} that there exist smooth initial data which develop splash--squeeze singularities. This is what we establish in Theorem~\ref{maintheorem}. To demonstrate this, we first construct an initial splash configuration which opens up as time advances, exhibiting a separated interface at some time $T>0$. We then essentially evolve the system backward in time, so that the smooth solution at time $t=T$ gives us a ``pre-splash'' initial datum, and the splash which was previously the initial state now becomes the final state of the system.

However, generic pre-splash initial data produced with the above method is not analytic, so demonstrating breakdown of analyticity, as we do in Theorem~\ref{maintheorem2}, requires further work. In summary, the central idea is to convert the argument of Section \ref{constructionsection} to a forward-in-time construction, starting from analytic initial data with a non-self-intersecting interface. This requires us to control and compare solution states for which the future splash times and splash points vary in a neighborhood of our original splash parameters. By realigning these states with special splash-dependent changes of variables, we are able to close a modified iteration scheme. This forward-in-time construction allows us simply to consider analytic initial data close to a splash, say with interface separation $\delta$, verify the solutions exist forward in time on a time interval independent of $\delta$, and ultimately verify they realize splash--squeeze singularities before the end of the time interval of existence is reached.\footnote{After the time of self-intersection, our solutions cease to be physically meaningful.} In fact, the construction is not sensitive to small changes in the initial data, implying the splash--squeeze and resulting phenomenon of analytic breakdown are robust.

\subsection{Literature review}\label{litreviewsection}

In the original paper on splash singularities \cite{splash2}, the authors constructed both Sobolev-regular and analytic splash solutions. Their construction relies on a conformal map similar to $z\mapsto\sqrt z$ to ``open'' the space between the converging arcs of the splash, while leaving the fluid region unbroken. This turns the splash curve into a chord–arc curve for which boundary singular integrals (like the Hilbert transform) are controlled. The system is evolved in those desingularized coordinates, and then the solution is pulled back to obtain a splash, with the help of the same time-reversal argument sketched earlier. It is important to keep in mind that the ``square-root trick'', prominently featured in the literature, is of quite limited use in the analysis of  squeeze singularities. This is because a square-root transform tears the vacuum apart and therefore cannot address the key issues at the heart of the construction: controlling a nontrivial field in the exterior and the operators that blow up as the squeeze forms.

Other past work exhibits splashes for water waves with surface tension \cite{splashtension} and proves stability of splash singularities \cite{splashstability}. Sobolev and analytic splash solutions were constructed for the Muskat equation in \cite{muskatsplash}. Splash constructions without time-reversibility were carried out for the free-boundary Navier--Stokes equations in \cite{naviersplash1,naviersplash2} and for viscoelastic fluids in \cite{oldroydsplash,viscosplash}. Splash singularity formation in three dimensions, where conformal maps cannot be used to simplify the analysis, was first established for the water wave system in~\cite{shkoller1}.   In general, across a range of fluid models, all the known splash singularities share the same basic features: the solution preserves the Sobolev or analytic regularity of the initial datum up to the splash time, and the pinched-off region does not contain any active quantity, just a field-free vacuum.

Closely related to this principle are several results in the literature ruling out the dynamical formation of smooth fluid pinch-offs in various free-boundary models. In particular, the works \cite{nosplash4,nosquirt,CordobaPernasCastano2017,nosplash3,nosplash2,nosplash1, JeonZlatos2024, KiselevLuo2023} rule out the formation of interface self-intersections that pinch off a region occupied by fluid under the assumption that the solution remains smooth up to the splash time. 

Among other splash-type variants for free-boundary models is the {splat singularity}, in which self-contact of the interface occurs along an arc, rather than a point. These have been constructed for models in \cite{splash2,shkoller1}, for example. 
A second splash-related solution type is the {stationary splash}, where stationary solutions to free-boundary models exhibit a splash profile, constructed in \cite{stationarysplash1,stationarysplash2}.

Regarding loss of smoothness, in \cite{CastroCordobaFeffermanGancedo2013Breakdown} it was shown that there exist analytic initial data in the stable regime for the Muskat problem such that the solution evolves into the unstable regime and later breaks down, i.e. it ceases to be in $C^4$. In the context of smooth interfaces that touch a fixed boundary, there are influential results on singularity formation for modified SQG~\cite{KiselevRyzhikYaoZlatos2016} (see also \cite{GancedoPatel},  \cite{jeon2025wellposednessfinitetimesingularity}, and \cite{zlatos2023localregularityfinitetime}), Euler~\cite{Coutand2019ARMA}, and Muskat~\cite{Zlatos2024MuskatII}. A  somewhat related recent result~\cite{Agrawal2023UniformGravity} proves singularity formation for water waves with angled crests without requiring a boundary. Based on different ideas, another scenario for singularity formation in angled-crested water waves was recently established in~\cite{CordobaEncisoGrubic2023AngledCrested}, relying  on the analysis of the asymptotic dynamics of the fluid near the corner points.

To conclude, let us discuss the existing results on singularity formation for plasma models. The only available results are~\cite{haoMHDsplash,hongMHDsplash}, which construct splash singularities for free-boundary viscous MHD (with external magnetic field identically zero), from the Eulerian perspective, by building on ideas of Coutand and Shkoller~\cite{CoutandShkoller2019NS}. These share the fundamental features of all previous splash singularities: no loss of Sobolev or analytic regularity up to the splash time, and no active quantity in the vacuum. No prior study of splash singularities treats a vacuum region with a nonzero exterior magnetic field. This regime is physically natural for astrophysical plasmas, such as the Sun, where the ambient magnetic field is observed to be nonzero.

As we outlined in the previous section, a key starting point of the present paper is to reformulate the equations in a novel way which makes their structure more transparent and amenable to analysis. Even beyond splash formation, ours appears to be the first Lagrangian local existence proof for a free-boundary MHD model with a nonzero exterior field. Moreover, earlier Lagrangian local existence constructions either involved the use of viscous models or introduced artificial smoothing in the ideal (i.e. inviscid, nonresistive) setting to recover regularity lost at the iteration step (see \cite{GuMHDtension,LeeLagrViscLimMHD,XieLuoALE}). By using magnetic Lagrangian coordinates and good surface unknowns $(\dot U^*,\dot B^*)$, which recast the free-boundary ideal MHD system as a coupled system of wave equations, we avoid this regularity loss and prove local existence with a basic iteration method for standard nonlinear wave equations. We believe that this formulation will be of independent interest for future studies of the free-boundary ideal MHD system.

\subsection{Organization of the paper}

Section 2 sets out the geometric framework used to analyze the free-boundary ideal MHD equations, formulating the problem in terms of magnetic Lagrangian variables. Section 3 derives the Lagrangian wave system for the “good unknowns” that underpins our analysis, together with the plasma-centric and vacuum-centric operator maps needed later on. Section 4 develops uniform estimates in an almost-pinched vacuum region. Using the time-reversibility of the equations, Section 5 establishes the dynamical formation of splash–squeeze singularities from classically smooth, i.e. Sobolev data, providing fairly general bounds later used also to construct the example demonstrating analytic breakdown. Section 6 collects several auxiliary results used throughout the proofs. Finally, Section 7 proves that splash--squeeze singularities cannot be analytic and, building on the previous uniform pinch estimates, establishes the existence of analytic data that lead to analytic breakdown and splash--squeeze singularity formation.

\subsection*{Acknowledgements}

This work has received funding from the European Research Council (ERC) under the European Union's Horizon 2020 research and innovation programme through the grant agreement~862342 (A.E.~and M.H.). The authors are partially supported by the grants CEX2023-001347-S, RED2022-134301-T,  PID2023-152878NB-I00 (D.C.) and PID2022-136795NB-I00 (A.E.) funded by MCIN/AEI/10.13039/501100011033.

\section{General setup}\label{generalsetup}

\subsection{The ideal free-boundary MHD equations}

\subsubsection*{The interior velocity, interior magnetic field, and exterior magnetic field system.}


We have the velocity field $u(t,x)$ and interior magnetic field $b(t,x)$ in $\Omega(t)$, and an external magnetic field $h(t,x)$ defined in the vacuum $\cV(t)$. The ideal free-boundary MHD equations consist of the following conditions, imposed for all time $t$ in some interval $[0,T]$.
\begin{align}
(x \in \Omega(t)) \qquad
& \left\{
    \begin{aligned}
    & \left.
        \begin{aligned}
        \del_t u + u \cdot \grad u &= b \cdot \grad b - \grad p \\
        \del_t b + u \cdot \grad b &= u \cdot \grad b \, ,
        \end{aligned}
        \right.  \\[1ex]
    &  \left.
    \begin{aligned}
        \grad \cdot u &= 0 \ \\
        \grad \cdot b &= 0 \, ,
    \end{aligned}
    \right.
    \end{aligned}
    \right.
\label{IdealMHD1} \\[1ex]
(x \in \cV(t) ) \qquad
&   \left\{
        \left.
        \begin{aligned}
            \grad \cdot h &= 0 \ \\
            \grad^\perp \cdot h &= 0 \, ,
        \end{aligned}
        \right.
    \right.
\label{IdealMHD2} \\[2ex]
(x \in \Gamma(t)) \qquad
& \left\{
    \begin{aligned}
        n \cdot b &= 0 \\
        n \cdot h &= 0 \\
        p = &\half |h|^2 \, ,
    \end{aligned}
\right.
\label{IdealMHD3} \\[1ex]
(x_2 = 0) \qquad
& \left\{
    \begin{aligned}
        n \cdot u &= 0 \\
        n \cdot b &= 0\,,
    \end{aligned}
    \right.
\label{IdealMHD4} \\
(x \in \mathcal{W}) \qquad
& \left\{
    \begin{aligned}
        n \cdot h &= 0\, ,
    \end{aligned}
\right. \label{IdealMHD5} \\
(i=1,2) \qquad
&  \int_{\mathcal W_i} h \cdot d \vec{r} = \eta_i \, .
\label{IdealMHD6}
\end{align}
Here and in all that follows, we use the notation~$\grad^\perp \cdot h = -\del_{x_2} h_1 + \del_{x_1} h_2$.

Additional boundary conditions are required for uniqueness, as outlined in the introduction and discussed in detail later. There is also the kinematic condition, which states that each particle on the free surface must move with the boundary. Consider the standard particle trajectory map $\X(t,x)$, defined by
\begin{align*}
(x \in \Omega_0) \qquad
\left\{
\begin{aligned}
\frac{d\X}{dt}(t,x) &= u(t,\X (t,x)) \\
    \X(0,x) &= x \phantom{|}_,
\end{aligned}
\right.
\end{align*}
where $\Omega_0=\Omega(0)$ and $\Gamma_0=\Gamma(0)$. The kinematic condition is then equivalent to
\begin{align}
\label{KinematicCondition}
\Gamma(t) = \X(t, \Gamma_0) .
\end{align}

\begin{remark}
As explained earlier, we include the two vacuum walls, $\mathcal W_1$ and $\mathcal W_2$, rather than treating an unrestricted vacuum above the interface,\footnote{It is trivial to adapt our construction of a splash to the setting without walls.} to exhibit external magnetic fields that get squeezed without being completely extinguished. If instead we had a picture like in Figure~\ref{pre_intersect}, with an unrestricted vacuum, any finite energy external field would be identically zero. Two vacuum walls are necessary to allow for nonzero fields in the two separate connected components of the vacuum that form during a splash. We discuss this further in Section~\ref{squeezedmagfieldssection}.
\end{remark}

\subsubsection*{The pressure system}

Taking the divergence of the first line in \eqref{IdealMHD1} and using the divergence-free conditions on $u$ and $b$ leads to the Poisson system for the pressure, which is
\begin{align}
&&   \Delta p &= \sum^2_{i,\, j=1} \del_{x_i} u_j \del_{x_j} u_i -\del_{x_i} b_j \del_{x_j} b_i
                                & (x \in \Omega(t) ) , \quad \label{PressureSystem} \\
&&          p &= \half |h|^2    & (x \in \Gamma(t) ) , \quad \notag \\[1ex]
&& \del_{x_2} p &= 0            & (x_2 = 0) .          \quad \notag
\end{align}
The determination of the pressure $p$ from the fields $u$, $b$, and $h$ as the unique solution to the above closes the system \eqref{IdealMHD1}--\eqref{IdealMHD6}.

\subsection{Proving the dynamic formation of splash--squeeze singularities}

Before continuing, let us give a slightly more detailed overview of the structure of the paper and of our constructions of splash--squeeze singularities which ultimately lead to a proof of analytic breakdown for free-boundary ideal MHD.

Our results are essentially built in three major steps:
\begin{itemize}
 \item {\em Step~1.}\/ Formulate a solvable system in Lagrangian coordinates, with well-defined terms and operators in appropriate spaces. In particular develop weighted estimates for the external magnetic field which hold across a range of vacuum regions with arbitrarily small pinch-widths.

\item {\em Step~2.}\/ 
Use the first step to establish the existence of splash--squeeze singularities arising from Sobolev initial data. This leads to Theorem~\ref{maintheorem}.

\item {\em Step~3.}\/ 
Upgrade the result of the second step by augmenting the approach at certain key points to prove the existence of splash--squeeze singularities arising from \emph{analytic} initial data. Verify these solutions lose analyticity at or before the time of splash, thus demonstrating analytic breakdown (Theorem~\ref{maintheorem2}).
\end{itemize}

The first step will occupy multiple dedicated sections; for the moment, we examine the second and third steps. Though the proof of Theorem~\ref{maintheorem2} technically does not require the second step, in which we construct singularities starting from Sobolev initial data, Step~2 provides a more transparent construction which offloads many technical difficulties. This allows us to focus on the core challenges of constructing splash--squeezes, providing a guide for later understanding the construction in the third step. The greater transparency of the proof in Step~2 is afforded by time-reversibility, which allows us to know a priori the precise parametrization of the self-intersecting interface at the time of splash and to effectively build the argument around it.

Let us give an outline below of the strategy we use up to the end of Step~2, paying special attention to the use of a backward-in-time approach (although the use of time-reversibility is abandoned in the third step) which is inspired by constructions of splash singularities in past work. Despite this, many new ideas are required. One particular new idea is that we may consider the external field $h$ as a nonlocal function of the interface itself, so that it becomes a kind of source term in the system. Plugging this into the well-known ``backward-in-time approach'' for splash singularities, we do the following to complete Step~2:
\begin{enumerate}[(i)]
    \item Construct good external fields $h=h(\Gamma)$ for general classes of splash and near-splash interfaces $\Gamma$ whose vacuum domains $\cV$ are of pinch-width \( \delta \in [0,\delta_0] \). This is done in Definition~\ref{formalhdefn}. The main Sobolev bound for controlling the $h$ produced is given in Proposition~\ref{hEstimateProp}.
    \item Substitute $h=h(\Gamma(t))$, as defined in (i), into the MHD system.
    \item Define suitable initial splash data $(u_0,b_0,h_0,\Gamma_0)$ with $H^k$ regularity, and interface $\Gamma_0$ exhibiting a single self-intersection as in Figure~\ref{ClosedSplash1}, with initial velocity near the self-intersection point directed in such a way that the glancing arcs in a neighborhood of the splash point will move apart in positive time. This is done in Definition~\ref{initialData}.
    \item Prove that in the resulting system, a solution with the initial data discussed above exists on a time interval $[0,T]$, maintaining $H^k$ regularity. This is the main focus of Section~\ref{solvinglagwavesystemsection}. The solution itself is provided by Proposition~\ref{maintheoremforward}.
    \item Observe that the interface $\Gamma(t)$ separates, resulting in an interface as in Figure~\ref{OpenedSplash1}. This is verified in Proposition~\ref{maintheoremforward}.
    \item Produce a new solution by ``running time backwards'', taking the initial state of the solution (at $t=0$) to be the old solution at $t=T$ (but with flipped velocity everywhere) so that we start with a picture as in Figure~\ref{OpenedSplash1}. Due to time-reversibility of the system, at time $t=T$ the corresponding interface evolves into the configuration which we used previously as initial data, depicted in Figure~\ref{ClosedSplash1}, thus demonstrating a splash at time $t_\splash=T$. This is done in detail in the proof of Theorem~\ref{maintheorem}.
\end{enumerate}
Note the splash--squeeze solution produced by reversing time at the end of Step~2 is thus in $H^k$ at all times from the initial moment up to the moment of splash.

Regarding Step~3, where we actually construct splash--squeeze singularities from analytic initial data, leading to solutions which lose analyticity, we must deviate from the strategy above. We still use (i) and (ii), which are really more a part of Step~1. However, due to reasons explained at the beginning of Section~\ref{analyticbreakdownsubsection}, to prove analytic breakdown, we adapt the tools of Steps 1 and 2 to a forward-in-time existence argument which shows the formation of a splash--squeeze from analytic initial data.

Step~3, the objective of Section~\ref{analyticbreakdownsection}, breaks down into two major insights, which we summarize below.

\begin{enumerate}[(a)]
\item There is a fundamental obstruction to analyticity for any free-boundary MHD splash--squeeze singularity. Since the zero set of a harmonic function cannot consist of a pair of glancing arcs, any $h$ that is not identically zero cannot be analytic at the moment of splash; this is proved in Theorem~\ref{nonexist2Dtheorem}.

\item There exists an analytic initial datum (with a non-self-intersecting interface) and a corresponding solution which is analytic for a short time and which realizes a splash singularity at some time $t=t_\splash$, with $h$ not identically zero at time $t=t_\splash$, so that we have a splash--squeeze. Moreover, the solution maintains $H^k$ regularity over the interval $[0,t_\splash]$. This is proved in Proposition~\ref{mainproposition}.

\end{enumerate}

Theorem~\ref{nonexist2Dtheorem} of (a) has a succinct, conceptual proof, whereas Proposition~\ref{mainproposition} of (b) takes a bit more work. By combining these two ingredients, we see that the solution produced in (b) necessarily encounters an analytic singularity, required by the splash--squeeze, all while the solution remains smooth. The content of (b) involves a forward-in-time local existence argument in Sobolev norms, going ``through the time of splash''. While the iteration argument is similar to that of (iv) in Step~2, the new scheme requires technical modifications since the individual iterates do not all realize the same splash states, as in a backward-in-time approach starting from a splash.

Another critical insight is that many of our most key estimates, such as those involving $h$, are purely built around the geometry of the vacuum at each fixed moment in time, so that the material established in Step~1 is still useful regardless of the direction of time and is robust enough to apply to a variety of splash states.

Additionally, in (b) we prove local existence in analytic spaces for analytic data with separated interfaces, thus demonstrating that the solution is in fact analytic for a short time before an analytic singularity is eventually induced by a splash--squeeze. For more discussion of the ideas and details of Step~3, we refer the reader to Section~\ref{analyticbreakdownsubsection}.

\subsection{Magnetic Lagrangian coordinates}\label{maglagvariablessection}

There are many reasons to use Lagrangian variables in studies of fluid dynamics, as the properties of the equations are often more transparent when written in such a form. For the ideal MHD equations, however, it is also known to be convenient to work in what is sometimes referred to as a \emph{magnetic coordinate system}. This is a curvilinear coordinate system in which one of the spatial variables is aligned with the magnetic field lines. Inspired by this, for our problem we work in a coordinate system which is simultaneously a Lagrangian coordinate system and a magnetic coordinate system, for suitably chosen initial data.

\subsubsection*{Setting and initial data}

We first discuss the initial plasma region $\Omega_0$, with corresponding initial interface $\Gamma_0$ and a convenient choice for our initial magnetic field $b_0$, in particular. The general shape of $\Omega_0$ is as depicted in Figure~\ref{ClosedSplash1}, with the corresponding interface $\Gamma_0$ intersecting itself at a single point $p_\splash$, the splash point.

We denote by $S^1$ the interval $[-\pi,\pi]$ with endpoints identified. For our Lagrangian domain we take $\Sigma = S^1 \times [-1,0]$, typically describing such points in $\Sigma$ with the notation $a=(\theta,\psi)$, whereas points in the Eulerian picture are typically written as $x=(x_1,x_2)$.

\begin{definition}\label{initialData}
\hfill
\begin{enumerate}[(i)]
\item For the initial Lagrangian labeling system, we choose a conformal map $\bX_0$, where
\begin{align}\label{bXinitIntroduction}
    \bX_0: \Sigma \to S^1 \times \R ,
\end{align}
such that the image $\Omega_0 = \bX_0(\Sigma)$ has the qualitative shape shown in Figure~\ref{ClosedSplash1}, symmetric about $x_1 = 0$, with a flat, lower boundary at $x_2 = 0$. We arrange that $\bX(\theta,-1)$ and $\bX(\theta,0)$ travel from left to right along $\{x_2=0\}$ and $\Gamma_0$, respectively, as we increase $\theta$. We also arrange that $\bX_0(0,-1)$ is the origin, $\bX_0(0,0)$ is the bottom of the trough of the splash, with $|\bX_0(0,0)-(0,1/2)|\leq 10^{-1}$, and for the splash point we have $p_\splash = \bX_0(\pm\frac{\pi}{2},0)$ and \(|p_\splash - (0, 2)|\leq 10^{-2} \). We denote by $\Gamma_0$ the upper bounding curve of the region $\Omega_0$. Additionally, denoting the trace of $\bX_0$ at $\psi=0$ by
\begin{align}
&& X_0(\theta) &= \bX_0(\theta,0) & (\theta\in S^1),
\end{align}
we arrange for all $\theta$ in $S^1$ that, where $\kappa_0(\theta)$ denotes the curvature of $\Gamma_0$ at $X_0(\theta)$,
\begin{align}\label{Xzerobds}
&&
\begin{aligned}
|\del_\theta X_0(\theta)| &\geq \epsilon_1 ,\\
\kappa_0(\theta) &\leq 10 ,
\end{aligned}
&&(\theta \in S^1).
\end{align}

\item To define our initial interior velocity field in $\Omega_0$, we define a velocity stream function on the Lagrangian side by
\begin{align}
&& \bvarphi_0(\theta,\psi) &= \underline\nu(1+\psi)\sin(2\theta) &( (\theta,\psi) \in \Sigma), \\ \intertext{for a constant $\underline \nu$. 
We then define the Eulerian stream function for the velocity field by}
&& \varphi_0(x) &= \bvarphi_0( \bX^{-1}_0(x)) & (x\in \Omega_0), \\
\intertext{and then define the initial velocity field to be}
&& u_0(x) &= \grad^\perp \varphi_0(x) & (x\in\Omega_0), \\ \intertext{with Lagrangian version}
&& \bU_0(a) &= u_0(\bX_0(a)) & (a\in\Sigma), \\ \intertext{with trace at $\psi=0$ denoted by}
&& U_0(\theta) &= \bU_0(\theta,0) & (\theta \in S^1), \\
\intertext{
\item We define the initial interior magnetic field in $\Omega_0$ by}
&&    b_0(x) &= \del_\theta \bX_0({\bm{X}}^{-1}_0(x)) & (x\in\Omega_0), \\ \intertext{and for the Lagrangian version we have}
&&    \bB_0(a) &= \del_\theta \bX_0(a) & (a\in\Sigma), \\ \intertext{with trace at $\psi=0$}
&&    B_0(\theta) &= \bB_0(\theta,0) & (\theta\in S^1).
\end{align}
\end{enumerate}
\end{definition}
\begin{remark}
To show that there exists a region $\Omega_0$ with the qualitative shape in Figure~\ref{ClosedSplash1} along with a conformal map $\bX_0$ as described in Definition~\ref{initialData}, one may begin by considering a variety of candidate regions for $\Omega$, say $\tilde\Omega_0$, and solving the Dirichlet problem for some function $\Phi$ in $\tilde \Omega_0$, where $\Phi=0$ on the top boundary and $\Phi=-1$ on the bottom boundary. With the use of the harmonic conjugate $\Psi$, we get a map $\Phi+i\Psi$ from $\tilde\Omega_0$ to a periodic rectangle $\tilde \Sigma$. By inverting this, we get a map from $\tilde \Sigma$ to $\tilde\Omega_0$. One can use the intermediate value theorem to verify that fine-tuned adjustments on $\tilde\Omega_0$ lead to $\tilde \Sigma=\Sigma$ and a map $\bX_0$ with the desired properties.
\end{remark}

Note that we must have that $u_0$ and $b_0$ are divergence-free.
\begin{lemma}
    For $u_0$ and $b_0$ as in Definition~\ref{initialData},
    \begin{align*}
        \grad \cdot u_0(x) = \grad \cdot b_0(x) = 0 \quad \quad (x\in\Omega_0).
    \end{align*}
\end{lemma}
\begin{proof}
The fact that $u_0(x)$ is divergence-free simply follows from the fact that it is given by $\grad^\perp \varphi_0(x)$. On the other hand, because $\bX_0$ satisfies
\begin{align*}
\del_\theta \bX_0(\theta,\psi) = b_0(\bX_0(\theta,\psi)),
\end{align*}
we have $b_0(x)$ is the pushforward by the conformal map $\bX_0$ of the constant vector field attaching $e_1$ to each point $(\theta,\psi)$ in $\Sigma$. Since the constant vector field is trivially divergence-free, $b_0(x)$ is also divergence-free.
\end{proof}

We now record some simple but important bounds for $U_0$, which cause the initial splash curve to separate as time goes forward.
\begin{lemma}\label{separationPrepLem}
For $U_0$ as in Definition~\ref{initialData}, we have the following, in which $p_\splash$ is the splash point of $\Gamma_0$ and $\angle(v_1,v_2)$ represents the angle in $[0,2\pi)$ between the two vectors $v_1,v_2$.
\begin{align}
&& |U_0(\theta)| & \geq 1 &(\theta \in S^1, \ |X(\theta)-p_\splash|\leq 10^{-1}),\\
&& \angle(U_0(\theta),\pm e_1) &\leq 10^{-1} &(\theta \in S^1, \  \theta\gtrless 0,\ |X(\theta)-p_\splash|\leq 10^{-1}).
\end{align}
\end{lemma}
\begin{proof}
This can be verified from the definitions of $u_0$, $\bvarphi_0(x)$, and $\bX_0$, given that $\underline \nu$ is large enough, which we ensure.
\end{proof}
\subsubsection*{Using a field-aligned Lagrangian label}

A magnetic field line is an integral curve of the magnetic field. In view of the construction above, noting that the magnetic field $b_0$ is (i) never zero and (ii) tangent to the boundary of $\Omega_0$, we find that the initial magnetic field lines foliate $\Omega_0$.

Consider again the parametrization $\bX_0(\theta,\psi): \Sigma \to \Omega_0$ of our initial domain. Rather than associating to the particles in $\Omega(t)$ their initial positions as Lagrangian labels, we associate to a particle the label $a = (\theta, \psi) \in \Sigma$  when its initial position is the point $\bX_0(a) \in \Omega_0$. Note that for each fixed $\psi$, as we vary $\theta$, $\bX_0(\theta,\psi)$ traces out a field line of the magnetic field $b_0$, essentially by construction.


Now we define the Lagrangian coordinate system $\bX(t,a)$ mapping $\Sigma$ into $\Omega(t)$ by imposing
\begin{align}\label{lagrSysOrig}
&&
\begin{aligned}
    \frac{d \bX}{d t}(t,a) &= u(t, \bX(t,a)) ,\\
    \bX(0,a) &= \bX_0(a) ,
\end{aligned}
&&(a\in\Sigma).
\end{align}
For the velocity and magnetic field in Lagrangian variables, we denote
\begin{align}\label{lagrUandB}
\bU(t,a) &= u(t,\bX(t,a) ,   & \bB(t,a) &= b(t,\bX(t,a)) .
\end{align}
The above choice of label helps simplify the MHD system. For a solution to ideal MHD, the second equation of \eqref{IdealMHD1} together with the fact
\begin{align*}
    \bB(0,a) = \bB_0(a) = \del_\theta \bX_0 (a) ,
\end{align*}
implies that as long as the solution continues to exist,
\begin{align}\label{bXbBorigIdent}
    \bB(t,a) = \del_\theta\bX (t,a).
\end{align}
One way to see this is simply to check that for any fixed $a$ in $\Sigma$, both $\bB(t,a)$ and $\del_\theta \bX (t,a)$ satisfy the ordinary differential equation for $V(t)$ below:
\begin{align*}
& &    \frac{d V_i}{d t}(t) &= \sum^2_{j=1}(\del_{x_j} u_i) (t, \bm X(t,a))V_j (t)    &   (i=1, \ 2),\\
& &    V(0) &= \bB_0(a) .   &
\end{align*}
In the lemma below, we record this and assert another interesting property of the coordinate system.
\begin{proposition}
\label{JacobianIdent}
Consider $u_0$, $b_0$, $\Omega_0$, etc. as in Definition~\ref{initialData}. Let us also define
\begin{align*}
& &    \bm{\sigma}(\psi) &= |\bB_0(0,\psi)|^2 & (\psi \in [-1,0]).
\end{align*}
    Given a solution $u$ and $b$ on the time interval $[0,T]$ to the ideal MHD system with initial data $u_0$, $b_0$, $\Omega_0$, consider the Lagrangian coordinate system with corresponding map $\bX$ satisfying \eqref{lagrSysOrig}, and $\bU$ and $\bB$ as defined by \eqref{lagrUandB}. Then for all $t$ in $[0,T]$, we have
\begin{align}
& &    \bU(t,a) &= \del_t\bX (t,a)      &   (a &\in \Sigma ), \\[1ex]
& &    \bB(t,a) &= \del_\theta \bX (t,a)&   (a &\in \Sigma), \\[1ex]
& &    \det (\grad \bX(t,a) ) &= \bm{\sigma}(\psi) & ( a &= (\theta,\psi) \in \Sigma ) .
\end{align}
\end{proposition}
\begin{proof}[Proof of Proposition~\ref{JacobianIdent}]
 The first two identities are immediate from the definition of $\bX$ and the comments just before the statement of the lemma. To begin the proof of the third, we verify
\begin{align}
\label{JacobianIdent1}
& &        \det(\grad \bX_0(a)) &= \bm{\sigma}(\psi) & ( a=(\theta,\psi) \in \Sigma ) .
\end{align}
To see this, we check the left-hand side is independent of $\theta$ by observing
    \begin{align*}
        \del_\theta \left( \det(\grad \bX_0(\theta, \psi)) \right) = \det(\grad \bX_0(\theta, \psi)) (\grad \cdot b_0)(\bX_0(\theta, \psi)) = 0,
    \end{align*}
    where we have used that $b_0$ is divergence-free. Meanwhile, since $\bX_0$ is conformal, we have
    \begin{align*}
        \grad \bX_0 = \begin{pmatrix}
            (\del_\theta \bX_0)_1   & (\del_\psi \bX_0)_1  \\
            (\del_\theta \bX_0)_2   & (\del_\psi \bX_0)_2
        \end{pmatrix}
        = \begin{pmatrix}
            (\del_\theta \bX_0)_1   & -(\del_\theta \bX_0)_2 \\
            (\del_\theta \bX_0)_2   & (\del_\theta \bX_0)_1
        \end{pmatrix}.
    \end{align*}
    Computing the determinant, applying $\bB_0 = \del_\theta \bX_0$, and plugging in $\theta=0$ then gives \eqref{JacobianIdent1}.
    Now observe from the divergence-free property of $u$ that we have
    \begin{align*}
        \del_t (\det (\grad \bX(t,a))) = \det (\grad \bX(t,a)) (\grad \cdot u)(t, \bX(t,a)) = 0,
    \end{align*}
    so the determinant is preserved in time.
\end{proof}
\begin{remark}
We may index the initial field lines $\gamma_\psi$ by $\psi$, resulting in a foliation of $\Omega_0$ with the collection $\{\gamma_\psi\}_{\psi \in [-1,0]}$. Recall the label $\psi$ tells us the field line on which the particle labeled $a=(\theta,\psi)$ initially lies. The above lemma tells us that the Jacobian determinant depends only on the associated magnetic field line.
\end{remark}
With the above lemma, we find that the interior evolution equations of ideal MHD, the first two equations of \eqref{IdealMHD1}, become the following simple-looking system in Lagrangian coordinates:
\begin{align}
\label{OriginalLagrangianSystem}
& &
\begin{aligned}
    \bU_t &= \bB_\theta - \bP_\grad \\
    \bB_t &= \bU_\theta
\end{aligned}
&   & ( a \in \Sigma, t \in [0,T] ),
\end{align}
where $\bP_\grad(t,a) = (\grad p)(t,\bX(t,a))$.
We note one more interesting property of the MHD system related to our choice of coordinate system.
\begin{remark}
\hspace{1mm}
\\
If a collection of particles forms a magnetic field line at time $t$, this will remain to be the case as it is carried by the flow. This is sometimes\footnote{Alfven's theorem actually asserts a property known as flux conservation for ideal MHD, which is technically slightly stronger than the one stated above.} referred to as the ``frozen-in condition'' of the field, asserted by Alfven's theorem. This is easy to see in our setting, as
\begin{align*}
    \del_\theta \bX (t, \theta, \psi) &= b(t, \bX(t, \theta, \psi))
\end{align*}
shows us that at any time $t$ all the integral curves of $b(t)$ are simply given by $\bX(t, S^1, \psi)$ for some $\psi$.\footnote{For each fixed $\psi$, by identifying $\gamma_\psi$ with the set of particles lying on it, one may think of the magnetic field line as being ``transported by the flow''.}

\end{remark}

In what follows, $\omega$ denotes the vorticity in the plasma, and $j$ denotes the magnetic current in the plasma:
\begin{align}
&&\omega(t,x) &= \grad^\perp \cdot u(t,x) ,\qquad & j(t,x)&=\grad^\perp\cdot b(t,x) &&(x\in \Omega(t)).\\
\intertext{We define their counterparts in Lagrangian variables by}
&&\bom(t,a) &= \omega(t,\bX(t,a)), &
\bj(t,a) &=j(t,\bX(t,a)) && (a\in \Sigma).
\end{align}
Moreover, we use $U$, $B$, and $X$ to represent the traces of $\bU$, $\bB$, and $\bX$ at the surface:
\begin{align}
&& U(t,\theta)=\bU(t,\theta,0), \qquad B(t,\theta) = \bB(t,\theta,0), \qquad X(t,\theta) = \bX(t,\theta,0) &&(\theta \in S^1).
\end{align}

\subsection{Decomposing the system into good surface and vorticity evolutions}
Let us give a very rough roadmap of how we finally arrive at a solvable Lagrangian system from the formulation \eqref{OriginalLagrangianSystem}. A more complete treatment is given in Sections \ref{refiningsurfsys} and \ref{basicsystemsection}.

In the original Lagrangian system \eqref{OriginalLagrangianSystem} there is too high a level of derivative of the unknowns arising in the tangential part of $\bP_\grad$ at $\psi=0$, i.e. the evaluation of $\grad p$ along the surface field line $\Gamma(t)$. To handle this, we introduce the operator $\bE$, which essentially projects out the tangential harmonic part of a vector field at the surface (see Proposition~\ref{bEDefns}). 
Using this, the next step is to decompose \eqref{OriginalLagrangianSystem} into a surface evolution system for the velocity and magnetic field traces $U$ and $B$ and an interior vorticity evolution. Thanks to the properties of $\bE$ there are no longer dangerous levels of derivatives appearing in the surface system thus produced, which is the following:
\begin{align}
(\theta \in S^1) \qquad
& \left\{
        \begin{aligned}
        U_t &= \bE^{-1}\bD \bE B_\theta + \mathit{l.o.t.} \\
        B_t &=  U_\theta \, ,
        \end{aligned}
        \right. 
        \label{step1lagrsurfsystem} \\[1ex]
( a \in \Sigma ) \qquad
&   \left\{
        \begin{aligned}
        \bom_t &= \bj_\theta \\
        \bj_t  &= \bom_\theta + \mathit{l.o.t.} \, ,
        \end{aligned}
        \right.
        \label{step1lagrsurfsystem2}
\end{align}
where $\mathit{l.o.t.}$ indicates lower order terms (see \eqref{BasicSurfSys1}--\eqref{BasicSysDefns} for details), and $\bD$ is given by
\[
\bD=
\begin{pmatrix}
1 & 0 \\
0 & 1+\frac{|H|^2}{|B|^2}
\end{pmatrix} .
\] 

Note \eqref{step1lagrsurfsystem2} forms a wave system pair. The pair of equations \eqref{step1lagrsurfsystem} for $U$ and $B$ has some similarities, but since the matrix operator $\bE^{-1}\bD \bE$ is not diagonal, it is not a wave system. Furthermore, $\bE^{-1}\bD \bE$ is not symmetrizable, so energy methods for general hyperbolic systems cannot be directly applied to solve \eqref{step1lagrsurfsystem}--\eqref{step1lagrsurfsystem2}.

To develop an equivalent system for which such methods do apply, we switch from using $U$ and $B$ to the good unknowns below, derived from the time derivatives of $U$ and $B$:
\begin{align}
&& \dot U^* &= \bE  U_t, \qquad \dot B^* = \bE B_t&& (\theta \in S^1).
\end{align}
In Section \ref{basicsystemsection}, we convert \eqref{step1lagrsurfsystem} to a system for $\dot U^*$ and $\dot B^*$ to derive the desired Lagrangian surface and vorticity evolution system with good unknowns, which we give below.
\begin{align}
(\theta \in S^1) \qquad
& \left\{
        \begin{aligned}
        \dot U^*_t &= \bD \dot B^*_\theta + \mathit{l.o.t.} \\
        \dot B^*_t &= \dot U^*_\theta + \mathit{l.o.t.} \, ,
        \end{aligned}
        \right. 
        \label{step2lagrsurfsystem} \\[1ex]
( a \in \Sigma ) \qquad
&   \left\{
        \begin{aligned}
        \bom_t &= \bj_\theta \\
        \bj_t  &= \bom_\theta + \mathit{l.o.t.} \, ,
        \end{aligned}
        \right.
        \label{step2lagrsurfsystemb}
\end{align}
See \eqref{symmetrizable2}--\eqref{symmetrizable1} for the full system. Due to the diagonal form of $\bD$, \eqref{step2lagrsurfsystem} indeed forms a wave sub-system (modulo couplings with $\bom$ and $\bj$ in the lower order terms). We are thus able to solve the resulting quasilinear Lagrangian wave system \eqref{step2lagrsurfsystem}--\eqref{step2lagrsurfsystemb} by standard methods.


\subsection{The external magnetic field}\label{squeezedmagfieldssection}

For ideal MHD, the exterior magnetic field $h$ is solenoidal and irrotational in the vacuum. In the case of perfectly conductive chamber walls, it is tangent to $\del \mathcal{V} = \Gamma \cup\cW$.
\begin{align}
\begin{aligned}
(x \in \cV ) \qquad
&   \left\{
        \left.
        \begin{aligned}
            \grad \cdot h &= 0 \ \\
            \grad^\perp \cdot h & = 0 \, ,
        \end{aligned}
        \right.
    \right.\\
(x \in \del \cV)  \qquad & \left\{n \cdot h = 0 \right. .
\end{aligned}
\label{ExternalFieldEq1}
\end{align}

In this section, we consider the picture at fixed moments in time and construct appropriate $h$ induced by the shape of the vacuum. Time is not relevant at this stage, and so we will omit the parameter $t$ from the objects we discuss for now. Let us consider the role of the boundary of vacuum domain, which consists of the vacuum walls $\cW_1,\cW_1$ and an interface $\Gamma$ in a near-splash or splash configuration.

\subsubsection{Vacuum chamber walls}\label{vacuumchamberwallssection}
If we were not to include any vacuum chamber wall $\cW$ in our picture, instead dealing with a setting as in Figure~\ref{pre_intersect} as opposed to that in Figure~\ref{OpenedSplash1}, any physically realistic\footnote{To assert this we need only assume the magnetic field $h$ decays at infinity.} vacuum magnetic field $h$ would have to be zero at all times.

To illustrate this, consider the system without any wall. Then \eqref{ExternalFieldEq1} is satisfied if and only if we have $h=\grad \phi + \grad^\perp \uppsi$ for some $\phi,\uppsi:\cV\to\R$ satisfying the following, where $\tau = -n^\perp$ gives the tangent to $\del\cV$.
\begin{align*}
\Delta \phi &= 0, &  \Delta \uppsi &= 0 ,     &       &(x \in \cV) , \\
\del_n \phi &= 0, &  \del_\tau \uppsi &= 0 ,     &      &(x \in \del \cV) .
\end{align*}
A calculation based on the divergence theorem implies $\phi$ must be trivial and thus $h=\grad^\perp\uppsi$. Meanwhile, the boundary condition implies $\uppsi$ is constant on each component of $\del\cV$. In the absence of a vacuum wall, $\cV$ forms a single simply-connected region (or two, in the event of a splash), in which $\uppsi$ can only be constant, and therefore $h=0$ in $\cV$.

Thus, incorporating a wall is necessary for nontrivial $h$. Due to the topology, however, a splash--squeeze requires a wall with more than one component. For example, if one takes a wall such as $\cW_1$ without an additional component like $\cW_2$, the argument above shows $h$ must vanish throughout the drop-shaped vacuum region that splits off as a splash forms. Each wall component $\cW_i$ serves the purpose of allowing $h$ to be nonzero in the corresponding region of the vacuum at the moment of splash.

\subsubsection*{Admissible glancing-permitted interfaces}

In general, we will need to work with a class of curves that look like the $\Gamma$ shown in Figure~\ref{OpenedSplash1}, with ``pinch'' that can be arbitrarily small or even zero.

\begin{definition}\label{PinchDefns}
We define the pinch of a closed curve $\Gamma$ parametrized by $X : S^1 \to \R^2$ and satisfying the bound $\lV X-X_0\rV_{C^1(S^2)}\leq r_0$ to be
\begin{align}
    \delta_\Gamma=\min\{|X(\theta)-X(\vartheta)|:|\theta -\vartheta|\geq 10^{-1}\} .
\end{align}
\end{definition}

Now we define the following, which, when bounded below, ensures a curve satisfies the chord-arc condition away from $X(\pm \pi/2)$, but allows for a single \emph{optional} self-glancing of the curve, i.e. a splash point, at $X(\pm \pi/2)$.
\begin{align}\label{glancingfunction}
\cG(X) = \left(\min_{\theta, \vartheta \in S^1}\frac{|X(\theta)-X(\vartheta)|}{|\theta-\vartheta|}+\frac{|X(\theta)-X(\vartheta)|}{(\theta-\pi/2)^2+(\vartheta+\pi/2)^2} \right) .
\end{align}
We may sometimes refer to curves for which $\cG(X)$ is uniformly bounded below as ``glancing-permitted~curves'', a generalization of chord-arc~curves which allows for splash and near-splash configurations.

\begin{remark}
\hfill
\begin{enumerate}[(i)]
\item The general heuristic behind $\cG(X)$ is that it gives a hybrid measurement between the tendency of a curve to obey the chord-arc condition and the tendency of a curve to sharply ``pull~away'' out of a potential point of chord-arc breakdown near $X(\pm \pi/2)$.

\item A simple computation for the initial interface of Definition~\ref{initialData} yields a positive lower bound on $\cG(X_0)$.
\end{enumerate}
\end{remark}

Let us use the above to specify some classes of glancing-permitted curves we consider.
\begin{definition}\label{HsgpDef}
We define the constant
\begin{align}
C^0_\gp = \cG(X_0).
\end{align}
We then define the family of \emph{admissible glancing-permitted curves} by
\begin{align}
\cXbar &= \{\Gamma \mbox{ parametrized by } X\in C^2(S^1) \mbox{ such that } \cG(X) \geq \half C^0_\gp \}, \label{cXbarDef} \\ \intertext{and for $s\ge 0$, we define}
H^s_\gp  &= H^s(S^1) \cap \{X\in C^2(S^1):\Gamma\in\cXbar, \ \lV X-X_0\rV_{C^1(S^2)}\leq r_0 \}.
\end{align}
\end{definition}

Now we select a fixed integer $k \geq 4$, which remains fixed in the work below.\footnote{Except in the statements of Theorems~\ref{maintheorem} and~\ref{maintheorem2}, which assert the existence of splash--squeeze singularities for interfaces with $H^{k+2}(S^1)$ regularity for any $k\geq 4$} The splash and near-splash curves we consider are parametrized by $X$ belonging to $H^{k+2}_\gp$. Regarding the general shape of the curves considered, by taking a good choice of the constant $r_0$ and working with a subset of $H^{k+2}_\gp$ satisfying
\begin{align}\label{closenesscondition}
\lV X-X_0 \rV_{H^{k+1}(S^1)} \leq r_0 ,
\end{align}
we ensure that the corresponding curves resemble $\Gamma$ as depicted in Figures~\ref{OpenedSplash1} or~\ref{ClosedSplash1}, with similar distance to the walls, curvature, and overall qualitative behavior.


\subsubsection*{Construction of the exterior magnetic field}
We now construct exterior magnetic fields for vacuum domains with admissible glancing-permitted interfaces.

\begin{definition}\label{formalhdefn}
Consider an interface $\Gamma \in \cXbar$ parametrized by $X$ in $H^{k+2}_\gp$, with vacuum region $\cV$ bounded between $\Gamma$ and $\mathcal{W}$. We then define the associated exterior magnetic field $h(\Gamma):\overline \cV \to \R^2$ to be the unique solution $h$ to the system
\begin{align}
(x \in \cV ) \qquad
&   \left\{
        \left.
        \begin{aligned}
            \grad \cdot h &= 0 \ \\
            \grad^\perp \cdot h & = 0 \, ,
        \end{aligned}
        \right.
    \right. \label{ExternalFieldEqDef1} \\
(x \in \Gamma\cup\mathcal{W} )  \qquad & \ \big\{  \, n \cdot h = 0 , \label{ExternalFieldEqDef2} \\
(i=1,2) \qquad & \int_{\cW_i} h \cdot d \vec r = 1 , \label{ExternalFieldEqDef3}
\end{align}
where in the directed line integral in \eqref{ExternalFieldEqDef3}, the path runs left to right along $\cW_1$ for $i=1$, and the path runs counter-clockwise around $\cW_2$ for $i=2$.
\end{definition}

\begin{remark}\label{hdefnremark}
(i) We comment that the conditions \eqref{ExternalFieldEqDef3} involving line integrals along the walls $\cW_1$ and $\cW_2$ are necessary for uniqueness of solutions to the system of Definition~\ref{formalhdefn} since the vacuum is not simply-connected. Note the above definition will yield a solution specifically in the case of $\eta_1=\eta_2=1$ in the boundary condition \eqref{IdealMHD6}, though our splash construction extends to arbitrary constant values of $\eta_1,\eta_2$ as well as to continuously differentiable $\eta_i(t)$. Taking $\eta_1,\eta_2$ to be constant models an experimental setting in which a time-independent line current density is supplied externally along the vacuum walls.

Regarding existence of solutions, one can produce an $h$ solving \eqref{ExternalFieldEqDef1}--\eqref{ExternalFieldEqDef3} by taking
\begin{align}\label{hDefn}
& & h(x) &= \grad^\perp  \uppsi (x)  &  (x  \in \cV ) ,
\end{align}
where $\uppsi$ solves the Dirichlet problem below for a particular choice of $\uplambda_1,\uplambda_2\in\R$,
\begin{align}
&& \Delta  \uppsi &= 0 & (x \in \cV),\\
&& \uppsi &= 0 &  (x \in \Gamma), \\
&&  \uppsi &= \uplambda_i &  (x \in \mathcal W_i, \ i=1,2).
\end{align}
We find the choice of $1$ as the circulation constants in \eqref{ExternalFieldEqDef3} results in $\uplambda_1<0$ and $\uplambda_2<0$. This implies that $h$ vanishes nowhere in the vacuum, as required for a splash--squeeze singularity.
\end{remark}

Let us record the following lemma, in which we discuss a few key properties of vacuum magnetic field constructed for a given interface.
\begin{lemma}\label{hTanDirLemma}
For a given interface $\Gamma$ which is parametrized by $X$ in $H^{k+2}_\gp$, consider $h(\Gamma)$ as in Definition~\ref{formalhdefn}. If $\delta_\Gamma>0$, $h(\Gamma)$ points left to right along $\Gamma$. If $\delta_\Gamma=0$, $h(\Gamma)$ points left to right along $\Gamma$ except at the splash point, which is the unique point at which $h(\Gamma)$ vanishes in $\overline \cV$.
\end{lemma}
\begin{proof}
Consider the $\uppsi$ corresponding to $\Gamma$ as discussed in Remark~\ref{hdefnremark}, noting $\uppsi$ is zero on $\Gamma$ and negative on $\mathcal{W}$. By the maximum principle, $\uppsi$ is negative throughout $\cV$. In particular, at every point in $\Gamma$ satisfying the tangent ball condition on the $\cV$ side of $\Gamma$ (i.e. everywhere on $\Gamma$ besides the single possible point of self-intersection), the Zaremba-Hopf-Oleinik Lemma implies $\grad \uppsi$ is nonzero. Moreover one finds that $\grad \uppsi$ gives a normal to $\Gamma$ pointing to the $\Omega$ side. The properties of $h(\Gamma)$ asserted in the statement of the lemma are easy to verify from these observations.
\end{proof}

\subsubsection*{Role of the squeezed external field in the analysis}

An important insight for our analysis is that the $h$ given above does not cause a repulsive force preventing the two approaching arcs of the nearly-self-intersecting interface from colliding. In other words, the presence of such an $h$ does not produce any physical resistance to the formation of the splash, although it brings forth significant mathematical obstacles.
Introducing any nontrivial external magnetic field gives rise to analytically delicate terms, such as $h$ itself and the vacuum-side Dirichlet-to-Neumann operator $\cN_+$, highly sensitive to the pinch.

In contrast to quantities that live in the plasma domain, neither $h$ nor the exterior pressure $p_+$ (defined in Proposition~\ref{heuristicPropPG}) can be smoothly transferred to a chord-arc domain with a square-root type map. Their behavior is unavoidably tied to the degenerating geometry. This requires us to prove estimates that remain valid across families of vacuum regions with varying degrees of pinch, which is the main topic of Section~\ref{vacuuminterfaceestimatessection}. We prove elliptic estimates with uniform constants for Sobolev spaces with weights amplifying the behavior of functions near the pinch (e.g. with Proposition~\ref{1stWeightedEstimate}). It is essential that $h$ becomes small in the pinching neck as the splash forms. Proposition~\ref{hEstimateProp} makes this precise with a weighted norm bound, which in fact implies that $h$ vanishes to infinite order at the splash point.\footnote{This reflects the necessary incompatibility of a splash--squeeze with analyticity, discussed in Section~\ref{analyticbreakdownsection}.}

Moreover, certain applications of the vacuum-side Dirichlet-to-Neumann operator $\cN_+$, which appears due to $h$, blow up as the pinch closes, in contrast with the plasma-side Dirichlet-to-Neumann operator $\cN_-$. To see this, observe that a harmonic function on $\cV$ equal to $1$ on the arc to one side of the narrowing gap in the vacuum and $0$ on the other must blow up in normal derivative as the two arcs approach each other. This apparent singularity is ultimately reconciled by the fact that whenever $\cN_+$ appears in our estimates, it acts on some quantity dependent on $h$ which, importantly, decays rapidly in the pinch. The related operator $\Nres = \cN_+ + \cN_-$ must likewise be carefully controlled to ensure its cancellation property, that it does not take a derivative overall (see Proposition~\ref{DtoNCancellation}). Using our weighted elliptic framework, we can show that all resulting terms remain uniformly bounded in standard Sobolev spaces. Lemma~\ref{sNilemma} can be thought of as summarizing how the various ``dangerous operators'' are ultimately balanced by the behavior of the magnetic field.

\section{Lagrangian wave system}\label{lagrangianwavesystemsection}

\subsection{Formal derivation}\label{formalderivationsubsection}
In this section, we explain how to derive a coupled system of nonlinear wave equations in Lagrangian coordinates which determines the evolution of the plasma and all the unknowns in the original problem. Instead of jumping straight to the end result, which will be amenable to well-understood solution methods, we first must explain how to decompose the original Lagrangian system into a coupled set of vorticity and interface systems with the ``correct number of derivatives'' in the right-hand side.

\subsubsection*{Decomposition into vorticity and interface systems}

As stated earlier, there is a dangerous level of derivative in the right-hand side in the system \eqref{OriginalLagrangianSystem} arising in the tangential part of $\bP_\grad$ at $\psi=0$. Here we show how to replace the equations for $\bU$ and $\bB$ in the interior with a system for the Lagrangian curl $\bom(t,a)$ and Lagrangian current $\bj(t,a)$ to separate the interior dynamics cleanly from the trickier interface dynamics. We will focus on the interface dynamics afterwards, adding in equations for the traces of $\bU$ and $\bB$ at $\psi=0$, which we denote by $U$ and $B$, respectively.

\subsubsection*{The Lagrangian vorticity evolution}
Let us derive the the vorticity evolution system for the Lagrangian vorticity $\bom$ and and current $\bj$.

Starting from \eqref{IdealMHD1} in vector form and applying the scalar curl, one finds for the Eulerian vorticity evolution we have
\begin{align}
&&&
\begin{aligned}
\del_t \omega + u\cdot\grad \omega &= b \cdot \grad j, \\
\del_t j + u \cdot \grad j &= b \cdot \grad \omega + \grad^\perp \cdot ((\grad u)^t b -  (\grad b)^t u ),
\end{aligned}
&
(x\in\Omega(t)).
\end{align}
where we use the convention that $(\grad u)_{ij} = \del_{x_j} u_i$ and $(\grad b)_{ij} = \del_{x_j} b_i$.

By using properties such as $\bX_\theta = \bB$, one is then able to show that in our magnetic Lagrangian coordinate system, these become equivalent to the following, in which the Lagrangian gradient is simply $\grad = (\del_\theta,\del_\psi)$, and as stated earlier, $\bm{\sigma}(\psi) = |\bB_0(0, \psi)|^2$.
\begin{align}
&&
\begin{aligned}
\bom_t &=\bj_\theta, \\
\bj_t &=\bom_\theta + \bm{\sigma}^{-1}\grad^\perp \cdot ((\grad \bU)^t \bB -  (\grad \bB)^t \bU ),
\end{aligned}
&& (a\in \Sigma).
\end{align}
Making a few additional computations, we find the system takes the simple form below.
\begin{align}\label{vortEvoSystem}
& &
    \col{\bom}{\bj}_t &=
        \begin{pmatrix}   0 & 1 \\
                          1 & 0 \end{pmatrix}   \col{\bom}{\bj}_\theta
        + 2\bm{\sigma}^{-1}\col{0}{\bU_\theta \cdot \bB_\psi - \bU_\psi \cdot \bB_\theta} & ( a \in \Sigma ).
\end{align}
Let us note the above gives a nice, semilinear wave system for the Lagrangian vorticity and current, as the additional terms in the right-hand side involving $\bU$ and $\bB$ are lower order. To close the system, we need to explain how $\bU$ and $\bB$ can be retrieved from $\bom$ and $\bj$, which we discuss next.

\subsubsection*{The Lagrangian div-curl systems}
Given $\bom(t,a)$, $\bj(t,a)$, and the restrictions of $\bU$ and $\bX$ to $\psi = 0$, $U(t,\theta)$ and $X(t,\theta)$, respectively, one may recover the Lagrangian flow $\bX(t,a)$ extended to the interior and the corresponding velocity $\bU(t,a)$ and magnetic field $\bB(t,a)$ by imposing the combined Lagrangian div-curl systems
\begin{align}
\label{OriginalLagrDivCurlSys}
\begin{aligned}
(a \in \Sigma)
          &&  \left\{ \Bsp \\ \Esp \right. & \\
(\psi = 0)
          &&  \Big\{\, & \\
(\psi = -1) 
          &&  \left\{ \Bsp \\ \Esp \right. &
\end{aligned} & 
    \Bsp \grad \cdot (\, \cof (\bM)  \bU )  &= 0                  & \grad \cdot (\, \cof (\bM)  \bB )      &= 0 \\
    \Msp \grad^\perp \cdot (\, \bM \bU )    &= \bm{\sigma} \bom \Lcm   & \grad^\perp \cdot (\, \bM  \bB )  &= \bm{\sigma} \bj \Lcm \\
    \Msp N \cdot \bU                          &= N \cdot U\,,       & N \cdot \bB                              &= 0 \,,\\
    \Msp \bm{N}\cdot \bU                      &= 0                  & \bm{N} \cdot \bB                         &= 0 \\
    \Msp \int^\pi_{-\pi}\bU \cdot \bX_\theta\, d\theta   &= \upalpha \,,       & \int^\pi_{-\pi} \bB \cdot \bX_\theta \, d\theta &= \upbeta \,,
    \Esp
\end{align}
\begin{align*}
& & \bX(t) &= \bX_0 + \int^t_0 \bU(\tau)\, d\tau &    &(a \in \Sigma ) ,
\end{align*}
where \(\bm{N}(t,\theta,-1)=(0,-1)\), $\upalpha$ and $\upbeta$ are the corresponding integrals computed from our initial data,\footnote{The extra conditions involving $\upalpha$ and $\upbeta$ are necessary for uniqueness, since $\Omega(t)$ is not simply connected.} and
\begin{align*}
& & \bM    &= \grad \bX^t  &                  &(a \in \Sigma ) .
\end{align*}
Above, we use the conventions that
\begin{align}
\label{GradXConvention}
    \grad \bX &= \begin{pmatrix}
\del_\theta \bX_1 & \del_\psi \bX_1\\
\del_\theta \bX_2 & \del_\psi \bX_2
\end{pmatrix} ,
\end{align}
and, for a $2\times 2$ matrix $\mathbf{A}$,
\begin{align}
    \cof(\mathbf{A}) = \begin{pmatrix}
\mathbf{A}_{22} & - \mathbf{A}_{21}\\
-\mathbf{A}_{12} & \mathbf{A}_{11}
\end{pmatrix} .
\end{align}
By applying Proposition~\ref{SolnMainDivCurlProp}, one is able to recover $\bU$, $\bB$, and $\bX$ as the solution to the system \eqref{OriginalLagrDivCurlSys} from $\bom$, $\bj$, and the boundary data.
\begin{remark}
We note that because Proposition~\ref{SolnMainDivCurlProp} is catered towards iterates solving a linearized MHD system, rather than an exact solution to ideal MHD, the div-curl system it solves directly, \eqref{AugmentedInteriorSystemL}, includes some additional error terms. One can verify that \eqref{AugmentedInteriorSystemL} reduces to \eqref{OriginalLagrDivCurlSys} for an exact solution to the Eulerian system \eqref{IdealMHD1}--\eqref{IdealMHD6}.
\end{remark}

\subsubsection{Refining the surface system by introducing the $\bE$ operator}\label{refiningsurfsys}
Now we explain how the original Lagrangian system \eqref{OriginalLagrangianSystem} evaluated at the boundary features a dangerous level of derivative in the right-hand side. This is ultimately resolved with the use of an invertible operator we denote by $\bE$. This linear operator acts on restrictions of vector fields to $\Gamma$ (viewed in the Lagrangian perspective), killing off the tangential component \emph{specifically} of the restriction to $\Gamma$ of a divergence-free, curl-free vector field, and leaving the normal component of a \emph{general} vector field unchanged. We will explain this in more detail momentarily.

Returning to the Lagrangian system, we obtain a \emph{surface system} by evaluating \eqref{OriginalLagrangianSystem} at $\psi=0$, with
\begin{align}
\label{UnmodSurfaceSystem}
& &
\begin{aligned}
    U_t &= B_\theta - P_\grad ,\\
    B_t &= U_\theta,
\end{aligned}
&   & ( \theta \in S^1),
\end{align}
where $P_\grad = (\grad p)\circ X$. Observe that if it were not for this $P_\grad$ term, we would have a trivial one-dimensional wave system for $(U,B)$. Let us consider the tangential component. Corresponding to the Eulerian tangent vector field $\tau$ pointing left-to-right along $\Gamma$, we have the Lagrangian tangent,
\begin{align}
\bigtau(t,\theta) = \tau(t,X(t,\theta)) = \frac{\del_\theta X(t,\theta)}{|\del_\theta X(t,\theta)|} .
\end{align}
The boundary condition $p=\half |h|^2$ at the interface implies
\begin{align}\label{problemtanglterm}
\bigtau \cdot P_\grad = \half(\tau \cdot \grad |h|^2)\circ X = H \cdot (\del_\tau h)\circ X,
\end{align}
where $H$ is at the regularity level of $U$ and $B$, so that, in addition to the term $B_\theta$ in the right-hand side of \eqref{UnmodSurfaceSystem}, we have another top order term arising from $P_\grad$. This obfuscates the wave equation structure we expect to see and causes problems for energy estimates.

 As it turns out, the above problem only arises in the tangential component of the equation for $U$. We claim the normal component, however, has a desirable structure. Let us take for granted the following identity, which we derive soon in Section \ref{goodformNPgradsection}:
\begin{align}\label{NPgradlot}
N\cdot P_\grad = -\frac{|H|^2}{|B|^2}N\cdot B_\theta + \mathit{l.o.t.}
\end{align}
Discarding the tangential component from the first equation in \eqref{UnmodSurfaceSystem} then results in
\begin{align*}
N\cdot U_t &= \left( 1+\frac{|H|^2}{|B|^2} \right) N \cdot B_\theta + \mathit{l.o.t.}\, ,\\
B_t &= U_\theta.
\end{align*}
If the tangential component of the first equation in \eqref{UnmodSurfaceSystem} had a similar structure to that in the normal component, we would thus get a system similar to the wave system \eqref{vortEvoSystem} for $\bom$ and $\bj$. In the present case, the tangential component behaves differently. While it may seem natural to drop it from the system, doing so appears to sacrifice control of tangential regularity in the energy estimates.\footnote{We believe this is related to difficulties encountered in past studies in developing a Lagrangian formulation of free-boundary ideal MHD which does not lose regularity at the iteration step, noted in \cite{GuMHDtension,XieLuoALE}, for instance.} Instead, by employing an appropriate projection operator, $\bE$, we get rid of the problematic term \eqref{problemtanglterm} while retaining critical information from the tangential component of the equation. Moreover, we will see this projection operator eventually yields a genuine wave system for our full set of good unknowns in Lagrangian coordinates.

The key insight for resolving the issue is the following: tangential components of harmonic vector fields in free-boundary fluid mechanics problems sometimes generate dangerous terms, but these can be sidestepped with the use of relations which hold among the tangential and normal components of such fields. For a vector field that is both divergence and curl free, the tangential part is determined by the normal part and shares its regularity. This is the motivation for our projection operator $\bE$, to be defined shortly.

The definition of $\bE$ involves harmonic extension off of $\Gamma$. Let us first introduce the harmonic extension operators $\bigh_\pm$ associated to a given interface.
\begin{definition}\label{harmonicextdefn}
Fix $\Gamma$ in $\cXbar$. We denote by $\Omega_-$ the unbounded region below $\Gamma$ and by $\Omega_+=\Omega_-^c \setminus \Gamma $ the region above $\Gamma$. Let us define $\bigh_+$ and $\bigh_-$ to be the harmonic extension operators off of $\Gamma$, each correspondingly acting on a function $f$ in $L^2(\Gamma)$ by
\begin{align}
\label{bighDefn1}
\bigh_\pm  : f \mapsto \phi_\pm ,
\end{align}
where $\phi_\pm$ is the unique solution which remains bounded as $x_2$ tends to $\pm \infty$ to the Dirichlet problem 
\begin{align}
& & \Delta \phi_\pm &= 0  &  &(x \in \Omega_\pm ),  \notag \\
& & \phi_\pm &= f         &  &(x \in \Gamma ). \label{bighDefn2}
\end{align}
\end{definition}

Let us also define the projection to mean-zero functions on $\Gamma$ by
\begin{align}
&&    \dcal{p} f &= f -  \frac{1}{|\Gamma|} \int_\Gamma f \, dS.
    &(f: \Gamma \to \R).
\end{align}
Using $\dcal{p}$, we are able to define the operator $\mathfrak H$, representing the Hilbert transform associated to the curve~$\Gamma$. Below, $\bign_-$ is the 
Dirichlet-to-Neumann operator into the unbounded region $\Omega_-$ below $\Gamma$ (see Definition~\ref{DirichToNeumDefs}). We define
\begin{align*}
    \mathfrak H(f) &= -\del_\tau g \qquad \mbox{for the unique } g \mbox{ on } \Gamma \mbox{ such that} \qquad \bign_- g = \dcal{p} f .
\end{align*}
Given this, we then define the Lagrangian version $\cH$ of the operator $\mathfrak H$ by the following:
\begin{align*}
    \cH F (\theta) &= (\mathfrak{H} f)(X (\theta) ) \qquad (F: S^1 \to \R, \  f = F \circ X^{-1} ) .
\end{align*}
We remark that the operators $\mathfrak H$ and $\cH$ preserve the Sobolev regularity of the fields to which they are applied. Due to the definition of $\cH$, for a harmonic vector field $v$, denoting $V = v\circ X$, we have
\begin{align*}
 & &   \bigtau \cdot V + \cH(N \cdot V) &= 0   & (\theta \in S^1).
\end{align*}

The operator $V\mapsto \bigtau \cdot V + \cH ( N\cdot V)$, which we may represent by writing $\bigtau + \cH \circ N$, gives the tangential projection trace of the orthogonal complement to the harmonic part of a vector field. 

 Let us apply $\bigtau + \cH \circ N$ to both sides of the evolution equation for $U$ in \eqref{UnmodSurfaceSystem}. For
 \begin{align*}
&& \bP^-_\grad &= \bP_\grad  - \half \left(\grad \bigh_- |h|^2\right) \circ \bX & (a\in\Sigma),
\end{align*}
and $P^-_\grad$ the trace at $\psi=0$ of $\bP^-_\grad$, we find
\begin{align*}
    (\bigtau + \cH \circ N) U_t   = (\bigtau + \cH \circ N)(B_\theta - P_\grad) = (\bigtau + \cH \circ N)(B_\theta - P^-_\grad).
\end{align*}
The term $P^-_\grad$ is lower order, at the level of $U$ and $B$,\footnote{This can be verified from the system \eqref{InteriorPressureGradSystem} solved by $\bP^-_\grad$.} unlike $P_\grad$, which behaves like a derivative of these quantities. The above identity provides a useful equation we may use to replace the tangential part of the first equation in \eqref{UnmodSurfaceSystem}. We thus arrive at a natural surface system for us to solve for $U$ and $B$:
\begin{align}
\begin{aligned}
    (\bigtau + \cH \circ N) U_t &= (\bigtau + \cH  \circ N)B_\theta +\mathit{l.o.t.}\, , \\
    N \cdot U_t &= \left( 1 + \frac{|H|^2}{|B|^2} \right) N \cdot B_\theta +\mathit{l.o.t.} \, ,\\
    B_t &= U_\theta . 
    \end{aligned}\label{naturalsurfsystem1}
\end{align}
The lower order terms are given explicitly in \eqref{ModSurfaceSystem2}. To write this more compactly, we introduce $\bE$ and $\bD$, defined by the following.\footnote{See also Proposition~\ref{bEDefns} for the rigorous definition of $\bE$.}
\begin{align}\label{bEopintro}
\bE V   = \begin{pmatrix} 1 & \cH \\ 0 & 1 \end{pmatrix} \col{\bigtau \cdot V}{N \cdot V} = \begin{pmatrix} 1 & \cH \\ 0 & 1 \end{pmatrix} \left[\col{\bigtau}{N} V \right]  \qquad (  \theta \in S^1 , \ V:S^1 \to \R^2  ) ,
\end{align}
\begin{align}
& &    \bD &= \begin{pmatrix} 1 & 0 \\ 0 & 1 + \frac{|H|^2}{|B|^2} \end{pmatrix} & (\theta &\in S^1). \label{bDfirstRef}
\end{align}

Using this, we thus find that we can ``project away''\footnote{To the system \eqref{UnmodSurfaceSystem} as a whole, one applies the invertible operator matrix \(\begin{pmatrix}
\bE & 0 \\ 0 & \bI \end{pmatrix} \), losing no information.} 
 the bad tangential part of the original surface system \eqref{UnmodSurfaceSystem} to produce the following \emph{basic surface system}:
\begin{align}
\label{ModSurfaceSystem0}
& &
\begin{aligned}
    \bE U_t &= \bD \bE B_\theta + \mathit{l.o.t.}\,, \\
    B_t &= U_\theta,
\end{aligned}
&   & ( \theta \in S^1).
\end{align}

Combining the basic surface system \eqref{ModSurfaceSystem0} with our interior vorticity systems yields a closed system involving only good levels of derivatives on the surface unknowns $X$, $U$, and $B$ and on the interior unknowns $\bom$ and $\bj$. While the levels of derivatives are as they should be, the system is still not quite in ``wave form'', which we address in Section \ref{basicsystemsection}. Before doing so, we now derive expressions for the lower order terms appearing in \eqref{NPgradlot}, and thus those in \eqref{naturalsurfsystem1} and \eqref{ModSurfaceSystem0}.

\subsubsection{Derivation of $N\cdot P_\grad$ and lower order terms in the basic surface system}\label{goodformNPgradsection}
Now that we have motivated the definition of the operator $\bE$, let us return to the claim that $N\cdot P_\grad$ can be given by an expression of the form \eqref{NPgradlot}. The explicit identity for $N\cdot P_\grad$ with lower order terms included is given by \eqref{HBIdent0} in the following proposition.
\begin{proposition}\label{heuristicPropPG}
    Consider the described initial data, and an evolving surface $\Gamma(t)$ traced out by $X(t, \theta)$ with corresponding $h=h(\Gamma(t))$ as in Definition~\ref{formalhdefn} and interior magnetic field $b$ pointing left to right along $\Gamma(t)$. Define $p_+(t,x)$ and $P^+_\grad(t,\theta)$ at each time $t$ by
\begin{align}
& &    p_+ (t,x) &= \half |h(t,x)|^2 - \half \bigh_+ \left( |h|^2 \right)(t,x) &     (x &\in \cV ), \label{Pplusdefn1}\\
& &    P^+_\grad (t,\theta) &= (\grad p_+) (t,X(t,\theta ))  &   (\theta & \in S^1 ), \label{Pplusdefn2}
\end{align}
and define $P^-_\grad(t,\theta) = \bP^-_\grad(t,\theta,0)$, for $\bP^-_\grad$ as defined in \eqref{InteriorPressureGradSystem}.

Suppose that $p$ satisfies the original pressure system \eqref{PressureSystem} and both $\bU = \bX_t$ and $\bB = \bX_\theta$ hold. Consider the exterior and interior Dirichlet-to-Neumann operators $\cN_\pm$ (see Definition~\ref{DirichToNeumDefs}). Then for $\Nres = \cN_+ + \cN_-$, we have
\begin{align}\label{HBIdent0}
    N \cdot P_\grad &= -\frac{|H|^2}{|B|^2} N \cdot B_\theta + N \cdot P^-_\grad + N \cdot P^+_\grad + \half \Nres |H|^2.
\end{align}
\end{proposition}
\begin{proof}
From the system for $\bP^+_\grad$,
\begin{align}\label{npg1}
N\cdot P^+_\grad =\half N\cdot \left(\grad|h|^2\right)\circ X - \half \cN_+ |H|^2.
\end{align}
From the systems for $\bP_\grad$ and $\bP^-_\grad$ we find
\begin{align}\label{npg2}
N\cdot(P_\grad- P^-_\grad) = \half \cN_- |H|^2.
\end{align}
Now we verify
\begin{align}\label{HBIdent}
    \frac{|H|^2}{|B|^2} N \cdot B_\theta = \frac{1}{2} N \cdot \left(\grad |h|^2\right) \circ X .
\end{align}
To see this, note $b$ and $h$ point point left to right along $\Gamma$, except possibly at one point where $h$ may vanish, by Lemma~\ref{hTanDirLemma}. Using the fact that $\grad^\perp \cdot h=0$ together with this, we find
\begin{align}
\half \del_n |h|^2 = \sum^2_{i,j=1} n_i \del_i h_j h_j = n_i \del_j h_i h_j = n \cdot (h \cdot \grad h) = \frac{|h|}{|b|} n \cdot (b \cdot \grad h).
\end{align}
Thus, using that $B$ is tangent to the interface in the last step,
\begin{align}
\half N \cdot (\grad|h|^2)\circ X = \frac{|H|}{|B|}N \cdot H_\theta = \frac{|H|}{|B|} N \cdot \left(\frac{|H|}{|B|} B\right)_\theta =\frac{|H|^2}{|B|^2} N\cdot B_\theta ,
\end{align}
as claimed. Combining \eqref{npg1}, \eqref{npg2}, and \eqref{HBIdent}, one arrives at \eqref{HBIdent0}.
\end{proof}

 In light of Proposition \ref{heuristicPropPG}, we use the identity \eqref{HBIdent0} to replace the tangential component of the equation \eqref{UnmodSurfaceSystem} for $U$. We thus arrive at \eqref{naturalsurfsystem1}, with lower order terms made explicit as follows.
\begin{align}
\begin{aligned}
    (\bigtau + \cH \circ N) U_t &= (\bigtau + \cH  \circ N)B_\theta - (\bigtau + \cH  \circ N)P^-_\grad , \\
    N \cdot U_t &= \left( 1 + \frac{|H|^2}{|B|^2} \right) N \cdot B_\theta - N \cdot P^-_\grad - N \cdot P^+_\grad - \half \Nres |H|^2 , \\
    B_t &= U_\theta . 
    \end{aligned}\label{ModSurfaceSystem2}
\end{align}

Regarding the level of the term $\Nres|H|^2$ in comparison with the unknowns $U$ and $B$, for now let us take for granted an important concept to be verified later with Proposition~\ref{DtoNCancellation}, which is that the operator $\Nres=\cN_++\cN_-$ is an order zero operator, despite the fact that $\cN_+$ and $\cN_-$ differentiate once. Combined with the facts that $H$ is at the level of $U$ and $B$, as are $P^-_\grad$ (which can be verified from \eqref{InteriorPressureGradSystem}) and $P^+_\grad$ (similarly verified), we find that the lower order terms in \eqref{naturalsurfsystem1} are indeed lower order. Using \eqref{ModSurfaceSystem2}, we may derive the complete form of the system \eqref{ModSurfaceSystem0} with lower order terms made explicit. We thus find
\begin{align}
\label{ModSurfaceSystem}
& &
\begin{aligned}
    \bE U_t &= \bD \bE B_\theta + \cF, \\
    B_t &= U_\theta,
\end{aligned}
&   & ( \theta \in S^1),
\end{align}
where we define $\cF$ by
\begin{align}
&& \cF  &= - \bE P^-_\grad - \col{0}{N \cdot P^+_\grad + \half \Nres |H|^2}   & (\theta &\in S^1). \label{cFdefn2}
\end{align}
\subsubsection{The non-symmetrizable basic system}\label{basicsystemsection}

Using the projection operator $\bE$ and combining the evolution system for $U$ and $B$ given by \eqref{ModSurfaceSystem} with that for $\bom$ and $\bj$ given by \eqref{vortEvoSystem}, we get the system below.
\begin{align}
(\theta \in S^1) \qquad
& \left\{
        \begin{aligned}
        U_t &= \bE^{-1}\bD \bE B_\theta + \bE^{-1} \cF \\
        B_t &=  U_\theta \\ 
        X_t &= U \, ,
        \end{aligned}
        \right. 
        \label{BasicSurfSys1} \\[1ex]
( a \in \Sigma ) \qquad
&   \left\{
        \begin{aligned}
        \bom_t &= \bj_\theta \\
        \bj_t  &= \bom_\theta + \dcF^2_2 \, ,
        \end{aligned}
        \right. 
        \label{BasicSysDefns} 
\end{align}
where
\begin{align}
& & \dcF^2_2 &= 2\bm{\sigma}^{-1} (\bU_\theta \cdot \bB_\psi - \bU_\psi \cdot \bB_\theta) & & (a\in \Sigma).\label{earlyF22defn}
\end{align}

\begin{remark}
Let us note from \eqref{bEopintro} that $\bE^{-1}$ is given by a relatively simple matrix operator via
\begin{align}\label{bEinverse}
& & \bE^{-1} V &= {\col{\bigtau}{N}}^{-1}\begin{pmatrix} 1 & -\cH \\ 0 & 1 \end{pmatrix} V & (V : S^1 \to \R^2 ).
\end{align}
\end{remark}
Thanks to our use of $\bE$ to get rid of the dangerous tangential derivative of the pressure in the original surface system \eqref{UnmodSurfaceSystem}, the remaining source term $\bE^{-1}\cF$ in the reformulated version \eqref{BasicSurfSys1} is lower order, so that the principal term in the first equation is $\bE^{-1} \bD \bE B_\theta$.


However, $\bE^{-1}\bD \bE$ \emph{is not symmetric}, and one can show that the above system is not symmetrizable (in the Lax-Friedrichs sense). This means that \emph{we do not yet have the structure of a wave system}, which is desirable for good energy estimates. Moreover, this means the basic system \eqref{BasicSurfSys1}--\eqref{BasicSysDefns} does not fall under the more general category of systems of conservation laws, and the standard, well-known techniques used to prove local existence for such systems are not directly applicable. Next we explain how to exchange our surface unknowns $(U,B)$ for the good unknowns $(\dot U^*,\dot B^*)$ to get a true system of wave equations.

\subsubsection{The good Lagrangian wave system}\label{thewavesystemsubsection}

The Lagrangian wave equation structure at the surface is only revealed when one takes one more time derivative of the equations for $U$ and $B$ in the basic non-symmetrizable system. For $\bE$ defined by \eqref{bEopintro}, we introduce the good surface unknowns $(\dot U^*,\dot B^*)$ by defining
\begin{align}
    \dot{ U }^*(t,\theta) &= \bE \,\del_t U(t,\theta) , \\
    \dot{ B }^*(t,\theta) &= \bE \,\del_t B(t,\theta) .
\end{align}
Using these together with our interior vorticity evolution, we now derive our good Lagrangian wave system, which takes the form
\begin{align}
(\theta \in S^1) \qquad
& \left\{
        \begin{aligned}
        \dot U^*_t &= \bD \dot B^*_\theta + \mathit{l.o.t.} \\
       \dot B^*_t &= \dot U^*_\theta + \mathit{l.o.t.}\,,
        \end{aligned}
        \right.  \\[1ex]
( a \in \Sigma ) \qquad
&   \left\{
        \begin{aligned}
        \bom_t &= \bj_\theta \\
        \bj_t  &= \bom_\theta + \mathit{l.o.t.}\, .
        \end{aligned}
        \right.
\end{align}
The system is given in full detail below in \eqref{symmetrizable2}--\eqref{symmetrizable1}.

Many lower order terms appearing in the system depend on $X$, for example, through the shape of the interface. With the goal of transitioning to the more rigorous formulation of the system as a whole in Section \ref{rigorousfmltnsection}, let us take on the task of writing the various quantities to appear in \eqref{BasicSurfSys1}--\eqref{BasicSysDefns} as maps acting on the complete \emph{state vector} of unknowns $(\dot{ U }^*, \dot{ B }^*, \bom, \bj , X, U)$. To account for the evolution of the unknowns $X$ and $U$, we take
\begin{align*}
(\theta \in S^1) \qquad
& \left\{
        \begin{aligned}
        X_t &= U \\
        U_t &= \bE^{-1} \dot{ U }^*\, ,
        \end{aligned}
        \right.
\end{align*}
This is more of an auxiliary part of the system to be formed, with right-hand side not involving derivatives of the state vector.

Now let us derive the evolution equations for $\dot{ U }^*$ and $\dot{ B }^*$. Applying $\bE$ to \eqref{ModSurfaceSystem}, we have
\begin{align}\label{dotUstarfmla}
\dot U^* = \bD \bE B_\theta + \cF.
\end{align}
Differentiating this equation in time and manipulating commutators, we get
\begin{align}
\dot U^*_t &= \bD\bE B_{\theta t} +[\del_t,\bD\bE]B_\theta + \cF_t \\
&= \bD\,\big( \dot B^*_\theta-[\del_\theta,\bE]B_t\big)+
[\del_t,\bD\bE]B_\theta + \cF_t .
\end{align}
Now by substituting in expressions for $B_t$ and $B_\theta$ by using $B_t = \bE^{-1} \dot B^*$ and \eqref{dotUstarfmla}, we arrive at
\begin{align}
\dot U^*_t &= \bD \dot B^*_\theta +[\del_t,\bD\bE](\bD\bE)^{-1}\dot U^* - \bD[\del_\theta,\bE]\bE^{-1}\dot B^* + \cF_t  - [\del_t,\bD\bE](\bD\bE)^{-1}\cF.
\end{align}
Similarly, one can use $B_t=U_\theta$ to derive an equation for $\dot B^*_t$, which we provide below. Let us write
\begin{align}\label{bRderivation2}
\bR^{11} &= \bD_t \bD^{-1} + \bD [\del_t,\bE] \bE^{-1} \bD^{-1} ,\\
\bR^{12} &= -\bD[\del_\theta, \bE] \bE^{-1} ,\\
\bR^{21} &= -[\del_\theta, \bE] \bE^{-1} ,\\
\bR^{22} &= [\del_t, \bE] \bE^{-1} .
\end{align}
We note that the expression for $\bR^{11}$ comes from the identity \[[\del_t, \bD \bE] (\bD \bE)^{-1}=\bD_t \bD^{-1} + \bD [\del_t,\bE] \bE^{-1} \bD^{-1} .\]
Writing $\mathring \cF = \cF_t - \bR^{11} \cF$, the evolution equations for $\dot U^*$ and $\dot B^*$ become
\begin{align}
    \dot{ U }^*_t &= \bD\, \dot{ B }^*_\theta + \bR^{11} \dot{ U }^* + \bR^{12} \dot{ B }^* + \mathring{\cF} , \label{SymizableSurfSys1}\\
    \dot{ B }^*_t &= \dot{ U }^*_\theta + \bR^{21} \dot{ U }^* + \bR^{22} \dot{ B }^* . \label{SymizableSurfSys2}
\end{align}
Due to the special form of the matrix $\bD$, a two-by-two off-diagonal matrix, the above gives a wave system for $\dot U^*$ and $\dot B^*$.

Let us describe the above in terms of $4\times 4$ matrices and matrix operators. We define
\begin{align}
    \bJ_1 = \begin{pmatrix} 0 & \bD \\ \bI_{(2 \times 2)} & 0 \end{pmatrix} ,\qquad \bR = \begin{pmatrix} \bR^{11} & \bR^{12} \\ \bR^{21} & \bR^{22} \end{pmatrix} , \qquad \dcF_1 = \begin{pmatrix}
        \mathring \cF \\
        0_{(2\times 1)} 
            \end{pmatrix}.
\end{align}
Then the system \eqref{SymizableSurfSys1}--\eqref{SymizableSurfSys2} becomes
\begin{align*}
& &
    \col{\dot{ U }^*}{\dot{ B }^*}_t &= \bJ_1 \col{\dot{ U }^*}{\dot{ B }^*}_\theta + \bR \col{\dot{ U }^*}{\dot{ B }^*} + \dcF_1 .
    & 
\end{align*}
Recall the interior system for $\bom$ and $\bj$, \eqref{BasicSysDefns}. Let us write the following
\begin{align}
    \bJ_2 &= \begin{pmatrix} 0 & 1 \\ 1 & 0 \end{pmatrix} , \qquad 
    \dcF_2(X_\theta,U,\bom, \bj) = 2\bm{\sigma}^{-1} \col{0}{\bU_\theta \cdot \bB_\psi - \bU_\psi \cdot \bB_\theta} \label{SurfEqnTerm2} .
\end{align}
Note $\dcF_2$ as written above depends on $X_\theta$ in particular due to the presence of $N=X^\perp_\theta/|X_\theta|$ in \eqref{OriginalLagrDivCurlSys}. Collecting our equations together gives the complete system for $\dot U^*$, $\dot B^*$, $\bom$, $\bj$, $X$, and $U$:
\begin{align}
& &    \col{\dot{ U }^*}{\dot{ B }^*}_t &= \bJ_1 \col{\dot{ U }^*}{\dot{ B }^*}_\theta + \bR \col{\dot{ U }^*}{\dot{ B }^*} + \dcF_1   &   (\theta &\in S^1 ),
        \label{symmetrizable2} \\[1ex]
& &    \col{\bom}{\bj}_t &= \bJ_2 \col{\bom}{\bj}_\theta + \dcF_2  &   (a &\in \Sigma ),
        \label{symmetrizable3} \\[1ex]
& &    \col{X}{U}_t &= \begin{pmatrix} \bI & 0 \\
                                  0 & \bE^{-1}
                              \end{pmatrix}
    \col{U}{\dot{ U }^*}    &   (\theta &\in S^1).
        \label{symmetrizable1}
\end{align}

Due to the off-diagonal forms of $\bJ_1$ and $\bJ_2$, the result is a coupled system of quasilinear wave equations, essentially formed by the pair \eqref{symmetrizable2} and \eqref{symmetrizable3}. The auxiliary, lower-order part \eqref{symmetrizable1} is included to close the system, since certain lower order terms depend on $X$ and $U$. The apparent wave structure means the system obeys good energy estimates and can be solved by traditional means, provided the terms $\bJ_1$, $\bJ_2$, $\bR$, $\dcF_1$, and $\dcF_2$ are all in fact lower order. While these are all indeed lower order, it is not obvious at the moment why this is so. In our derivation of the Lagrangian wave system \eqref{symmetrizable2}--\eqref{symmetrizable1}, we have glossed over the details of many dependencies hidden in these terms. Below, we provide one possible description of these quantities in terms of objects we have previously introduced, without going into explicit detail.
\begin{align}
\begin{aligned}
\bJ_1 &= \bJ_1(X_\theta,H) ,\\
\bR &= \bR(\bE,[\del_t,\bE],[\del_\theta,\bE],X_\theta,U_\theta, H, H_t) ,\\
\dcF_1 &=  \dcF_1(\bE,\Nres,[\del_t,\bE],[\del_t,\Nres],X_\theta,U_\theta,H,H_t,P^\pm_\grad,\del_t P^\pm_\grad) ,\\
\dcF_2 &= \dcF_2(X_\theta,U,\bom,\bj) ,
\end{aligned}
\label{dependencies}
\end{align}
where
\begin{itemize}
\item $H$, $P^+_\grad$, $\bE$, and $\Nres$ depend on $X$ (through the shape of $\Gamma$),
\item $P^-_\grad$ depends on $\bU$, $\bB$, and $X$ (through the shape of $\Gamma$),
\item $\bU$ and $\bB$ depend on $\bom$, $\bj$, $U$, and $X$ (through the shape of $\Gamma$).
\end{itemize}

To make things less notationally burdensome but still keep track of the dependencies, we collect the various unknowns into a single state vector, which we call $\xi$, taking values in $\R^{10}$. We define it (typically as a column vector) by
\begin{align}
\xi = (\dot U^*, \dot B^*, \bom, \bj, X,U).
\end{align}
We then find the Lagrangian wave equation system \eqref{symmetrizable2}--\eqref{symmetrizable1} can be written in the form
\begin{align}
\del_t \xi = \bJ(\xi) \, \del_\theta \xi + \cR(\xi) \, \xi + \dcF(\xi)  ,
\label{XiSystem}
\end{align}
where
\begin{align}
\dcF = (\dcF_1,\dcF_2,0,0,0,0),
\end{align}
and we have the $10\times 10$ matrices $\cR(\xi)$ given by \eqref{bRderivation} and $\bJ(\xi)$ by
\begin{align} \label{bJdefn1} 
    \bJ = \begin{pmatrix} \bJ_1 & 0     & 0 \\
                            0    & \bJ_2 & 0 \\
                            0    &  0    & 0_{(4 \times 4)} \end{pmatrix} .
\end{align}
 It should be apparent that the terms $\bJ$, $\cR$, and $\dcF$, though lower order, depend on $\xi$ in complicated ways. We have yet to rigorously define many of the terms appearing in \eqref{dependencies} as maps acting on $\xi$, taking values in certain spaces (sometimes operator spaces). This task mostly boils down to carefully defining objects such as the following as maps acting on $\xi$.
\begin{align}\label{genximaps}
\begin{aligned}
H &=  H(\xi), & H_t &= \dot H(\xi),\\
P^\pm_\grad &=  P^\pm_\grad(\xi), & \del_t P^\pm_\grad &= \dot{P}^\pm_\grad(\xi), \\
\bU &=   \bU(\xi), & \bB &=   \bB(\xi), \\
\bE &= \bE(\xi), & \Nres &= \Nres(\xi), \\
[\del_t,\bE] &= [\del_t;\bE](\xi), & [\del_t,\Nres] &= [\del_t;\Nres](\xi), \quad \mbox{etc.}
 \end{aligned}
\end{align}
 The overhead dot notation used, for example, in $\dot H(\xi)$, emphasizes that we express the object in terms of $\xi$, \emph{without actually taking a time derivative}, in such a way that it coincides with $\del_t(H(\xi))$. The notation $[\del_t ; \bE](\xi)$, $[\del_t ; \Nres](\xi)$, etc. is used to play the analogous role for commutators of operators with $\del_t$.
 
 Rigorous definitions of some of the maps in \eqref{genximaps} become especially subtle when we consider a nearly self-intersecting or self-intersecting interface $\Gamma$. For example, the exterior Dirichlet-to-Neumann operator $\cN_+(\xi)$ (notably tangled up in $[\del_t;\Nres](\xi)$, which is defined in Definition~\ref{DtNCommDefns}) becomes unbounded in standard operator spaces as the pinch tends to zero, as discussed near the end of Section~\ref{squeezedmagfieldssection}. Similarly, the rigorous definition of $[\del_t;\bE](\xi)$ (given in Proposition~\ref{bEcommBounds}) must be handled carefully. The remainder of this section focuses on precisely asserting that all the maps constituting the terms $\bJ(\xi)$, $\cR(\xi)$, and $\dcF(\xi)$ of \eqref{XiSystem} are well-defined, allowing us to put the Lagrangian wave system \eqref{symmetrizable2}--\eqref{symmetrizable1} and its compact form \eqref{XiSystem} on firm footing.

\subsection{Rigorous formulation}\label{rigorousfmltnsection}
To form precise definitions of the terms $\bJ(\xi)$, $\cR(\xi)$, and $\dcF(\xi)$ as maps acting on $\xi$, we first define appropriate classes for $\xi$ to serve as domains.

We begin the main definitions for quantities that live in the plasma region (e.g. $\bU$, $\bB$, $\bX$) in Section~\ref{plasmacentricmapssection} by constructing solutions to div-curl type systems, such as \eqref{OriginalLagrDivCurlSys}. Proposition~\ref{SolnMainDivCurlProp} in particular defines $\bU(\xi)$, $\bB(\xi)$, and $\bX(\xi)$.

In Section~\ref{vacandopmapssection}, we carefully define the essential vacuum quantities, such as $H(\xi)$ and the exterior pressure gradient $P^+_\grad(\xi)$, as maps acting on $\xi$. Then, we focus on operator valued maps, such as $\cN_\pm(\xi)$, $\Nres(\xi)$, and $\bE(\xi)$. 

Since the definitions of some of the ``time derivative maps'', $\dot H(\xi)$ and $[\del_t ;\bE](\xi)$ for example, are more technical than conceptually important, their comprehensive treatment is left till Section~\ref{timederivdefnsection}.
\subsubsection*{State spaces and classes for the state vector $\xi$}
Let us introduce a natural function space in which our state vector $\xi(t)$ will reside, for each $t$ in $[0,T]$. We define the $\sH^s$ norm and the corresponding space by
\begin{align}
    \lV \xi \rV_{\sH^s} = \lV \dot{ U }^* \rV_{H^s(S^1)} + \lV \dot{ B }^* \rV_{H^s(S^1)}+ \lV \bom \rV_{H^{s+1/2}(\Sigma)} + \lV \bj \rV_{H^{s+1/2}(\Sigma)} + \lV X \rV_{H^{s+2}(S^1)} + \lV U \rV_{H^{s+1}(S^1)} ,
\end{align}
\begin{align}
    \sH^s &= (H^s(S^1))^2 \times (H^{s+1/2}(\Sigma))^2 \times H^{s+2}(S^1) \times H^{s+1}(S^1) , \\ \intertext{along with the space}
    \sH^s_\dagger &= (H^s(S^1))^2 \times (H^{s+1/2}(\Sigma))^2,
\end{align}
for the truncated state vector $\xi_\dagger$, defined by
\begin{align}
\xi=(\xi_\dagger,X,U),\qquad \mbox{i.e.} \qquad \xi_\dagger=(\dot U^*, \dot B^*, \bom,\bj),
\end{align}
meant for situations in which we need not pay attention to $(X,U)$, the ``lower order'' components of $\xi$.

Recall from \eqref{HsgpDef} the definition of the family $H^s_\gp$ for parametrizations of glancing-permitted curves, or splash and near-splash curves, we would like to consider. We will often need to consider compatible families of $\xi$. For example, we define
\begin{align}\label{sHsgp}
    \sH^s_\gp = \{\xi\in\sH^s: X \in H^{s+2}_\gp\} .
\end{align}

In order to define another important class for our state vector, we must first properly define the initial splash state $\xi_0$.
\begin{definition}\label{initialxiData}
Recalling the terms given in Definition~\ref{initialData}, let us define
\begin{align}
&& \bom_0(a) &= (\grad^\perp \cdot u_0)(\bX_0(a)) & (a\in \Sigma), \\
&& \bj_0(a) &= (\grad^\perp \cdot b_0)(\bX_0(a)) & (a\in\Sigma),
\end{align}
and define $\bE_0$ to be the operator associated to the interface $\Gamma_0$ analogous to the $\bE(\xi)$ operator determined by \eqref{projecdefn}--\eqref{bEdefnFmla} of Proposition~\ref{bEDefns}. Correspondingly, we define initial $\cN_\pm$, $\Nres$ operators associated to $\Gamma_0$ as in Definitions~\ref{DirichToNeumDefs} and~\ref{Nresdefn}. We also take the corresponding initial external magnetic field trace $H_0$ associated to $\Gamma_0$ (see Definition~\ref{prelimHmapDefn}). Together with the obvious initial versions of $P^+_\grad$ (see Definition~\ref{ExtPressureDefnBd}) and $P^-_\grad$ associated to $\bU_0$ and $\bB_0$ (see \eqref{InteriorPressureGradSystem}), we define the analogous term $\cF_0$ as in \eqref{cFdefn2}. With these, we define the following functions of $\theta$ in $S^1$:
\begin{align}
\bD_0 &=\begin{pmatrix}1 & 0 \\ 0 & 1 + \frac{|H_0|^2}{|B_0|^2}\end{pmatrix} , \\
\dot U^*_0 &=\bD_0 \bE_0 \del_\theta B_0 +  \cF_0, \\
\dot B^*_0&= \bE_0 \del_\theta U_0.
\end{align}
We then define our initial state to be $\xi_0 = (\dot U^*_0, \dot B^*_0 , \bom_0, \bj_0, X_0, U_0)$ and for the truncated version, we denote $\xi_{\dagger,0} = (\dot U^*_0, \dot B^*_0 , \bom_0, \bj_0)$.

Let us also define the families of functions below.
\begin{align}
    \sB_\dagger &=\{\xi_\dagger\in C^0([0,T];\sH^k_\dagger) : \sup_f \lV \xi_\dagger \rV_{\sH^k_\dagger} \leq M, \ \sup_{t} \lV \xi_\dagger - \xi_{\dagger,0} \rV_{\sH^{k-1}_\dagger} \leq r_0\} , \label{sBdagdefn}\\
    \sB &= \{\xi \in C^0([0,T]; \sH^k) : \sup_{t} \lV \xi \rV_{\sH^k} \leq M, \ \xi_\dagger \in \sB_\dagger, \ \xi(0) = \xi_0,  \ X_t = U , \  \lV U\rV_{C^1_{t,\theta}}\leq M\} ,\label{sBdefn}
\end{align}
where $\lV U\rV_{C^1_{t,\theta}}\leq M$ is taken to mean $U(t,\theta)$ is in $C^1([0,T]\times S^1)$, with $\lV U\rV_{C^1([0,T]\times S^1)}\leq M$.
\end{definition}
Some of the criteria in the definitions of $\sB$ and $\sB_\dagger$ are motivated by especially technical details of our construction. On the other hand, for now we can at least verify that for $\xi$ in $\sB$, the corresponding evolving interface $\Gamma(t)$ satisfies the following: the interface opens up for $t>0$ after beginning in a splash at the initial time.

\begin{proposition}\label{dcXlemma}
Suppose that $\xi$ is in $\sB$. Then it follows that the corresponding $X$ and $\Gamma$, $X(t)$ is in $H^{k+2}_\gp$ for each $t$ in $[0,T]$, and
\begin{align}
\lV X-X_0 \rV_{H^{k+1}(S^1)}\leq Ct ,\label{Xshapetbound}\\
c t \leq \delta_\Gamma(t) \leq C t. \label{pinchtbound}
\end{align}
\end{proposition}
\begin{proof}
The bound \eqref{Xshapetbound} is easy to verify due to the fact that $\lV U \rV_{H^{k+1}(S^1)}$ remains uniformly bounded for $\xi$ in $\sB$.

Lemma~\ref{separationPrepLem} verifies the initial velocity $u_0$ is both bounded below in magnitude and essentially pointing left and right along the left and right arcs of $\Gamma_0$ near the splash point. Since the acceleration $U_t$ is bounded above in magnitude, the maximum change in angle of the velocities is $O(t)$. Thus, as long as $T$ is small enough, the particles on the left and right arcs can only move farther to the left and farther to the right, respectively, and we are able to deduce that $\delta_\Gamma(t)\geq ct$. Similarly, one can use the upper bound on the acceleration to verify $\delta_\Gamma(t)\leq Ct$.

The bound on the acceleration together with \eqref{Xshapetbound} can be used to show that the change in $\cG(X(t))$, defined by \eqref{glancingfunction}, is also $O(t)$. Thus for small enough $T$, we can guarantee that $\cG(X(t))\geq \half \cG(X_0)$ for all $t$ in $[0,T]$. In view of Definition~\ref{HsgpDef}, we find $X(t)$ stays in $H^{k+2}_\gp$.\end{proof}

We now make the following remark regarding a notation to be used in the statements of many bounds which are relevant to the rigorous definitions for various $\xi$-dependent quantities.

\begin{remark}\label{PowerDep}
Occasionally we need to keep track of the algebraic dependence of a bounding constant on some quantity, perhaps a norm of $X$ or the state vector $\xi$. Whenever we write $C(\mbox{quantity})$, $C$ clearly being used to denote a function, with a single, isolated quantity inside, either the norm of some object, say $\lV\xi\rV$, or a constant like $M$, we mean
\begin{align}
C(\lV \xi \rV) = c(1+\lV \xi \rV^p),\qquad\qquad\mbox{or}\qquad\qquad
C(M) = c(1+M^p),
\end{align}
where in the right-hand side, we have some universal multiplicative constant $c$ and power $p$, which do not depend on $M$, in particular. We often use the same notation in different instances to refer to a different pair of universal constants $c$ and $p$.

Regarding upper bounds featuring $C(M)$ in the right-hand side, throughout the corresponding proof, if other constants are simply written as $C$, $C'$, $C_1$, etc., (as opposed to $C(M)$), they generally do not depend on $M$. However, we make an exception for Lipschitz type bounds: when stating these, we need not keep track of the dependence of constants on $M$, and so any time we write an expression like
\begin{align}
\lV\dcF(\xi)-\dcF(\uxi)\rV \leq C \lV \xi - \uxi\rV,
\end{align}
either in the statement of a result or in a proof, the constant $C$ may in fact depend on $M$.
\end{remark}

\subsubsection{Plasma-centric maps}\label{plasmacentricmapssection}

Now we prepare to define the quantities dependent on the state vector $\xi$ which are primarily associated with the plasma domain, such as the interior Lagrangian velocity $\bU$ and magnetic field $\bB$, for example.

Rather than considering a situation in which a state vector $\xi$ solves the nonlinear system \eqref{XiSystem}, in this section, we think of $\xi$ as an example of a single iterate $\xi=\xi_\npo$ arising in our main iteration scheme, which is the focus of Section~\ref{solvinglagwavesystemsection}:
\begin{align}
\label{IterationSchemeMainEq}
  \del_t \xi_{n+1} = \bJ(\xi_n) \, \del_\theta \xi_{n+1} + \cR(\xi_n) \, \xi_{n+1} + \dcF(\xi_n) .
\end{align}

Since $\xi=\xi_\npo$ does not necessarily solve the exact system \eqref{XiSystem}, we must compensate by adding error terms to the Lagrangian div-curl system \eqref{OriginalLagrDivCurlSys} in which the Lagrangian vorticity ${\bom}$ and current ${\bj}$ appear in the right-hand side. 

Now we state and solve an appropriate \emph{augmented} Lagrangian div-curl system. In particular, this allows us to define the second component of $\dcF_2$ (see \eqref{SurfEqnTerm2}) as a map acting on ${\xi}$, an important step in clarifying the precise definition of the map $\dcF({\xi})$.

\paragraph{The augmented Lagrangian div-curl system.}

Consider \({\xi} = ({\dot{ U }^*},{\dot{ B }^*},{\bom},{\bj},{X},{U})\) in $\sB$. Below, we take a corresponding auxiliary vector field $w_\Omega(t,x)$ in $\Omega(t)$ as a kind of error term whose definition is not essential at the moment but can be found in \eqref{ulwformula}. The \emph{augmented Lagrangian div-curl system} we associate to the given $\xi$ and solve to produce $\bU(\xi)$, $\bB(\xi)$, and $\bX(\xi)$ is
\begin{align}
 &  \label{AugmentedInteriorSystemL} \\
\begin{aligned}
(a \in \Sigma)
          &&  \left\{ \Bsp \\ \Esp \right. & \\
(\psi = 0)
          &&  \left\{ \Msp \right. & \\
(\psi = -1) 
          &&  \left\{ \Bsp \\ \Esp \right. &
\end{aligned} & 
    \Bsp \grad \cdot (\, \cof (\bM)  \bU )         &= 0                                         & \grad \cdot (\, \cof (\bM)  \bB )   &= 0 \\
    \Msp \grad^\perp \cdot (\, \bM \bU )           &= \bm{\sigma} {\bom} \Lcm              & \grad^\perp \cdot (\, \bM  \bB )    &= \bm{\sigma} {\bj} \Lcm \\
    \Msp \bm{N} \cdot \bU=\bm{N} \cdot \underline{u} & (\bX ) -\bm{N} \cdot w_\Omega (\bX)\,, &\bm{N} \cdot \bB                       &= 0 \,,\\
    \Msp \bm{N} \cdot \bU                            &= 0                                         &\bm{N} \cdot \bB                       &= 0 \\
    \Msp \int^\pi_{-\pi}\bU\cdot\bX_\theta \, d\theta           &= \upalpha \,,                              &\int^\pi_{-\pi}\bB \cdot \bX_\theta \, d\theta &= \upbeta \,,
    \Esp \notag
\end{align}
\begin{align*}
\bX(t) &= \bX_0 + \int^t_0 \bU(\tau)+w_\Omega (\tau,\bX(\tau))\, d\tau &    &(a \in \Sigma ) , \\
\shortintertext{where}
\bM_{ij}(t,a) &=\del_{a_i} \bX_j(t,a)  &                  &(a \in \Sigma, \ i,j = 1,2) , \\
\bm{N}(t,\theta,\psi) &= n(t,\bX(t,\theta,\psi)) &      &(\theta \in S^1, \ \psi = 0, \ -1) ,\notag \\
\underline{u}(t,x) &= {U}(t,{X}^{-1}(t,x)) &      &(x \in {\Gamma}(t) ) ,\notag \\
 \mbox{and } &w_\Omega(t,x) \mbox{ is defined by \eqref{ulwformula}.} &      & \notag 
\end{align*}

For the proposition below, we remind the reader of our notational convention discussed in Remark~\ref{PowerDep}. We also comment here that before we give the actual proof of the proposition (which follows after \eqref{compatConditExpl}), we explain the introduction of the auxiliary field $w_\Omega$ and its relation to boundary condition compatibility.

\begin{proposition}
\label{SolnMainDivCurlProp}
Suppose $\xi$ is in $\sB$. 
Then there exists a unique solution $(\bU, \bB,\bX)$ to the augmented Lagrangian div-curl system \eqref{AugmentedInteriorSystemL} on the time interval $[0,T]$, and the following bounds hold:
\begin{align}
\label{InteriorSysLBound}
\begin{aligned}
    \sup_{t}\lV \bU \rV_{H^{k+3/2}(\Sigma)} +
    \sup_{t}\lV \bB \rV_{H^{k+3/2}(\Sigma)} +
    \sup_{t}\lV \bX \rV_{H^{k+3/2}(\Sigma)}
    \leq C(M),
\end{aligned}
\end{align}
\begin{align}\label{gradInvBd}
&&\sup_{t}\lV(\grad \bX)^{-1} v\rV_{L^\infty(\Sigma)} &\leq C\left(\sup_{t} \lV \xi \rV_{\sH^{k-1}}\right) |v|&(v\in\R^2).
\end{align}
We also have the bound \eqref{InteriorSysLBound} in the case that $k$ in the left-hand side is replaced by $k-1$ and $M$ in the right-hand side is replaced by $\sup_{t} \lV \xi \rV_{\sH^{k-1}}$.
\end{proposition}
\begin{definition}\label{UBXdefn}
We define the maps \(\bU(\xi), \ \bB(\xi), \ \bX(\xi)\) on $\sB$ by
\begin{align}
(\bU(\xi),\bB(\xi),\bX(\xi)) = (\bU,\bB,\bX) \spaced{solving 
 the system \eqref{AugmentedInteriorSystemL},}
\end{align}
for \({\xi} = ({\dot{ U }^*},{\dot{ B }^*},{\bom},{\bj},{X},{U})\) together with the corresponding interface $\Gamma(t)$ and plasma domain $\Omega(t)$.
\end{definition}
\begin{remark}
It is worth noting that one finds that an $\bX$ solving the augmented problem \eqref{AugmentedInteriorSystemL} (which does not require that $\bX_\theta = \bB$, in spite of the identity \eqref{bXbBorigIdent} for exact solutions to the ideal MHD system) is only guaranteed to be up to the same regularity level as $\bU$ and $\bB$, though one might expect it to be one degree of regularity better. On the other hand, although the iterate $\xi=\xi_n$ we consider does not solve the exact ideal MHD system, it \emph{does} satisfy the relations $X_t=U$, $X_\theta = B$, and $B_t=U_\theta$, for $B$ defined by $B_t=E^{-1}\dot B^*$, \emph{which imply $X$ has one more derivative} than $U$ and $B$. This is a crucial observation in the proof of Proposition~\ref{MainInductiveBoundProp}, the main inductive bound for our iteration scheme.

To close our iteration, it turns out that we only need the extra degree of regularity for the \emph{surface} iterate $X_n$ at each stage of the main iteration scheme for solving the $\xi$ system, as opposed to the corresponding $\bX_n$, which \emph{need not satisfy $\bX_n(\theta,0)=X_n(\theta)$ and thus need not have the analogous extra regularity}.

Nevertheless, after arriving at the exact solution to \eqref{XiSystem} by taking $\lim_{n\to\infty} \xi_n$, one finds that the limiting trajectory map $\lim_{n\to\infty} \bX_n$ does indeed have the expected extra degree of regularity.
\end{remark}

Before proving Proposition~\ref{SolnMainDivCurlProp}, let us explain the auxiliary term $w_\Omega$ introduced in the augmented div-curl system \eqref{AugmentedInteriorSystemL}, absent in the original version \eqref{OriginalLagrDivCurlSys}.

Compatibility issues can arise in imposing certain boundary conditions for solutions to div-curl systems. Note, for instance, that a quantity such as $U$ will not generally correspond to a divergence-free flow, since (i) it arises from the evolution equation $U_t=\bE^{-1}\dot U^*$ rather than a div-curl system with zero divergence and (ii) it does not solve the exact original nonlinear system. Thus, given $U$, if one were to attempt to produce $\bU$ and naively impose that the normal component be \( U \cdot X^\perp_\theta/|X_\theta|\), this would generally conflict with the first equation of \eqref{AugmentedInteriorSystemL}, i.e. the divergence-free property of $\bU$.

Let us explain how introducing the vector field $w_\Omega$ defined by \eqref{ulwformula} deals with this discrepancy. It will make the following discussion slightly more transparent to refer to the equivalent Eulerian version of the augmented Lagrangian div-curl system \eqref{AugmentedInteriorSystemL}. Here, we seek $u,b$ defined in $\Omega(t)$ at each time $t$ which satisfy
\begin{align}
\label{AugmentedInteriorSystemE}
\begin{aligned}
(x \in {\Omega}(t))
          &&  \left\{ \Bsp \\ \Esp \right. & \\
(x \in {\Gamma}(t))
          &&  \ \big\{ \, & \\
(x_2 = 0) 
          &&  \left\{ \Bsp \\ \Esp \right. &
\end{aligned} & 
    \Bsp \grad \cdot  u         &= 0                                                 & \grad \cdot b       &= 0 \\
    \Msp \grad^\perp \cdot u    &= \sigma^\flat {\bom} \circ \bX^{-1} \Lcm & \grad^\perp \cdot b &= \sigma^\flat {\bj} \circ \bX^{-1} \Lcm \\
    \Msp n \cdot u              &= n \cdot \underline{u} - n \cdot w_\Omega \,,  & n \cdot b           &= 0 \,,\\
    \Msp n \cdot u              &= 0                                                 & n \cdot b           &= 0 \\
    \Msp \int^\pi_{-\pi} u\cdot \tau \,dx_1&= \upalpha \,,                                      &\int^\pi_{-\pi} b\cdot\tau\,dx_1&= \upbeta \,,
    \Esp
\end{align}
\begin{align*}
\bX(t) &= \bX_0 + \int^t_0 u(\tau,\bX(\tau))+w_\Omega (\tau,\bX(\tau))\, d\tau &    &(a \in \Sigma ) , \\
\intertext{where}
\underline{u}(t,x) &= {U}(t,{X}^{-1}(t,x))&           &(x \in {\Gamma}(t) ) ,\notag \\
\sigma^\flat(t,x)  &= \frac{\bm{\sigma} (\bX^{-1}(t,x))}{\det \grad \bX (t,\bX^{-1}(t,x))} &           &(x \in {\Omega}(t) ) ,\notag \\
{\Omega}(t) &= \bX(t,\Sigma),\ \ \mbox{and} \ \ w_\Omega(t,x) \mbox{ is defined by \eqref{ulwformula}.}&&
\end{align*}
\begin{remark}
For $\xi$ corresponding to an exact solution to ideal MHD, in addition to finding the corresponding $w_\Omega$ reduces to zero, one also finds $\sigma^\flat = 1$, since in that case $\det \grad \bX=\bm{\sigma}$, as verified by Proposition~\ref{JacobianIdent}.
\end{remark}

We define the field $w_\Omega$ so that its divergence has a specific constant value at each time $t$, with
    \begin{align*}
& &     \grad \cdot w_\Omega (t,x) &= \frac{{\Phi}(t)}{|{\Omega}(t)|} & (x \in {\Omega}(t) ) ,
    \end{align*}
where $|{\Omega}(t)|$ is the area of ${\Omega}(t)$, and the quantity ${\Phi}(t)$, defined below in \eqref{Phiformula}, represents the flux of ${U}$ across ${\Gamma}$. In an exact solution to the nonlinear problem, one finds that ${\Phi}(t)=0$, and thus $w_\Omega$ is zero, as expected. To be precise, we define the field $w_\Omega$ corresponding to $\Omega(t)$ and $U$ by
\begin{align}
& &   w_\Omega (t, x) &= \frac{{\Phi}(t)}{|{\Omega}(t)|} \int_{{\Omega}(t)} K_R(x, y) \, dy  &    (x \in {\Omega}(t) ) , 
\label{ulwformula}\\
& &    {\Phi}(t) &= \int^\pi_{-\pi} X^\perp_\theta(t,\theta) \cdot {U}(t,\theta)\, d\theta ,  \label{Phiformula}\\
& &    K_R(x, y) &= K(x - y) + K(x - \bar{y}),
\end{align}
where \(\bar{y}=(y_1,-y_2)\) and $K(x)$ is the convolution kernel of $\grad \Delta^{-1}$ for the periodic domain $S^1 \times \R$, i.e., 
\begin{align}
    K(x) &= \grad G(x) , \\
    G(x) &= \frac{1}{\pi} \log \left| \sin \left( \frac{x_1 + i x_2}{2} \right) \right| . \label{Gdefn}
\end{align}
We choose the kernel $K_R(x)$ over $K(x)$ to arrange \(n\cdot w_\Omega = 0\) at $x_2=0$, simplifying calculations.

        Let us verify that the boundary prescription for $u$ is compatible with the div-curl prescription in the system \eqref{AugmentedInteriorSystemE}. Due to the definition of $w_\Omega$, one can show that the problem \eqref{AugmentedInteriorSystemE} is solvable for $u$ if and only if the problem below is solvable for $u_{\mathrm{aug}}=u+w_\Omega$.
        \begin{align*}
            & & \grad \cdot (u+w_\Omega ) &= \frac{{\Phi}(t)}{|{\Omega}(t)|} & &(x \in {\Omega}(t) ) ,  \\
            & & n \cdot (u+w_\Omega )&= n \cdot \underline u & &(x\in\Gamma(t)),\\
            & &  n \cdot (u+w_\Omega ) &= 0 & &(x_2 =0 ) .
        \end{align*}
        It is easy to verify solvability of the above problem for $u_{\mathrm{aug}}=u+w_\Omega$ essentially by checking compatibility with the divergence theorem via the following calculation, which only uses the definitions of $\Phi(t)$ and $\underline u$.
        \begin{align}
        \int_{\Omega(t)} \frac{{\Phi}(t)}{|\Omega(t)|} \, dx = \int^\pi_{-\pi} {X}^\perp_\theta \cdot {U}  \, d\theta  = \int_{{\Gamma}(t)} n \cdot \underline u \, dS. \label{compatConditExpl}
        \end{align}
This justifies our boundary condition for $\bm{N} \cdot \bU$ in the augmented Lagrangian div-curl problem \eqref{AugmentedInteriorSystemL}.

Now we give the proof of Proposition~\ref{SolnMainDivCurlProp}, describing the solution of the Lagrangian div-curl system \eqref{AugmentedInteriorSystemL} via an iteration scheme (not to be confused with the iteration scheme of Section~\ref{solvinglagwavesystemsection} for the solution to the $\xi$ system \eqref{XiSystem} itself).

\begin{proof}[Proof of Proposition~\ref{SolnMainDivCurlProp}]
For clarity, we present the construction of the solution to the Lagrangian div-curl system \eqref{AugmentedInteriorSystemL} partially from the point of view of the equivalent Eulerian system \eqref{AugmentedInteriorSystemE}, although it is relatively easy to adapt this to a purely Lagrangian proof.

Consider $\xi$ in $\sB$, with corresponding $X,U,\bom,\bj,\Omega(t)$, etc. Let $w_\Omega$ and $\underline u$ be defined as in \eqref{ulwformula}. Let us form iterates to find $\bX$ and $u$ solving \eqref{AugmentedInteriorSystemE}.

Note that $\lV X - X_0\rV_{H^{k+1}(S^1)}\leq C T$ and $X\in H^{k+2}_\gp$, by Proposition~\ref{dcXlemma}, so that $\Omega(t)$ has the overall shape depicted in Figures ~\ref{OpenedSplash1} or~\ref{ClosedSplash1}. For the first iteration of the Lagrangian map, for each fixed $t$ in $[0,T]$ we shall sweep out a parametrization of $\Omega(t)$ with a deformation (smoothly depending on the parameter $\psi$ in $[-1,0]$) from the bottom $\theta \mapsto (\theta,-1)$ to the top $\theta\mapsto X(t,\theta)$. This may be done in such a way that it defines a bijection from $\Sigma$ to $\Omega(t)$ for each $t$, say $\bX_1(t,\theta,\psi):\Sigma\to\Omega(t)$, satisfying $\bX_1(t,\theta,-1) = (\theta,-1)$ and $\bX_1(t,\theta,0) = X(t,\theta)$. Moreover, we may arrange that $\bX_1$ is in $C^0([0,T];H^{k+3/2}(\Sigma))$.

For $m\geq 1$, assume we have an iterate $\bX_m$ in $C^0([0,T];H^{k+3/2}(\Sigma))$ with invertible $\bX_m(t,\cdot)$ for each $t$ satisfying 
\(\bX_m(t,\Sigma) = {\Omega}(t)\). We now solve the following system for \( u_{m+1}(t):{\Omega}(t) \to \R^2 \).
\begin{align}
\label{AugmentedInteriorSystemIt}
(x \in {\Omega}(t)) \quad &
\left\{
    \begin{aligned}
        \grad \cdot u_{m+1}       &= 0 \\
        \grad^\perp \cdot u_{m+1} &= \sigma^\flat_m {\bom} \circ \bX^{-1}_m \Lcm
    \end{aligned}
\right. \\
(x \in {\Gamma}(t)) \quad & \ \big\{\, n \cdot u_{m+1}  = n \cdot \underline{u} - n \cdot {w}\,, \\
(x_2 = 0) \quad &
\left\{
    \begin{aligned}
                  n\cdot u_{m+1}             &= 0 \\
       \quad  \int^\pi_{-\pi} u_{m+1} \cdot\tau \, dx_1 &= \upalpha \,,
    \end{aligned}
\right.
\end{align}
where we define \(\bm{\sigma}^\flat_m  = \bm{\sigma} (\bX^{-1}_m)/\det \grad \bX_m (t,\bX^{-1}_m)\) for $x$ in $\Omega(t)$.
A similar computation to \eqref{compatConditExpl} verifies compatibility of the boundary conditions, ensuring the existence of $u_{m+1}$. Regularity results for solutions to such div-curl systems in domains such as ${\Omega}(t)$ are standard, though they can also be deduced from the estimates of Section~\ref{basicEstSection}. Using such a method, one finds for either of the cases $j=k$ or $j=k-1$ that
\begin{align}\label{iteratebound1}
\lV u_{m+1}\rV_{H^{j+3/2}({\Omega}(t))} \leq C\left(\lV \xi \rV_{\sH^j}\right).
\end{align}

Now we construct a map $\bX_{m+1}(t) : \Sigma \to \R^2$ satisfying 
\begin{align*}
\bX_{m+1}(t) &= \bX_0 + \int^t_0 u_{m+1}(\tau,\bX_{m+1}(\tau))+w_\Omega(\tau,\bX_{m+1}(\tau))\, d\tau &    &(a \in \Sigma ) , \\
\bX_{m+1}(t,\Sigma) &= {\Omega}(t) . & &
\end{align*}
Note we have the same $\bX_{m+1}$ appearing in both the left and right-hand sides of the first equation. We leave it to the reader to verify that the existence of such a map $\bX_{m+1}$ is clear under the assumption that $u_{m+1}+w_\Omega$ transports ${\Omega}(t)$. This assumption is justified by recalling that $U=\frac{dX}{dt}$ and observing
\begin{align*}
& &    (n \cdot (u_{m+1}+w_\Omega))(t, {X} ) &= {N} \cdot {U} = {N} \cdot \frac{d{X}}{dt} & &(\theta \in S^1 ) .
\end{align*}
Furthermore, one finds
\begin{align}
\sup_{t}\lV \bX_{m+1} -\bX_0\rV_{H^{j+3/2}(\Sigma)}\leq C\left(\sup_{t}\lV \xi \rV_{\sH^j}\right) T,
\end{align}
which implies
\begin{align}
&&\sup_{t}\lV \bX_{m+1}\rV_{H^{j+3/2}(\Sigma)}&\leq C\left(\sup_{t}\lV \xi \rV_{\sH^j}\right), &\label{iteratebound2}\\
&&\sup_{t}\lV(\grad\bX_{m+1})^{-1}v\rV_{L^\infty(\Sigma)} &\leq C\left(\sup_{t}\lV \xi \rV_{\sH^j}\right)|v| &(v\in\R^2), \label{iteratebound3}
\end{align}
for small enough $T$.

A standard inductive argument for the above iteration proves $u_m$ and $\bX_m$ converge in $C^0([0,T];H^{k+3/2}(\Omega(t)))$ and $C^0([0,T];H^{k+3/2}(\Sigma))$, respectively, with limits $u$ and $\bX$ solving \eqref{AugmentedInteriorSystemE} and satisfying the bounds \eqref{iteratebound1}, \eqref{iteratebound2}, and \eqref{iteratebound3}.

Moreover, now that we have $\bX$, the div-curl system in \eqref{AugmentedInteriorSystemE} for $b$ is directly solvable, and one easily verifies the bound \eqref{iteratebound1} with $b$ in place of $u_{m+1}$.

Given $u$, $b$, and $\bX$, we are able to define
\begin{align*}
& & &
\begin{aligned}
    \bU(t,a) &= u(t,\bX(t,a)) \\
    \bB(t,a) &= b(t,\bX(t,a))
\end{aligned}
&   (a \in \Sigma, \ t \in [0,T]),
\end{align*}
which, in combination with $\bX(t,a)$, give a solution to the system \eqref{AugmentedInteriorSystemL}. The claimed bounds \eqref{InteriorSysLBound} (as well as the version with $k$ replaced by $k-1$) and \eqref{gradInvBd} then easily follow from the estimates deduced above for $u$, $b$, and $\bX$.
\end{proof}

We also will make use of the following proposition, which establishes Lipschitz bounds as well as continuity in time for the maps of Definition~\ref{UBXdefn}.

\begin{proposition}\label{LipBdMainDivCurlSysProp}
For the maps $\bU$, $\bB$, and $\bX$ of Definition~\ref{UBXdefn}, we have
\begin{align}
\bU , \  \bB, \  \bX  :  \sB \to C^0( [0,T];H^{k+3/2}(\Sigma) ) .
\end{align}
Moreover, for $\xi$ and $\uxi$ in $\sB$,
\begin{align}\label{bUbBbXlipBdStatement}
\sup_{t}\lV \bU (\xi) - \bU (\uxi) \rV_{H^{7/2}(\Sigma)} + \sup_{t}\lV  {\bB}(\xi) -  {\bB}(\uxi) \rV_{H^{7/2}(\Sigma)} + \sup_{t}\lV  {\bX}(\xi) -  {\bX}(\uxi) \rV_{H^{7/2}(\Sigma)} & \\
 \leq C \sup_{t} \lV \xi - \uxi \rV_{\sH^2}.&
\end{align}
\end{proposition}
\begin{proof}
The Lipschitz bound \eqref{bUbBbXlipBdStatement} is readily shown by considering the evolution of $\bX(\xi)-\bX(\uxi)$, writing out the Lagrangian div-curl systems satisfied by $\bU(\xi)-\bU(\uxi)$ and $\bB(\xi)-\bB(\uxi)$, and applying the div-curl estimate of Lemma~\ref{GeneralDivCurlSysLem}. Following the proof of Proposition~\ref{SolnMainDivCurlProp}, it is not hard to show the maps $\bU(\xi)(t)$, $\bB(\xi)(t)$, and $\bX(\xi)(t)$ continuously take values in $H^{k+3/2}(\Sigma)$ on the interval $[0,T]$.
\end{proof}
Now that we have properly introduced the maps $\xi\mapsto\bU(\xi)$ and $\xi\mapsto\bB(\xi)$, we may define the map giving the $\dcF_2$ component of the source term $\dcF$ appearing in the $\xi$ system \eqref{XiSystem}.
\begin{definition}\label{dcF2Defn}
For $\xi$ in $\sB$, we define $\dcF_2(\xi):\Sigma\to\R$ by
\begin{align}
\dcF_2(\xi) = 2 \bm{\sigma}^{-1}\col{0}{\del_\theta(  \bU(\xi)) \cdot \del_\psi ( \bB(\xi)) - \del_\psi( \bU(\xi))\cdot \del_\theta( \bB(\xi))} .
\end{align}
\end{definition}
Propositions~\ref{SolnMainDivCurlProp} and~\ref{LipBdMainDivCurlSysProp} now directly imply the following.
\begin{proposition}\label{dcF2UnifBd}
We have
\begin{align}
\dcF_2 : \sB \to C^0([0,T]; H^{k+1/2}(\Sigma)) ,
\end{align}
and for $\xi$ in $\sB$,
\begin{align}
\sup_{t}\lV \dcF_2(\xi)\rV_{H^{k+1/2}(\Sigma)} \leq C(M).
\end{align}
\end{proposition}

\subsubsection*{The interior pressure gradient $\bP^-_\grad$}

Recall the original Lagrangian pressure gradient $\bP_\grad=(\grad p) \circ \bX$, where the system for $p$ is \eqref{PressureSystem}, and the modified version $\bP^-_\grad$ which we simply refer to as the \emph{interior pressure gradient}, represented by
\begin{align}\label{intPressureGradrep}
&& \bP^-_\grad &= \bP_\grad  - \half \left(\grad \bigh_- |h|^2\right) \circ \bX & (a\in\Sigma).
\end{align}
The interior pressure gradient system is as follows.
\begin{align}
\label{InteriorPressureGradSystem}
\begin{aligned}
(a \in \Sigma) \qquad &
\left\{
    \begin{aligned}
        \grad \cdot (\, \cof (\bM) \bP^-_\grad ) &= \grad \cdot (\, \cof (\grad \bU^t) \bU - \cof(\grad \bB^t) \bB )  \\
        \grad^\perp\cdot (\, \bM \bP^-_\grad )  &= 0 \, \Lcm
    \end{aligned}
\right. \\
(\psi = 0)  \qquad & \ \big\{ \, \bX_\theta \cdot \bP^-_\grad = 0 , \\
(\psi = -1) \qquad & \ \Big\{ \, \bm{N} \cdot \bP^-_\grad = \half \left(\del_{x_2}  \left[ \bigh_-  |h|^2  \right]\right)(\bX).
\end{aligned}
\end{align}
As before, $\bM_{ij} = \del_{a_i} \bX_j$, and $\bX$, $\bU$, $\bB$, and $h$ are those associated to a given $\xi$. This corresponds to a mixed Dirichlet-Neumann boundary problem for a function $p_-:\Omega\to\R$ for which $\bP^-_\grad = (\grad p_-)\circ \bX$. This problem is then found to be uniquely solvable by Lax-Milgram, due to the coercivity property below, which is holds for $\xi$ in $\sB$ by Proposition~\ref{SolnMainDivCurlProp}.
\begin{align}
&& (  \cof(\bM)\bM^{-1} v)\cdot v &\geq c |v|^2 & (v\in\R^2).
\end{align}
Let us define the map associating to $\xi$ the corresponding interior pressure gradient $\bP^-_\grad(\xi)$.

\begin{definition}\label{PminusDefns}

Consider $\xi$ in $\sB$, with corresponding interface $\Gamma(t)$. In the system \eqref{InteriorPressureGradSystem}, we take $\bU$, $\bB$, and $\bX$ to be $\bU(\xi)$, $\bB(\xi)$, and $\bX(\xi)$, respectively, of Definition~\ref{UBXdefn}. Additionally, at each fixed time $t$ in $[0,T]$ we take $h$ in \eqref{InteriorPressureGradSystem} to be the $h(\Gamma(t))$ given by Definition~\ref{formalhdefn}. 

We then define $\bP^-_\grad(\xi)(t,a)$ to be the corresponding solution to \eqref{InteriorPressureGradSystem} guaranteed by Lax-Milgram. We define the corresponding map for the trace at $\psi=0$ by $P^-_\grad(\xi)(t,\theta)=\bP^-_\grad(\xi)(t,\theta,0)$.
\end{definition}
\begin{proposition}\label{PminusBdProp}
Regarding $\bP^-_\grad$ and $P^-_\grad$ as maps acting on $\xi$, we have
\begin{align}\label{ctsvaluesPminus}
 {\bP}^-_\grad : \sB \to C^0([0,T];H^{k+3/2}(\Sigma)),\qquad P^-_\grad : \sB \to C^0([0,T];H^{k+1}(\Sigma))  .
\end{align}
Moreover, given $\xi$ in $\sB$,
\begin{align}\label{PminusBds2}
\sup_{t}\lV  {\bP}^-_\grad (\xi)\rV_{H^{k+3/2}(\Sigma)} \leq C (M),\qquad
\sup_{t}\lV  P^-_\grad (\xi)\rV_{H^{k+1}(S^1)} &\leq C(M).
\end{align}
Alternate versions of the bounds \eqref{PminusBds2} also hold in the case that one replaces $k$ with $k-1$ and $M$ with $\sup_{t}\lV \xi \rV_{\sH^{k-1}}$. In addition, for $\uxi$ in $\sB$, we have
\begin{align}\label{PminusLipBd}
\sup_{t}\lV \bP^-_\grad(\xi) - \bP^-_\grad(\uxi) \rV_{H^{5/2}(\Sigma)} +\sup_{t}\lV   P^-_\grad(\xi) -   P^-_\grad(\uxi) \rV_{H^2(S^1)} 
&\leq C \sup_{t} \lV \xi - \uxi \rV_{\sH^1} ,\\
\end{align}
\end{proposition}

\begin{proof}
In the interest of brevity we merely sketch the arguments here. Standard techniques using Lax-Milgram show that the solution $\bP^-_\grad$ to \eqref{InteriorPressureGradSystem} depends continuously on $\bM$, $X$, and the right-hand sides in \eqref{InteriorPressureGradSystem}, including the boundary data.

First observe that for $\xi$ in $\sB$, in particular $\bM(t)$ and the right-hand side of the divergence equation in \eqref{InteriorPressureGradSystem} continuously take values in $H^{k+1/2}(\Sigma)$ on $[0,T]$, in view of Proposition~\ref{LipBdMainDivCurlSysProp} and the Piola identity. Moreover, one finds a uniform bound on the $H^{k+1}(S^1)$ norm of $(\del_{x_2}\bigh_- |h|^2)|_{x_2=0}$ can be verified from elementary estimates for harmonic extensions coupled with a uniform bound on $\lV h \rV_{L^\infty(\Gamma(t))}$. Such a bound on $h|_{\Gamma(t)}$ is implied by Proposition~\ref{hEstimateProp}. With a bit more work, one verifies $t\mapsto(\del_{x_2}\bigh_- |h|^2)\circ\bX(t,\theta,-1)$ continuously takes values in $H^{k+1}(S^1)$. Using these ideas, one is able to verify \eqref{ctsvaluesPminus} in addition to both claimed versions of the bounds \eqref{PminusBds2}. Regarding \eqref{PminusLipBd}, one may derive a div-curl problem for the difference $\bP^-_\grad(\xi) - \bP^-_\grad(\uxi)$ and apply similar techniques to show the Lipschitz bound holds.
\end{proof}


\subsubsection{Vacuum-centric maps and operator-valued maps}\label{vacandopmapssection}

Now we properly define key objects appearing in the Lagrangian wave system as maps acting on $\xi$ which require more subtle analysis, due to either (i) being significantly affected by the degenerating pinch in the vacuum or (ii) taking values in operator spaces.

 For some of these definitions, it matters whether we are considering interfaces with positive pinch or interfaces with pinch zero. To prepare for this, we define classes of $\xi$ with interfaces whose pinches lie in certain ranges. 
 \subsubsection*{Classes for the state vector $\xi$ with specified pinch ranges}
 The classes below are defined specifically for $\xi$ at a frozen instant in time, subsets of the class $\sH^k_\gp$, defined in \eqref{sHsgp}. They are useful in the context of vacuum-centric maps, which are sensitive to the pinch, but dependent only on the shape of the vacuum at a single instant.
\begin{definition}\label{sBboxdefns}
Let us define
\begin{align}
    \sB^\bbox(\delta) &= \{\xi \in \sH^k_\gp :\lV \xi \rV_{\sH^k} \leq M, \ \lV X - X_0 \rV_{H^{k+1}(S^1)} \leq C\delta , \ c\delta \leq \delta_\Gamma \leq C\delta \}, \\
    \sB^\bbox &= \bigcup_{\delta \in [0,\delta_0]}\sB^\bbox(\delta) ,\\
    \sB^\wbox &= \bigcup_{\delta \in (0,\delta_0]}\sB^\bbox(\delta) .
\end{align}
\end{definition}
\begin{proposition}\label{xiInPinchRanget}
For $\xi$ in $\sB$, we have
\begin{align}
&& \xi(t)&\in \sB^\bbox(t) &(t\in[0,T]).
\end{align}
\end{proposition}
\begin{proof}
This can be ensured for sufficiently small $T$ by using Proposition~\ref{dcXlemma}.
\end{proof}
\subsubsection*{The exterior magnetic field and exterior pressure}
We now return to the exterior magnetic field as a map acting on $\xi$. After defining the external magnetic field trace map $\xi\mapsto H(\xi)$ below, we give an important basic bound which asserts $H(\xi)$ remains highly regular for vacuums with arbitrarily small pinches, including the case of pinch zero.
\begin{definition}\label{prelimHmapDefn}
For $\xi$ in $\sB^\bbox$, with corresponding interface $\Gamma$ parametrized by $X$, we define
\begin{align*}
&& H (\xi)(\theta) &=  h(\Gamma)(X (\theta)) &(\theta\in S^1),
\end{align*}
for $h(\Gamma)$ given by Definition~\ref{formalhdefn}.
\end{definition}
Later, in Section~\ref{vacuuminterfaceestimatessection}, we discuss the details of vacuum estimates for a nearly self-intersecting or self-intersecting interface, and show how the induced $h$ can be controlled in weighted Sobolev norms dominating the standard ones with Proposition~\ref{hEstimateProp}. We get the result below as an immediate consequence.
\begin{proposition}
Regarding $H$ as a map acting on $\xi$, we have
\begin{align}
 H  &: \sB^\bbox \to H^{k+1}(S^1) .
\end{align}
\end{proposition}

Now we consider the vacuum counterpart to the interior pressure gradient $\bP^-_\grad$ introduced with the system \eqref{InteriorPressureGradSystem}, the \emph{exterior pressure gradient}, which we denote by $P^+_\grad$. Defined in terms of both $h$ and the harmonic extension operator $\bigh_+$, it is also sensitive to the pinch, and carefully estimated in Proposition~\ref{PplusbdProp}, which implies the proposition following the definition below.
\begin{definition}\label{ExtPressureDefnBd}
    Consider $\xi$ in $\sB^\bbox$ with associated $X$ and vacuum region $\cV$. For $h=h(\Gamma)$ and \(p_+ : \cV \to \R\) defined by
    \begin{align}
        && p_+(x) &= \half(1-\bigh_+) |h|^2(x) & (x\in\overline{\cV}),
    \end{align}
    we define $ P^+_\grad{}(\xi)$ by
    \begin{align}
         && P^+_\grad{}(\xi)(\theta) &= (\grad p_+) (X(\theta)) & (\theta\in S^1).
    \end{align}
\end{definition}
\begin{proposition}\label{ExtPressureDefnBd2}
    Regarding $P^+_\grad$ as a map acting on $\xi$, we have
    \begin{align}
     P^+_\grad : \sB^\bbox \to H^{k+1}(S^1) .
    \end{align}
\end{proposition} 
\subsubsection*{Operator-valued maps acting on $\xi$}
Here, we present some basic definitions and propositions revolving around how Dirichlet-to-Neumann maps $\cN_\pm$ and the operator $\bE$, for example, are described as operator-valued maps acting on $\xi$, relying on bounds taken from Sections~\ref{vacuuminterfaceestimatessection} and~\ref{auxiliaryestimates}.

\begin{definition}\label{DirichToNeumDefs}
Consider $\xi$ in $\sB^\bbox$ with associated $\Gamma$. First consider the case that the pinch $\delta_\Gamma$ is positive. We use $\Omega_+$ to denote the unbounded region above the associated $\Gamma$, and we use $\Omega_-$ to denote the unbounded region below $\Gamma$. For $1\leq s \leq k+2$, and a function $f$ in $H^s(\Gamma)$, consider the unique solutions $\phi_\pm$ (which remain bounded in their domains) to the problems
\begin{align}
& & \Delta \phi_\pm &=0 &   &(x\in \Omega_\pm), \\
& & \phi_\pm &= f & &(x\in \Gamma).
\end{align}
We then define the interior Dirichlet-to-Neumann map, $\bign_-(\xi)$, and the exterior Dirichlet-to-Neumann map, $\bign_+(\xi)$, by the following, where we take $n$ to be the normal to $\Gamma$, outward-pointing relative to $\Omega$.
\begin{align}
& & \bign_\pm(\xi) f &= \del_n \phi_\pm & (x\in\Gamma).
\end{align}
We correspondingly define the ``Lagrangian versions'' of the Dirichlet-to-Neumann maps, $\cN_\pm(\xi)$, by the following.
\begin{align}\label{Npmdefns}
& & \cN_\pm(\xi) F(\theta) &= \big(\bign_\pm(\xi)(F\circ X^{-1})\big) (X(\theta)) & (\theta \in S^1, \ F\in H^s(S^1)).
\end{align}
Now consider the case that $\delta_\Gamma=0$. We define $\cN_-(\xi)F$ by replacing $\Gamma$ with an approximating sequence of non-self-intersecting curves and taking a limit.
\end{definition}
\begin{remark}\label{blowupreminder}
Note that in the case that $\delta_\Gamma=0$, using a sequence of non-self-intersecting curves to approximate $\Gamma$ only leads to a bounded interior Dirichlet-to-Neumann map $\cN_-(\xi)$. We do not define $\cN_+(\xi)$ when the pinch is zero. This is related to the discussion of blow up for the exterior Dirichlet-to-Neumann map at the end of Section~\ref{squeezedmagfieldssection}.
\end{remark}
The proof of the following proposition is standard.
\begin{proposition}
For integer $s$ with $1\leq s \leq k+2$ we have bounded maps:
\begin{align}
\cN_-(\cdot):\sB^\bbox \to B(H^s(S^1),H^{s-1}(S^1)),\\
\cN_+(\cdot):\sB^\wbox \to B(H^s(S^1),H^{s-1}(S^1)).
\end{align}
\end{proposition}
Earlier, we alluded to the fact that a critical cancellation occurs upon adding the operators $\cN_+$ and $\cN_-$. The key regularity-preserving property of the corresponding operator $\Nres$ is demonstrated in Proposition~\ref{NresIsBoundedMap}, below.
\begin{definition}\label{Nresdefn}
Given $\xi$ in $\sB^\wbox$, for $\cN_\pm(\xi)$ given by Definition~\ref{DirichToNeumDefs}, we define
\begin{align}
\Nres(\xi) = \cN_+(\xi) + \cN_-(\xi) .
\end{align}
\end{definition}
\begin{proposition}\label{NresIsBoundedMap}
For integer $s$ with $1\leq s \leq k+1$ we have a bounded map:
\begin{align}
\Nres(\cdot) : \sB^\wbox \to B( H^s(S^1), H^s(S^1) ) .
\end{align}
\end{proposition}
\begin{remark}
The above proposition is an easy consequence of the stronger result proved in Section~\ref{vacuuminterfaceestimatessection}, Proposition~\ref{DtoNCancellation}, which provides a uniform bound for the operator $\Nres$ that holds even as the pinch of the interface shrinks to zero.
\end{remark}

Note that $\Nres$ is only defined for interfaces with nonzero pinch. Because our initial data starts in a splash, we must find another way to define related terms in our system for $\xi$ such that $\delta_\Gamma=0$. In Section~\ref{pinchzeroextsection} we address this subtlety.

Now we give a rigorous definition of the operator $\bE$ we motivated in Section~\ref{refiningsurfsys}.

\begin{proposition}\label{bEDefns}
Consider $\xi$ in $\sB^\bbox$, and let $\Gamma$ and $X$ be the associated interface and parametrization. For any $F\in L^1( S^1)$ let us define
\begin{align}\label{projecdefn}
    \langle F \rangle &= \frac{1}{|\Gamma|} \int^\pi_{-\pi} F |X_\theta| \, d\theta .
\end{align}
Then for integer $s$ with \(0\leq s \leq k+1\), for any real-valued function $F$ in $H^s(S^1)$, there is a unique $G$ in $H^{s+1}(S^1)$ such that
\begin{align}
\cN_-(\xi) G = F - \langle F \rangle .
\end{align}
We define $\cH(\xi)$ by
\begin{align}
&& \cH(\xi) F(\theta) &= -\frac{1}{|X_\theta|}\frac{dG}{d\theta} &(\theta\in S^1, \ F\in H^s(S^1)).
\end{align}
Let us define the operator $\bE(\xi)$ acting on $\R^2$-valued $V$ in $H^s(S^1)$ by the following, for $\bigtau=\tau\circ X:S^1\to \R^2$, where $\tau$ is the left-to-right pointing unit tangent along $\Gamma$:
\begin{align}\label{bEdefnFmla}
\bE(\xi) V   = \begin{pmatrix} 1 & \cH(\xi) \\ 0 & 1 \end{pmatrix} \left[\col{\bigtau}{N} V \right] = \col{\bigtau \cdot V + \cH(\xi)(N \cdot V)}{N \cdot V} .
\end{align}
We have that $\xi \mapsto \bE(\xi)$ defines a bounded map for $0\leq s \leq k+1$:
\begin{align}
\bE(\cdot) &: \sB^\bbox \to B( H^s(S^1), H^s(S^1) ) . \\ \intertext{Moreover, for $\bE^{-1}(\cdot)$ mapping $\xi$ to $(\bE(\xi))^{-1}$, given by the formula \eqref{bEinverse},}
\bE^{-1}(\cdot) &: \sB^\bbox \to B( H^s(S^1), H^s(S^1) ) .
\end{align}
\end{proposition}
This proposition follows from a simple application of Proposition~\ref{bEbounds}. Let us remark that because $\cH$ involves only harmonic extension into $\Omega_-$, a region which does not get pinched, both $\bE$ and $\bE^{-1}$ are estimated by standard methods.

{\subsubsection*{Time derivative quantities}}

Note in the preceding definitions, our maps are defined for time-independent $\xi$, elements of $\sB^\bbox$ or $\sB^\wbox$.\footnote{Despite this we do in fact extend the maps of the previous section to $\sB$, the class for time-dependent $\xi$, in the obvious way, by plugging in $\xi(t)$ at each fixed time $t$ in $[0,T]$.} In our Lagrangian wave system, we often have expressions that represent derivatives in time of certain quantities, and yet we would like to consider them as maps acting on $\xi$ without explicitly having to take a time derivative or necessarily restrict our attention to evolving, time-dependent states.

As an example, suppose we would like to define a map $\xi\mapsto\dot N(\xi)$ for a state $\xi$ ``frozen in time'', say $\xi\in\sB^\bbox$, such that we get $\del_t(N(\uxi(t)))=\dot N(\uxi(t))$ when plugging in time-dependent $\uxi$ in $\sB$ for each fixed time $t$. Writing
\begin{align}
N= i \bigtau = i e^{i\arg X_\theta}
\end{align}
and assuming $U=X_t$, we calculate
\begin{align}
N_t = - \bigtau \del_t (\operatorname{Im}\log X_\theta) = -\bigtau \operatorname{Im}\left(\frac{U_\theta}{X_\theta}\right) = -\bigtau\frac{U_\theta\cdot N}{|X_\theta|}.
\end{align}
Therefore a good definition of $\dot N(\xi)$ for $\xi$ in $\sB^\bbox$ would be $\dot N(\xi) = -\bigtau (U_\theta\cdot N)/|X_\theta|$.



Generally speaking we typically use a dot to denote a corresponding ``time derivative mapping''. We briefly note some of these maps below, which we will soon need in order to finish defining $\cR(\xi)$ and $\dcF_1(\xi)$.
\begin{itemize}
\item
\(\dot H(\cdot):\sB^\bbox \to H^k(S^1)\) \quad ($\dot H(\xi)$ defined to coincide with $\del_t( H(\cdot))$)
\item
\(\dot P^\pm_\grad(\cdot): \sB^\bbox \to H^k(S^1)\) \quad (defined to coincide with $\del_t(  P^\pm_\grad(\cdot))$)
\item
\([\del_t;\bE](\cdot): \sB^\bbox \to B(H^s(S^1),H^s(S^1))\) \quad (defined to coincide with $[\del_t,\bE(\cdot)]$)
\item
\([\del_t;\Nres](\cdot): \sB^\wbox \to B(H^s(S^1),H^s(S^1))\) \quad (defined to coincide with $[\del_t,\Nres(\cdot)]$)
\end{itemize}
The definitions of the various time-derivative type maps we use are given in detail in Section~\ref{timederivdefnsection}.

\subsubsection*{Defining $\cR(\xi)$ and $\bJ(\xi)$}
Now, we are able to define $\cR(\xi)$ and $\bJ(\xi)$ (though not yet $\dcF(\xi)$) for $\xi$ in $\sB$, which includes situations in which the interface has pinch zero.
\begin{definition}\label{DEcommDefns}
For $\xi$ in $\sB^\bbox$, we define the $10\times 10$ matrix $\bJ(\xi)$, the $4\times 4$ matrix $\bJ_1(\xi)$, and the $2\times 2$ matrices $\bD(\xi)$ and $\dot \bD(\xi)$ by the following, for $\bJ_2 = \begin{pmatrix} 0 & 1 \\ 1 & 0 \end{pmatrix}$ and $\dot H(\xi)$ given by Definition ~\ref{dotHdefns}:
\begin{align}
 \bJ(\xi) &= \begin{pmatrix} \bJ_1(\xi) & 0_{(4\times 2)}     & 0 \\
                            0_{(2\times 4)}    & \bJ_2 & 0 \\
                            0    &  0    & 0_{(4 \times 4)} \end{pmatrix} ,
&
\bJ_1(\xi) &= \begin{pmatrix} 0 & \bD(\xi) \\ \bI_{(2\times 2)} & 0 \end{pmatrix}, \\
 \bD(\xi) &= \begin{pmatrix} 1 & 0 \\ 0 & 1 + \frac{| H (\xi)|^2}{|X_\theta|^2} \end{pmatrix}, & \dot \bD  (\xi) &= 
   \left(\frac{|X_\theta|^{2}\dot H(\xi) \cdot  H(\xi)-| H(\xi)|^2 U_\theta \cdot X_\theta}{|X_\theta|^4} \right)   \begin{pmatrix} 0 & 0 \\ 0 &  2  \end{pmatrix}.
\end{align}
Additionally, we define operator matrices $R(\xi)$, $\bR^{ij}(\xi)$ for $i,j=1,2$, and $\cR(\xi)$ by\footnote{In the matrices $R(\xi)$ and $\cR(\xi)$, each entry (besides $R(\xi)$ itself) represents a $2\times 2$ matrix.}
\begin{align}
R(\xi) &= \begin{pmatrix}
\bR^{11}(\xi) & \bR^{12}(\xi) & 0 \\
\bR^{21}(\xi) & \bR^{22}(\xi) & 0 \\
0 & 0 & 0_{(2\times 2)}
\end{pmatrix} ,\hspace{1.5 cm}
\begin{aligned}
\bR^{11} &= \dot \bD   \bD^{-1} + \bD [\del_t;\bE] \bE^{-1} \bD^{-1} ,\\
\bR^{12} &= -\bD[\del_\theta, \bE] \bE^{-1} ,\\
\bR^{21} &= -[\del_\theta, \bE] \bE^{-1} ,\\
\bR^{22} &= [\del_t;\bE] \bE^{-1},
\end{aligned}\label{bRderivation} \\
\cR(\xi) &= \begin{pmatrix}
                    &         &   & 0 & 0   \\
                    &  R(\xi) &   & 0 & 0   \\
                    &         &   & 0 & 0   \\
       0            &  0      & 0 & 0 & \bI \\
    (\bE(\xi))^{-1} &  0      & 0 & 0 & 0
    \end{pmatrix} ,
\end{align}
where $\bE(\xi)$ and $\bE^{-1}(\xi)$ are defined in Proposition~\ref{bEDefns} and $[\del_t;\bE](\xi)$ is defined in Proposition~\ref{bEcommBounds}.
\end{definition}
\begin{proposition}\label{cRbJBoundProp}
For the map $\bJ(\cdot)$ of Definition~\ref{DEcommDefns}, we have
\begin{align}
 (\bJ(\cdot))_{24} &: \sB^\bbox \to H^{k+1}(S^1) ,
\end{align}
and $(\bJ(\cdot))_{ij}$ is $0$ or $1$ for all other $i, j \in \{1,\ldots,10\}$.
\end{proposition}
\begin{proof}
This follows from the bound on $H(\xi)$ of Proposition~\ref{hEstimateProp}.
\end{proof}

Let us collect some bounds which follow from the definitions of $\bE(\xi)$, the main bound on $\bE(\xi)$ provided by Proposition~\ref{bEbounds}, and the bound on $[\del_t;\bE](\xi)$ provided by Proposition~\ref{bEcommBounds}.
\begin{lemma}\label{bEboundLemma}
For $\xi$ in $\sB^\bbox$,
\begin{align*}
&\lV\bE(\xi) \rV_{B(H^k(S^1),H^k(S^1))} \leq C(\lV \xi \rV_{\sH^{k-1}}), \\
&\lV\bE(\xi) \rV_{B(H^{k+1}(S^1),H^{k+1}(S^1))}+\lV [\del_t;\bE](\xi) \rV_{B(H^k(S^1),H^k(S^1))} \leq C(M) ,
\end{align*}
and the same bounds hold with $(\bE(\xi))^{-1}$ in place of $\bE(\xi)$.
\end{lemma}
Similarly, we are easily able to justify the following bounds in view of the definitions of the corresponding objects in Definition~\ref{DEcommDefns} and the bounds on $H(\xi)$ and $\dot H(\xi)$ provided by Propositions~\ref{hEstimateProp} and~\ref{dotHEstimateProp}.
\begin{lemma}\label{bDboundLemma}
For $\xi$ in $\sB^\bbox$,
\begin{align}\label{exampleExplBd}
&\lV\bD(\xi) \rV_{H^k(S^1)} \leq C(\lV \xi \rV_{\sH^{k-1}}),\\
&\lV\bD(\xi) \rV_{H^{k+1}(S^1)} + \lV \dot \bD(\xi) \rV_{H^k(S^1)}
\leq C(M),
\end{align}
and the same bounds hold with $(\bD(\xi))^{-1}$ in place of $\bD(\xi)$.
\end{lemma}
Using the above two lemmas, we easily establish the following.
\begin{proposition}\label{BasicRBound}
For the map $R(\cdot)$ of Definition~\ref{DEcommDefns}, for any integer $s$ with $0\leq s \leq k$ we have
\begin{align}
R(\cdot) &: \sB^\bbox \to B(\sH^s_\dagger,\sH^s_\dagger),
\end{align}
and for $\xi_\dagger$ in $\sH^s_\dagger$ and $\uxi$ in $\sB^\bbox$,
\begin{align}
\lV R(\uxi)\xi_\dagger \rV_{\sH^s_\dagger} &\leq C(M)\lV\xi_\dagger\rV_{\sH^s_\dagger} .
\end{align}
\end{proposition}

\subsubsection{Defining $\dcF(\xi)$ for pinches $\delta_\Gamma \geq 0$}\label{pinchzeroextsection}

We would like to define $\dcF(\xi)$, notably the component $\dcF_1(\xi)$, for $\xi$ in $\sB$. However, the associated interface $\Gamma(t)$ to $\xi$ in $\sB$ has pinch zero at $t=0$, and currently some of the constituent terms of $\dcF_1(\xi)$ are only defined if the pinch is positive. With the following remark, we break up $\dcF_1(\xi)$ to isolate the terms at the root of the problem.
\begin{remark}\label{dcF1forBcircDefn}
For appropriate $\xi$, we will take
\begin{align}\label{dcF1defn}
    \dcF_1(\xi) &= \begin{pmatrix}
    \mathring{\cF}(\xi) \\
    0_{(2\times 1)}
    \end{pmatrix} ,
\end{align}
where $\mathring{\cF}(\xi)$ is composed of three maps, to be explained:
\begin{align}\label{mathrcFdef}
\mathring \cF(\xi) &=  \mathring \cF_-(\xi) + \mathring \cF_+(\xi)+
\mathring \cF_\sN(\xi) .
\end{align}
We can immediately define $\cF_-(\xi)$ and $\cF_+(\xi)$ for $\xi$ in $\sB$, but we only define $\cF_\sN(\xi)$ at first for $\xi$ in $\sB^\wbox$. Recalling the nonrigorous derivation of $\dcF_1$ in Section~\ref{thewavesystemsubsection}, we take
\begin{align}
 \mathring \cF_+(\xi) &=\bR^{11}(\xi)\,N\cdot  P^+_\grad(\xi) e_2 + \left(\frac{N\cdot U_\theta}{|X_\theta|} \bigtau \cdot P^+_\grad(\xi)  - N \cdot \dot P^+_\grad(\xi) \right) e_2  &&(\xi\in\sB),\label{ringcFpmDef2}\\
 \mathring \cF_-(\xi) &=\bR^{11}(\xi) \bE(\xi)  P^-_\grad(\xi)- [\del_t;\bE](\xi)P^-_\grad(\xi) -\bE(\xi) \dot P^-_\grad(\xi)   &&(\xi\in\sB),\label{ringcFpmDef1}\\
  \mathring \cF_\sN(\xi) &=\bR^{11}(\xi) \,\sN(\xi) e_2 -
 \dot\sN(\xi)e_2   &&(\xi\in\sB^\wbox), \label{ringcFsNDef} \\
\shortintertext{where}
\sN(\xi) &= \half \Nres(\xi) | H(\xi)|^2    &&(\xi\in\sB^\wbox),\label{sNdefn1}\\
\dot\sN(\xi) &= \Nres(\xi) \, (H(\xi) \cdot \dot H(\xi))  +\half [\del_t;\Nres](\xi) | H(\xi)|^2    &&(\xi\in\sB^\wbox),\label{sNdefn2}
\end{align}
where $P^\pm_\grad(\xi)$, $R^{11}(\xi)$, $\bE(\xi)$, and $\Nres(\xi)$ are defined above, and $\dot P^+_\grad(\xi)$, $\dot P^-_\grad(\xi)$, $[\del_t;\bE](\xi)$, and $[\del_t;\Nres](\xi)$ are defined in Section~\ref{timederivdefnsection}. In particular, note we only defined $\sN(\xi)$ and $\dot\sN(\xi)$ for $\xi$ in $\sB^\wbox$, due to the limitations of $\Nres(\xi)$ (see Definition~\ref{Nresdefn} and Remark~\ref{blowupreminder}).
\end{remark}

Let us explain the rough idea behind the work-around which allows us to define $\mathring \cF_\sN(\xi)$, and thus $\dcF_1(\xi)$, for $\xi$ in $\sB$. Although $\Nres(\xi)$ itself is not uniformly bounded in standard operator norms as the pinch $\delta$ tends to zero, in \eqref{sNdefn1} and \eqref{sNdefn2}, it acts on quantities, such as $|H(\xi)|^2$, that vanish at the pinch, so the composition ultimately stays uniformly bounded. The required weighted bounds (in particular Proposition~\ref{hEstimateProp}) are proved in the next section. Using this uniform control, we approximate a given $\xi$ with zero pinch by $\xi_n$ with corresponding pinches $\delta_n>0$, pass to the limit, and thus extend $\mathring \cF_\sN(\xi)$ continuously from $\sB^\wbox$ to $\sB^\bbox(0)$, hence to all $\xi\in\sB^\bbox$.

Though the following lemma relies heavily on the machinery of Section~\ref{vacuuminterfaceestimatessection}, we present it earlier to motivate the technical work to come. It gives the finishing touch needed to define the source term $\dcF(\xi)$ driving the Lagrangian wave system.
\begin{lemma}\label{sNiExt}
For the map $\xi\mapsto\mathring \cF_\sN(\xi)$ defined for $\xi$ in $\sB^\wbox$ by \eqref{ringcFsNDef}, we have a continuous extension
\begin{align}
\mathring \cF_\sN : \sB^\bbox \to H^k(S^1) .
\end{align}
Additionally, the resulting map satisfies the bounds below for $\xi$ and $\uxi$ in $\sB^\bbox$:
\begin{align}
\lV \mathring \cF_\sN(\xi) \rV_{H^k(S^1)} &\leq C(M)   , \label{sNiUnifBdCopy}\\
 \lV \mathring \cF_\sN(\xi)-\mathring \cF_\sN(\uxi) \rV_{H^2(S^1)} &\leq C \lV \xi - \uxi \rV_{\sH^2} .\label{sNiLipBdCopy}
\end{align}
\end{lemma}
\begin{proof}
Let us begin by stating the key result of Lemma~\ref{sNilemma}, which is that
\begin{align}
\mathring \cF_\sN : \sB^\wbox \to H^k(S^1) ,
\end{align}
and that we have the estimates \eqref{sNiUnifBdCopy} and \eqref{sNiLipBdCopy} for arbitrary $\xi$ and $\uxi$ in $\sB^\wbox$. To extend $\mathring \cF_\sN$ to $\sB^\bbox$, we rely on the density of $\sB^\wbox$ in $\sB^\bbox$.

Given $\xi$ in $\sB^\bbox(0)$, let $\{\xi_n\}_{n\geq 1}$ be a sequence in $\sB^\wbox$ converging to $\xi$ in $\sH^k$ norm. From \eqref{sNiLipBdCopy} we find that the sequence \(\{\mathring \cF_\sN(\xi_n)\}_{n\geq1}\) converges to a limit $Y$ in $H^2(S^1)$. We define $\mathring \cF_\sN(\xi) = Y$ and proceed to show that in fact $Y$ is in $H^k(S^1)$.

Using \eqref{sNiUnifBdCopy}, we find by Banach--Alaoglu that a subsequence \(\{\mathring \cF_\sN(\xi_{n_m})\}_{m\geq1}\) converges weakly to a limit in $H^k(S^1)$. It follows that this limit can only be $Y$, and so $Y$ is in $H^k(S^1)$. Thus we have a well-defined extension $\mathring \cF_\sN : \sB^\bbox \to H^k(S^1)$. It is not hard to show from the construction that the extension inherits the bounds \eqref{sNiUnifBdCopy} and \eqref{sNiLipBdCopy}.
\end{proof}
\begin{definition}\label{finalDefdcF1}
Consider $\xi$ in $\sB$. We define $\mathring \cF(\xi)$ by \eqref{mathrcFdef}, using \eqref{ringcFpmDef2}, \eqref{ringcFpmDef1}, and Lemma~\ref{sNiExt} to define $\mathring \cF_+(\xi)$, $\mathring \cF_-(\xi)$, and $\mathring \cF_\sN(\xi)$, respectively. Here,  $P^\pm_\grad(\xi)$, $R^{11}(\xi)$, and $\bE(\xi)$ are defined earlier in this section, and we take $\dot P^+_\grad(\xi)$ from Definition~\ref{dotPplusdefn}, $\dot P^-_\grad(\xi)$ from Definition~\ref{dotPintdefn}, $[\del_t;\Nres](\xi)$ from Definition~\ref{DtNCommDefns}, and $[\del_t;\bE](\xi)$ as defined in Proposition~\ref{bEcommBounds}.

We then define $\dcF_1(\xi)$ by \eqref{dcF1defn}, and, for $\dcF_2(\xi)$ of Definition~\ref{dcF2Defn}, we define $\dcF(\xi)=(\dcF_1(\xi),\dcF_2(\xi),0,0,0,0)$, typically regarded as a column vector.
\end{definition}

This concludes the rigorous definition for $\xi$ in $\sB$ of the terms $\bJ(\xi)$, $\cR(\xi)$, and $\dcF(\xi)$ appearing in the Lagrangian wave system.


\section{Pinch domain and above-interface estimates}\label{vacuuminterfaceestimatessection}

Now we turn to the main technical challenge: to obtain uniform elliptic estimates for $\delta$-pinched domains which hold all the way to complete closure $(\delta=0)$.

This is necessary to tie up the loose thread of Lemma~\ref{sNilemma}, which is needed to define the source term $\dcF(\xi)$ in the Lagrangian wave system when the interface self-intersects; it is also essential for our local existence argument and construction of splash--squeeze solutions.

The main objectives for the section are as follows:
\begin{enumerate}[(1)]
\item introduce the comparison domain $\Sigma_+(\delta)$, weighted spaces, and general weighted elliptic estimates;
\item derive estimates for vacuum-centric maps, e.g. the weighted bound for $h(\Gamma)$ of Proposition~\ref{hEstimateProp};
\item establish bounds for Dirichlet-to-Neumann maps, including the cancellation bound for $\Nres(\xi)$ of Proposition~\ref{DtoNCancellation}.
\end{enumerate}

\subsubsection*{Auxiliary Lagrangian variable domain for the vacuum}

To carry out our analysis, we must estimate an external magnetic field $h(\Gamma)$ defined on a vacuum domain $\cV$. For the local existence argument, we also need to compare a field $h(\Gamma)$ to another field, say $h(\underline \Gamma)$ with vacuum domain $\underline \cV$. Since these fields have different domains, we must transfer them to a common one to make these comparisons. For unknowns like $u$ defined in the plasma region $\Omega$, this type of comparison is handled naturally on the fixed reference domain $\Sigma$ by using our Lagrangian coordinate system. Given that our formulation already uses Lagrangian coordinates for the plasma, it is natural to adopt a similar framework for the vacuum region, transferring external magnetic fields to a coordinate domain lying above $\Sigma$, which plays a kind of analogous Lagrangian role, but for the vacuum.

Consider
\begin{align}\label{SigmaPlusDefn}
    \Sigma_+(\delta) = \{(\theta,\psi) : \theta \in S^1, \ 0 < \psi < \delta + \cos^2 \theta  \}.
\end{align}
The set $\Sigma_+(\delta)$ serves as a vacuum counterpart to our Lagrangian domain $\Sigma$ for the plasma.

\begin{figure}[!htbp]
    \centering
    
    \begin{minipage}[t]{0.48\textwidth}
        \centering
        
        \begin{tikzpicture}
              \begin{axis}[
    height=4.5cm,
    width=10cm,
    axis lines=none,
    domain=-pi:pi,
    samples=200,
    clip=false,
    enlargelimits=0.2
  ]

  \begin{scope}
    \fill[gray!20] plot[domain=-pi:pi, samples=200] (\x,{0.2+2*(cos(deg(\x)))^2})
          -- (pi,0) -- (-pi,0) -- cycle;
  \end{scope}

  \addplot [thick, black, domain=-pi:pi, samples=200] {0.2+2*(cos(deg(x)))^2};

  \addplot [thick, black, domain=-pi:pi, samples=2] {0};
 
  \draw[thick, black] 
    (axis cs:-pi,-0.1) -- (axis cs:-pi,0.1)
    (axis cs:pi,-0.1) -- (axis cs:pi,0.1);

  \draw[thick, black, dashed]
    (axis cs:pi/2-0.4,0.18) -- (axis cs:pi/2+0.4,0.18);
    
  \node[left] at (axis cs:-pi,0) {$-\pi$};
  \node[right] at (axis cs:pi,0) {$\pi$};
  \node[above] at (axis cs:0,0) {$\Sigma_+(\delta)$};
  \node[above right] at (axis cs:pi/2+0.3,-0.05) {$\delta$}; 
  
  \end{axis}
        \end{tikzpicture}
        \caption{}
        \label{fig:Sigmadelta}
    \end{minipage}
    \begin{minipage}[t]{0.4\textwidth}
        \centering
        \includegraphics[width=\linewidth]{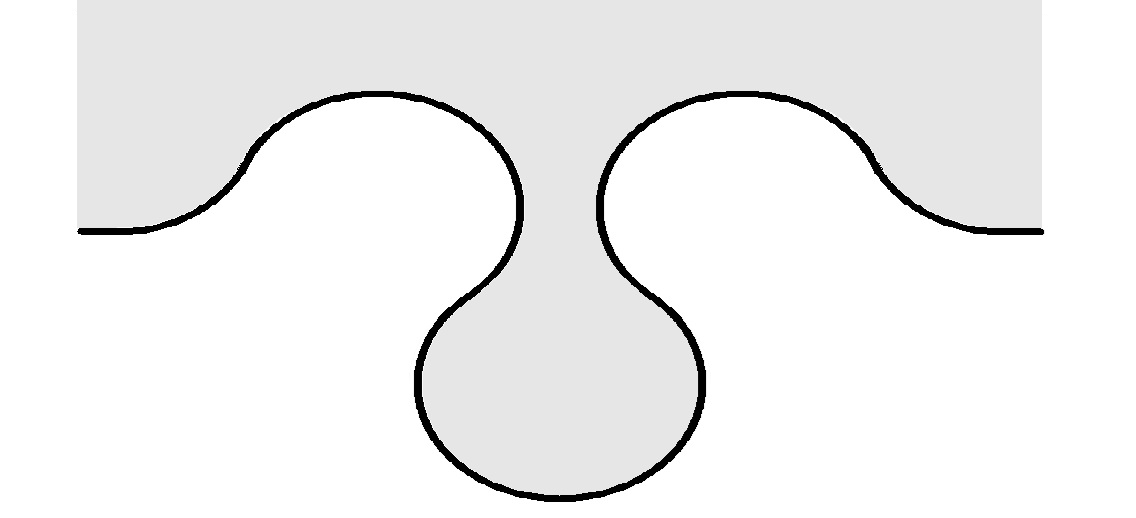}

    \begin{tikzpicture}[overlay, remember picture]
    \node at (2.9,2.1) {$\Gamma$};
    \node at (0,1.2) {$\Omega_+$};
    \end{tikzpicture}
        \caption{}
        \label{fig:Omegaplus}
    \end{minipage}
\end{figure}
Note that $\Sigma_+(\delta)$ has $\delta$-pinches at $\theta = \pm \pi/2$, precisely where $X_0(\theta)$ gives the splash point for $\Gamma_0$, and that it sits above $\Sigma = S^1\times[-1,0]$, analogous to $\Omega_+$ lying above $\Omega$. The pinch of $\Sigma_+(\delta)$ corresponds to the pinch of $\Omega_+$. For a variation on the domain $\Sigma_+(\delta)$ analogous to $\cV$, which has holes in it due to the walls $\mathcal{W}=\cW_1\cup\cW_2$, we use the domain $\Sigma^\circ_+(\delta)$ defined below, where we denote $a_\star=(0,(1+\delta)/2)$ and $a_\infty=(\pi,(1+\delta)/2)$:
\begin{align}
\Sigma^\circ_+(\delta) &= \{a\in \Sigma_+(\delta):|a-a_\star|> 10^{-2}\mbox{ and }|a-a_\infty|> 10^{-2}\}, \\
\tilde \cW_1 &= \{|a-a_\infty|=10^{-2}\},\label{tildeW1}\\
\tilde \cW_2 &= \{|a-a_\star|=10^{-2}\}.\label{tildeW2}
\end{align}

In our analysis, we frequently use mappings between $\Sigma_+(\delta)$ and the domain $\Omega_+$. A minor complication is that ``flattening'' $\Omega_+$ in this way leads to a map which is a double cover\footnote{See Proposition~\ref{extensionProp2}.} of $\Omega_+$ rather than a bijection. Despite this, both pictures arise naturally in the analysis, and so we define weighted spaces for both settings.

\subsection{Weighted Sobolev spaces and elliptic estimates in pinch domains}\label{weightedspacessection}

We now discuss weighted function spaces for $\Sigma_+(\delta)$ as well as $\Omega_+$, $\cV$, and $\Gamma$. Our weight functions, indexed by $\delta \geq 0$ and integer $m$, are chosen such that their order of magnitude is $\delta^{-m}$ in regions where a domain has pinch $\delta$ and $1$ away from these regions.

Taking a fixed smooth cutoff $\chi$ whose support contains the splash point $p_\splash$, we define the weights
\begin{align}\label{weightfcndef}
& &    \mu_\delta(x) &= \frac{1}{(\delta^2 + |x-p_\splash|^4 )^{1/2}}\chi(x) + 1-\chi(x) & &(x \in S^1\times\R), \\
& &    M_\delta(\theta) &= \frac{1}{\delta+\cos^2\theta} & & ( \theta \in S^1 ).
\end{align}

We develop the following weighted Sobolev norms for functions that, near the pinch, are very small if $m$ is positive and possibly very large if $m$ is negative. Before defining weighted norms on $\Sigma_+(\delta)$, $\Omega_+$, and $\cV$, we define simpler integer order Sobolev spaces for the boundaries $S^1$, $\Gamma$, and $\del\cV$.

\begin{definition}\label{weightedSpaceDefs}
Consider a fixed $\delta$ in $[0,\delta_0]$ and $\xi$ in $\sB^\bbox(\delta)$ with associated $\Gamma$ parametrized by arc-length by $\gamma:S\to \Gamma$. Let $j$ and $m$ be integers, with $0\leq j \leq k+2$. We define
\begin{align}
    \lV f \rV^2_{H^j (\Gamma, \delta, m)} &= \sum^j_{i=0}\lV(\mu_\delta \circ \gamma)^m \, \del^i_s (f\circ \gamma)\rV^2_{L^2(S)} &&(f:\Gamma\to\R^n), \label{Gammaweightednorm} \\
    \lV f \rV^2_{H^j (\del\cV, \delta, m)} &= \lV f \rV^2_{H^j (\Gamma, \delta, m)} +\lV f \rV^2_{H^j(\mathcal{W})} &&(f:\del\cV\to\R^n), \\
    \lV F \rV^2_{H^j(S^1,\delta, m)}&=\sum^j_{i=0}\lV M^m_\delta \, \del^i_\theta F \rV^2_{L^2(S^1)} &&(F:S^1\to\R^n),\\
   \lV F \rV^2_{H^j(\del \Sigma_+,\delta, m)}&= \left\lV F|_{\psi=0} \right\rV^2_{H^j(S^1,\delta, m)}+
   \left\lV F|_{\psi=\delta+\cos^2(\cdot)}\right\rV^2_{H^j(S^1,\delta, m)} &&(F:\del\Sigma_+(\delta)\to\R^n).\\
   \lV F \rV^2_{H^j(\del \Sigma^\circ_+,\delta, m)}&= \lV F \rV^2_{H^j(\del \Sigma_+,\delta, m)} +\lV F \rV^2_{H^j(\del \tilde \cW_1 \cup \del \tilde \cW_2)}   &&(F:\del\Sigma^\circ_+(\delta)\to\R^n).
\end{align}
\end{definition}
Regarding the domains $\Omega_+$, $\cV$, and $\Sigma_+(\delta)$, which require half-integer order derivative Sobolev spaces, we define the corresponding norms differently. Instead, we localize to a $\delta$-dependent collection of squares $\{Q_\nu\}$ with disjoint interiors satisfying $\cup_{\nu} Q_\nu = S^1\times\R$. Near $p_\splash$, similar to a Whitney decomposition, the squares are selected to match the ``local pinch scale'' for interfaces $\Gamma$ with pinch $\delta$. Due to the geometry, this amounts to arranging for the sidelength $l_\nu$ of each square $Q_\nu$
\begin{align}\label{sidelengthrule}
c \ell_\nu \leq \max\left(\delta,\left(\operatorname{dist}(Q_\nu,p_\splash)\right)^2\right) \leq C \ell_\nu .
\end{align}
One can form a collection $\{Q_\nu\}$ for which \eqref{sidelengthrule} holds for all $\nu$
by starting with a uniform grid of candidate squares, and discarding a given candidate, say $Q_{\textrm{cand}}$, precisely when
\begin{align}
\textrm{sidelength}(Q_{\textrm{cand}})>c'\max\left(\delta,\left(\operatorname{dist}(Q_\nu,p_\splash)\right)^2\right).
\end{align}
One then replaces $Q_{\textrm{cand}}$ by four new candidate squares produced by bisecting  $Q_{\textrm{cand}}$. Running this test on each starting square of the uniform grid, iterating the above process, and keeping all squares at which bisection halts yields the collection $\{Q_\nu\}$. Note that this collection is finite if $\delta>0$ and infinite if $\delta=0$.
\begin{remark}\label{finiteoverlap}
The collection $\{Q_\nu\}$ produced in the above way satisfies some standard properties for Whitney-type decompositions which we state here without proof.
\begin{itemize}
\item If a pair of squares $Q_\nu$ and $Q_\mu$ are adjacent, one has $c \ell_\mu \leq \ell_\nu \leq C \ell_\mu$.
\item For a given $Q_\nu$, the number of $Q_\mu$ for which the dilates (about centers) $1.1 Q_\nu$ and $1.1 Q_\mu$ intersect is bounded by a constant $C$.
\end{itemize}
\end{remark}

\begin{definition}\label{whitneyweightdef}
Fix $\delta$ in $[0,\delta_0]$. Let $\{Q_\nu\}$ be the disjoint Whitney-type cover of $S^1\times\R$ satisfying \eqref{sidelengthrule} for all $\nu$ as described above. Take $\xi$ in $\sB^\bbox(\delta)$ with associated $\Gamma$, $\cV$, and $\Omega_+$.

For real $s$ with $0\le s \le k+5/2$, integer $m$, we define
\begin{align}
&& \|f\|_{H^s(\Omega_+,\delta,m)}^2
&= \sum_\nu\ell_\nu^{-2m}\,\|f\|_{H^s(1.01Q_\nu\cap\Omega_+)}^2 &(f:\Omega_+\to\R^n),\label{OmegaPweightednorm}\\
&& \|f\|_{H^s(\cV,\delta,m)}^2 
&= \sum_\nu\ell_\nu^{-2m}\,\|f\|_{H^s(1.01Q_\nu\cap \cV)}^2 &(f:\cV\to\R^n).
\end{align}

Similarly, for the Lagrangian domains, we construct a collection of squares $\{\tilde Q_\nu\}$ with sidelengths $\tilde \ell_\nu$ whose union is $S^1\times\R$ such that the analogue of \eqref{sidelengthrule} holds with $\{(\pi/2,0),(-\pi/2,0)\}$ in place of $p_\splash$. For $s\ge 0$ and integer $m$, we define
\begin{align*}
&& \|F\|_{H^s(\Sigma_+,\delta,m)}^2
&= \sum_\nu\tilde \ell_\nu^{-2m}\,\|F\|_{H^s(1.01\tilde Q_\nu\cap\Sigma_+(\delta))}^2 & (F:\Sigma_+(\delta)\to\R^n), \\
&& \|F\|_{H^s(\Sigma^\circ_+,\delta,m)}^2
&= \sum_\nu\tilde \ell_\nu^{-2m}\,\|F\|_{H^s(1.01\tilde Q_\nu\cap\Sigma^\circ_+(\delta))}^2 & (F:\Sigma^\circ_+(\delta)\to\R^n).
\end{align*}
\end{definition}

\begin{remark}
Let us note that in the case that $s$ is an integer and $\delta>0$ the norm of \eqref{OmegaPweightednorm} is equivalent to the natural analogue of \eqref{Gammaweightednorm}:
\begin{align}
c\|f\|_{H^s(\Omega_+,\delta,m)}^2\leq
\sum_{|\beta|\leq s}\| \mu^m_\delta \del^\beta f \|^2_{L^2(\Omega_+)} 
\leq C\|f\|_{H^s(\Omega_+,\delta,m)}^2.
\end{align}
\end{remark}

Now we move on to the main general weighted elliptic estimate for exterior domains pinched by a given interface $\Gamma$. The key is that we are able to prove bounds where the constants in the right-hand sides do not depend on the pinch $\delta$.

\begin{proposition}\label{1stWeightedEstimate}
Consider half-integer $s$ with $1/2\le s \le k + 1/2$. For any $m \geq 0$, there is an $m'\geq m $ such that the following holds. Take any $\delta\in[0,\delta_0]$ and $\xi$ in $\sB^\bbox(\delta)$ with corresponding vacuum $\cV$, interface $\Gamma$, and upper region $\Omega_+$ (see Figure~\ref{fig:Omegaplus}). We then have the estimates below for a function $f$ in $H^{s+2}(\Omega_+)$ or $H^{s+2}(\cV)$:
\begin{align}
\lV f \rV_{H^{s+2}(\Omega_+,\delta,m)} &\leq C(M) \left( \lV \Delta f \rV_{H^s(\Omega_+,\delta,m')}
+
\lV f \rV_{H^{s+3/2}(\Gamma,\delta,m')}
\right) ,\label{1stWeightedEstimateBd2}\\
\lV f \rV_{H^{s+2}(\cV,\delta,m)} &\leq C(M) \left( \lV \Delta f \rV_{H^s(\cV,\delta,m')}
+
\lV f \rV_{H^{s+3/2}(\del \cV,\delta,m')}
\right) ,
\label{1stWeightedEstimateBd1} \\
\lV \grad f \rV_{H^{s+1}(\Omega_+,\delta,m)} &\leq C(M) \left( \lV \Delta f \rV_{H^s(\Omega_+,\delta,m')}
+
\lV \del_n f \rV_{H^{s+1/2}(\Gamma,\delta,m')}
\right) ,
\label{1stWeightedEstimateBd2b}\\
\lV \grad f \rV_{H^{s+1}(\cV,\delta,m)} &\leq C(M) \left( \lV \Delta f \rV_{H^s(\cV,\delta,m')}
+
\lV \del_n f \rV_{H^{s+1/2}(\del \cV,\delta,m')}
\right) ,
\label{1stWeightedEstimateBd1b}
\end{align}
where $M$ can be replaced by $\lV X\rV_{H^{k+1}(S^1)}$ in the case $s\leq k-1/2$. 
\end{proposition}
\begin{proof}
    Let us sketch only the proof of \eqref{1stWeightedEstimateBd2}, since the proofs of the other bounds are similar.
    
    Decomposing the weighted norm of $f$ as described by Definition~\ref{whitneyweightdef} in terms of localized norms over the collection $\{Q_\nu\}$, one finds that the estimate of Corollary~\ref{buildingcor} can be used to give a satisfactory upper bound. This relies in particular on the fact that the intersections $Q_\nu\cap\Omega_+$ can be shown to satisfy the hypotheses for $\mathpzc{V}$ in the statement of the corollary.
    
    One also uses the upper bound discussed in Remark~\ref{finiteoverlap} on the number of dilates $1.1Q_\mu$ intersecting a given $1.1 Q_\nu$. This controls the ``overcount'' produced by summing up the contributions from the right-hand side of \eqref{ScaledSquareEst} for overlapping regions. The additional power of $\ell_\nu$ from the corollary leads to a potential increase in the weight index from $m$ to $m'$.
    
    Without too much trouble, one also verifies that the weighted boundary norms of Definition~\ref{weightedSpaceDefs} give an upper bound for the weighted sum involving the norms over the various $1.1Q_\nu \cap\Gamma$ which results from the above process. Carefully patching together all the estimates yields \eqref{1stWeightedEstimateBd2}.

\end{proof}

\begin{remark}\label{WeightRemark}
Notice that the bounds given in Proposition~\ref{1stWeightedEstimate} are almost identical to basic elliptic estimates in unweighted Sobolev spaces. When applying these bounds, we pay the price that we have a potentially ``weaker'' weight index $m$ in the left-hand side than the $m'$ in the upper bound. This means that to control some quantity $f:\cV\to\R$, we may need tighter control on $\Delta f$, in a norm with a more singular weight.

Despite this increase in weight index, it is not necessary in our construction to keep track of the precise dependence of $m'$ on $m$. The same general principle applies to the other weighted estimates we prove. Our weighted bounds always boil down to bounds on $h$, which decays rapidly enough in the pinched region to be controlled in the $H^s(\cV,\delta,m)$ norm for any integer $m\geq 0$ (Proposition~\ref{hEstimateProp}).
\end{remark}

Now we collect several useful bounds for our weighted Sobolev spaces that will be used throughout our arguments.

\begin{proposition}\label{standardWeightedSobNormEstims}
Fix $\delta$ in $[0,\delta_0]$. For the norms as given by Definition~\ref{weightedSpaceDefs}, for $\xi$ in $\sB^\bbox(\delta)$ with corresponding $\Gamma$ and $\cV$, the following statements are true for any integers $j$ and $m$, with $0\le j \le k+2$.
\begin{enumerate}[(i)]
    \item Given a function $f$ in $H^{j+1/2}(\cV)$,
\begin{align}
    \lV f \rV_{H^j(\del\cV,\delta,m)} \leq C \lV f \rV_{H^{j+1/2}(\cV,\delta,m)} .
\end{align}
    \item Given $F_+$ in $H^{j+1/2}(\Sigma_+(\delta))$ and $F(\theta)=F_+(\theta,0)$,
\begin{align}
    \lV F \rV_{H^j(S^1,\delta,m)} \leq C \lV  F_+  \rV_{H^{j+1/2}(\Sigma_+,\delta,m)} .
\end{align}
    \item Given $f$ in $H^j(\Gamma)$ and $F$ in $H^j(S^1)$ with $F(\theta) = f(X(\theta))$,
\begin{align}
    c \lV F \rV_{H^j(S^1,\delta,m)} \leq \lV f \rV_{H^j(\Gamma,\delta,m)} \leq C \lV F \rV_{H^j(S^1,\delta,m)}.
 \end{align}
    \item For $s \geq 3/2$, and any integers $m_1, m_2\geq 0$, there is an $m_3 \geq m_1+m_2$ such that, for $F_1,F_2$ in $H^s(\Sigma_+(\delta))$,
\begin{align}
\lV F_1 F_2 \rV_{H^s(\Sigma_+,\delta,m_1)} \leq C \lV F_1 \rV_{H^s(\Sigma_+,\delta,-m_2)} \lV F_2 \rV_{H^s(\Sigma_+,\delta,m_3)} .
\end{align}
 \end{enumerate}
\end{proposition}
\begin{proof}
The above assertions, which mirror standard bounds for unweighted Sobolev spaces, can be proved by decomposing the domains of consideration with the collections $\{Q_\nu\}$ and $\{\tilde Q_\nu\}$ of Definition~\ref{whitneyweightdef} and localizing to bounds in small squares in pinched regions. One then patches the estimates together as in the proof of Proposition~\ref{1stWeightedEstimate}.
\end{proof}

\subsubsection*{Mapping between $\Sigma_+(\delta)$ and $\Omega_+$}

Now we return to the idea of mapping $\Sigma_+(\delta)$ to $\Omega_+$ (and $\Sigma^\circ_+(\delta)$ to $\cV$), as discussed at the beginning of the section. Referring back to Figures~\ref{fig:Sigmadelta} and~\ref{fig:Omegaplus}, note that $X$ gives a mapping onto $\Gamma$ of the domain $S^1$, which we can also think of as the bottom boundary of $\Sigma_+(\delta)$. In the proposition below, we show how a useful mapping from $\Sigma_+(\delta)$ onto $\Omega_+$ can be obtained as an extension of $X$ to the domain $\Sigma_+(\delta)$. This proposition is proved in Section~\ref{auxiliaryestimates}.
\begin{myprop}{\ref{extensionProp2}}
For $0<\delta\leq \delta_0$ there is an extension operator
\begin{align}
\cE_\delta : \{X:\xi \in \sB^\bbox(\delta)\}  \to \{\mbox{\emph{Continuous maps from}} \ \Sigma_+(\delta) \ \mbox{to} \ \Omega_+\cup\{i\infty\}\} ,
\end{align}
where, for $\xi$ in $\sB^\bbox(\delta)$ with corresponding $X$, $\cE_\delta(X)(\theta,0)=X(\theta)$. We also have that $\cE_\delta(X)$

\begin{enumerate}[(i)]

\item maps the upper boundary as well as the lower boundary of $\Sigma_+(\delta)$ onto $\Gamma$,

\item maps $a_\star = (0,(1+\delta)/2)$ to $z_\star=i$ and $a_\infty = (\pi,(1+\delta)/2)$ to $i\infty$,

\item provides a two-to-one covering map from $\Sigma_+(\delta)\setminus \{a_\star ,a_\infty\}$ to $\Omega_+\setminus\{z_\star\}$ and a two-to-one covering map from $\Sigma^\circ_+(\delta)$ to $\cV$,

\item maps $\tilde \cW_i$ (defined in \eqref{tildeW1} and \eqref{tildeW2}) onto $\cW_i$ for $i=1,2$, and

\item satisfies the following for an $m\geq 0$ dependent only on $k$, where we denote 
\[\Sigma^\star=\Sigma_+(\delta)\cup\{|a-a_\star|\leq 10^{-1}\},\]
with the corresponding weighted norm defined in the obvious way, and we take any integer $j$ with $0\leq j \leq k$:
\begin{align}\label{EdeltaBd}
\lV \cE_\delta(X)\rV_{H^{j+5/2}(\Sigma^\star,\delta,-m)} +\lV (\grad\cE_\delta(X))^{-1}\rV_{H^{j+3/2}(\Sigma^\star,\delta,-m)}\leq C \left(\lV X \rV_{H^{j+2}(S^1)}\right).
\end{align}
Additionally, for a pair $X,\uX$ corresponding to $\xi,\uxi$ in $\sB^\bbox(\delta)$,
\item  we have
\begin{align}
& & \cE_\delta(X)(a) &= \cE_\delta(\uX)(a)  &   \left(|a-a_\star|\leq 10^{-1} \mbox{ or }|a-a_\infty|\leq 10^{-1} \right),
\end{align}

\item and for any integer $j$ with $0\leq j \leq k$ (and $m$ as before),
 \begin{align}
 \lV \cE_\delta(X)-\cE_\delta(\uX) \rV_{H^{j+5/2}(\Sigma_+,\delta,-m)} \leq C \lV X-\uX \rV_{H^{j+2}(S^1)} .
 \end{align}

\end{enumerate}
\end{myprop}

We now have in place a suitable framework in which we will compare vector fields (such as $h$) that live in exterior domains associated to different interfaces, as discussed at the beginning of this section. These comparisons will typically be made by applying the following proposition, which establishes weighted estimates of a similar flavor to those of Proposition~\ref{1stWeightedEstimate}, but for vector fields $V$ solving div-curl-type problems in the domain $\Sigma_+(\delta)$.

\begin{proposition}\label{LagrVacDivCurlEst}
Fix $\delta$ with $0<\delta\leq \delta_0$. Given $\xi$ in $\sB^\bbox(\delta)$, for the associated $X$, let us denote the extension into $\Sigma_+(\delta)$ by $X_+= \cE_\delta(X)$. Let $\sP_\delta$ be the transformation defined in \eqref{sPmapdefn} in the proof of Proposition~\ref{extensionProp2}, mapping $\Sigma_+(\delta)$ onto $\Sigma^\sharp_+(\delta)$.

In the following, let us take $\Sigma^*_+$ to be either the domain $\Sigma_+(\delta)$ or $\Sigma^\circ_+(\delta)$. Assume we are given $V$, $\rho$, and $\varpi$ defined on $\Sigma^*_+(\delta)$, and $\eta$ defined on $\del\Sigma^*_+(\delta)$, where $V=V^\sharp \circ \sP_\delta$, $\rho = \rho^\sharp \circ \sP_\delta$, $\varpi = \varpi^\sharp \circ \sP_\delta$, and $\eta = \eta^\sharp \circ \sP_\delta$ for $V^\sharp$, $\rho^\sharp$, and $\varpi^\sharp$ continuous on $\sP_\delta(\Sigma^*_+(\delta))$ and $\eta^\sharp$ continuous on $\sP_\delta(\del\Sigma^\sharp_+(\delta))$. 
Moreover, let us assume that $\supp(\rho),\supp(\varpi)\subset \overline{\Sigma^\circ_+(\delta)}$.

Suppose $V$ solves the problem\footnote{In the case $\Sigma^*_+=\Sigma_+(\delta)$, we of course exclude the point $a_\infty$ from $\Sigma^*_+=\Sigma_+(\delta)$ in \eqref{gendivcurlprob} since $X_+$ maps this point to $i\infty$.}
\begin{align}\label{gendivcurlprob}
(a \in \Sigma^*_+) \qquad &
\left\{
    \begin{aligned}
        \grad \cdot (\, \cof (\grad X^t_+) V ) &= \rho \\
        \grad^\perp\cdot (\, \grad X^t_+ V )   &= \varpi \, \Lcm
    \end{aligned}
\right. \\
(a \in \del\Sigma^*_+)  \qquad &  \ \big\{\,(n \circ X_+ ) \cdot V = \eta , \\
 \mbox{and, if $\Sigma^*_+=\Sigma^\circ_+(\delta)$,}\qquad 
 &\int_{\tilde\cW_i} (V \grad X_+) \cdot d\vec r = \mu_i 
 \qquad(i=1,2),\label{lineIntegCond2} 
\end{align}
where $n$ is the inward unit normal for $\cV$ and 
in \eqref{lineIntegCond2} we take counter-clockwise paths around $\tilde \cW_1$ and $\tilde \cW_2$.\footnote{In the case $\Sigma^*_+=\Sigma^\circ_+(\delta)$, the conditions 
\eqref{lineIntegCond2} are required to uniquely determine $V$, since the domain has holes. This is analogous to \eqref{IdealMHD6}, necessary for uniqueness of $h$.} Again, we use the convention $(\grad X_+)_{ij}=\del_{a_j}(X_+)_i$, where $a=(\theta,\psi)$.

Then for half-integer $s$ with $1/2 \leq s \leq k+1/2$ and any integer $m$, there is an $m' \geq m$ such that a given solution $V$ to the problem \eqref{gendivcurlprob}--\eqref{lineIntegCond2} obeys the estimate below. 
\begin{align}\label{lagrDCest}
\lV V \rV_{H^{s+1}(\Sigma^*_+,\delta,m)} &\leq C \Bigg( \lV \rho \rV_{H^s(\Sigma^*_+,\delta,m')}+\lV \varpi \rV_{H^s(\Sigma^*_+,\delta,m')}
+\lV \eta \rV_{H^{s+1/2}(\del\Sigma^*_+,\delta,m')} \\
&\hspace{4cm}+\begin{cases}0 & \mbox{\emph{if}} \ \Sigma^*_+ = \Sigma_+(\delta) \\|\mu_1|+|\mu_2| & \mbox{\emph{if}} \ \Sigma^*_+ = \Sigma^\circ_+(\delta)\end{cases}\Bigg) .
\end{align}
Above, we have $C=C\left(\lV X\rV_{H^{k+2}(S^1)}\right)$ for $s=k+1/2$ and $C=C\left(\lV X\rV_{H^{k+1}(S^1)}\right)$ for $s\leq k-1/2$.

Moreover, the same bound holds when $V$ satisfies the system \eqref{gendivcurlprob}--\eqref{lineIntegCond2} with $n$ replaced by $\tau = -n^\perp$ and $V$ replaced by $V^\perp$ in \eqref{lineIntegCond2}.
\end{proposition}

\begin{proof}
By using the properties of $X_+=\cE_\delta(X)$, along with the construction from the proof of Proposition~\ref{extensionProp2}, we deduce from the fact that $V$, $\rho$, $\varpi$, and $\eta$ correspond to continuous functions in the $\Sigma^\sharp_+(\delta)$ picture that we can pull back the problem above on $\Sigma^*_+(\delta)$ to a corresponding problem on either $\Omega_+$ or $\cV$, where one solves for a vector field of the form $v=\grad \phi +\grad^\perp \varphi$. This leads to a Dirichlet problem for $\phi$ and a Neumann problem for $\varphi$, ultimately reducing the desired weighted estimates for $V$ to weighted estimates on $\phi$ and $\varphi$, which are provided by Proposition~\ref{1stWeightedEstimate}.
\end{proof}

\begin{remark}
The point of assuming $\supp(\rho),\supp(\varpi)\subset \overline{\Sigma^\circ_+(\delta)}$ is essentially to ensure that, in the case $\Sigma^*_+=\Sigma_+(\delta)$, upon pulling back $V$ to $\Omega_+$, the associated $\phi$ and $\varphi$ discussed in the proof are harmonic above the wall $\cW_1$, corresponding to the region where $X_+$ becomes unbounded.
\end{remark}

Now we record an auxiliary lemma linking vector fields with zero divergence and curl in $\Omega_+$ to the corresponding problem in $\Sigma_+(\delta)$.

\begin{lemma}\label{DivFreeCurlFreeConversion}
Fix $\delta$ with $0<\delta\leq \delta_0$. Consider $\xi$ in $\sB^\bbox(\delta)$ and define $X_+ = \cE_\delta(X)$. For the domain $\Omega^*_+$ below, we choose either $\Omega_+$ (in which case we define $\Sigma^*_+=\Sigma_+(\delta)$) or $\cV$ (in which case we define $\Sigma^*_+=\Sigma^\circ_+(\delta)$).

Suppose a vector field $v:\Omega^*_+\to\R^2$ solves the problem
\begin{align}\label{divcurlfreesys}
(x \in \Omega^*_+) \qquad &
\left\{
    \begin{aligned}
        \grad \cdot v &= 0 \\
        \grad^\perp\cdot v   &= 0 \, \Lcm
    \end{aligned}
\right. \\
(x \in \del \Omega^*_+ 
)  \qquad &  \ \big\{\, n \cdot v =  g , \\ \shortintertext{and in the case $\Omega^*_+=\cV$,}
(i=1,2) \qquad & \int_{\cW_i} v \cdot d\vec r = \upnu_i.
\end{align}
Then $V:\Sigma^*_+\to\R^2$, defined by $V = v\circ X_+$ satisfies the hypotheses of Proposition~\ref{LagrVacDivCurlEst}, in particular solving the system \eqref{gendivcurlprob}--\eqref{lineIntegCond2} for $\rho=\varpi=0$, $\eta=g\circ X_+$, and $\mu_i=2\upnu_i$.
\end{lemma}
\begin{proof}
With the help of the construction of $X_+=\cE_\delta(X)$ in the proof of Proposition~\ref{extensionProp2}, the above assertions are relatively easy to check. Regarding the property that $V=v\circ X_+$ satisfies \eqref{lineIntegCond2}, with $\mu_i = 2\upnu_i$, this follows as a consequence of the fact that each directed line integral over a given $\tilde \cW_i$ corresponds to a line integral in the $\Omega_+$ picture in which one loops around $\cW_i$ twice.
\end{proof}
\subsection{Magnetic field estimates}\label{magfieldestimatessection}

Now that the general weighted estimates are in place, we can apply them to prove various critically important bounds for maps such as $\xi \mapsto  H(\xi)$ and $\xi\mapsto   P^+_\grad(\xi)$ which we simply stated and used in the latter part of Section~\ref{lagrangianwavesystemsection} to define terms in the Lagrangian wave system that we intend to solve. We can also use our setup to compare vector fields associated to different exterior domains, establishing Lipschitz bounds on maps like $\xi \mapsto  H(\xi)$ which are necessary for our local existence argument.

The next main objective is to apply our weighted estimates to prove bounds on the exterior magnetic field with constants that are independent of the interface pinch $\delta$.
\begin{proposition}
\label{hEstimateProp}
Consider $\xi$ in $\sB^\bbox(\delta)$ for $\delta$ in $[0,\delta_0]$, with associated $\Gamma$. We take $h(\Gamma)$ and $H(\xi)$ as in Definitions~\ref{formalhdefn} and~\ref{prelimHmapDefn}.  Then we have the following estimates for any $m\geq 0$:

(i) We have
\begin{align}\label{mainhestimates}
\begin{aligned}
    \lV  h(\Gamma) \rV_{H^{k+3/2}(\cV,\delta,m)} &\leq C(M) , \\
    \lV  H(\xi) \rV_{H^{k+1}(S^1,\delta,m)} &\leq C(M),
\end{aligned}
\end{align}
as well as the analogous bounds where we replace $k$ by $k-1$ and $M$ by $\lV \xi \rV_{\sH^{k-1}}$.

(ii) Given $\uxi$ in addition to $\xi$ in $\sB^\bbox(\delta)$, if $\delta>0$,
\begin{align}
  \lV  H (\xi)- H (\uxi) \rV_{H^2(S^1,\delta,m)} \leq C \lV X - \uX \rV_{H^{3}(S^1)} .
\end{align}
\end{proposition}
\begin{proof}
Recall the $\uppsi$ associated to $\Gamma$ as discussed in Remark~\ref{hdefnremark}. Noting in particular $\uppsi=0$ on $\Gamma$, $\uppsi=\uplambda_1$ on $\cW_1$, and $\uppsi=\uplambda_2$ on $\cW_2$. It is not difficult to verify uniform bounds on the constants $\uplambda_i$. We then find a satisfactory bound on $\uppsi$ follows immediately from Proposition~\ref{1stWeightedEstimate}, regardless of the value of $m'$.
    \begin{align}
    \lV \uppsi \rV_{H^{k+5/2}(\cV,\delta,m)}
            &\leq C(M) \lV \uplambda_1 {\mathbbm 1}_{\cW_1}+\uplambda_2 {\mathbbm 1}_{\cW_2}  \rV_{H^{k+2}(\del\cV,\delta,m')} \\
            &\leq  C(M ).
    \end{align}
The bounds of \eqref{mainhestimates} are then readily verified from the above. The proof of the corresponding bounds with $k-1$ in place of $k$ is similar.

    Now, to prove (ii), we claim that there are extensions of $ H (\xi)$ and $ H (\uxi)$, $H_+$ and $\underline{H}_+$, respectively, to $\Sigma^\circ_+(\delta)$, for which we find
    \begin{align}    
\lV H_+-\underline{H}_+\rV_{H^{5/2}(\Sigma^\circ_+,\delta,m)} \leq C \lV X-\uX \rV_{H^3(S^1)} .
\end{align}

To justify this claim, we use Proposition~\ref{extensionProp2} to define the extension $X_+=\cE_\delta(X)$ on $\Sigma_+(\delta)$ and consider $h=\grad^\perp \uppsi$ corresponding to $X$. We define the extension of $H$ to $\Sigma^\circ_+(\delta)$ by $H_+=h\circ X_+$. With the aid of the system for $h$ given in Definition~\ref{formalhdefn} and Lemma~\ref{DivFreeCurlFreeConversion}, we find
\begin{align}\label{HplusSys}
\begin{aligned}
(a \in \Sigma^\circ_+(\delta)) \qquad &
\left\{
    \begin{aligned}
        \grad \cdot (\, \cof (\grad X^t_+) H_+ ) &= 0 \\
        \grad^\perp\cdot (\, \grad X^t_+ H_+ )   &= 0 \, \Lcm
    \end{aligned}
\right. \\
(a \in \del\Sigma^\circ_+(\delta))  \qquad &  (n \circ X_+ ) \cdot H_+ = 0 , \\
(i=1,2)\qquad &\int_{\tilde \cW_i} ( H_+ \grad X_+) \cdot  d\vec r = 2 .
\end{aligned}
\end{align}

Defining the analogous quantities $\uX$, $\uX_+$, $\underline h$, $\underline{H}_+$, $\underline n$, etc. associated to $\uxi$, we get that $\underline H_+$ satisfies the analogous problem with underlines everywhere. Now we prepare to use Proposition~\ref{LagrVacDivCurlEst} to estimate $H_+ - \underline H_+$. By reviewing the definitions and properties of $X_+=\cE_\delta(X)$ and $\uX_+=\cE_\delta(\uX)$ from Proposition~\ref{extensionProp2} we find the difference satisfies
\begin{align}\label{HplusDiffSys}
\begin{aligned}
(a \in \Sigma^\circ_+(\delta)) \qquad &
\left\{
    \begin{aligned}
        \grad \cdot (\, \cof (\grad X^t_+) (H_+ - \underline H_+) ) &=  (\del_\theta X_+ - \del_\theta \uX_+ ) \cdot \del_\psi \underline H^\perp_+ - (\del_\psi X_+ - \del_\psi \uX_+)\cdot \del_\theta \underline H^\perp_+\\
        \grad^\perp\cdot (\, \grad X^t_+ (H_+-\underline H_+) )   &= (\del_\theta X_+ - \del_\theta \uX_+)\cdot \del_\psi \underline H_+ - (\del_\psi X_+ - \del_\psi \uX_+) \cdot \del_\theta \underline H_+ \, \Lcm
    \end{aligned}
\right. \\
(a \in \del\Sigma^\circ_+(\delta)) \qquad &  \ \Big\{(n \circ X_+ ) \cdot (H_+-\underline H_+) = \begin{cases}- (n \circ X_+ - \underline n \circ \uX_+ ) \cdot \underline H_+ & a\in  \del \Sigma_+(\delta) \\ 0 & a\in  \tilde \cW_1\cup \tilde \cW_2 \ ,
\end{cases}\\
(i=1,2) \qquad &\int_{\tilde \cW_i} ((H_+ - \underline H_+)\grad X_+ )\cdot  d\vec r = 0.
\end{aligned}
\end{align}
For appropriate choices of $m_i$ for $i=1,\ldots, 5$, we make the series of estimates below. Using the estimates of Proposition~\ref{LagrVacDivCurlEst}, Proposition~\ref{standardWeightedSobNormEstims}, and Proposition~\ref{extensionProp2}, we find
\begin{align}\label{beginHLipbd}
\lV H_+ -  \underline{H}_+\rV_{H^{5/2}(\Sigma^\circ_+,\delta,m)} &\leq C_1 \Big(\max_{i,j} \lV(\grad X_+ - \grad \uX_+)\del_i \underline H^j_+ \rV_{H^{3/2}(\Sigma^\circ_+,\delta,m_1)} \\
& \qquad\qquad + \lV (n\circ X_+ - \underline{n} \circ \uX_+ )  \cdot \underline H_+ \rV_{H^2(\del\Sigma_+,\delta,m_1)} \Big) \\
&\leq C_2 \Big( \lV \grad X_+ -\grad \uX_+ \rV_{H^{3/2}(\Sigma^\circ_+,\delta,-m_2)} \lV  \underline H_+ \rV_{H^{5/2}(\Sigma^\circ_+,\delta,m_3)} \\
& \qquad\qquad + \lV \grad X_+ - \grad \uX_+ \rV_{H^2(\del\Sigma_+,\delta,-m_2)} \lV \underline H_+ \rV_{H^2(\del\Sigma_+,\delta,m_3)} \Big) \\
&\leq C_3 \lV X_+ - \uX_+ \rV_{H^{7/2}(\Sigma^\circ_+,\delta,-m_4)} \lV \underline{H}_+\rV_{H^{5/2}(\Sigma^\circ_+,\delta,m_5)} \\
&\leq C_4 \lV X - \uX \rV_{H^3(S^1)} . \label{endHLipbd}
\end{align}
This directly implies (ii) of the statement of the proposition.
\end{proof}

\begin{corollary}\label{Jbounds}
For integer $j$ with $2\leq j \leq k$, given $\xi$ in $\sH^j$ and $\uxi_1,\uxi_2$ in $\sB^\bbox$, we have
\begin{align}
\lV \bJ(\uxi_1)\xi \rV_{\sH^j} &\leq C(\lV \uxi_1\rV_{\sH^{k-1}})\lV \xi\rV_{\sH^j}, \\
\lV (\bJ( 
\uxi_1) - \bJ(\uxi_2)) \xi \rV_{\sH^2} &\leq C \lV \uxi_1-\uxi_2 \rV_{\sH^1}\lV \xi \rV_{\sH^2}.
\end{align}
\end{corollary}
\begin{proof}
These follow from our above bounds on $ H(\xi)$ and the definitions of $\bJ(\xi)$ (see Definition~\ref{DEcommDefns}), $\sB^\bbox$, and the $\sH^k$ norm. We also use here that we can ensure $|X_\theta|\geq c$ for $\xi$ in $\sB^\bbox$, as long as $r$ is small enough, since then $\lV \xi-\xi_0\rV_{\sH^{k-1}}\leq r$.
\end{proof}
Just as we proved bounds for the map $\xi \mapsto  H(\xi)$ above, we do the same for the external pressure gradient map, $\xi\mapsto P^+_\grad (\xi)$, in Proposition~\ref{PplusbdProp}. First we record an easy lemma to be used in the proof.

\begin{lemma}\label{harmExtLinftyBd}
Consider $\xi$ in $\sB^\bbox$ with corresponding interface $\Gamma$ and harmonic extension operator $\bigh_+$. For any $0\leq s\leq k+2$,
\begin{align}
    \lV \bigh_+ f \rV_{H^s(\mathcal{W})} \leq C \lV f \rV_{L^\infty (\Gamma) } .
\end{align}
\end{lemma}
\begin{proof}
    Let us take $\sV$ to be the set of all points within distance $10^{-1}$ of $\mathcal{W}$. By standard elliptic theory, it is not hard to verify that
    \begin{align}
        \lV \bigh_+ f \rV_{H^s(\mathcal{W})} \leq C \lV \bigh_+ f \rV_{L^2(\sV)}.
    \end{align}
    With the use of the maximum principle, we can then show the estimate claimed by the lemma.
\end{proof}
\begin{proposition}\label{PplusbdProp}
    Consider $\xi$ in $\sB^\bbox$. Then we have the bound
    \begin{align}\label{Pplusbd}
    \lV  P^+_\grad (\xi)\rV_{H^{k+1}(S^1) } &\leq C(M ) ,
\end{align}
as well as the analogous bound with $k-1$ in place of $k$ and $\lV \xi \rV_{\sH^{k-1}}$ in place of $M$.

In addition, if both $\xi$ and $\uxi$ are in $\sB^\bbox(\delta)$ for $\delta$ with $0<\delta\leq \delta_0$,
    \begin{align}\label{PplusDiffbd}
    \lV  P^+_\grad (\xi) - P^+_\grad (\uxi)\rV_{H^2(S^1) } &\leq C \lV X-\uX \rV_{H^3(S^1)} .
\end{align}
\end{proposition}
\begin{proof}
    For \eqref{Pplusbd}, recall the quantity $p_+$ introduced in Definition~\ref{ExtPressureDefnBd}.
    Note that, in $\cV$,
    \begin{align}
        \Delta p_+ = \Delta \left(\half(1-\bigh_+) |h|^2 \right) = |\grad h|^2 .
    \end{align}
    Note $\Omega_+\setminus \cV$ gives the set of points outside the walls $\mathcal{W}$. Let us then take a smooth cutoff function $\chi:\R^2\to\R^+$ vanishing at all points within distance $10^{-3}$ of $\Omega_+\setminus \cV$, and equal to $1$ at all points at least a distance $10^{-2}$ from $\Omega_+\setminus \cV$.
    We find that in $\cV$
    \begin{align}
    \Delta (\chi p_+)
    =\Delta \chi p_+ + 2 \grad \chi \cdot \grad p_+ + \chi |\grad h|^2  .
    \end{align}
    Using this in an application of Proposition~\ref{1stWeightedEstimate}, noting that $\chi p_+$ vanishes on both $\Gamma$ and $\mathcal{W}$, we find for an $m \geq 0$
    \begin{align}\label{chipplusbd}
        \lV \chi p_+ \rV_{H^{k+5/2}(\cV) } \leq C(M) \left(\lV p_+ \rV_{H^{k+3/2}(\cV,\delta,m)}+\lV h \rV_{H^{k+3/2}(\cV,\delta,m)}\right).
    \end{align}
    Applying Proposition~\ref{1stWeightedEstimate} again and then applying Lemma~\ref{harmExtLinftyBd}, we get for an $m'\geq 0$
    \begin{align}
        \lV p_+ \rV_{H^{k+3/2}(\cV,\delta,m) } &\leq C(M) \left(\lV |\grad h|^2 \rV_{H^{k-1/2}(\cV,\delta,m')} + \lV (1-\bigh_+) |h|^2 \rV_{H^{k+1}(\mathcal{W})}\right) \\
       &\leq C(M) \left(\lV |\grad h|^2 \rV_{H^{k-1/2}(\cV,\delta,m')} + \lV h \rV_{H^{k+1}(\mathcal{W})} +  \lV h \rV_{L^\infty(\Gamma)}\right).
    \end{align}
    Using this in \eqref{chipplusbd} and then applying Propositions~\ref{standardWeightedSobNormEstims} and~\ref{hEstimateProp}, we thus obtain
    \begin{align}
    \lV \chi p_+ \rV_{H^{k+5/2}(\cV)} \leq C(M).
    \end{align}
    Since $\chi$ is $1$ in a neighborhood of $\Gamma$, recalling that $  P^+_\grad(\xi)=(\grad p_+)\circ X$, we easily use this to conclude that \eqref{Pplusbd} holds. The alternate version with $k-1$ in place of $k$ is similar.

Regarding the proof of the Lipschitz bound \eqref{PplusDiffbd}, we sketch the main steps below. The first is to extend $P^+_\grad$ to $\Sigma^\circ_+(\delta)$ with $\bP^+_\grad = (\grad p_+)\circ X_+$, for $X_+ = \cE_\delta(X)$. Recall from above the identity $\Delta p_+ = |\grad h|^2$ and the cutoff function $\chi$. Defining $H_+ = h\circ X_+$ and $\tilde\chi=\chi\circ X_+$, one finds that $\tilde\chi\bP^+_\grad$ satisfies the following:
\begin{align}
\begin{aligned}
(a \in \Sigma^\circ_+(\delta)) \qquad &
\left\{
    \begin{aligned}
        \grad \cdot (\, \cof (\grad X^t_+) \tilde\chi\bP^+_\grad ) &= \grad \tilde\chi \cdot (\, \cof(\grad X^t_+) \bP^+_\grad )+ \tilde\chi\det \grad X_+ | \grad H_+(\grad X_+)^{-1}|^2 \\
        \grad^\perp\cdot (\, \grad X^t_+  \tilde\chi\bP^+_\grad )   &= \grad^\perp \tilde\chi \cdot (\, \grad X^t_+ \bP^+_\grad ) \, \Lcm
    \end{aligned}
\right. \\
(a \in \del\Sigma^\circ_+(\delta))  \qquad &  \ \big\{ \, (\tau \circ X_+ ) \cdot \tilde\chi \bP^+_\grad = 0 , \\
(i=1,2)\qquad &\int_{\tilde \cW_i} ( (\tilde\chi\bP^+_\grad)^\perp \grad X_+) \cdot  d\vec r = 0 .
\end{aligned}
\end{align}
For $\uX$ corresponding to $\uxi$, defining $\underline P^+_\grad$, etc. in the analogous way, we obtain the analogous system for $\tilde \chi\underline \bP^+_\grad$. Similar to the way \eqref{HplusDiffSys} was obtained from \eqref{HplusSys} in the proof of Proposition~\ref{hEstimateProp}, one obtains a system for \(\tilde\chi(\bP^+_\grad-\underline \bP^+_\grad)\) in \(\Sigma^\circ_+(\delta)\).

After repeated applications of Proposition~\ref{LagrVacDivCurlEst} in which one uses a similar strategy to the proof of (ii) of Proposition~\ref{hEstimateProp}, one eventually arrives at the bound \eqref{PplusDiffbd} by estimating \(\tilde\chi(\bP^+_\grad-\underline \bP^+_\grad)\) in $H^{5/2}(\Sigma^\circ_+(\delta))$. Let us note that to estimate this quantity, when one applies Proposition~\ref{LagrVacDivCurlEst}, one uses the alternative version of the system \eqref{gendivcurlprob}--\eqref{lineIntegCond2} where $n$ is replaced by $\tau$ and $V$ by $V^\perp$ in \eqref{lineIntegCond2}. Furthermore, after handling the obvious terms, since $\bP^+_\grad -\underline \bP^+_\grad$ still appears in the right-hand sides (though without being hit by derivatives), this first application still leaves one with the task of proving a bound on this difference in $H^{3/2}(\Sigma^\circ_+(\delta))$, similar to the step we took in \eqref{chipplusbd}. After a second application of a div-curl system bound, it still remains to bound $(\bigh_+|h|^2)\circ X-(\underline\bigh_+|\underline h|^2)\circ \uX$ on $\tilde\cW_1\cup\tilde\cW_2$ in the appropriate norm. By deriving yet another elliptic problem in $\Sigma^\circ_+(\delta)$, applying another such elliptic estimate, and repeatedly following the strategy of \eqref{beginHLipbd}--\eqref{endHLipbd}, one finishes the proof.
\end{proof}

Now that we have established bounds for the maps $\xi\mapsto H(\xi)$ and $\xi\mapsto  P^+_\grad(\xi)$, we merely state the corresponding bounds for the time derivative versions. Their proofs can be found in Section~\ref{auxiliaryestimates}.
\begin{myprop}{\ref{dotHEstimateProp}}
Fix $m\geq 0$. For $\xi$ in $\sB^\bbox(\delta)$ with $\delta$ in $[0,\delta_0]$, we have the estimates
\begin{align}
\lV \dot{h}(\xi) \rV_{H^k(\Gamma, \delta, m)} &\leq C(M) , \\
\lV \dot{H}(\xi) \rV_{H^k(S^1,\delta,m)} &\leq C(M) .
\end{align}
Additionally, if $\delta>0$, for $\uxi$ also in $\sB^\bbox(\delta)$,
\begin{align}
\lV \dot{H}(\xi)-\dot{H}(\uxi) \rV_{H^2(S^1,\delta,m)} &\leq C \lV \xi - \uxi \rV_{\sH^2} .
\end{align}
\end{myprop}
\begin{myprop}{\ref{dotPplusSysProp}}
Fix $m\geq 0$. For $\xi$ in $\sB^\bbox$, we have the estimate
\begin{align}
\lV  \dot{P}^+_\grad(\xi) \rV_{H^k(S^1)} &\leq C(M) ,
\end{align}
and if $\xi$ and $\uxi$ are in $\sB^\bbox(\delta)$ for $0<\delta\leq \delta_0$,
\begin{align}
\lV  \dot{P}^+_\grad(\xi) - \dot{P}^+_\grad(\uxi) \rV_{H^2(S^1)} &\leq C \lV \xi - \uxi \rV_{\sH^2} .
\end{align}
\end{myprop}

\subsection{Dirichlet-to-Neumann estimates}\label{DtNestimatessection}
        Now we discuss bounds involving Dirichlet-to-Neumann operators $\bign_\pm(\xi)$ and related objects, like $\Nres(\xi)$. This analysis is more nuanced than that in Section~\ref{magfieldestimatessection} for several reasons. For instance, operators such as these, which involve harmonic extensions off of a curve, are very sensitive to changing pinches in the domain. Additionally, they can either gain or lose regularity depending on where in space the function they act on is supported and where the output of the map is evaluated.


        Before we proceed, let us outline the intuition behind the way the Dirichlet-to-Neumann operator $\bign_+$ for extensions into $\Omega_+$ behaves to help explain what one should expect for the most basic estimates.

        Consider a situation in which the vacuum with a pinch given by a very small (but positive) number $\delta$. In this case, the output of $\bign_+ f$ can be quite large even if $f:\Gamma\to\R$ is very smooth. For example, if $f$ takes a different constant value on each of the two arcs of $\Gamma$ that come close together in the event of a near-splash, the magnitude of the gradient of the harmonic extension would generally be inversely proportional to $\delta$ in the pinched region. Thus classical bounds on the normal derivative cannot hold uniformly for small $\delta$.

        On the other hand, the proposition below establishes that if one has sufficient $\delta$-dependent weighted control on the input function $f$, then the output of $\bign_+ f$ can be uniformly bounded in standard Sobolev spaces, and thus in $L^\infty$. For example, the weighted bound on $|h|$ of Proposition~\ref{hEstimateProp} implies that $|h|$ must be very small where the vacuum is very pinched. In fact the magnitude of $h$ near the pinch points becomes small enough to guarantee that $\bign_+ |h|$ remains bounded, even for small $\delta$.

        The proposition below clarifies the sense it which the operator produces bounded output if the input function is small in the pinch. Note that it remains consistent with the standard, well-known facts that harmonic extensions off of a curve have a half-degree higher index of Sobolev regularity in the domain of extension, and that the Dirichlet-to-Neumann operator generally takes one derivative when applied to a function on the interface.
        \begin{proposition}\label{basicHarmExtBds}
            Let $\xi$ be in $\sB^\bbox(\delta)$ for $0<\delta\leq \delta_0$, with corresponding $\Gamma$ (and $\cV$) and the associated harmonic extension operator $\bigh_+$ to $\Omega_+$. Let $\bign_+(\xi)$ and $\cN_+(\xi)$ be as defined in Definition~\ref{DirichToNeumDefs}. 
            Then for some integer $m \geq 0$ and any integer $s$ with \(0 \leq s \leq k+1 \), we have the estimates
            \begin{align}
&&                \lV \bigh_+ f \rV_{H^{s+3/2}(\cV)} &\leq C(M)  \lV f \rV_{H^{s+1}(\Gamma,\delta,m)} &(f:\Gamma\to \R),\\
&&                \lV \bign_+(\xi) f \rV_{H^{s}(\Gamma)} &\leq C(M)  \lV f \rV_{H^{s+1}(\Gamma,\delta,m)}  & (f:\Gamma\to\R), \\
&&                \lV \cN_+(\xi) F \rV_{H^{s}(S^1)} &\leq C(M)  \lV F \rV_{H^{s+1}(S^1,\delta,m)} &(F:S^1\to\R).
            \end{align}
Moreover, in the case that $s\leq k$, we can replace $M$ by $\lV \xi\rV_{\sH^{k-1}}$.
        \end{proposition}
        \begin{proof}
            This follows easily from Proposition~\ref{1stWeightedEstimate} and Lemma~\ref{harmExtLinftyBd}.
        \end{proof}
The lemma below states that the output of a Dirichlet-to-Neumann operator $\bign_\pm f$ does not lead to a decrease in regularity away from the support of the input function $f$. Our estimate is far from sharp but sufficient. To prepare for the lemma, we define the distance-within-$\Gamma$ function by
     \begin{align}
     \operatorname{dist}_\Gamma(x,y) = \inf \{\operatorname{length}(\gamma):\gamma\subset\Gamma\mbox{ is an arc joining} \ x \ \mbox{and} \ y\}   &&  (x,\, y \in \Gamma ).
     \end{align}
    We then define $\operatorname{dist}_\Gamma(S_1,S_2)$ for a pair of subsets $S_1, S_2$ of $\Gamma$ in the obvious way.
\begin{lemma}\label{SuppsBddAwayLem}
     There exists an $m \geq 0$ so the following holds, for some $d_0>0$. Consider $\xi$ in $\sB^\bbox(\delta)$ for $0<\delta\leq \delta_0$ and some arc $\gamma\subset \Gamma$. Then for $ f : \Gamma \to \R$ such that
     \begin{align}
     \operatorname{dist}_\Gamma (\supp(f),\gamma ) \geq d_0 ,
     \end{align}
    given integer $s$ with $1 \leq s \leq k+1$ and $\bign_\pm(\xi)$ of Definition~\ref{DirichToNeumDefs}, we have
    \begin{align}\label{offsuppDtN}
        \lV \bign_{\pm}(\xi) f \rV_{H^s(\gamma)}\leq C(M)   \lV f \rV_{H^s(\Gamma,\delta,m)},
    \end{align}
    where we can replace $M$ by $\lV \xi \rV_{\sH^{k-1}}$ for $s\leq k$.
    Moreover, in the case of $\pm=-$, we have $m=0$ (that is, the weighted norm can be replaced by the standard Sobolev norm).
\end{lemma}
\begin{proof}
We give the proof only of the estimate for $\bign_+(\xi)$, since the pinching domain plays no role in the estimate for $\bign_-(\xi)$, which is strictly easier. Consider a pair of parabolas $\cP_1$ and $\cP_2$ as pictured in Figure~\ref{fig:parabolas}, separated by distance $\delta_\Gamma/2$. Let $\tilde\gamma_1$ be the set of all points on $\Gamma$ lying to the left of $\cP_1$ and $\tilde\gamma_2$ the points on $\Gamma$ to the right of $\cP_2$. For simplicity, let us assume $\supp (f)\subset \tilde\gamma_1$, and $\gamma\subset\joinrel\subset \tilde\gamma_2$. It is easy to adapt the proof to general $f$ as in the statement of the lemma.
\begin{figure}
    \centering
    \includegraphics[width=0.45\textwidth]{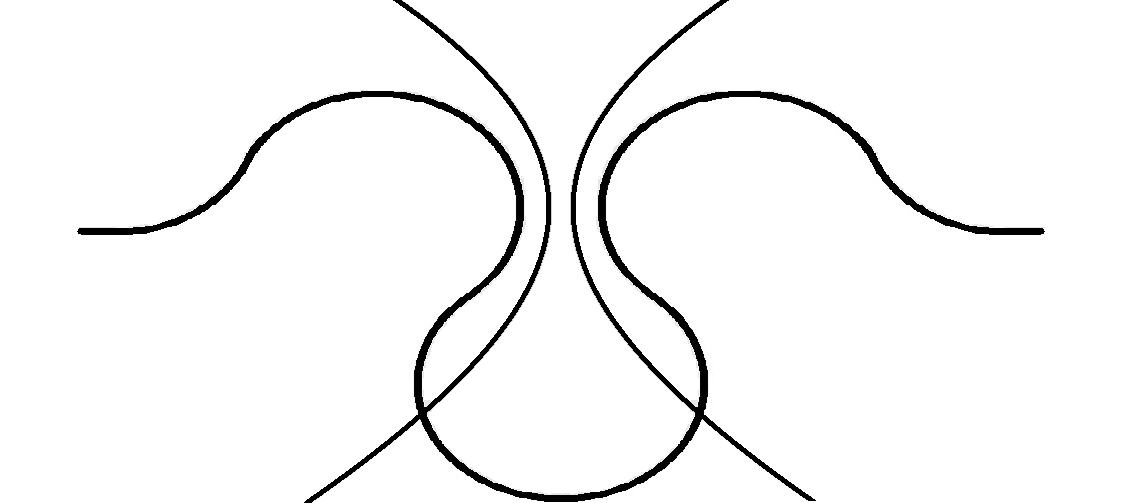}
    \begin{tikzpicture}[overlay, remember picture]
        \node at (-2.5,3) {$\cP_2$};
        \node at (-4,3) {$\cP_1$};
    \end{tikzpicture}
    \caption{}\label{fig:parabolas}
\end{figure}
Define a smooth cutoff $\chi$ on $S^1\times\R$ with $\chi=1$ to the right of $\cP_2$ and $\chi=0$ to the left of $\cP_1$, and such that for any multi-index $\beta$ with $|\beta|\leq k+2$ one has $|\del^\beta \chi|\leq C \mu^{|\beta|}_\delta $ for $\mu_\delta$ defined by \eqref{weightfcndef}. Note that because $\supp (f)\subset\tilde\gamma_1$, the trace of $\chi \bigh_+ f$ is zero on $\Gamma$. We apply Proposition~\ref{1stWeightedEstimate} to $\chi \bigh_+ f$ to get
\begin{align}
\lV\chi \bigh_+ f \rV_{H^{s+3/2}(\Omega_+)} 
\leq C \lV \Delta (\chi \bigh_+ f) \rV_{H^{s-1/2}(\Omega_+,\delta,m_1)} .
\end{align}
When all of the Laplacian falls on $\bigh_+ f$, one gets zero. Alternatively, at least one derivative falls on $\chi$ instead, and we may bound the result above by a weighted norm of $\bigh_+ f$ with an increased weight index $m_2$. We end up with
\begin{align}
\lV\Delta (\chi \bigh_+ f) \rV_{H^{s-1/2}(\Omega_+,\delta,m_1)} 
\leq C \lV  \bigh_+ f \rV_{H^{s+1/2}(\Omega_+,\delta,m_2)} .
\end{align}
Applying Proposition~\ref{1stWeightedEstimate} again results in
\begin{align}
\lV\chi \bigh_+ f \rV_{H^{s+3/2}(\Omega_+)}
\leq C \lV  \bigh_+ f \rV_{H^{s+1/2}(\Omega_+,\delta,m_2)}
\leq C' \lV  f \rV_{H^{s}(\Gamma,\delta,m_3)} .
\end{align}
We then get \eqref{offsuppDtN} from the facts that $\bign_+ f = n\cdot \grad \bigh_+f$ along $\Gamma$ and $\chi=1$ in a neighborhood of $\gamma$.
\end{proof}

\subsubsection*{The Dirichlet-to-Neumann cancellation bound (Proposition~\ref{DtoNCancellation})}
        As discussed above, the operators $\bign_\pm$ and $\cN_\pm$ tend to reduce regularity by one degree. However, now we discuss an important subtlety, which is that the operators defined by
        \begin{align}
            \nres &= \bign_+ + \bign_-  ,\\ 
         \Nres &= \cN_+ + \cN_- ,
        \end{align}
        maintain regularity. A cancellation occurs at the leading order between $\bign_+ f$ and $\bign_- f$ when the two are combined. While estimates verifying this are already well-established for chord-arc domains, we prove uniform estimates demonstrating this for the first time for near-splash curves. Unique to our setting is the need for $\delta$-dependent weighted bounds on the operator which hold uniformly even when the pinch $\delta$ is very small. In the proposition below, we formulate an appropriate uniform bound demonstrating the Dirichlet-to-Neumann cancellation for our pinching vacuum scenario.
        \begin{proposition}
        \label{DtoNCancellation}
        \hfill \\
        Let $\xi$ be in $\sB^\bbox(\delta)$ for $0<\delta\leq \delta_0$. 
   There exists an $m \geq 0$ such that for integer $s$ with \( 1 \leq s \leq k+1 \),
        \begin{align}
            & & \lV \nres(\xi) f \rV_{H^s(\Gamma)} &\leq C(M) \lV f \rV_{H^s(\Gamma,\delta,m)} & (f:\Gamma \to \R), \label{nresCancelBd}\\
            & & \lV \Nres(\xi) F \rV_{H^s(S^1)} &\leq C(M) \lV F \rV_{H^s(S^1,\delta,m)} & (F:S^1 \to \R). \label{NresCancelBd}
        \end{align}
        In the case that $s\leq k$, $M$ can be replaced by $\lV \xi \rV_{\sH^{k-1}}$.
        \end{proposition}
        The proof of this proposition (which begins after the statement of Lemma~\ref{EvalSLPBdLem}) is somewhat involved and first requires more supporting framework. A key idea of ours is to reveal the cancellation which occurs by performing elementary calculations involving layer potentials. Below, we provide an overview of the notation we use in this section for layer potentials and related operators.
        
        For the Newtonian potential (periodic in $x_1$), we denote
\begin{align}\label{newtpotdefn}
    & & G(x) = G(x_1+i x_2) &= \frac{1}{\pi}\log\left|\sin\left(\frac{x_1+ix_2}{2}\right)\right|  & \qquad ( x \in \R^2 \setminus \{0\} ) .
\end{align}
For $ f : \Gamma \to \R$, we define the single layer potential $\Slp f $ and double layer potential $\Dlp f$ by the following.
\begin{align}
& & \Slp f (x) &= \int_\Gamma G(x-y) f (y) \, dS(y)    &&   (x \in \R^2 ) , \\
& & \Dlp f (x) &= -\int_\Gamma \left(\frac{\del G}{\del n_y }\right)(x-y) f (y) \, dS(y) && (x \in \R^2 \setminus \Gamma ) ,\\ \intertext{
where $\frac{\del G}{\del n_y}=n(y)\cdot \grad G$, the normal $n$ pointing into the region above $\Gamma$, i.e. $\Omega_+$. For the restriction of the single layer potential to $\Gamma$, we denote}
& & \slp f (x) &= \Slp f (x)  &&  (x \in \Gamma ) .
\\ \intertext{
Let us define the operator $\dcal{T}$ acting on $f: \Gamma \to \R$ by}
& &    \dcal{T} f (x) &= -\int_\Gamma \left(\frac{\del G}{\del n_y}\right)(x-y) f (y)\, dS(y) && (x \in \Gamma ) .\label{TOpDefn}
\end{align}
\begin{remark}
The following well-known identity regarding the limits of the double layer potential holds for periodic, $C^1$ non-self-intersecting $\Gamma$ and $f$ in $C^{0,\alpha}(\Gamma)$ for $\alpha>0$:
\begin{align}\label{pottheorylim}
& &    \lim_{\epsilon\to 0^+} \Dlp f (x\pm\epsilon n) &= \left(\pm\half+\dcal{T}\right) f (x)    &   
    ( x \in \Gamma ) .
\end{align}
\end{remark}
    \subsubsection*{Using the key cancellation identity and tailored chord-arc curves}   

    The Dirichlet-to-Neumann cancellation boils down to the following idea. It turns out that one can express the Dirichlet-to-Neumann operators $\bign_\pm$ associated\footnote{In the material that follows, we frequently drop the argument $\xi$ from expressions like $\bign_\pm(\xi)$ when the curve of consideration is understood.} to a given curve $\Gamma$ in terms of layer potentials, with
    \begin{align}
    \bign_\pm f  &= \slp^{-1}\left(\pm \half+\dcal{T}\right) f .
\end{align}    
    Thus, for the combination $\nres = \bign_+ + \bign_-$, one finds that the top order terms cancel, yielding
    \begin{align}\label{essidentity}
    \nres f  &= 2\slp^{-1}\dcal{T} f .
    \end{align}
    Since $\dcal T$ generally gains a derivative and $\slp^{-1}$ loses one, this suggests $\nres$ preserves regularity, as asserted by Proposition~\ref{DtoNCancellation}. It is worth noting that in the proof of Proposition~\ref{DtoNCancellation}, we only invert $\slp$ for non-self-intersecting curves, a case in which invertibility is well-known. In the following Lemma, essentially, we record the identity \eqref{essidentity}.
\begin{lemma}\label{keycancellationident}
Consider $\Slp$, $\slp$, $\dcal T$, $\bigh_\pm$, and $\nres$ associated to a non-self-intersecting curve $\Gamma$ parametrized by $X$ in $H^2_\gp$. Given $f$ in $C^{0,\alpha}(\Gamma)$ for $\alpha>0$,
\begin{align}
& &    \slp\nres f  &= 2\dcal{T} f    &   & (x\in \Gamma) , \label{keynresIdent1}\\
& &    \Slp \nres f  & = 2 \bigh_\pm(\dcal{T} f ) &  & (x \in \Omega_\pm) \label{keynresIdent2}.
\end{align}
\end{lemma}
\begin{proof}
The identities below follow from a short calculation\footnote{The authors have only found these identities derived or stated in \cite{taylorPDE2}.} using Green’s formula and \eqref{pottheorylim}:
\begin{align}\label{layerpotident2}
& &     \Dlp f & = \Slp \bign_\mp f  =\Slp \bign_\pm f  \mp\bigh_\pm f &   (x &\in \Omega_\pm ),
\end{align}
\begin{align}\label{slpRef}
& &    \slp \bign_\pm f  &= \left(\pm\half+\dcal{T}\right) f  & & (x \in \Gamma ) .
\end{align}
Thus by adding the equation \eqref{slpRef} with $\pm=+$ together with the version with $\pm=-$, one arrives at \eqref{keynresIdent1}. The identity \eqref{keynresIdent2} then arises from the fact that the output of the operator $\Slp$ evaluated in $\Omega_\pm$ yields the appropriate harmonic extension (recall Definition~\ref{harmonicextdefn}) of the left-hand side of \eqref{keynresIdent1}. In particular this uses the fact that $\Slp g(x)$ remains bounded as $x_2$ tends to $\pm\infty$ if $\int_\Gamma g \, dS=0$, which holds in the case $g=\nres f$.
\end{proof}

Note that directly using the identity \eqref{keynresIdent1} to get the formula \eqref{essidentity} for $\nres$ requires one to invert $\slp$. The estimates we state momentarily in Lemma~\ref{ChordArcOpBdLem} include a bound on the inverse of a single layer potential. However, one should note that these only hold uniformly for curves whose chord-arc constant remains uniformly bounded below. These are in contrast with the near-splash curves associated to $\xi$ in $\sB^\wbox$, which we are forced to deal with.

Let us say a curve $\Gamma$ is an \emph{admissible chord-arc curve} if it is parametrized by $X$ in $C^1(S^1)$ such that $\|X-X_0\|_{C^1(S^1)}\leq r_0$ and, for a universal constant $c_\ca$,
\begin{align}\label{chordarc}
\min_{\theta,\vartheta\in S^1}\frac{|X(\theta)-X(\vartheta)|}{|\theta-\vartheta|}\geq c_\ca .
\end{align}
For such curves classical methods\footnote{See \cite{MazyaShaposhnikova} for example.} yield the conclusions of Lemma~\ref{ChordArcOpBdLem}.
\begin{lemma}\label{ChordArcOpBdLem}
    Consider an admissible chord-arc curve $\Gamma_{\ca}$ parametrized by $X$ in $H^{k+2}(S^1)$ with corresponding single layer potential $\slp_{\ca}$ and operator $\dcal{T}_{\ca}$ defined analogously to $\dcal{T}$ in \eqref{TOpDefn}. Then for $0 \leq s \leq k + 1$, the operator $\slp_{\ca}:H^s(\Gamma_{\ca})\to H^{s+1}(\Gamma_{\ca})$ is invertible, and
    \begin{align}
    \lV \dcal{T}_{\ca} f \rV_{H^{s+1}(\Gamma_{\ca})} &\leq C \lV f \rV_{H^s(\Gamma_{\ca})} , \\
    \lV \slp_{\ca}^{-1} f \rV_{H^s(\Gamma_{\ca})} &\leq C\lV f \rV_{H^{s+1}(\Gamma_{\ca})} ,
    \end{align}
    where we have $C=C\left(\lV X \rV_{H^{k+2}(S^1)} \right)$ if $s=k+1$ and $C=C\left(\lV X \rV_{H^{k+1}(S^1)} \right)$ if $s\leq k$.
\end{lemma}
    It is not obvious how to prove useful bounds for $\slp^{-1}$ for the operator $\slp$ associated to a curve $\Gamma$ with a very small pinch, even in terms of our weighted Sobolev spaces. To get around this, in the proof of Proposition~\ref{DtoNCancellation}, we will occasionally replace a nearly-self-glancing curve $\Gamma$ by a modified version, $\Gamma_\vardiamond$, which coincides with $\Gamma$ along the entire ``left portion'' $X([-\pi,0])$, but dips into $\Omega$ along the ``right portion'', so as to avoid the breakdown of \eqref{chordarc} and form an admissible chord-arc curve.
\begin{definition}\label{tailoredCA}
Fix a smooth bump function $\eta_\vardiamond:S^1\to\R$ supported in $[0.2,\pi-0.2]$ and equal to a small constant $\epsilon_0$ on $[0.3,\pi-0.3]$. Given $\Gamma$ parametrized by $X$ in $H^{k+2}_\gp$ and corresponding $N=n\circ X$, we define
\begin{align}
&& X_\vardiamond(\theta) &= X(\theta) - \eta_\vardiamond(\theta) N(\theta) & (\theta\in S^1), 
\end{align}
and we then define $\Gamma_\vardiamond$, the tailored chord-arc version of $\Gamma$, by
\begin{align}
\Gamma_\vardiamond = X_\vardiamond(S^1) .
\end{align}
\end{definition}


Some minor error terms will arise from replacing $\Gamma$ by $\Gamma_\vardiamond$ at certain points in the proof of Proposition~\ref{DtoNCancellation}. The following two lemmas are designed to handle some of the errors introduced by doing so. The first gives an estimate for a single layer potential $\Slp$ associated to $\Gamma$ when evaluated along $\Gamma_\vardiamond$.
\begin{lemma}\label{EvalSLPBdLem}
Consider $\xi$ in $\sB^\bbox$, with corresponding interface curve $\Gamma$ and its tailored chord-arc version $\Gamma_\vardiamond$. Let $\Slp$ denote the corresponding single layer potential operator for $\Gamma$. For integer $s$ with $0 \leq s \leq k+1$ we have the estimate
\begin{align}\label{evalBd}
    \lV (\Slp f )|_{\Gamma_\vardiamond}\rV_{H^{s+1}(\Gamma_\vardiamond)}\leq C(M)  \lV f \rV_{H^s(\Gamma)},
\end{align}
where we may replace $M$ by $\lV \xi \rV_{\sH^{k-1}}$ for $s\leq k$.
\end{lemma}
\begin{proof}


After plugging in the parametrization $X_\diamond(\theta)$ into $\cS f$, one applies derivatives to the resulting expression, and verifies that the resulting integral kernel is well-behaved, leading to a gain of one derivative. This calculation is completely standard when evaluating the result at $\theta$ corresponding to the part of the curve where $\Gamma$ and $\Gamma_\diamond$ overlap. The only slightly nonstandard part of the argument arises when one evaluates $\Slp f$ along $X_\diamond(\theta)$ where it pushes into $\Omega$. In a neighborhood (bounded away from the pinch) of the points where $\Gamma_\diamond$ branches off from $\Gamma$, one verifies the kernel is no more singular than it is where the curves coincide. Over the rest of $\Gamma_\diamond$, the kernel is evaluated away from its singularity, which results in a smoothing remainder.
\end{proof}
We discuss one last identity to be applied in the proof of Proposition~\ref{DtoNCancellation}, which relates cutoffs and single layer potential operators $\Slp$ and $\Slp'$ associated to partially overlapping curves $\Gamma$ and $\Gamma'$. For a given $\Gamma$ in $\cXbar$, $\Gamma'$ will play the role of the tailored chord-arc version $\Gamma_\vardiamond$.
\begin{lemma}\label{PartOverlapIdLem}
Consider two $C^1$ non-self-intersecting closed curves $\Gamma$ and $\Gamma'$. We denote the single layer potential operator associated to $\Gamma$ by $\Slp$ and its restriction by $\slp$. Additionally, denote the corresponding operators for $\Gamma'$ by $\Slp'$ and $\slp'$. Suppose the intersection $\Gamma\cap \Gamma'$ contains a nonempty curve segment. Consider functions $ f :\Gamma\to\R$ and $ \chi : \Gamma\to\R$ with $\supp(\chi)\subset \Gamma\cap \Gamma'$ and $\chi=1$ on some subset $\Gamma''\subset\Gamma\cap\Gamma'$. Then
\begin{align}\label{curveSwapIdent}
f &= \slp'^{-1}((\Slp f)|_{\Gamma'})+\slp'^{-1}((\Slp[(\chi-1) f ])|_{\Gamma'})   &   ( x \in \Gamma'' ) .
\end{align}
\end{lemma}
\begin{proof}
First one uses linearity of $\Slp$ and $\slp'^{-1}$ to reduce the right-hand side of \eqref{curveSwapIdent} to
\begin{align}
\slp'^{-1}((\Slp[\chi f])|_{\Gamma'}).
\end{align}
Then one observes that $g=\chi f:\Gamma\to\R$, which is supported in $\Gamma\cap\Gamma'$, can be extended trivially to $\Gamma'\setminus\Gamma$, and then one has $(\Slp g)|_{\Gamma'}=\slp' g$. At the last step, after recalling $\chi=1$ on $\Gamma''$, one finds the right-hand side of \eqref{curveSwapIdent} reduces to $g=f$ there.
\end{proof}

\subsubsection*{Proof of Proposition~\ref{DtoNCancellation}}

        \begin{proof}
We proceed to argue there exists an $m\geq 0$ such that for $\xi$ in $\sB^\bbox(\delta)$ with pinch $\delta>0$, for integer $s$ with $1\leq s \leq k+1$,
        \begin{align}
            \lV  \nres f \rV_{H^s(\Gamma)} \leq C(M) \lV f \rV_{H^s(\Gamma,\delta,m)} ,\label{nresCancelBd2}
        \end{align}
        where $M$ may be replaced by $\lV \xi \rV_{\sH^{k-1}}$ if $s\leq k$. Once we show this, the remaining bounds \eqref{NresCancelBd} easily follow.

First, let us prove \eqref{nresCancelBd2} holds under the assumption that $F = f\circ X$ is supported in $[-\pi,0]$.

For $i=1,2,3$, let us define $a_i=-\pi-10^{-1}(i-1)$, $b_i=10^{-1}(i-1)$, and $\Gamma_i = X([a_i,b_i])$. We also define the complimentary arcs $\Gamma^c_i=\Gamma\setminus \Gamma_i$, for $i=1,2,3$. Now note
\begin{align}\label{NresPf1}
    \lV \nres f \rV_{ H^s (\Gamma) }\leq \lV \nres f \rV_{H^s(\Gamma_2)} + \lV \bign_+ f \rV_{H^s(\Gamma^c_2)} + \lV \bign_- f \rV_{H^s(\Gamma^c_2)}
\end{align}
Let us remark here that our choices of $a_i$, $b_i$, and the constant $d_0$ in the statement of Lemma~\ref{SuppsBddAwayLem} guarantee that $\operatorname{dist}_\Gamma(\Gamma_1,\Gamma^c_2) \geq d_0$. Note $\supp (f) \subset \Gamma_1$, and that the lemma then gives us the following, for some $m \geq 0$:
\begin{align}
    \lV \bign_\pm f \rV_{H^s(\Gamma^c_2)}\leq C(M) \lV f \rV_{H^s (\Gamma_1, \delta, m)} .
\end{align}
Note that applying this immediately gives a satisfactory bound on the second and third terms in \eqref{NresPf1}, so that only the first term remains to be bounded.

Now we use the tailored chord-arc curve $\Gamma_\vardiamond$ of Definition~\ref{tailoredCA}, which we note is a deformation of $\Gamma$ where we have pushed $\Gamma^c_3$ slightly into the plasma region $\Omega$ while preserving $\Gamma_3$. Let us denote by $\tilde \Gamma^c_3$ the altered portion of the curve, so that $\Gamma_\vardiamond = \Gamma_3\cup \tilde{\Gamma}^c_3$.

We take a smooth cutoff $\chi:\Gamma\to\R$ with $\supp(\chi)\subset\Gamma_3$ and $\chi=1$ on $\Gamma_2$. We then apply the identity of Lemma~\ref{PartOverlapIdLem} to $\nres f $, decomposing
\begin{align*}
\begin{aligned}
    \nres f &=
    \slp_{\vardiamond}^{-1}((\Slp \nres f )|_{\Gamma_\vardiamond})
    +
    \slp_{\vardiamond}^{-1}((\Slp[(\chi-1)\nres f ])|_{\Gamma_\vardiamond})  &   \\
    &=\nres^1 f + \nres^2 f  & 
\end{aligned} \quad (x \in \Gamma_2 ).
\end{align*}
We think of $\nres^1 f $ as the main term and $\nres^2 f$ as a remainder term. Let us note for now that $(\chi - 1)=0$ on $\Gamma_2$. We bound
\begin{align}\label{collectednresBd}
    \lV \nres f 
    \rV_{H^s(\Gamma_2)}
    \leq
    \lV \nres^1 f \rV_{H^s(\Gamma_2)}+\lV \nres^2 f \rV_{H^s(\Gamma_2)} ,
\end{align}
 Note we may extend $f$ from $\Gamma$ to $\Gamma_\vardiamond$ by defining it to be zero on $\Gamma_\vardiamond\setminus\Gamma$. Denote by $\dcal T_\vardiamond$ the operator defined by \eqref{TOpDefn} for the curve $\Gamma_\vardiamond$. Since $\supp (f ) \subset \Gamma \cap \Gamma_\vardiamond$, inspecting the expression for $\dcal T_\vardiamond f:\Gamma_\vardiamond\to\R$ reveals it is precisely the restriction to $\Gamma_\vardiamond$ of a harmonic extension of $\dcal{T} f:\Gamma\to\R$. Using identity \eqref{keynresIdent2} of Lemma~\ref{keycancellationident}, we thus obtain
\begin{align}
& &    \Slp \nres f &= 2 \bigh_- \dcal{T} f =2\dcal{T}_\vardiamond f  &   (x \in \Gamma_\vardiamond ) .
\end{align}
Since $\Gamma_\diamond$ is an admissible chord-arc curve, Lemma~\ref{ChordArcOpBdLem} and the above observations lead us to
\begin{align}
    \lV \nres^1 f \rV_{H^s(\Gamma_2)}&=
    \lV \slp^{-1}_\vardiamond
    (\Slp \nres f )|_{\Gamma_\vardiamond}
    \rV_{H^s(\Gamma_2)}\\
    &\leq C(M) \lV (\Slp \nres f)|_{\Gamma_\vardiamond}\rV_{H^{s+1}(\Gamma_\vardiamond)}\\
    &\leq C(M)\lV \dcal{T}_\vardiamond f \rV_{H^{s+1}(\Gamma_\vardiamond)}\\
    &\leq C(M)\lV f \rV_{H^s(\Gamma)} . \label{nres1bound}
\end{align}
For the remainder term $\nres^2 f$, a crude bound suffices. Using Lemmas~\ref{ChordArcOpBdLem} and~\ref{EvalSLPBdLem}, we bound
\begin{align}
    \lV \nres^2 f \rV_{H^s(\Gamma_2)}
    &=\lV\slp_{\vardiamond}^{-1}((\Slp[(\chi-1)\nres f ])|_{\Gamma_\vardiamond})\rV_{H^s(\Gamma_2)}\\
    &\leq C(M)\lV(\Slp[(\chi-1)\nres f ])|_{\Gamma_\vardiamond}\rV_{H^{s+1}(\Gamma_\vardiamond)}\\
    &\leq C(M)\lV (\chi-1)\nres f \rV_{H^s(\Gamma)} . \label{nres2bound}
\end{align}
Now we use the fact that $\supp(\chi-1)\subset \Gamma^c_2$ is bounded away from $\supp (f) \subset \Gamma_1$ to apply Lemma~\ref{SuppsBddAwayLem}:
\begin{align}
    \lV (\chi-1)\nres f \rV_{H^s(\Gamma)}
    &\leq C
    \lV \nres f \rV_{H^s(\Gamma^c_2)}\\
    &\leq C(M)
    \lV f \rV_{H^s (\Gamma_1,\delta,m)} .
\end{align}
Using this in \eqref{nres2bound} to bound $\nres^2 f$, and combining the result with the bound \eqref{nres1bound} for $\nres^1 f$, we get
\begin{align}
    \lV \nres f \rV_{H^s(\Gamma_2)}\leq C \lV f \rV_{H^s(\Gamma,\delta,m)}.
\end{align}
This takes care of the last remaining term in the right-hand side of \eqref{NresPf1}.

Thus, \eqref{nresCancelBd2} holds in the case that $F = f \circ X$ is supported in $[-\pi,0]$. Using a similar argument to the one above, one proves the bound in the case that $F$ is supported in $S^1\setminus [-3\pi/4,-\pi/4]$. By patching these estimates together with a partition of unity, we verify \eqref{nresCancelBd2} holds in general. For $s\leq k$, we similarly prove the bound with $M$ replaced by $\lV\xi\rV_{\sH^{k-1}}$.
        \end{proof}

\subsubsection*{Lipschitz bound for $\Nres$}
Now we establish a basic Lipschitz estimate for $\xi\mapsto\Nres(\xi)$. For our Lipschitz estimates, we only need bounds at lower orders, so a cancellation-type estimate like in Proposition~\ref{DtoNCancellation}, where the regularity level of the input function is preserved, is unnecessary. 

\begin{proposition}\label{NresLipBd}
Then there exists an $m\geq 0$ such that the following holds for $\xi$ and $\uxi$ in $\sB^\bbox(\delta)$, where $0<\delta\leq \delta_0$.
\begin{align}\label{lipNres}
            \lV (\Nres(\xi)-\Nres(\uxi)) F \rV_{H^2(S^1)} \leq C \lV \xi - \uxi \rV_{\sH^1}\lV F \rV_{H^3(S^1,\delta,m)} ,
        \end{align}
\end{proposition}
\begin{proof}
Let us use underlines to denote obvious analogues of quantities corresponding to $\uxi$. First, we verify an estimate of the form
\begin{align}
\lV (\cN_+ - \underline \cN_+) F \rV_{H^2(S^1)} \leq C \lV \xi - \uxi \rV_{\sH^1}\lV F \rV_{H^3(S^1,\delta,m)} .
\end{align}
Recalling Proposition~\ref{extensionProp2}, let us define $X_+=\cE_\delta(X)$ and $\uX_+=\cE_\delta(\uX)$, noting in particular that for some $m\geq 0$
\begin{align}
\lV X_+ - \uX_+ \rV_{H^{5/2}(\Sigma_+,\delta,-m)} \leq C \lV X - \uX \rV_{H^2(S^1)} .
\end{align}
Let us define $G:\del\Sigma_+(\delta)\to\R$ by
\begin{align}
& & G&=\frac{F'\circ X^{-1}\circ X_+}{|X_\theta \circ X^{-1}\circ X_+|} & (a\in\del\Sigma_+(\delta)).
\end{align}
Now we will express $\cN_+ F$ in terms of the unique solution $V$ to the following problem:
\begin{align}\label{DtNDivCurl}
\begin{aligned}
(a \in \Sigma_+(\delta)) \qquad &
\left\{
    \begin{aligned}
        \grad \cdot (\, \cof (\grad X^t_+) V ) &= 0 \\
        \grad^\perp\cdot (\, \grad X^t_+ V )   &= 0 \, \Lcm
    \end{aligned}
\right. \\
(a \in \del\Sigma_+(\delta))  \qquad &  \ \big\{ \,(n \circ X_+ ) \cdot V = G , \\
 \qquad &\int_{\del \Sigma_+(\delta)} (V \grad X_+) \cdot d\vec r = 0,
 \end{aligned}
\end{align}
where the directed line integral along $\del\Sigma_+(\delta)$ is taken with paths going left to right along the top and bottom boundaries of $\Sigma_+(\delta)$.

By using Lemma~\ref{DivFreeCurlFreeConversion}, one finds the $V$ above is produced by taking $V=v\circ X_+$, where $v$ is the solution to \eqref{divcurlfreesys} in the case $\Omega^*_+=\Omega_+$, $\nu_1=\nu_2=0$, and $g=(F'\circ X^{-1})/|X_\theta\circ X^{-1}|$, which satisfies the necessary compatibility condition $\int_\Gamma g \,dS=0$.\footnote{That the line integral in \eqref{DtNDivCurl} is zero is verified by using properties related to the fact that $X_+=\cE_\delta(X)$ is a double cover and the construction of Proposition~\ref{extensionProp2}.} Moreover, one verifies that $v=\grad^\perp \phi$ for $\phi$ satisfying the Dirichlet problem in $\Omega_+$ with boundary data $F\circ X^{-1}$ on $\Gamma$. 
Using these observations and referring back to Definition~\ref{DirichToNeumDefs}, we find $\cN_+ F = -\bigtau \cdot V$.

We also have a unique $\underline V:\Sigma_+(\delta)\to\R^2$ solving the analogous system in which the various quantities have been replaced by the analogous versions corresponding to $\uxi$, and we find $\underline \cN_+ F = -\underline \bigtau \cdot \underline V$. Thus
\begin{align}\label{cNdifference1}
\lV (\cN_+  - \underline \cN_+ )F\rV_{H^2(S^1)}
\leq C \left( \lV \bigtau - \underline \bigtau \rV_{H^2(S^1)} \lV V \rV_{H^2(S^1)} + \lV \underline \bigtau \rV_{H^2(S^1)}\lV V - \underline V \rV_{H^2(S^1)}\right) .
\end{align}
Meanwhile, by using that $V$ solves the system \eqref{DtNDivCurl} and $\underline V$ solves the analogous one, one is able to formulate a system for the difference $V-\underline V$. By applying Proposition~\ref{LagrVacDivCurlEst}, we are able to show the following, for appropriate $m_1$ and $m_2$. Let us note that the steps here are similar to those used to prove the Lipschitz bound, (ii), of Proposition~\ref{hEstimateProp}.
\begin{align}
\lV V - \underline V \rV_{H^{5/2}(\Sigma_+(\delta))} &\leq C \Big( \lV \grad \cdot [(\cof(\grad X^t_+) - \cof(\grad \uX^t_+)) V] \rV_{H^{3/2}(\Sigma_+,\delta,m_1)} \\
&\qquad +\lV \grad^\perp \cdot [(\grad X^t_+ - \grad \uX^t_+) V] \rV_{H^{3/2}(\Sigma_+,\delta,m_1)} + \lV G-\underline G \rV_{H^2(\del\Sigma_+,\delta,m_1)} \\
&\qquad + \lV (n\circ X_+ - \underline n \circ \uX_+)\cdot V \rV_{H^2(\del\Sigma_+,\delta,m_1)} + \sup_{a\in \del\Sigma_+(\delta)}| V(\grad X_+ - \grad \uX_+)| \Big)\\
&\leq C'\left( \lV \uX_+ - X_+ \rV_{H^{7/2}(\Sigma_+,\delta,-m)}\lV  V \rV_{H^{5/2}(\Sigma_+,\delta,m_2)} + \lV G-\underline G \rV_{H^2(\del\Sigma_+,\delta,m_1)}\right) \\
&\leq C'' \left( \lV \uX-X \rV_{H^3(S^1)} \lV F \rV_{H^3(S^1,\delta,m_2)}+\lV G-\underline G \rV_{H^2(\del\Sigma_+,\delta,m_1)}\right).
\end{align}
We use the above to bound $\lV V - \underline V \rV_{H^2(S^1)}$ and incorporate the result in \eqref{cNdifference1}, again using the bound \(\lV V \rV_{H^2(S^1)} \leq C \lV F \rV_{H^3(S^1,\delta,m_2)}\). With the same techniques, we then bound $\lV G-\underline G \rV_{H^2(\del\Sigma_+,\delta,m_1)}$, finding
\begin{align}
\lV (\cN_+ - \underline \cN_+) F\rV_{H^2(S^1)}
&\leq C \left(\lV X - \uX \rV_{H^3(S^1)} \lV F \rV_{H^3(S^1,\delta,m_2)}+\lV G- \underline G \rV_{H^2(S^1)}\right) \\
&\leq C'\lV X - \uX \rV_{H^3(S^1)} \lV F \rV_{H^3(S^1,\delta,m_3)} .
\end{align}
Combining this with Proposition~\ref{LipDirichNeumBd} gives \eqref{lipNres}.
\end{proof}

\subsubsection*{Remaining bounds needed for $[\del_t;\Nres](\xi)$}

Similar to the cancellation bound for $\Nres(\xi)$, we also need a kind of cancellation bound for the time derivative commutator type operator $[\del_t;\Nres](\xi)$. Section~\ref{commutatormapssection} includes a rigorous definition of this operator as well as the proof of the proposition below.
\begin{myprop}{\ref{CommutatorCancelBd}}
For some $m \geq 0$ the following holds. Let $\xi$ be in $\sB^\bbox(\delta)$ for $0<\delta\leq \delta_0$. Then for the corresponding operator \([\del_t;\Nres](\xi)\), for any integer $s$ with \(0 \leq s \leq k+1 \), we have the estimate
\begin{align}
            \lV [\del_t;\Nres](\xi) F \rV_{H^s(S^1)} \leq C(M) \lV F \rV_{H^s (S^1,\delta,m)} .
        \end{align}
\end{myprop}
Likewise, we need a Lipschitz estimate for $[\del_t;\Nres](\xi)$ we state here, proved in Section~\ref{commutatormapssection}.
\begin{myprop}{\ref{commNresLipBound}}
For some $m\geq0$, given any $\xi$ and $\uxi$ in $\sB^\bbox(\delta)$ for $0<\delta\leq \delta_0$ we have the estimate
\begin{align}
\lV ([\del_t;\Nres](\xi)-[\del_t;\Nres](\uxi) )F\rV_{H^2(S^1)} \leq C \lV \xi-\uxi \rV_{\sH^2} \lV F \rV_{H^3(S^1,\delta,m)}.
\end{align}
\end{myprop}

\subsubsection*{Concluding bounds for the main nonlinear terms}

Now that we have the bounds necessary to define nearly all the maps in the Lagrangian wave system for $\xi$ in $\sB$, let us recall $\sN(\xi)$, $\dot\sN(\xi)$, and $\mathring \cF_\sN(\xi)$, first discussed in Remark~\ref{dcF1forBcircDefn}:
\begin{align}
\sN(\xi) &= \half \Nres(\xi) | H(\xi)|^2 ,\\
\dot \sN(\xi) &=  \Nres(\xi) ( H(\xi) \cdot \dot H(\xi) ) +\half [\del_t;\Nres](\xi) | H(\xi)|^2 , \\
 \mathring \cF_\sN(\xi) &=\bR^{11}(\xi)  (\sN(\xi) e_2 )  -
 \dot \sN(\xi)e_2  .
\end{align}
Having $\xi$ in $\sB$ requires that the interface $\Gamma$ has zero pinch at time $t=0$ despite the fact that we only define $\Nres(\xi)$ and $[\del_t ; \Nres](\xi)$ for $\xi$ with positive pinch. 

It is necessary for us to define $\dcF_1(\xi)$ for any $\xi$ in $\sB$. We are able to do so with Lemma~\ref{sNiExt} only after having established the essential uniform bounds of Lemma~\ref{sNilemma}, below.
\begin{lemma}\label{sNilemma}
We have for $\mathring \cF_\sN$ given by \eqref{ringcFsNDef}
\begin{align}
\mathring \cF_\sN: \sB^\wbox \to H^k(S^1) ,
\end{align}
and for $\xi$ and $\uxi$ in $\sB^\wbox$, we have the following estimates:
\begin{align}
\lV \mathring \cF_\sN(\xi) \rV_{H^k(S^1)} &\leq C(M)   , \label{sNiUnifBd}\\
\lV \mathring \cF_\sN(\xi)-\mathring \cF_\sN(\uxi) \rV_{H^2(S^1)} &\leq C \lV \xi - \uxi \rV_{\sH^2} .\label{sNiLipBd}
\end{align}
\end{lemma}
\begin{proof}
         The above is an easy consequence of the uniform weighted bounds for $ H(\xi)$ and $\dot H(\xi)$ given by Propositions~\ref{hEstimateProp} and~\ref{dotHEstimateProp} in combination with the main cancellation bounds of Propositions~\ref{DtoNCancellation} and~\ref{CommutatorCancelBd}, and the Lipschitz bounds for $\Nres(\xi)$ and $[\del_t;\Nres](\xi)$ of Propositions~\ref{NresLipBd} and~\ref{commNresLipBound}. We also use Proposition~\ref{BasicRBound} to control $\bR^{11}(\xi)$ (defined in \eqref{bRderivation}), and in order to bound $\bR^{11}(\xi)-\bR^{11}(\uxi)$, we use in particular the estimates for $\bE(\xi)-\bE(\uxi)$ and $[\del_t;\bE]-[\del_t;\bE](\uxi)$ given by Propositions~\ref{bEbounds} and~\ref{bEcommBounds}.
\end{proof}
Now let us recall how Lemma~\ref{sNilemma} allows us to extend $\mathring \cF_\sN(\xi)$ to a map on $\sB^\bbox$, as stated in Lemma~\ref{sNiExt}, which in turn allows us to define $\dcF_1(\xi)$ for all $\xi$ in $\sB$ with Definition~\ref{finalDefdcF1}. This means that all the terms are well-defined for $\xi$ in $\sB$ in the $\xi$ system
\begin{align}
\del_t \xi = \bJ(\xi) \, \del_\theta \xi + \cR(\xi) \, \xi + \dcF(\xi)  .
\label{XiSystem2}
\end{align}
Furthermore, the bounds we have established imply that the source term $\dcF(\xi)$ is truly lower order in the system above.

With Proposition~\ref{dcF2UnifBd}, we verified that $\dcF_2(\xi)$ remains uniformly bounded in the necessary norm for the class of $\xi$ considered. Now we conclude with the following proposition, which uses many of the bounds proved above to confirm that, similarly, $\dcF_1(\xi)$ remains bounded in the appropriate norm. In the next section, we will use these bounds to construct the solution to \eqref{XiSystem2} which ultimately gives us our splash--squeeze singularity.
\begin{proposition}\label{dcF1BoundProp}
For $\xi$ in $\sB$ and $\dcF_1(\xi)$ given by Definition~\ref{finalDefdcF1},
\begin{align}\label{estimatedcF1}
\sup_{t}\lV\dcF_1(\xi)\rV_{H^k(S^1)} &\leq C(M).
\end{align}

\end{proposition}
\begin{proof}

Referring back to the definitions of Remark~\ref{dcF1forBcircDefn} and Definition~\ref{finalDefdcF1}, we prove the estimate \eqref{estimatedcF1} by using Lemmas~\ref{bEboundLemma} and~\ref{sNiExt} together with Propositions~\ref{PminusBdProp},~\ref{BasicRBound},~\ref{PplusbdProp},~\ref{dotPplusSysProp}, and~\ref{PminusDotBds}.
\end{proof}
Now we get the following immediate corollary from Proposition~\ref{dcF1BoundProp} and Proposition~\ref{dcF2UnifBd}.
\begin{corollary}\label{maindcFBound}
For $\xi$ in $\sB$ and $\dcF(\xi)$ given by Definition~\ref{finalDefdcF1},
\begin{align}
\sup_{t}\lV \dcF(\xi)\rV_{\sH^k} \leq C(M) .
\end{align}
\end{corollary}
Let us also prove an additional estimate for a term which mainly arises at a more technical point of our local existence construction in Section~\ref{constructionsection}.
\begin{proposition}\label{PenultBoundConstDepProp}
For $\cF(\xi)$ defined for $\xi$ in $\sB$ by\footnote{Note $\sN(\xi)$ is not originally defined for $\xi$ with pinch zero. By using the same technique as in the proof of Lemma~\ref{sNiExt}, we first extend it to $\xi$ in $\sB^\bbox$.}
\begin{align}\label{cFdefn}
\cF(\xi) &= -\bE(\xi) P^-_\grad(\xi) - N\cdot P^+_\grad(\xi) e_2 - \sN(\xi) e_2 ,
\end{align}
we have the estimate
\begin{align}\label{polynomialbounds}
\sup_{t}\lV \cF(\xi)  \rV_{H^k(S^1)}
&\leq C \left(\sup_{t}\lV \xi \rV_{\sH^{k-1}} \right) .
\end{align}
\end{proposition}
\begin{proof}
The proof is similar to that of Proposition~\ref{dcF1BoundProp}.
\end{proof}

Finally, we finish off the Lipschitz bounds for some remaining terms in the $\xi$ system by combining our various Lipschitz estimates together.

\begin{proposition}\label{finalLipBdSummary}
For $\xi,\uxi$ in $\sB$, we have the following estimates, where $\cF(\xi)$ is defined by \eqref{cFdefn}, $R(\xi)$ by Definition~\ref{DEcommDefns}, and $\dcF(\xi)$ by Definition~\ref{finalDefdcF1}.
\begin{align}
\lV R(\xi)-R(\uxi) \rV_{B(\sH^2_\dagger,\sH^2_\dagger)} &\leq C\lV \xi -\uxi \rV_{\sH^2} ,\label{lipcRBd}\\
\sup_{t}\lV\dcF(\xi)-\dcF(\uxi)\rV_{\sH^2} &\leq C\sup_{t}\lV \xi -\uxi \rV_{\sH^2},\label{lipdcFBd}\\
\sup_{t}\lV\cF(\xi)-\cF(\uxi)\rV_{H^2(S^1)} &\leq C\sup_{t}\lV \xi -\uxi \rV_{\sH^1}.\label{lipcFBd}
\end{align}
\end{proposition}
\begin{proof}
Recall from Definition~\ref{finalDefdcF1} that $\dcF(\xi)$ consists of only the components $\mathring \cF(\xi)$ and $\dcF_2(\xi)$. In the proof below we use the notation $\mathring\cF = \mathring \cF(\xi)$, $\dcF_2=\dcF_2(\xi)$, $\underline {\mathring\cF} = \mathring \cF(\uxi)$, $\underline \dcF_2 = \dcF_2(\uxi)$, etc., dropping the arguments $\xi$ and $\uxi$ from various terms and instead using underlines to denote the quantities associated to $\uxi$. From Definition~\ref{dcF2Defn} we find
\begin{align*}
\dcF_2 - \underline \dcF_2 &= 2 \bm \sigma^{-1} e_2 \big(   ( \bU_\theta - \underline \bU_\theta) \cdot \bB_\psi  + \underline \bU_\theta \cdot(\bB_\psi - \underline \bB_\psi)- ( \bU_\psi - \underline \bU_\psi) \cdot \bB_\theta - \underline \bU_\psi \cdot (\bB_\theta - \underline \bB_\theta)     \big),
\end{align*}
and so by using the Lipschitz bounds given by Proposition~\ref{LipBdMainDivCurlSysProp} we get
\begin{align}
\sup_{t}\lV \dcF_2 - \underline \dcF_2 \rV_{H^{5/2}(\Sigma)}\leq C \sup_{t} \lV \xi - \uxi \rV_{\sH^2}  .\label{dcF2lipBd}
\end{align}

Now we consider $\mathring \cF -\underline{\mathring\cF}$. Recall the expressions given for $\mathring \cF_\pm$ in \eqref{ringcFpmDef2} and \eqref{ringcFpmDef1}, along with the expression for $\bR^{11}$ in \eqref{bRderivation}. Recall the expressions for $\bD$ and $\dot \bD$ from Definition~\ref{DEcommDefns}. Using the fact that $|X_\theta|\geq c$ for $\xi$ in $\sB$, it is not hard to verify from Propositions~\ref{hEstimateProp} and~\ref{dotHEstimateProp}, which bound $H-\underline H$ and $\dot H-\underline{\dot H}$, that
\begin{align}
\lV\bD - \underline \bD \rV_{H^2(S^1)}+\lV\dot\bD - \underline{\dot\bD}\rV_{H^2(S^1)} &\leq C \lV \xi-\uxi\rV_{\sH^2}.
\end{align}
Using this along with the bounds on $\bE-\underline \bE$ and $[\del_t;\bE]-[\del_t;\underline \bE]$ from Propositions~\ref{bEbounds} and~\ref{bEcommBounds} we easily get a suitable bound for $\bR^{ij}-\underline \bR^{ij}$ for $i,j=1,2$ (see \eqref{bRderivation} for the definitions of the $\bR^{ij}$):
\begin{align}
\lV (\bR^{ij}-\underline \bR^{ij})V \rV_{H^2(S^1)} &\leq C\lV \xi-\uxi \rV_{\sH^2} \lV V \rV_{H^2(S^1)} .\label{RijLipbd}
\end{align}
Note that the first bound \eqref{lipcRBd} follows from \eqref{RijLipbd}. Using additionally the Lipschitz bounds for $  P^\pm_\grad$ and $\dot P^\pm_\grad$ given by Propositions~\ref{PplusbdProp},~\ref{dotPplusSysProp},~\ref{PminusBdProp}, and~\ref{PminusDotBds}, we also find
\begin{align}
\sup_{t}\lV \mathring \cF_\pm - \underline{\mathring \cF}_\pm \rV_{H^2(S^1)} \leq C \sup_{t}\lV \xi-\uxi \rV_{\sH^2}.
\end{align}
Recalling \eqref{mathrcFdef}, we thus conclude from the bound on $\mathring \cF_\sN-\underline{\mathring \cF}_\sN$ of Lemma~\ref{sNilemma} that
\begin{align}
\sup_{t}\lV \mathring \cF -\underline{\mathring \cF} \rV_{H^2(S^1)} \leq C\sup_{t} \lV \xi -\uxi \rV_{\sH^2}
\end{align}
which, when combined with \eqref{dcF2lipBd}, yields \eqref{lipdcFBd}. The estimate \eqref{lipcFBd} is proven similarly.
\end{proof}

\section{Backward-in-time splash--squeeze construction in Sobolev spaces}\label{constructionsection}


The Lagrangian wave system for $\xi$ is the quasilinear equation below:
\begin{align}\label{xiSystemRep}
& &    \xi_t &= \bJ(\xi) \xi_\theta + \cR(\xi) \xi + \dcF (\xi)    &(t&\in[0,T]), \\
& &     \xi &= \xi_0      &(t &= 0) .
\end{align}
For $\xi$ in $\sB$, $\bJ(\xi)$ and $\cR(\xi)$ above are given by Definition~\ref{DEcommDefns}, $\dcF(\xi)$ is given by Definition~\ref{finalDefdcF1}, and $\xi_0$ is given by Definition~\ref{initialxiData}.

With the vast majority of the setup out of the way, in this section we construct a solution to ideal MHD demonstrating a splash--squeeze singularity. Let us remind the reader of the overall structure of our argument, now with all the working parts in place. In Section~\ref{solvinglagwavesystemsection}, we construct a solution $\xi$ to \eqref{xiSystemRep} with initial data $\xi_0=(\dot U^*_0, \dot B^*_0, \bom_0,\bj_0,X_0,U_0)$, where $X_0$ parametrizes the interface starting out as a splash curve. For this solution, the self-intersecting interface opens up, becoming non-self-intersecting for $t>0$. Once we have the solution $\xi$ to \eqref{xiSystemRep}, which is a kind of reformulation of the original ideal MHD equations \eqref{IdealMHD1}--\eqref{IdealMHD6}, we then demonstrate in Section~\ref{backtoorigsystemsubsection} exactly how it yields a solution to \eqref{IdealMHD1}--\eqref{IdealMHD6}, with Proposition~\ref{maintheoremforward}. In Section~\ref{maintheoremsubsection}, we invoke time reversibility for the ideal MHD system to produce a solution which starts with the interface separated at time $t=0$ and which terminates in a splash at $t=t_\splash$. This is done in the proof of Theorem~\ref{maintheorem}, which follows almost immediately from the previously mentioned results.

\subsection{Solving the Lagrangian wave system}\label{solvinglagwavesystemsection}

To construct the solution which starts in a splash, we generate a sequence of iterates $\xi_n$ solving the linearized systems
\begin{align}\label{xiItSystRep}
& &    \frac{d\xi_{n+1}}{dt} &= \bJ(\xi_n) \del_\theta\xi_{n+1} + \cR(\xi_n) \xi_{n+1} + \dcF (\xi_n)  &   (t&\in[0,T]), \\
& &     \xi_{n+1} &= \xi_0      &(t &= 0) .
\end{align}
To explain the setting for the iteration scheme, let us consider $\xi_\dagger$, the truncated version of a state vector $\xi=(\dot U^*, \dot B^*, \bom, \bj,X,U)$. For the truncated version, we omit the less crucial components $X$ and $U$, which depend implicitly on $\xi_\dagger$, defining
\[\xi_\dagger = (\dot U^*, \dot B^*, \bom, \bj) . \]

The iteration scheme is then run in the Banach space ``ball'' $\sB$, defined by
\begin{align}
\sB_\dagger &=\{\xi_\dagger\in C^0([0,T];\sH^k_\dagger) : \sup_{t} \lV \xi_\dagger \rV_{\sH^k_\dagger} \leq M, \ \sup_{t} \lV \xi_\dagger - \xi_{\dagger,0} \rV_{\sH^{k-1}_\dagger} \leq r\} ,\\
    \sB &= \{\xi \in C^0([0,T]; \sH^k) : \sup_{t} \lV \xi \rV_{\sH^k} \leq M, \ \xi_\dagger \in \sB_\dagger, \ \xi(0) = \xi_0,  \ X_t = U , \  \lV U\rV_{C^1_{t,\theta}}\leq M\}  .
\end{align}

To sketch the iteration step, let us suppose we are either starting only with $\xi_0$, the initial datum, which is in $\sB^\bbox$, or that we have already constructed an iterate $\xi_n$ in $\sB$ for some $n\geq 1$. Instead of directly solving the system \eqref{xiItSystRep} for $\xi_{n+1}$, we first solve the system for the truncated vector $\xi_{\dagger,n+1}$. This system is basically just the version of \eqref{xiItSystRep} in which we drop the last two equations, which are for $\del_t X_{n+1}$ and $\del_t U_{n+1}$.

Let us provide the definition now for the lower order terms to appear in our iteration scheme. Of course, it will only make sense in tandem with the explanation of how the iterates themselves are generated, which we explain shortly afterwards.
\begin{definition}\label{lowerordertermsIterateSys}
Assume we have already defined the $n^\text{th}$ iterate, $\xi_n$, for some $n\geq 0$. Suppose either $n=0$, so that $\xi_n=\xi_0$ is our initial datum, or that $\xi_n$ is in $\sB$. Referring to the maps given in Definitions~\ref{DEcommDefns},~\ref{finalDefdcF1}, and~\ref{dcF2Defn}), we define
\begin{align}\label{AnForm}
J_n &=\begin{pmatrix} \bJ_1(\xi_n) & 0 \\ 0 & \bJ_2 \end{pmatrix} && (n\geq 0), \\
R_n &= \begin{pmatrix}\bR^{11}(\xi_n) & \bR^{12}(\xi_n) & 0 \\ 
\bR^{21}(\xi_n) & \bR^{22}(\xi_n) & 0 \\
0 & 0 & 0_{2\times 2}
\end{pmatrix} && (n\geq 1),\qquad R_0 = \begin{pmatrix}  0 & \bR^{12}(\xi_0) & 0 \\ 
\bR^{21}(\xi_0) & 0 & 0 \\
0 & 0 & 0_{2\times 2}
\end{pmatrix},\\
F_n &= \col{\dcF_1(\xi_n)}{\dcF_2(\xi_n)} && (n\geq 1), \qquad F_0 = 0 .
\end{align}
\end{definition}
Given $\xi_n$ for $n\geq 0$, the system for the subsequent truncated iterate $\xi_{\dagger,n+1}$ is then
\begin{align}\label{definingSysIte}
\frac{d \xi_{\dagger,n+1}}{dt} &= J_n \del_\theta \xi_{\dagger,n+1} + R_n \xi_{\dagger,n+1} + F_n(t) , \\
\xi_{\dagger,n+1}(0) &= \xi_{\dagger,0} .
\end{align}
Due to the form of $J_n$, \eqref{definingSysIte} is simply a linear wave system with variable coefficients, although in a slightly more abstract setting than usual, since $R_n$ is a nonlocal operator. We use Kato's semigroup method to construct the solution $\xi_{\dagger,n+1}$. Finally, we use the $\xi_{\dagger,n+1}$ obtained to get $U_{n+1}$ and $X_{n+1}$ via \eqref{Undefn}--\eqref{xindefn}, defining the remaining components of the $\xi_{n+1}$ vector and thus finishing the iteration step. Of course, we must show $\xi_{n+1}$ is in $\sB$ for the scheme to close.

The above process is handled in detail in Proposition~\ref{MainInductiveBoundProp}, used to rigorously construct the sequence of iterates $\{\xi_n\}_{n\geq 1}$. Following the construction of the sequence are Propositions~\ref{LimitExistence} and~\ref{improvedxireg}, used to prove the existence of a limit $\xi$ in $\sB$ which solves the quasilinear system \eqref{xiSystemRep}.
\subsubsection{Kato's method for the linearized wave evolution}\label{katosmethodsection}

Constructing the solutions to the linear systems \eqref{definingSysIte} for the iterates with Kato's semigroup method amounts to checking certain criteria and applying Theorem 1 from \cite{kato2} to define a semigroup solution operator for a homogeneous problem.

To explain how this works, let us define for $n\geq 0$
\begin{align}
& & A_n(t) &= J_n(t) \del_\theta + R_n(t) & (t\in[0,T]),
\end{align}
and suppose we can solve the homogeneous problem below for each $s$ in $[0,T]$ and any data in $\sH^k_\dagger$ to be taken at time $s$, which we denote by $\xi_{\dagger;s}$.
\begin{align}\label{homogsys1}
& & \frac{d\xi_\dagger}{dt} &= A_n(t) \xi_\dagger & (t\in[0,T]),\\
& & \xi_\dagger(t)&=\xi_{\dagger;s} &   (t=s).
\end{align}
Consider a solution operator $S_n(t,s)$ mapping the choice of data to the corresponding solution at time $t$,
\begin{align}
\xi_\dagger(t)=S_n(t,s)\xi_{\dagger;s} .
\end{align}
We can then use Duhamel's principle to describe the solution to \eqref{definingSysIte} as
\begin{align}
\xi_{\dagger,n+1}(t) = S_n(t, 0) \xi_{\dagger,0} + \int^t_0 S_n(t, s) F_n(s) \, ds .
\end{align}

To apply Theorem 1 from \cite{kato2}, which is quite general, in order to verify the existence of an operator $S_n(t,s)$ which maps a ``state vector at time $s$'' to ``the corresponding solution state at time $t$'', the key properties we observe among the pieces $J_n(t)$ (matrix-valued) and $R_n(t)$ (operator-valued) making up the operator $A_n(t) = J_n(t)\del_\theta + R_n(t)$ are the following:
\begin{enumerate}
\item We have that $J_n(t)$ is symmetrizable and bounded in $\sH^k_\dagger$ (see Definition~\ref{DEcommDefns} and Proposition~\ref{cRbJBoundProp}).
\item We have that $R_n(t)$ preserves regularity in $\sH^j_\dagger$ for $0\leq j\leq k$ (See Proposition~\ref{BasicRBound}).
\end{enumerate}

The details are explained in Proposition~\ref{KatoProp}. Before diving into the proof, let us outline some ideas behind the enumerated points above and their relation to the existence of solutions.

Regarding the first point, to say $J_n(t)$ is symmetrizable (in the Lax-Friedrichs sense) means that there exists a positive definite matrix $M_n(t)$ such that $M_n(t) J_n(t)$ is symmetric. This essentially implies that the system \eqref{homogsys1} yields good energy estimates. 

Let us relate this to a major criterion of Kato's method, that the associated family of semigroup generators is \emph{stable}\footnote{See \cite{kato1} for the definition of a stable family of generators.}. Let us fix an $n$ and for any $s$ in $[0,T]$ consider the problem
\begin{align}\label{homogsys2}
& & \frac{d\xi_\dagger}{dt} &= J_n(s)\del_\theta \xi_\dagger & (t\geq 0), \\
& & \xi_\dagger &= \xi_{\dagger;0} & (t=0).
\end{align}
We express the solution in terms of the semigroup generator as $\xi_\dagger(t)=e^{t J_n(s) \del_\theta}\xi_{\dagger,0}$. Thanks to the structure of $J_n$, standard energy estimates show for any $j$ with $0\leq j \leq k$,
\begin{align}\label{stabilityest1}
& & \left\lV e^{t J_n(s)\del_\theta} \right\rV_{B(\sH^j_\dagger,\sH^j_\dagger)} & \leq C e^{C(M) t} &   ( s \in [0,T], \ t\geq 0 ).
\end{align}
The property \eqref{stabilityest1} implies that the family of generators \(\{J_n(t)\del_\theta\}_{t\in[0,1]}\) is stable, Similarly, using that $R_n(s)$ preserves regularity, one finds \eqref{stabilityest1} holds with $A_n(s)=J_n(s)\del_\theta + R_n(s)$ in place of $J_n(s)\del_\theta$, and therefore the family \(\{A_n(t)\}_{t\in[0,1]}\) is also stable. Besides this, to guarantee the existence of the solution operator $S_n(t,s)$ to \eqref{homogsys1}, Theorem 1 of \cite{kato2} only requires continuity with respect to $t$ and easy-to-check bounds involving $A_n(t)$.

\begin{proposition}\label{KatoProp}
Assume we have already defined the $n^\text{th}$ iterate, $\xi_n$, for some $n\geq 0$. Suppose either $n=0$ or that $\xi_n$ is in $\sB$.
Suppose either $n=0$ or that we have \( \xi_n \) in \( \sB \) for an $n\geq 1$.
\begin{enumerate}[(i)]
\item There exists a semigroup solution operator $S_n(t,s)$ for \eqref{homogsys1}, where $S_n(t,s)$ is in $B(\sH^j_\dagger,\sH^j_\dagger)$ for $0\leq j\leq k$, and it is strongly continuous with respect to $(t,s)$ in $[0,T]^2$. In other words, for any fixed $s$ in $[0,T]$ and $\xi_{\dagger;s}$ in $\sH^k_\dagger$, the map $t\mapsto S_n(t,s) \xi_{\dagger;s}$ is a solution to \eqref{homogsys1}, and for a fixed $\xi_{\dagger;0}$ in $\sH^k_\dagger$, the map $(t,s)\mapsto S_n(t,s)\xi_{\dagger;0}$ continuously takes values in $\sH^k_\dagger$.

\item Additionally, for $t,s$ in $[0,T]$, $0\leq j \leq k$,
\begin{align}\label{SntsBound}
\lV S_n(t,s) \rV_{B(\sH^j_\dagger,\sH^j_\dagger)} \leq Ce^{C(M)(t-s)} .
\end{align}
\end{enumerate}
\end{proposition}
\begin{proof}

Let us define \( \bm{\langle \grad \rangle} \) acting on a function $f$ either with domain $S^1$ or $\Sigma$, via
\begin{align}
&& \bm{\langle \grad \rangle} f& =(1-\del^2_\theta)^{1/2}f= \langle \del_\theta \rangle f &&(f:S^1\to\R^N, \ \theta \in S^1),\\
&& \bm{\langle \grad \rangle} f& = (1-\del^2_\theta-\del^2_\psi)^{1/2}f &&(f:\Sigma\to\R^N, \ a\in\Sigma).
\end{align}

We verify three hypotheses in order to apply Theorem 1 from \cite{kato2}, which implies the existence of an operator $S_n(t,s)$ satisfying (i). They are as follows:

\begin{enumerate}[(1)]
\item the family of generators $\{A_n(t)\}_{t}$ is stable;

\item the commutators \( [\bm{\langle \grad \rangle}^{k-2}, A_n(t)] \) are uniformly bounded in $B(\sH^k_\dagger, \sH^2_\dagger)$;

\item $A_n(t)$ is in $B(\sH^k_\dagger, \sH^2_\dagger)$, with $A_n(t)$ continuous in $t$ on $[0,T]$ in the operator norm topology.
\end{enumerate}

We explain how to check the above in detail for $n\geq 1$. Adapting the argument to the case $n=0$ is trivial. Let us justify (1). From Proposition~\ref{BasicRBound},
\begin{align}\label{RnBound}
& & \lV R_n(t) \xi_\dagger \rV_{\sH^2_\dagger} &\leq C(M) \lV  \xi_\dagger \rV_{\sH^2_\dagger}  &   ( \xi_\dagger \in \sH^2_\dagger, \ t \in [0,T] ) .
\end{align}
Meanwhile, we also find the bound \(\sup_{t}\lV J_n\rV_{H^k(S^1)}\leq C(M)\) follows from Corollary~\ref{Jbounds}. Using this in combination with \eqref{RnBound}, standard energy estimate techniques for wave equations yield the following: for each $s$ in $[0,T]$, the corresponding solution $\xi_\dagger$ to the homogeneous problem
\begin{align}\label{homogsys3}
& & \frac{d\xi_\dagger}{dt} &= A_n(s)\del_\theta \xi_\dagger & (t\geq 0), \\
& & \xi_\dagger &= \xi_{\dagger;0} & (t=0),
\end{align}
satisfies for $j$ with $0\leq j \leq k$ the estimate below.
\begin{align}
&& \lV \xi_\dagger(t) \rV_{\sH^j_\dagger} &\leq C\lV \bD_n(s)\rV_{L^\infty} e^{C(M) t} \lV \xi_{\dagger;0}\rV_{\sH^j_\dagger} &(t\geq 0).
\end{align}
By using the Lipschitz bound on $H(\xi)$ of Proposition~\ref{hEstimateProp} and the continuity of $\xi_n(t)$, for $T$ sufficiently small we may bound $\lV \bD_n(s)\rV_{L^\infty}$ by a constant dependent only on the initial interface. Therefore we have
\begin{align}\label{stabilityest2}
& & \left\lV e^{t A_n(s)\del_\theta} \right\rV_{B(\sH^j_\dagger,\sH^j_\dagger)} & \leq C e^{C(M) t} &   ( s \in [0,T], \ t\geq 0, \ 0\leq j \leq k ),
\end{align}
which implies \(\{A_n(t)\del_\theta\}_{t\in[0,1]}\) is stable.

Now we consider (2). First we assert that
\begin{align}\label{commBound0}
\sup_{t}\lV [\bm{\langle \grad \rangle}^{k-2},J_n\del_\theta] \rV_{B(\sH^k_\dagger,\sH^2_\dagger)} \leq C(M).
\end{align}
Thanks to the form of $J_n$, we have
\begin{align}
[\bm{\langle \grad \rangle}^{k-2},J_n\del_\theta] = 
\begin{pmatrix}[\bm{\langle \grad \rangle}^{k-2},\bJ_1(\xi_n)\del_\theta] &  0 \\ 0 & 0_{2\times 2}
\end{pmatrix} .
\end{align}
Using the above together with the fact that $\bm{\langle \grad \rangle}=(1-\del^2_\theta)^{1/2}=\langle \del_\theta \rangle$ when acting on functions independent of $\psi$, for $\xi_\dagger$ in $\sH^k_\dagger$,
\begin{align}\label{commBound1}
\lV [\bm{\langle \grad \rangle}^{k-2},J_n \del_\theta] \xi_\dagger \rV_{\sH^2_\dagger}
=\left\lV [\langle \del_\theta \rangle^{k-2},\bJ_1(\xi_n) \del_\theta] \col{\dot U^*}{\dot B^*}\right \rV_{H^2(S^1)} .
\end{align}
Now, one is able to get the desired bound by using the standard bound below for the commutator of $\langle \del_\theta \rangle^j$ with multiplication by a function $f(\theta)$:
\begin{align}\label{commestimate}
\lV[\langle \del_\theta \rangle^{k-2}, f] g \rV_{H^2(S^1)} \leq C \lV f  \rV_{H^k(S^1)}\lV g \rV_{H^{k-1}(S^1)} .
\end{align}
Indeed, using that $\del_\theta$ commutes with $\langle \del_\theta \rangle^j$ and applying \eqref{commestimate}, we get
\begin{align*}
\left\lV [\langle \del_\theta \rangle^{k-2},\bJ_1(\xi_n)\del_\theta] \col{\dot U^*}{\dot B^*}\right \rV_{H^2(S^1)} 
&=\left\lV [\langle \del_\theta \rangle^{k-2},\bJ_1(\xi_n) ] \del_\theta \col{\dot U^*}{\dot B^*}\right \rV_{H^2(S^1)} \\
&\leq C \lV \bJ_1(\xi_n)  \rV_{H^k(S^1)}( \lV \dot U^* \rV_{H^k(S^1)} + \lV \dot B^* \rV_{H^k(S^1)}) \\
&\leq C(M) \lV \xi_\dagger\rV_{\sH^k_\dagger}.
\end{align*}
Combining this with \eqref{commBound1} gives us \eqref{commBound0}. In light of \eqref{commBound0}, the statement of (2) follows provided we show
\begin{align}
\lV [\bm{\langle \grad \rangle}^{k-2},R_n] \rV_{B(\sH^k_\dagger,\sH^2_\dagger)} \leq C(M).
\end{align}
This is confirmed by using Proposition~\ref{BasicRBound}, from which we find for $\xi_\dagger$ in $\sH^k_\dagger$
\begin{align}
\lV [\bm{\langle \grad \rangle}^{k-2},R_n]\xi_\dagger \rV_{B(\sH^k_\dagger,\sH^2_\dagger)}
&\leq \lV \bm{\langle \grad \rangle}^{k-2} (R_n \xi_\dagger) \rV_{\sH^2_\dagger} + \lV  R_n \bm{\langle \grad \rangle}^{k-2} \xi_\dagger \rV_{\sH^2_\dagger} \\
&\leq C(M) \lV \xi_\dagger \rV_{\sH^k_\dagger}.
\end{align}

Regarding (3), by using \eqref{RnBound} together with Corollary~\ref{Jbounds} it is not hard to show
\begin{align}
\lV A_n(t) \rV_{B(\sH^k_\dagger,\sH^2_\dagger)} \leq C(M).
\end{align}
Regarding continuity of $A_n(t)$ in the operator norm topology, it only remains to verify for $t$ in $[0,T]$
\begin{align}
\lim_{s\to t}\left(\sup_{\lV V\rV_{H^k} = 1} \lV(A_n(t)-A_n(s)) V\rV_{H^2(S^1)}\right) = 0.
\end{align}
The above is not hard to show by using the Lipschitz bounds with respect to $\xi$ satisfied by $H(\xi)$, $\dot H(\xi)$, $\bE(\xi)$, and $[\del_t;\bE](\xi)$ of Propositions~\ref{hEstimateProp},~\ref{dotHEstimateProp},~\ref{bEbounds}, and~\ref{bEcommBounds} together with the fact that $\xi_n$ is in $C^0([0,T];\sH^k)$.

Properties (1), (2), and (3) imply the hypotheses of Theorem 1 from \cite{kato2}, which then asserts the existence of the solution operator $S_n(t,s)$ satisfying (i). Furthermore, (1), (2), and (3) imply the hypotheses of Theorem 5.1 of \cite{kato1}. The bound \eqref{stabilityest2} in combination with Theorem 5.1 of \cite{kato1} implies \eqref{SntsBound}.
\end{proof}

\subsubsection{Iteration and convergence}

We now define the $(n+1)^\text{st}$ iterate $\xi_{n+1}$ given $\xi_n$. This definition consists of solving the linear system for $\xi_{\dagger,n+1}$, followed by a reconstruction of the full state $\xi_{n+1}$. The next proposition formalizes this step.

\begin{proposition}
\label{MainInductiveBoundProp}
Assume we have already defined the $n^\text{th}$ iterate, $\xi_n$, for some $n\geq 0$. Suppose either $n=0$ or that $\xi_n$ is in $\sB$. We use Definition~\ref{lowerordertermsIterateSys} to define the terms $A_n$ and $F_n$ corresponding to the $n^\text{th}$ iterate. It follows that

\begin{enumerate}[(i)]
\item there exists a solution $\xi_{\dagger,n+1}=(U^*_\npo,B^*_\npo,\bom_\npo,\bj_\npo)$ in $\sB_\dagger$ to the system
\begin{align}
\label{iteratedXiSys}
\frac{d \xi_{\dagger,n+1}}{dt} &= A_n(t) \xi_{\dagger,n+1} + F_n(t) , \\
\xi_{\dagger,n+1}(0) &= \xi_{\dagger,0} .
\end{align}

\item Moreover, let us define
\begin{align}
U_{n+1} &= U_0 + \int^t_0 (\bE(\xi_{n}))^{-1} \dot U^*_{n+1}(\tau)\, d\tau ,\label{Undefn}\\
B_{n+1} &= B_0 + \int^t_0 (\bE(\xi_{n}))^{-1} \dot B^*_{n+1}(\tau)\, d\tau ,\label{Bndefn}\\
X_{n+1} &= X_0 + \int^t_0 U_{n+1}(\tau) \, d\tau ,\label{Xndefn} \\
\xi_{n+1} &= (\xi_{\dagger,n+1}, X_{n+1}, U_{n+1} ) .\label{xindefn}
\end{align}
Then $U_{n+1}$ and $B_{n+1}$ are in $H^{k+1}(S^1)$ and $X_{n+1}$ is in $H^{k+2}(S^1)$. Moreover, the new iterate $\xi_{n+1}$ is in $\sB$ and solves \eqref{xiItSystRep}.
\end{enumerate}

\begin{remark}
Note the proposition asserts that $U_{n+1}$ is in $H^{k+1}(S^1)$ and $X_{n+1}$ is in $H^{k+2}(S^1)$, implying both are of higher regularity than $\dot U^\star_\npo \in H^k(S^1)$. On the other hand, from the definition of $U_{n+1}$ in terms of $\xi_{\dagger,n+1}$ in \eqref{Undefn} (and $X_\npo$ in terms of $U_\npo$ in \eqref{Xndefn}) there is no reason to expect that $X_{n+1}$ and $U_{n+1}$ have more regularity than $\dot U^*_{n+1}$. If they did not, we would lose regularity in the iteration step, since $\xi_n\in\sB$ requires that $X_{n}$ has one more degree of regularity than $U_{n}$ and that $U_n$ has one more than $\dot U^*_{n+1}$. However, as claimed by the proposition, \emph{our iteration step loses no regularity}, and thus we need no artificial smoothing of any kind for our iterates.\footnote{The absence of regularity loss at the iteration step and the introduction of artificial smoothing in our construction sets it apart from other Lagrangian proofs of existence for free-boundary ideal MHD models, with or without splash (see \cite{GuMHDtension,UnifEstFBViscMHD}). The apparent difficulty in avoiding tangential regularity loss for Lagrangian methods is observed in \cite{GuMHDtension,XieLuoALE}.} Due to the wave structure in the system, we are able to relate $X_{n+1}$, $B_{n+1}$, and $U_{n+1}$ with identities that imply they have the regularities claimed above. This is an important step in the proof of Proposition~\ref{MainInductiveBoundProp}
\end{remark}

\begin{proof}

By Proposition~\ref{KatoProp}, we get a solution semigroup operator $S_n(t,s)$ to the associated homogenous problem. We use it to write the desired solution to \eqref{iteratedXiSys} as
\begin{align}
\xi_{\dagger,n+1}(t) = S_n(t, 0) \xi_{\dagger,0} + \int^t_0 S_n(t, s) F_n(s) \, ds .
\end{align}
By using \eqref{SntsBound} and the bound on
\(\lV F_n \rV_{\sH^k_\dagger}\) implied by  Corollary~\ref{maindcFBound}, we get
\begin{align}
\lV \xi_{\dagger,n+1}(t)\rV_{\sH^k_\dagger} &\leq \sup_{\tau,s\in[0,T]} \lV S_n(\tau, s)  \rV_{B(\sH^k_\dagger,\sH^k_\dagger)}( \lV \xi_{\dagger,0} \rV_{\sH^k_\dagger} + T\lV F_n(s) \rV_{\sH^k_\dagger} )  \notag \\
&\leq Ce^{C(1+M^p)T}(\lV \xi_{\dagger,0} \rV _{\sH^k_\dagger}+ T(1+M^p) ) ,\\ \intertext{
which, for $T$ small enough (dependent on $M$ and the initial data), gives the bounds}
\lV \xi_{\dagger,n+1}(t)\rV_{\sH^k_\dagger}&\leq 2C(\lV \xi_{\dagger,0} \rV _{\sH^k_\dagger}+ T(1+M^p)),\\
&\leq 4C \lV \xi_{\dagger,0} \rV _{\sH^k_\dagger}, \label{xiprimeTbound}\\
&\leq M .\label{sBdagBound1}
\end{align}
Above, we use the fact that we may choose $M$ larger than $4C\lV \xi_{\dagger,0} \rV _{\sH^k_\dagger}$.
Now we verify
\begin{align}\label{sBdagBound2}
\sup_{t}\lV \xi_{\dagger,\npo} - \xi_{\dagger,0} \rV_{\sH^{k-1}_\dagger} \leq r .
\end{align}
In view of \eqref{definingSysIte}, the evolution equation for $\xi_{\dagger,\npo}$, we find
\begin{align}
\left\Vert \frac{d\xi_{\dagger,\npo}}{dt} \right\Vert_{\sH^{k-1}_\dagger} &\leq C(1+M^p) (\lV \xi_{\dagger,\npo}\rV_{\sH^k_\dagger}+\lV F_n\rV_{\sH^{k-1}_\dagger})\leq C'(1+M^{p'}),
\end{align}
and thus
\begin{align}
\sup_{t}\lV \xi_{\dagger,\npo} - \xi_{\dagger,0} \rV_{\sH^{k-1}_\dagger} \leq CT(1+M^p).
\end{align}
Thus for small enough $T$ we guarantee \eqref{sBdagBound2} holds in addition to \eqref{sBdagBound1}. The only other property required for $\xi_{\dagger,n+1}$ to be in $\sB_\dagger$ is that $t\mapsto \xi_{\dagger,\npo}$ continuously takes values in $\sH^k_\dagger$. This is standard for iteration schemes for symmetrizable systems like ours. We refer the reader to the proof of Theorem 2.1 (b) from \cite{majdaconservationlaws} for a proof of the analogous result for a class of similar symmetrizable systems.

Now let us show that $\xi_{n+1}$ is in $\sH^k$. We will first use the equations satisfied by $\dot U^*_{n+1}$ and $\dot B^*_{n+1}$ to deduce $U_{n+1}$ and $B_{n+1}$ are of one degree higher regularity than $\dot U^*_{n+1}$ and $\dot B^*_{n+1}$.

We define the following for $n\geq 1$, using the subscript $n$ to denote corresponding objects with $\xi_n$ as the input state vector, using, for example, the definitions of $\dot \bD_n=\dot  \bD(\xi_n)$ from Definition~\ref{DEcommDefns}, of $[\del_t;\bE]_n=[\del_t;\bE](\xi_n)$ from Proposition~\ref{bEcommBounds}, and of $\sN_n=\sN(\xi_n)$ and $\dot \sN_n=\dot \sN (\xi_n)$ from Remark~\ref{dcF1forBcircDefn}.\footnote{Using the same method as in Lemma~\ref{sNiExt}, we extend the definitions of $\sN(\xi)$ and $\dot \sN(\xi)$ to $\xi$ in $\sB^\bbox$.}
\begin{align}
\begin{aligned}
\bR^{11}_n &= \dot \bD _n\bD_n^{-1}+ \bD_n[\del_t;\bE]_n \bE^{-1}_n \bD^{-1}_n , \hspace{15mm}   & \bR^{12}_n &= -\bD_n [\del_\theta, \bE_n] \bE^{-1}_n , \\
\bR^{21}_n &= -[\del_\theta, \bE_n] \bE^{-1}_n , & \bR^{22}_n &= [\del_t; \bE]_n \bE^{-1}_n ,
\end{aligned}\label{bRnidents}
\end{align}
\begin{align}\label{cFnidents}
\begin{aligned}
\cF_n &= -\bE_n P^-_{\grad,n}-(N_n\cdot P^+_{\grad,n}+\sN_n)e_2 , \\
\mathring{\cF}_n &= -\bR^{11}_n
\cF_n-\bE_n \dot P^-_{\grad,n}-[\del_t;\bE]_n P^-_{\grad,n} +\left(\frac{N_n\cdot \del_\theta U_n}{|X_\theta|}\bigtau_n\cdot P^+_{\grad,n}-N_n\cdot \dot P^+_{\grad,n}-\dot\sN_n\right)e_2 .
\end{aligned}
\end{align}
For the case $n=0$, we take $\bD_0$, $\bE_0$, and $\cF_0$ as in Definition~\ref{initialxiData}, and we correspondingly define $\bR^{12}_0$ and $\bR^{21}_0$ as above. On the other hand, we define each of $\bR^{11}_0$, $\bR^{22}_0$, and $\mathring \cF_0$ to be zero to ensure consistency with the fact that $\xi_0$ is constant in time, and thus various corresponding objects should be time-independent.

For $n\geq 1$, by following our derivation of the system in Section~\ref{formalderivationsubsection} and using
\begin{itemize}
\item the facts \(\del_t X_n = U_n\) and \(\xi_n(0)=\xi_0\), due to $\xi_n\in \sB$, and
\item the definitions of the various time derivative maps, such as $\dot H_n$, $\dot P^\pm_{\grad,n}$, $\dot \bD_n$, $[\del_t;\bE]_n$, etc.,
\end{itemize}
we find $\dot H_n=\del_t(H_n)$, $\dot P^\pm_{\grad,n}=\del_t(P^\pm_{\grad,n})$, $\dot \bD_n=\del_t(\bD_n)$, $[\del_t;\bE]_n=[\del_t,\bE_n]$, and so on.

Now, referring back to \eqref{definingSysIte}, using the identities $\dot H_n=\del_t(H_n)$, etc., in the case $n\geq 1$, and unpacking the various terms from \eqref{bRnidents}--\eqref{cFnidents}, one verifies\footnote{To do this, one works backwards through the derivation which took us from the basic system discussed in Section~\ref{basicsystemsection} to the wave system \eqref{SymizableSurfSys1} for $\dot U^*$ and $\dot B^*$.} for $n\geq 0$
\begin{align}
& &\mathring{\cF}_n &= \del_t \cF_n - [\del_t, \bD_n\bE_n] (\bD_n \bE_n)^{-1} \cF_n , \\
   & &  \del_t\dot{ U }^*_{n+1} &= \bD_n \del_\theta \dot{ B }^*_{n+1} + \bR^{11}_n \dot{ U }^*_{n+1}  + \bR^{12}_n \dot{ B }^*_{n+1}  + \mathring{\cF}_n ,\label{dotUstarnp1ident}\\
& &    \del_t\dot{ B }^*_{n+1} &= \del_\theta\dot{ U }^*_{n+1} + \bR^{21}_n \dot{ U }^*_{n+1}  + \bR^{22}_n \dot{ B }^*_{n+1} .\label{dotBstarnp1ident}\\
\intertext{Integrating \eqref{dotUstarnp1ident} and \eqref{dotBstarnp1ident} in time leads to}
& & \dot{ U }^*_{n+1}  &= \bD_n \bE_n \del_\theta B_{n+1} + \cF_n  ,\\
& & \del_t B_{n+1} &= \del_\theta U_{n+1},
\intertext{which we rearrange to get}
& & \del_\theta B_{n+1} &= (\bD_n \bE_n)^{-1}( \dot{ U }^*_{n+1}  - \cF_n) , \label{DthetaBformula}\\
& & \del_\theta U_{n+1} &= \bE^{-1}_n \dot{ B }^*_{n+1}  . \label{DthetaUformula}
\end{align}
Using these equations together with bounds of Lemmas~\ref{bEboundLemma} and~\ref{bDboundLemma}, we are able to deduce that $U_{n+1}$ and $B_{n+1}$ are in $H^{k+1}(S^1)$. Similarly, we are able to deduce that $\del_t X_{n+1} = U_{n+1}$, that $\del_\theta X_{n+1} = B_{n+1}$, and thus that $X_{n+1}$ is in $H^{k+2}(S^1)$. Thus, it follows that $\xi_{n+1}$ is in $\sH^k$.

Now we argue that in fact $\xi_{n+1}$ is in $\sB\subset \sH^k$. We already established above that $\xi_{\dagger,n+1}$ is in $\sB_\dagger$ and $U_{n+1} = \del_t X_{n+1}$, and it clearly follows from the definition of $\xi_{n+1}$ that $\xi_{n+1}(0)=\xi_0$. It only remains to verify that for our constant $M$ we have
\begin{align}\label{Mkbound}
\sup_{t}\lV \xi_{n+1} \rV_{\sH^k} \leq M\quad \mbox{and}\quad\lV U_{n+1} \rV_{C^1_{t,\theta}}\leq M .
\end{align}
Note by using the bounds on the objects $\bE(\xi)$, $\bD(\xi)$, and $\cF(\xi)$ given by Proposition~\ref{PenultBoundConstDepProp} together with the formulas for $\del^2_\theta X_{n+1}$ and $\del_\theta U_{n+1}$ given by \eqref{DthetaBformula} and \eqref{DthetaUformula}, followed by an application of the upper bound for $\xi_{\dagger,n+1}$ of \eqref{xiprimeTbound}, we find
\begin{align}
\lV \del^2_\theta X_{n+1} \rV_{H^k(S^1)}+ \lV \del_\theta U_{n+1}  \rV_{H^k(S^1)}
&\leq
C (1+ \sup_{t} \lV \xi_n \rV_{\sH^{k-1}} + \lV \xi_{\dagger,n+1} \rV_{\sH^k_\dagger} )^p \\
&\leq C' (1+  \sup_{t}\lV \xi_n \rV_{\sH^{k-1}} + \lV \xi_{\dagger,0} \rV _{\sH^k_\dagger})^p.\label{Xnplusonebound}
\end{align}
Meanwhile, using the relation $\xi_n = \xi_0 + \int^t_0 \del_\tau \xi_n(\tau) d\tau$ with the equation for $\del_t \xi_n$ and bounds on $\bJ(\xi_{n-1})$, $\cR(\xi_{n-1})$, and $\dcF(\xi_{n-1})$ following from Corollary~\ref{Jbounds}, Proposition~\ref{BasicRBound}, and Corollary~\ref{maindcFBound}, we find
\begin{align}
\lV \xi_n \rV_{\sH^{k-1}}
&\leq 
\lV \xi_0 \rV_{\sH^{k-1}} + C T(1+M^p)  (
\lV \del_\theta \xi_n \rV_{\sH^{k-1}} + 
\lV \xi_n \rV_{\sH^{k-1}} +1 ) \\
&\leq \lV \xi_0 \rV_{\sH^{k-1}} + C'T(1+M^{p'}) .
\end{align}
Using this in the upper bound \eqref{Xnplusonebound} yields
\begin{align}
\lV \del^2_\theta X_{n+1} \rV_{H^k(S^1)}+ \lV \del_\theta U_{n+1}  \rV_{H^k(S^1)}
&\leq
C (1+ \lV \xi_0 \rV^p_{\sH^k} + T(1+M^p) ).
\end{align}
By combining this with our upper bound of \eqref{xiprimeTbound} on $\xi_{\dagger,n+1}$ one then easily obtains
\begin{align}
\lV \xi_{n+1} \rV_{\sH^k} \leq C (1+ \lV \xi_0 \rV^p_{\sH^k} + T(1+M^p) ),
\end{align}
and so as long as $M$ is chosen large enough that $2C(1+ \lV \xi_0 \rV^p_{\sH^k}) \leq M$ and $T$ is small enough that $T(1+M^p)\leq \lV \xi_0 \rV^p_{\sH^k}$, we guarantee that for all $t$ in $[0,T]$
\begin{align}\label{latexibound}
\lV \xi_{n+1} \rV_{\sH^k} \leq M .
\end{align}
To verify that also
\begin{align}\label{accelBound}
\lV U_{n+1} \rV_{C^1_{t,\theta}} \leq M ,
\end{align}
one starts by using the equation below, which can be shown from \eqref{dotUstarnp1ident}:
\begin{align}
\del_t U_{n+1} = \bE^{-1}_n \left( \dot U^*_0 + \int^t_0 (\bD_n \del_\theta \dot{ B }^*_{n+1} + \bR^{11}_n \dot{ U }^*_{n+1}  + \bR^{12}_n \dot{ B }^*_{n+1}  + \mathring{\cF}_n) d\tau \right) .
\end{align}
Proceeding in the obvious way, we obtain
\begin{align}\label{finalpfbound}
\lV  U_{n+1} \rV_{C^1_{t,\theta}} \leq C( 1 + T M^p )( \lV \xi_0 \rV_{\sH^k} + T M^p) .
\end{align}
With the above, by choosing sufficiently large $M$ (larger than a constant depending on our initial data) and a corresponding choice of sufficiently small $T$, we conclude that \eqref{accelBound} holds in addition to \eqref{latexibound}. It is not hard to show from the above calculations together with the fact that $\xi_\dagger$ is in $C^0([0,T];\sH^k_\dagger)$ that $\xi$ continuously takes values in $\sH^k$, finishing the proof that $\xi_{n+1}$ is in $\sB$.
\end{proof}
\end{proposition}

Using Proposition~\ref{MainInductiveBoundProp}, we thus construct a sequence of iterates $\{\xi_n\}_{n\geq 1}$ with the property that
\begin{align}\label{sBreview}
& & \xi_n &\in \sB &  ( n \geq 1 ) .
\end{align}
The next task is to show that the sequence converges to a solution $\xi(t)$ to the problem \eqref{xiSystemRep} in $\sH^k$ for $t$ in the interval $[0,T]$. Convergence in the $\sH^k$ norm can be deduced by showing the sequence converges in a weaker norm, such as that of $\sH^2$, and combining this with the uniform bound in the stronger norm given by \eqref{sBreview}, that is,
\begin{align}
\label{MainUniformBound}
& & \sup_{t}\lV \xi_n \rV_{\sH^k} &\leq M &  ( n \geq 1 ) .
\end{align}
With the next proposition, we show that the sequence indeed converges in the $\sH^2$ norm.
\begin{proposition}\label{LimitExistence}
For $t$ in $[0,T]$, there exists $\xi(t)$ in $\sH^2$ such that
\begin{align}\label{convergence}
\lim_{n\to \infty} \sup_{t}\lV \xi_n - \xi \rV_{\sH^2} = 0 .
\end{align}
\end{proposition}
\begin{proof}
First, we bound the difference $\xi_{\dagger,n+1}-\xi_{\dagger,n}$. We have
\begin{align}
\frac{d}{dt}(\xi_{\dagger,n+1}-\xi_{\dagger,n}) &= A_n(t)(\xi_{\dagger,n+1}-\xi_{\dagger,n}) + (A_n-A_{n-1})(t) \xi_{\dagger,n} + F_n(t) - F_{n-1}(t) , \\
(\xi_{\dagger,n+1}-\xi_{\dagger,n})(0) &= 0 ,
\end{align}
and so
\begin{align}
\xi_{\dagger,n+1}-\xi_{\dagger,n} = \int^t_0 S_n(t,s)((A_n-A_{n-1})(s) \xi_{\dagger,n} + F_n(s) - F_{n-1}(s)) \, ds ,
\end{align}
giving the following bound, in which we have used that $\xi_n$ is uniformly bounded in $\sH^k$ and that $S_n$ is bounded from $\sH^2_\dagger$ to $\sH^2_\dagger$
\begin{align}\label{diffexpress1}
\lV \xi_{\dagger,n+1}-\xi_{\dagger,n} \rV_{\sH^2_\dagger}
& \leq
C T \sup_{s\in[0,T]}( \lV (A_n-A_{n-1} )\xi_{\dagger,n} \rV_{\sH^2_\dagger} + \lV F_n - F_{n-1} \rV_{\sH^2_\dagger} ) .
\end{align}
Using \(A_n-A_{n-1} = (J_n-J_{n-1}) \del_\theta +R_n-R_{n-1} \) we then find after using uniform bounds on $\xi_{\dagger,n}$
\begin{align}
\lV (A_n - A_{n-1})\xi_{\dagger,n} \rV_{\sH^2_\dagger}
&\leq C  ( \lV J_n-J_{n-1} \rV_{\sH^2_\dagger} + \lV R_n - R_{n-1} \rV_{B(\sH^2_\dagger,\sH^2_\dagger)} ).
\end{align}

Note Corollary~\ref{Jbounds} and Proposition~\ref{finalLipBdSummary} imply
\begin{align}
 \lV J_n-J_{n-1} \rV_{\sH^2_\dagger}+\lV R_n - R_{n-1} \rV_{B(\sH^2_\dagger,\sH^2_\dagger)}+\lV F_n - F_{n-1} \rV_{\sH^2_\dagger} \leq C \sup_{s\in[0,T]}\lV \xi_n - \xi_{n-1} \rV_{\sH^2}.
\end{align}
Using these in \eqref{diffexpress1} leads to
\begin{align}\label{xiprimeiteratediffbound}
\lV \xi_{\dagger,n+1}-\xi_{\dagger,n} \rV_{\sH^2_\dagger}
\leq C T \sup_{s\in[0,T]} \lV \xi_n - \xi_{n-1}\rV_{\sH^2} .
\end{align}
It remains to bound the remaining components of the difference $\xi_{n+1}-\xi_n$, which are just $X_{n+1}-X_n$ and $U_{n+1}-U_n$. Recall our expressions \eqref{DthetaBformula} and \eqref{DthetaUformula}, and the fact that $\del_\theta X_n$ = $B_n$. Using these we are able to derive
\begin{align}
\del^2_\theta X_{n+1} - \del^2_\theta X_n &= (\bD_n \bE_n)^{-1} ( \dot U^*_{n+1} - \cF_n ) - (\bD_{n-1} \bE_{n-1})^{-1} ( \dot U^*_n - \cF_{n-1} ), \\
\del_\theta U_{n+1} - \del_\theta U_n &= \bE_n^{-1} \dot B^*_{n+1} - \bE_{n-1}^{-1} \dot B^*_n .
\end{align}
Using in the above expressions the Lipschitz bounds of Propositions~\ref{hEstimateProp},~\ref{bEbounds}, and~\ref{finalLipBdSummary} in combination with \eqref{xiprimeiteratediffbound}, we then obtain
\begin{align*}
\lV \del^2_\theta X_{n+1} - \del^2_\theta X_n \rV_{H^2(S^1)} + \lV \del_\theta U_{n+1} - \del_\theta U_n \rV_{H^2(S^1)}
&\leq C \big( \lV \xi_{\dagger,n+1} - \xi_{\dagger,n}\rV_{\sH^2_\dagger} + \lV X_n - X_{n-1}\rV_{H^3(S^1)}\\
&\quad  +\lV \bE_n - \bE_{n-1} \rV_{B(H^2(S^1),H^2(S^1))}+\lV \cF_n - \cF_{n-1}\rV_{H^2(S^1)}\big)\\
&\leq C'  \left( T \sup_{s\in[0,T]}\lV \xi_n - \xi_{n-1} \rV_{\sH^2} + \sup_{s\in[0,T]}\lV \xi_n - \xi_{n-1} \rV_{\sH^1} \right)  .
\end{align*}
By writing $X_{n+1}-X_n$ and $U_{n+1} - U_n$ in terms of time integrals (as in \eqref{Undefn} and \eqref{Xndefn}), we find from the above bound
\begin{align}\label{XnUnbounds}
&\\
\lV X_{n+1} - X_n \rV_{H^4(S^1)} + \lV U_{n+1} - U_n \rV_{H^3(S^1)}
&\leq C' \left( T  \sup_{s\in[0,T]}\lV \xi_n - \xi_{n-1} \rV_{\sH^2}+ \sup_{s\in[0,T]}\lV \xi_n - \xi_{n-1} \rV_{\sH^1} \right) .
\end{align}
Meanwhile
\begin{align}
\frac{d}{dt}(\xi_n-\xi_{n-1})=&\big(\bJ(\xi_{n-1})\del_\theta + \cR(\xi_{n-1})\big)(\xi_n - \xi_{n-1})\\
&+\Big(\big(\bJ(\xi_{n-1})-\bJ(\xi_{n-2})\big)\del_\theta + \cR(\xi_{n-1})-\cR(\xi_{n-2})\Big)\xi_{n-1}+\dcF(\xi_{n-1})-\dcF(\xi_{n-2})
\end{align}
implies after integrating and using our Lipschitz bounds on $\bJ$ and $\cR$ that
\begin{align}
\lV \xi_n - \xi_{n-1} \rV_{\sH^1} \leq CT\left(\sup_{s\in[0,T]}\lV\xi_n - \xi_{n-1} \rV_{\sH^2} + \sup_{s\in[0,T]}\lV\xi_{n-1} - \xi_{n-2} \rV_{\sH^2}\right) .
\end{align}
Incorporating this bound in \eqref{XnUnbounds} above gives us
\begin{align*}
\lV X_{n+1} - X_n \rV_{H^4(S^1)} + \lV U_{n+1} - U_n \rV_{H^3(S^1)}
&\leq C T \left( \sup_{s\in[0,T]}\lV \xi_n - \xi_{n-1} \rV_{\sH^2} + \sup_{s\in[0,T]}\lV \xi_{n-1} - \xi_{n-2} \rV_{\sH^2} \right) ,
\end{align*}
which, when combined with \eqref{xiprimeiteratediffbound}, yields
\begin{align}\label{CauchySeqProved}
 \sup_{t}\lV \xi_{n+1}-\xi_n \rV_{\sH^2}
\leq C T\left( \sup_{t} \lV \xi_n - \xi_{n-1} \rV_{\sH^2} + \sup_{t} \lV \xi_{n-1} - \xi_{n-2} \rV_{\sH^2} \right).
\end{align}
By ensuring $T$ is sufficiently small, it is straightforward to verify from \eqref{CauchySeqProved} that \eqref{convergence} holds.
\end{proof}

Now that we have proved that the sequence $\{\xi_n\}_{n\geq 1}$ converges to a limit $\xi$ which is in $\sH^2$ at each time $t$, we use relatively standard arguments from functional analysis to upgrade this to convergence in $\sH^k$, using the uniform boundedness of the sequence in $\sH^k$.
\begin{proposition}\label{improvedxireg}
There exists a solution $\xi$ in $\sB$ to the system \eqref{xiSystemRep}, namely the limit $\xi$ given by Proposition~\ref{LimitExistence} of the sequence $\{\xi_n\}_{n\geq 1}$.
\end{proposition}
\begin{proof}
Note that the space $\sH^k$ is reflexive. For any $t$ in $[0,T]$, the sequence \( \{\xi_n(t)\}_{n\geq 1} \) is contained in a closed ball in $\sH^k$, and thus for each fixed $t$ there exists a weakly convergent subsequence, by Banach--Alaoglu. One then deduces that its limit must agree with the limit $\xi(t)$ in $\sH^2$ given by Proposition~\ref{LimitExistence}. From this, it follows that $\xi(t)$ is in $\sH^k$ for each $t$ in $[0,T]$. It is not difficult to upgrade this to deduce that $\xi$ is in $C^0([0,T];\sH^k)$. Moreover, one easily checks that $\xi$ solves \eqref{xiSystemRep}, satisfies $X_t=U$, obeys the bounds necessary to conclude $\xi$ is in $\sB$.
\end{proof}
\subsection{Returning to the original system}\label{backtoorigsystemsubsection}
In the previous subsection, we constructed a solution $\xi \in \sB$ on the time interval $[0,T]$ to
\begin{align}\label{xiSystemRep2}
& &    \xi_t &= \bJ(\xi) \xi_\theta + \cR(\xi) \xi + \dcF (\xi),  &    \\
& &     \xi(0) &= \xi_0  . & 
\end{align}
In this section, we will verify that our solution $\xi$ to the above system in fact results in a solution to the original ideal MHD equations. Recall again the definitions of $\bJ$ and $\cR$ from Definition~\ref{DEcommDefns} and that of $\dcF$ from Definition~\ref{finalDefdcF1}. In view of these, the system \eqref{xiSystemRep2} is equivalent to
\begin{align}
    \col{\dot{ U }^* }{\dot{ B }^* }_t &= \bJ_1 \col{\dot{ U }^* }{\dot{ B }^* }_\theta + \bR \col{\dot{ U }^* }{\dot{ B }^* } + \dcF_1 \,, \label{surfacewavesysRep}\\
    \col{\bom}{\bj}_t &= \bJ_2 \col{\bom}{\bj}_\theta + \dcF_2 \,,  \\
    \col{X}{U}_t &= \begin{pmatrix} \bI & 0 \\
                                       0 & \bE^{-1} \end{pmatrix} \col{U}{\dot{ U }^* } , 
                                       \label{auxiliarysys}\\
                                      (\dot U^*, \dot B^*, \bom, \bj, X,U)(0) &= (\dot U^*_0, \dot B^*_0, \bom_0, \bj_0, X_0,U_0),
\end{align}
To begin with our derivation of a solution to the original ideal MHD system, we essentially work backwards through the derivation of \eqref{symmetrizable2}--\eqref{symmetrizable1} in Section~\ref{formalderivationsubsection}. Note \eqref{auxiliarysys} implies $U$ is given by
    \begin{align}
        &&U(t,\theta) &= U_0(\theta) + \int_0^t \bE^{-1} \dot{ U }^* (\tau,\theta) d\tau &(\theta\in S^1, \ t\in[0,T]).\\ \intertext{
    We use the analogous relation to define $B$:}
       && B(t,\theta) &= B_0(\theta) + \int_0^t \bE^{-1} \dot{ B }^* (\tau,\theta) d\tau &(\theta\in S^1, \ t\in[0,T]).\label{Bdefn}
    \end{align}
    By carefully undoing the commutators in the terms appearing in our evolution equation above for $\dot{ B }^* $, as a result of the way the terms $\bD$ and $\bR^{ij}$ were constructed, one finds that the evolution equation for $\dot B^*$ is equivalent to $B_t = U_\theta$, and, finally, due to the fact that $\del_\theta X_0 = B_0$, one easily verifies that, in addition to $X_t = U$, we have the identity $X_\theta = B$.
    
Gradually, we will retrieve more identities which bring us closer to verifying we indeed have a solution to the original MHD system. Next we define the other corresponding quantities for our solution $\xi$, such as the associated $\bU$, $\bB$, and $\bX$ in particular. Momentarily, we will address related identities they must satisfy.

\begin{definition}\label{finaldefns}
For our exact solution $\xi$ in $\sB$ to the Lagrangian wave system, we define $B(t,\theta)$ by \eqref{Bdefn}, and we define
\begin{align}
\bU(t,a) =  \bU(\xi)(t,a),\qquad \bB(t,a) =  \bB(\xi)(t,a),\qquad \bX(t,a) =  \bX(\xi)(t,a),\qquad (a\in\Sigma, \ t\in [0,T]),
\end{align}
referring to the maps $\xi\mapsto\bU(\xi)$, $\xi\mapsto\bB(\xi)$, and $\xi\mapsto\bX(\xi)$ given by Definition~\ref{UBXdefn}.

We also define the quantities $H(t,\theta)= H(\xi)(t,\theta)$, and analogously define $\dot H(t,\theta)$, $h(t,x)$, $\dot h(t,x)$, $p_\pm(t,x)$, $P^\pm_\grad(t,\theta)$, $\dot P^\pm_\grad(t,\theta)$, $\bP^-_\grad(t,a)$, $\dot \bP^-_\grad(t,a)$, $\bD(t,\theta)$, $\bE$, $\Nres$, and $\bigh_\pm$, all in terms of the obvious corresponding maps evaluated at $\xi$ (see Sections~\ref{lagrangianwavesystemsection} and~\ref{auxiliaryestimates} for definitions).

Moreover, we denote the accompanying interface and plasma region by $\Gamma(t)$ and $\Omega(t)$, and define the quantities $|\Omega(t)|= \operatorname{area}(\Omega(t))$ as well as
    \begin{align}
        \Phi(t) &= \int^\pi_{-\pi} B^\perp(t,\theta) \cdot U(t,\theta) \, d\theta & & (t\in[0,T]), \label{fluxDefn} \\
        w_\Omega(t,x)&=\frac{\Phi(t)}{|\Omega(t)|}\int ^\pi_{-\pi}\left(G_R(x,\vartheta)e_2 -G_R(x,X(t,\vartheta))\del_\vartheta X^\perp (t,\vartheta) \right)\,d\vartheta   &&  (x \in \Omega(t), \ t\in[0,T] ) , \label{wDefn}
\end{align}
where $G_R(x,y)=G(x-y)+G(x-\bar y)$, for $\bar y = (y_1,-y_2)$ and
\begin{align}
G(x) &= \frac{1}{\pi} \log \left| \sin \left( \frac{x_1 + i x_2}{2} \right) \right| .
\end{align}
Additionally, we define
\begin{align}
&& W(t,\theta) &= w_\Omega(t,X(t,\theta)) &&(\theta\in S^1, \ t\in[0,T]), \\
&&\bW(t,a) &= w_\Omega(t,\bX(t,a))&&(a\in \Sigma, \ t\in[0,T]) .
\end{align}
\end{definition}

\begin{remark}
Note we cannot assume that $\bX$ agrees with $X$ at $\psi=0$. Rather than defining $X$ to be the trace of $\bX$, we produce $X$ by solving \eqref{xiSystemRep2}. Similarly, we cannot assume $\bU|_{\psi=0}=U$ or $\bB|_{\psi=0}=B$.
\end{remark}

Note that if the $\bU(t,a)$ defined above is the true Lagrangian velocity in the bulk, we \emph{should} have $\del_t \bX (t,a)=\bU(t,a)$. Meanwhile, with the lemma below, we only know $\del_t \bX (t,a)=\bU(t,a)+\bW(t,a)$, not yet having verified $\bW(t,a)$ is in fact zero. In this lemma, we record some basic identities which follow from the definitions above.
\begin{lemma}\label{variousIdents}
For the quantities below as given by Definition~\ref{finaldefns}, the following relations all hold for $t$ in $[0,T]$, $\theta$ in $S^1$, $a$ in $\Sigma$, and $x$ in $\cV(t)$:
\begin{align} 
P^-_\grad(t,\theta) &= \bP^-_\grad(t,\theta,0), &\del_t h(t,x) &= \dot h(t,x),\\
\del_t X(t,\theta) &= U(t,\theta),  &     \del_t H(t,\theta) &= \dot H(t,\theta)  ,\\
\del_t U(t,\theta) &= \bE^{-1}\dot U^*(t,\theta) ,&   \del_t\bU(t,a) &= \dot \bU(t,a) ,\\
\del_t B(t,\theta) &=\del_\theta U(t,\theta)= \bE^{-1}\dot B^*(t,\theta), \qquad &               \del_t\bB(t,a) &= \dot \bB(t,a) ,\\
\del_\theta X(t,\theta)&= B(t,\theta), &    \del_t \bP^-_\grad(t,a) &= \dot \bP^-_\grad(t,a) ,\\
&      &  \del_t P^+_\grad(t,\theta) &= \dot P^+_\grad(t,\theta).
\end{align}
Additional identities include
\begin{align}
\del_t \bX(t,a) &= \bU(t,a) + \bW(t,a) , \label{XUWident} \\
\del_t \bD(t,\theta) &= \dot \bD (\xi)(t,\theta) ,\label{operatortimederividents}\\
[\del_t, \bE] &= [\del_t;\bE](\xi) ,\label{operatortimederividents2}\\
[\del_t, \Nres] &= [\del_t;\Nres](\xi) ,\label{operatortimederividents3}
\end{align}
and for $\cF$ and $\mathring \cF$ as given by \eqref{cFdefn2} and \eqref{mathrcFdef}, and $\bR^{11}$ as in \eqref{bRderivation} corresponding to $\xi$,
\begin{align}\label{dtcFident}
\del_t \cF(t,\theta) - \bR^{11} \cF(t,\theta) = \mathring \cF(t,\theta) .
\end{align}
\end{lemma}
\begin{proof}
The fact that $\bP^-_\grad$ agrees with $P^-_\grad$ at $\psi=0$ follows from Definition~\ref{PminusDefns}.

The observations made just before Definition~\ref{finaldefns} lead to $\del_t X(t,\theta)=U(t,\theta)$, $\del_\theta X(t,\theta) = B(t,\theta)$, and the identities for $\del_t U(t,\theta)$ and $\del_t B(t,\theta)$ stated in the lemma.

One is able to verify from the definitions given in Section~\ref{VacuumQuantitiesTime} of $\dot h$ and $\dot H$ corresponding to a given $\xi$, using mainly that $\del_t X=U$, that we get the corresponding identities involving $\del_t h$ and $\del_t H)$. With these, it is not hard to establish the identity for $\dot P^+_\grad$.

The identity \eqref{XUWident} follows directly from the equation for $\bX$ in the system \eqref{AugmentedInteriorSystemL}. Using this while reexamining the systems given in Section~\ref{PlasmaQuantitiesTime} for $\dot \bU$ and $\dot \bB$ reveals that the unique solutions $\dot \bU$ and $\dot \bB$ can only be $\del_t \bU$ and $\del_t \bB$, respectively. Once we have this, a similar approach also leads to the identity for $\del_t \bP^-_\grad$.

Careful examination of our definitions of the maps in the right-hand sides of \eqref{operatortimederividents}--\eqref{operatortimederividents3} reveals that these identities hold as well. Once those are established, one may verify from the definitions of $\bR^{11}$, $\cF$, and $\mathring \cF$ that \eqref{dtcFident} holds.
\end{proof}

The following identity is slightly more obscure than those above, but is equally important, serving as a stepping stone to get from the surface Lagrangian wave system \eqref{surfacewavesysRep}--\eqref{auxiliarysys} to the original surface system \eqref{UnmodSurfaceSystem} without $(\dot U^*,\dot B^*)$.
\begin{lemma}\label{ProjectionIdentityLemma}
For quantities defined in Definition~\ref{finaldefns},
    \begin{align}\label{ProjIdent1}
        \bE(U_t - B_\theta + P^-_\grad) = \col{0}{-\frac{1}{2} \cN_- |H|^2} .
    \end{align}
\end{lemma}
\begin{proof}
To establish the above, first we verify that
\begin{align}\label{bEUtident}
    \bE U_t = \bD \bE B_\theta + \cF .
\end{align}
Due to $\xi$ satisfying \eqref{xiSystemRep2}, one can show by using Lemma~\ref{variousIdents} that
\begin{align*}
\dot U^*_t &= \bD \dot B^*_\theta + \bR^{11} \dot U^* + \bR^{12} \dot B^*-[\del_t,\bD] \bD^{-1}\cF - \bD[\del_t,\bE]\bE^{-1}\bD^{-1}\cF + \del_t \cF \\
&= \bD \dot B^*_\theta  + \bR^{11} \dot U^* + \bR^{12} \dot B^* - [\del_t,\bD\bE](\bD \bE)^{-1} \cF + \del_t \cF \\
&= \bD \dot B^*_\theta + [\del_t,\bD\bE](\bD \bE)^{-1} \dot U^* -\bD[\del_\theta,\bE] \bE^{-1} \dot B^*  - [\del_t,\bD\bE](\bD \bE)^{-1} \cF + \del_t \cF \\
&= \bD \bE B_{t\theta} +[\del_t,\bD\bE](\bD \bE)^{-1} \dot U^* + \bD \bE \del_t \left( (\bD \bE)^{-1} \cF \right) .
\end{align*}
This then implies
\begin{align}
\del_t \left( (\bD\bE)^{-1} \dot U^* \right) = B_{t\theta} + \del_t \left( (\bD\bE)^{-1} \cF \right).
\end{align}
Integrating and making basic observations about our initial data, we find
\begin{align}
 (\bD\bE)^{-1} \dot U^* = B_{\theta} + (\bD\bE)^{-1} \cF,
\end{align}
which leads to \eqref{bEUtident} after applying $\bD\bE$ and recalling $\bE U_t = \dot U^*$.

Regarding the terms appearing in \eqref{bEUtident}, let us note from the form of $\bD$ that
\begin{align}
\bD \bE B_\theta = \bE B_\theta +\col{0}{\frac{|H|^2}{|B|^2}N\cdot B_\theta},
\end{align}
and recall the expression from \eqref{cFdefn2}. Together, these give us that
\begin{align}\label{Utident2}
\bE U_t = \bE B_\theta - \bE P^-_\grad + \col{0}{\frac{|H|^2}{|B|^2}N\cdot B_\theta-N\cdot P^+_\grad - \half\Nres |H|^2} .
\end{align}
From Definition~\ref{ExtPressureDefnBd} it follows
\begin{align}\label{NPIdent}
    N \cdot P^+_\grad = \frac{1}{2} N \cdot (\grad |h|^2) \circ X - \half \cN_+ |H|^2 .
\end{align}
As in the proof of Proposition~\ref{heuristicPropPG}, we are able to verify that the identity \eqref{HBIdent} holds. Incorporating \eqref{NPIdent} and \eqref{HBIdent} into \eqref{Utident2} yields
\begin{align}
\bE U_t = \bE B_\theta - \bE P^-_\grad + \col{0}{\half \cN_+ |H|^2 - \half \Nres |H|^2} ,
\end{align}
so that \eqref{ProjIdent1} follows upon recalling that $\Nres = \cN_+ +\cN_-$.
\end{proof}

Now we use the lemma above to obtain the evolution of the Lagrangian surface velocity, $U$, almost bringing us back to the original Lagrangian surface system \eqref{UnmodSurfaceSystem}.
\begin{corollary}\label{SurfIdentLems}
On the time interval $[0,T]$ we have the identities
\begin{align}
\label{MainSurfIdent}
& & U_t - B_\theta + P^-_\grad &= - \half \left( \grad \bigh_- |h|^2 \right) (X) & (\theta \in S^1 ) ,
\end{align}
and
\begin{align}\label{averageZeroEvo}
\int^\pi_{-\pi}  B^\perp \cdot (U_t - B_\theta + P^-_\grad) \, d\theta = 0 .
\end{align}
\end{corollary}
\begin{proof}
To prove the first identity, one verifies from the definitions of $\bE$, $\bigh_-$, and $\cN_-$ that
\begin{align}
\bE \left[ -\half\left( \grad \bigh_- |h|^2 \right)(X)\right] = \col{0}{ -\half\cN_- |H|^2 } .
\end{align}
Equation \eqref{MainSurfIdent} then follows by using \eqref{ProjIdent1} and then applying $\bE^{-1}$.

To get \eqref{averageZeroEvo}, one uses \eqref{MainSurfIdent} together with the following, where we have used $B^\perp=X^\perp_\theta$ and the divergence theorem.
\begin{align}\label{nbighZero}
\int^\pi_{-\pi}  B^\perp \cdot \left( \grad \bigh_- |h|^2 \right) (X) \, d \theta = \int_\Gamma \del_n \bigh_- |h|^2\, d S = 0 .
\end{align}
\end{proof}

Now that we have obtained the surface evolution identity \eqref{MainSurfIdent}, we essentially need to make sure our ``surface quantities'' are compatible with our ``interior quantities'' for various pairings, e.g. $X=\bX|_{\psi=0}$, $U=\bU|_{\psi=0}$, and that the error term $\bW$ appearing in the identity $\bX_t = \bU+\bW$ is zero. Once we have verified these properties, we will be able to prove we have satisfied the original div-curl system without errors, that is, \eqref{OriginalLagrDivCurlSys}, and use the compatibility of the surface quantities with the interior quantities to basically translate 
\eqref{MainSurfIdent} into the statement that the original evolution equations \eqref{OriginalLagrangianSystem} hold at $\psi=0$ as well as everywhere else in $\Sigma$.


\subsubsection*{Evolution of the error}
Proving various quantities agree on the boundary and that our proposed solution fulfills other critical requirements comes down to showing that the quantity $\Xi$ we define below is zero for all time.
\begin{definition}\label{Xidefn}
For the quantities appearing in the right-hand side below all as given by Definition~\ref{finaldefns}, let us define
\begin{align}
&&\Xi (t,\theta,\psi) &=
\begin{pmatrix}
\bU (t,\theta,0) - U(t,\theta) \\
\bB (t,\theta,0) - B(t,\theta)  \\
\bX(t,\theta,0) - X(t,\theta)  \\
W(t,\theta)    \\
\bW(t,\theta,\psi)   \\
 \bB(t,\theta,\psi)-\del_\theta\bX(t,\theta,\psi) \\
\Phi(t) 
\end{pmatrix} 
& (\theta\in S^1,\ \psi\in[-1,0],\ t\in[0,T]).
\end{align}
\end{definition}
The task now at hand is to show that $\Xi$ satisfies a system of the following form, where $\bm{\mathrm{J}}_0$ is a simple symmetrizable matrix and $\cT_\xi$ is a regularity-preserving operator:
\begin{align}\label{ErrorEvoDemo}
\frac{d \Xi }{d t}
    &= \bm{\mathrm{J}}_0 \del_\theta \Xi + \cT_\xi \Xi  ,\\
    \Xi (0) &= 0 .
\end{align}
The point is that $\Xi(t)=0$ is the unique solution. Now we begin with the process of deriving such an evolution equation. For the space in which $\Xi(t)$ takes values, let us define
\begin{align}
 \sZ = (H^2(S^1))^4 \times (H^{5/2}(\Sigma))^2\times \R , \end{align}
 defining the corresponding norm in the corresponding way, so that, for example,
 \begin{align}
 \lV \Xi \rV_\sZ = & \lV \bU|_{\psi=0}-U\rV_{H^2(S^1)}+\lV \bB|_{\psi=0}-B\rV_{H^2(S^1)}+\lV \bX|_{\psi=0}-X\rV_{H^2(S^1)}\\
 & + \lV W \rV_{H^2(S^1)}+ \lV \bW \rV_{H^{5/2}(\Sigma)} + \lV \bB-\bX_\theta \rV_{H^{5/2}(\Sigma)} + |\Phi| .
 \end{align}

First, we make note of the equation driving the corresponding error for the Lagrangian trajectory maps at the surface, $\bX-X$ at $\psi=0$.
\begin{lemma}\label{cTDeltaxLem}
For each $t$ in $[0,T]$ we have a map $\cT_{\Delta x}(t)$ in $B(\sZ,H^2(S^1))$ with
\begin{align}
& & \del_t( \bX - X ) &= (\bU|_{\psi=0} - U ) + \bW|_{\psi=0} = \cT_{\Delta x} \Xi,&
\end{align}
and
\begin{align}
\sup_t\lV \cT_{\Delta x} \rV_{B(\sZ,H^2(S^1))} \leq C .
\end{align}
\end{lemma}
\begin{proof}
The identity follows immediately from the identities for $\bX_t$ and $X_t$ given by Lemma~\ref{variousIdents}. One simply defines $\cT_{\Delta x}\Xi$ in the obvious way in terms of the components of $\Xi$, trivially resulting in the claimed bound.
\end{proof}

Now we derive an equation for the evolution of $\Phi(t)$, the flux of the velocity across $\Gamma(t)$.
\begin{lemma}
\label{DerivDiffLem1}
We have a continuous vector-valued function $Y_\Phi(t,\theta)$ on \([0,T]\times S^1\), and a corresponding linear functional $\cT_\phi(t)$ in $\sZ'$ defined for each $t$ in $[0,T]$ with
\begin{align}
\label{DPhiDt}
    \frac{d \Phi}{dt}(t) = \int^\pi_{-\pi} Y_\phi(t,\theta) \cdot \Xi (t,\theta,0) \, d\theta =\cT_\phi(t) \Xi(t),
\end{align}
and
\begin{align}
\sup_t\lV \cT_\phi \rV_{\sZ'} \leq C .
\end{align}
\end{lemma}
\begin{proof}
Recall the definition of $\Phi(t)$ from \eqref{fluxDefn}. Using \eqref{averageZeroEvo} and $B_t=U_\theta$, we find
\begin{align}
    \frac{d \Phi}{dt} &= \int^\pi_{-\pi} U_t \cdot B^\perp - U^\perp \cdot B_t \, d\theta \\
    &= \int^\pi_{-\pi} B_\theta \cdot B^\perp- U^\perp \cdot U_\theta \, d\theta -\int B^\perp \cdot P^-_\grad \, d\theta\\
    &= \int^\pi_{-\pi} B_\theta \cdot B^\perp -U^\perp \cdot U_\theta \, d\theta
    - \int^\pi_{-\pi} (B - \bX_\theta )^\perp \cdot P^-_\grad \, d\theta - \int^\pi_{-\pi} \bX_\theta^\perp \cdot P^-_\grad \, d\theta \qquad (\psi = 0 ). \label{PhiEvoComp}
\end{align}
Recall that \( P^-_\grad(t,\theta) = \bP^-_\grad(t,\theta,0) \) and that from \eqref{InteriorPressureGradSystem} we have
\begin{align}
& &     \grad \cdot (\, \cof (\bM) \bP^-_\grad ) &= \grad \cdot (\, \cof ( \grad \bU^t) \, \bU - \cof( \grad \bB^t) \bB ) &   (a & \in \Sigma), \label{InteriorPressureDiv}
\intertext{where}
& &     \bM_{ij} &= \del_{a_i} \bX_j .&   & \label{MdefnReminder}
\end{align}
First, note the two integrals below are equal due to the boundary conditions for the system \eqref{InteriorPressureGradSystem}, and moreover the right-hand side is zero, shown similarly as \eqref{nbighZero} in the proof of Corollary~\ref{SurfIdentLems}.
\begin{align}
\label{BottomAvgZero1}
\begin{aligned}
\int^\pi_{-\pi} \bX_\theta^\perp \cdot \bP^-_\grad  \, d\theta \, \Big|_{\psi=-1}  &= 
\frac{1}{2} \int^\pi_{-\pi} \del_n \left( \bigh_- |h|^2 \right) (x_1, -1) \, dx_1  \\
&= 0 .
\end{aligned}
\end{align}
Now we observe that $\bX_\theta^\perp$ appears as the second row of $\cof(\bM)$. This means that with the use of \eqref{BottomAvgZero1}, we get the first equality below simply by using the fundamental theorem of calculus. Following this, we apply \eqref{InteriorPressureDiv} and perform a similar computation.
\begin{align*}
\int^\pi_{-\pi} \bX_\theta^\perp \cdot \bP^-_\grad \, d\theta \, \Big|_{ \psi = 0 }
    &= \int_\Sigma \grad \cdot (\cof(\bM) \bP^-_\grad) \, da \\
    &= \int_\Sigma \grad \cdot (\, \cof(\grad \bU^t) \bU - \cof( \grad \bB^t) \bB) \, da \\
    &= \int^\pi_{-\pi}  \bU_\theta^\perp \cdot \bU - \bB_\theta^\perp \cdot \bB  \, d\theta \, \bigg|^{\psi=0}_{\psi=-1} \\
    &= - \int^\pi_{-\pi} \bU^\perp \cdot \bU_\theta - \bB^\perp \cdot \bB_\theta  \, d\theta \, \bigg|^{\psi=0}_{\psi=-1} \, .
\end{align*}
Now we note from the boundary conditions for $\bU$ and $\bB$ that $\bU_2 = \bB_2 = 0$ along $\psi = -1$. Thus also $\del_\theta \bU_2 = \del_\theta \bB_2 = 0$ along $\psi = -1$. It follows
\begin{align*}
& &   \bU^\perp \cdot \bU_\theta - \bB^\perp \cdot \bB_\theta &= 0 & (\psi = -1).
\end{align*}
So we conclude
\begin{align*}
    -\int^\pi_{-\pi} \bX_\theta^\perp  \cdot P^-_\grad \, d\theta
    = \int^\pi_{-\pi} \bU^\perp \cdot \bU_\theta - \bB^\perp \cdot \bB_\theta \, d\theta \qquad \qquad ( \psi = 0 ) .
\end{align*}
Substituting this back into \eqref{PhiEvoComp}, we find the following holds (in which we evaluate all terms at $\psi=0$).
\begin{align*}
    \frac{d \Phi}{dt}
    =& \int^\pi_{-\pi} B^\perp \cdot B_\theta- U^\perp \cdot U_\theta \, d\theta
    +\int^\pi_{-\pi} (\bX_\theta - B)^\perp \cdot P^-_\grad \, d\theta
    + \int^\pi_{-\pi} \bU^\perp \cdot \bU_\theta - \bB^\perp \cdot \bB_\theta  \, d\theta
    \\
    =& \int^\pi_{-\pi} \del_\theta(\bX-X)^\perp \cdot P^-_\grad \, d\theta + \int^\pi_{-\pi} (\bU - U)^\perp \cdot U_\theta + \bU^\perp \cdot \del_\theta(\bU-U) \, d\theta  \\
    &\qquad+ \int^\pi_{-\pi} (B-\bB )^\perp \cdot B_\theta + \bB^\perp \cdot \del_\theta(B - \bB) \, d\theta \\
    =& \int^\pi_{-\pi} (\del_\theta P^-_\grad)^\perp \cdot ( \bX - X) -(U_\theta+\bU_\theta)^\perp \cdot (\bU - U) +(B_\theta +\bB_\theta)^\perp \cdot (\bB - B) \, d\theta  \qquad(\psi=0).
\end{align*}
Thus we obtain \eqref{DPhiDt}. The bound on the right-hand side easily follows.
\end{proof}

Now that we have derived the evolution of the flux term $\Phi(t)$, we can use this to calculate some expressions for certain partial derivatives of the error terms $W(t,\theta)$ and $\bW(t,a)$.
\begin{lemma}
\label{WPartialsLem}
We have vector fields $Y^1_w(t,\theta)$, $Y^2_w(t,\theta)$, $Y^3_w(t,\theta)$, $Y^4_w(t,a)$, $Y^5_w(t,\theta)$, $Y^6_w(t,a)$, and $Y^7_w(t,a)$, continuous in $t$ and $C^3$ in $\theta$ and $a$, where, for $t$ in $[0,T]$ and $a=(\theta,\psi)$ in $\Sigma$,
\begin{align}
\del_t W(t,\theta) &= \Phi(t) Y^1_w(t,\theta) + \left( \int^\pi_{-\pi} Y^2_w(t,\vartheta) \cdot \Xi(t,\vartheta,0)\, d\vartheta \right) Y^3_w(t,\theta) , \label{WtimeDeriv}\\
\del_t\bW (t,a) &= \Phi(t) Y^4_w(t,a) + \left( \int^\pi_{-\pi} Y^5_w(t,\vartheta) \cdot \Xi(t,\vartheta,\psi)\, d\vartheta \right) Y^6_w(t,a), \label{bWtimeDeriv} \\
\del_\theta\bW(t,a) &= \Phi(t) Y^7_w(t,a) .
\end{align}
Moreover, defining for $t$ in $[0,T]$ the linear map $\cT^1_w(t)$ in the obvious way so that the right-hand side of \eqref{WtimeDeriv} is $\cT^1_w(t)\Xi(t)$, and defining the analogous linear map $\cT^2_w(t)$ corresponding to the right-hand side of \eqref{bWtimeDeriv}, we have
\begin{align}
\sup_t\lV \cT^1_w \rV_{B(\sZ,H^2(S^1))} +\sup_t \lV \cT^2_w \rV_{B(\sZ,H^{5/2}(\Sigma))}\leq C.
\end{align}
\end{lemma}
\begin{proof}
Recall the definition of $w_\Omega(t,x)$ from \eqref{wDefn}. For $W(t,\theta)=w(t,X(t,\theta))$, we calculate that
\begin{align*}
    \del_t W(t,\theta) =
    &\frac{\Phi}{|\Omega|} \left( \int_{-\pi}^{\pi} \del_t X(t,\theta) \cdot \grad G_R(X(t,\theta), \vartheta) \, d\vartheta \right) e_2 + \frac{d}{dt}\left( \frac{\Phi}{|\Omega|} \right) \left( \int_{-\pi}^{\pi} G_R(X(t,\theta), \vartheta) \, d\vartheta \right) e_2 \\
    & -\int_{-\pi}^{\pi} G_R(X(t,\theta) - X(t,\vartheta)) \left( \frac{d}{dt}\left( \frac{\Phi}{|\Omega|} \right) \del_\vartheta X^\perp (t,\vartheta) + \frac{\Phi}{|\Omega|} \del^2_{t \vartheta} X^\perp(t,\vartheta) \right) d\vartheta \\
    &- \frac{\Phi}{|\Omega|} \int_{-\pi}^{\pi} (\del_t X(t,\theta) - \del_t X(t,\vartheta)) \cdot K(X(t,\theta) - X(t,\vartheta)) \del_\vartheta X^\perp(t,\vartheta) \, d\vartheta \\
    &- \frac{\Phi}{|\Omega|} \int_{-\pi}^{\pi} (\del_t X(t,\theta) - \del_t \bar{X}(t,\vartheta)) \cdot K(X(t,\theta) - \bar{X}(t,\vartheta)) \del_\vartheta X^\perp(t,\vartheta) \, d\vartheta .
\end{align*}
It is not hard to verify that $\frac{d}{dt} |\Omega(t)| = \mbox{harmless terms}$. Whenever the time derivative hits $\Phi(t)$, Lemma~\ref{DerivDiffLem1} implies a term as in the right-hand side of \eqref{DPhiDt} is produced. As a result, one arrives at an expression for $W_t$ as in the statement of the lemma. The calculations for $\bW_t$ and $\bW_\theta$ are similarly straightforward, and, noting $|\Phi|\leq \lV \Xi \rV_{\sZ}$ for example, the bounds claimed are easily verified.
\end{proof}

Now that we have evolution equations for $\bX-X$ at $\psi=0$, $W$, $\bW$, and $\Phi$, it remains to do the same for $\bU-U$ and $\bB-B$ at $\psi=0$ and for $\bB-\bX_\theta$. Recall our comment that in \eqref{ErrorEvoDemo}, the evolution equation we are in the process of deriving, we find that $\bm{\mathrm{J}}_0$ is symmetrizable. In fact, it takes the form
\begin{align}\label{bJ0def}
\bm{\mathrm{J}}_0 =
\begin{pmatrix}
0 & \bI_{(2\times 2)} & 0 \\
\bI_{(2\times 2)} & 0 & 0 \\
0 & 0 & 0_{(9\times 9)}
\end{pmatrix}.
\end{align}
This means that we will ultimately show
\begin{align}
&&\del_t(\bU-U)-\del_\theta(\bB-B) &= \cT_u \Xi  &(\psi=0), \label{UdiffEvoEx}\\
&&\del_t(\bB-B)-\del_\theta(\bU-U) &=\cT_b \Xi  &(\psi=0),\label{BdiffEvoEx}
\end{align}
for some uniformly bounded maps $\cT_u(t)$ and $\cT_b(t)$ from $\sZ$ into $H^2(S^1)$.

First, we verify the normal component of the left-hand side of \eqref{UdiffEvoEx} can be expressed as follows.
\begin{lemma}
For some vector-valued functions $Y^1_u(t,\theta)$ and $Y^2_u(t,\theta,\vartheta)$ continuous in $t$, $C^2$ in $\theta$ and $\vartheta$, we have on the time interval $[0,T]$ that
\begin{equation}
\label{ProjIdent2}
    N \cdot ( \del_t(\bU-U) - \del_\theta(\bB-B) )|_{\psi=0} = Y^1_u(t,\theta) \cdot \Xi (t,\theta,0) + \int^\pi_{-\pi} Y^2_u (t,\theta,\vartheta) \cdot \Xi (t,\vartheta,0)\, d\vartheta .
\end{equation}
\end{lemma}
\begin{proof}
Observe that in solving the system \eqref{AugmentedInteriorSystemL} in the proof of Proposition~\ref{SolnMainDivCurlProp} to construct $\bU$, $\bB$, and $\bX$, we get nice vector-valued functions \(u(t), b(t):\Omega(t) \to \R^2 \), for which $\bU = u \circ \bX$ and $\bB = b \circ \bX$. By examining the boundary conditions in the system \eqref{AugmentedInteriorSystemL} we find $N\cdot (u \circ X) = N\cdot U - N\cdot W$, and correspondingly
\begin{align}
 & &   N \cdot (\bU - U) &= N \cdot (u \circ \bX - u \circ X) - N \cdot W  & & (\psi = 0) , \\
 & &   N \cdot (\bB - B) &= N \cdot (b \circ \bX - b \circ X)  & & (\psi = 0) .
\end{align}
By applying $\del_t$ to the first identity, $\del_\theta$ to the second, and rearranging terms, one can show for some vector-valued $Y^1(t,\theta)$, $Y^2(t,\theta)$ that
\begin{align}
\label{NdotIdent1}
& & N \cdot ( \del_t( \bU - U) - \del_\theta(\bB - B) ) &= Y^1 \cdot \Xi + Y^2 \cdot W_t & ( \psi = 0 ) .
\end{align}
Now we use \eqref{WtimeDeriv} to express $W_t$ in particular in terms of vector fields dotted with $\Xi $ to get a right-hand side as claimed in the statement of the lemma.
\end{proof}

Let us use the fundamental theorem of calculus to define the \( 2 \times 2 \) matrix-valued function $\mathrm{Q} (t,\theta)$ in the obvious way which leads to the identity below, valid on the time interval $[0,T]$:
\begin{align}
\label{QDefn}
 \left(\grad \bigh_- |h|^2 \right) (\bX(t,\theta,0)) - \left(\grad \bigh_- |h|^2 \right) (X(t,\theta)) &= \mathrm {Q}(t,\theta) (\bX(t,\theta,0) - X(t,\theta))  & ( \theta\in S^1 ) .
\end{align}
Then we find
\begin{align}
& & \del_t(\bU - U) - \del_\theta(\bB - B) &= \bV_* - \mathrm{Q} ( \bX - X ) & (\psi = 0 ) ,
\end{align}
where we define for $t$ in $[0,T]$ and $a$ in $\Sigma$
\begin{align}\label{Vstardef}
\bV_*(t,a) = \bU_t(t,a)- \bB_\theta(t,a) + \bP^-_\grad(t,a) + \half \left( \grad \bigh_- |h|^2 \right)(\bX(t,a)).
\end{align}
Thus we should be able to prove a relation of the form \eqref{UdiffEvoEx} holds if we can show that $\bV_*$ solves an appropriate div-curl system. This is the idea behind the next proposition.
\begin{proposition}
\label{bVstarProp}
For $t$ in $[0,T]$ and $a$ in $\Sigma$, let us define $\bV_*(t,a)$ by \eqref{Vstardef}. Then for each such $t$ we have some $\cT_*(t)$ in $B(\sZ,H^{5/2}(\Sigma))$ for which
\begin{align}
\bV_*(t) = \cT_*(t) \Xi(t) ,
\end{align}
and
\begin{align}
&& \sup_t\lV \cT_* \rV_{B(\sZ,H^{5/2}(\Sigma))} &\leq C.&
\end{align}
\end{proposition}

\begin{proof}
To prove this, we will apply Lemma~\ref{GeneralDivCurlSysLem} after verifying that $\bV_*$ solves a nice system of the following form:
\begin{align}
(a \in \Sigma) \qquad &
\left\{
    \begin{aligned}
        \grad \cdot (\, \cof (\bM) \bV_* ) &= Y^1_* \cdot \Xi_\theta + Y^2_* \cdot \Xi_\psi \\
        \grad^\perp\cdot (\, \bM \bV_*) &= Y^3_* \cdot \Xi_\theta + Y^4_* \cdot \Xi_\psi \, \Lcm
    \end{aligned}
\right. \label{divcurlsysVstar} \\
(\psi = 0)  \qquad & \ \Big\{ \, \bm{N} \cdot \bV_* = Y^5_* \cdot \Xi + \int^\pi_{-\pi} Y^6_* \cdot \Xi \, d\vartheta , \label{NormalIdentity1}\\
(\psi = -1) \qquad & \ \Big\{ \, \bm{N} \cdot \bV_* = 0 , \qquad\quad \int^\pi_{-\pi} \bV_* \cdot \bX_\theta \, d\theta =  \int^\pi_{-\pi} Y^7_* \cdot \Xi \, d\theta,
\end{align}
where \( \bm{N} = n(\bX) \) for $\psi = 0$ and \( \bm{N} = (0, -1) \) for $\psi = -1$.
Moreover,
First we show \eqref{NormalIdentity1}. 
Using the identity \eqref{MainSurfIdent} of Corollary~\ref{SurfIdentLems} and taking $\mathrm{Q}(t,\theta)$ defined by \eqref{QDefn}, we find for some vector field $Y(t,\theta)$
\begin{align}\label{NormalIdentity2}
N\cdot(U_t - B_\theta + P^-_\grad +\half (\grad \bigh_- |h|^2)(\bX) ) = Y \cdot (\bX - X ) \qquad\qquad (\psi = 0 ) .
\end{align}
Note that the left-hand side of \eqref{NormalIdentity2} is almost the same as $\bm{N} \cdot \bV_*$ at $\psi=0$, except with $N$, $U$, and $B$ in place of $\bm{N}$, $\bU$, and $\bB$.
Using \eqref{ProjIdent2} we can essentially replace $N\cdot(U_t-B_\theta)$ with $\bm{N} \cdot (\bU_t-\bB_\theta)$ in \eqref{NormalIdentity2} at the price of harmless error terms such as $n\circ \bX - n \circ X$, etc., leading to \eqref{NormalIdentity1}.

Now we verify that
\begin{align}
\label{NormalIdentity3}
 \bm{N}\cdot \bV_* =  \bm{N} \cdot \left(\bU_t - \bB_\theta + \bP^-_\grad +\half ( \grad \bigh_- |h|^2 )( \bX )\right) &= 0  & (\psi = -1 ).
\end{align}
We remind the reader that $\bm{N}$ is simply $(0,-1)$ at \( \psi = -1 \). Regarding the terms $\bm{N} \cdot \bU_t$ and $\bm{N} \cdot \bB_\theta$ appearing above, it is not hard to show from the div-curl systems satisfied by $\bU$ and $\bB$ that
\begin{align}
& &\bm{N} \cdot \bU_t &= \bm{N} \cdot \bB_t = 0 & (\psi = -1) .
\end{align}
Regarding the two terms left in \eqref{NormalIdentity3}, one observes from the boundary condition at $\psi = -1$ in the system \eqref{InteriorPressureGradSystem} for $\bP^-_\grad$ that these remaining terms indeed cancel.

Now we claim
\begin{align}
\label{CirculationIdent1}
\int^\pi_{-\pi} \bV_* \cdot \bX_\theta \, d\theta &= \int^\pi_{-\pi} \bX_{\theta t} \cdot \bW + \bB_\theta \cdot (\bB- \bX_\theta ) \, d\theta & (\psi = -1 ) .
\end{align}
Since we have that $\bP^-_\grad = (\grad \phi)(\bX)$ for a nice function $\phi$ defined on $\Omega(t)$,
\begin{align}
\label{CirculationIdent2}
\int^\pi_{-\pi}  (\bP^-_\grad +\half ( \grad \bigh_- |h|^2 )( \bX ) ) \cdot \bX_\theta \, d\theta = 0 \qquad \qquad (\psi = -1 ).
\end{align}
Now, we use the condition $\int^\pi_{-\pi} \bU \cdot \bX_\theta \, d\theta = \upalpha$ at $\psi = -1$ imposed on $\bU$ in the system \eqref{AugmentedInteriorSystemL} to find that at $\psi = -1$ we have
\begin{align}
&&\int^\pi_{-\pi} \bU_t \cdot \bX_\theta \, d\theta &= - \int^\pi_{-\pi} \bU \cdot \bX_{\theta t} \, d\theta &\notag \\
&&&= -\int^\pi_{-\pi} (\bX_t - \bW) \cdot \bX_{\theta t} \, d\theta &\notag \\
&&&= \int \bW \cdot \bX_{\theta t} \, d\theta  & (\psi = -1 ). \label{CirculationIdent3}
\end{align}
Similarly, the analogous condition imposed on $\bB$ implies
\begin{align}
\label{CirculationIdent4}
& & -\int \bB_\theta \cdot \bX_\theta \, d\theta &= \int \bB_\theta \cdot (\bB - \bX_\theta ) \, d\theta & (\psi = -1 ) .
\end{align}
Adding identities \eqref{CirculationIdent2}, \eqref{CirculationIdent3}, and \eqref{CirculationIdent4}, one obtains \eqref{CirculationIdent1}.

Now we derive expressions for the right-hand sides in \eqref{divcurlsysVstar}. First let us recall the $\bom$ and $\bj$ evolution equations, which can be written in the form
\begin{align}
\begin{aligned}
    \bom_t &= \bj_\theta , \\
    \bj_t &= \bom_\theta + \bm{\sigma}^{-1}\grad^\perp\cdot( \grad \bU^t \bB - \grad \bB^t \bU ) .
    \end{aligned}\label{divcurlEvoNice}
\end{align}
Using the above together with the systems \eqref{AugmentedInteriorSystemL} for $\bU$ and $\bB$ and the system \eqref{InteriorPressureGradSystem} for $\bP^-_\grad$, we are able to find that
\begin{align}
\label{preDivCurlSystembV}
\begin{aligned}
    \grad\cdot(\, \cof(\bM) (\bU_t-\bB_\theta + \bP^-_\grad)) &= \grad\cdot(\, \cof(\grad \bU^t - \bM_t) \bU - \cof(\grad \bB^t - \bM_\theta) \bB ) , \\
    \grad^\perp \cdot(\, \bM (\bU_t-\bB_\theta + \bP^-_\grad)) &= \grad^\perp \cdot(\, \bM_\theta \bB - \bM_t \bU ) .
\end{aligned}
\end{align}
Now we record a few identities which hold for any vector fields $\bm{Y}$ and $\bm{Z}$ on $\Sigma$.
\begin{align}
    \grad\cdot( \, \cof(\grad \bm{Y}^t) \bm{Z} )&= \grad\cdot( \, \cof(\grad \bm{Z}^t) \bm{Y} ) ,\label{cofId}\\
    \grad^\perp \cdot ( \grad \bm{Y}^t \bm{Z} ) &= -\grad^\perp \cdot ( \grad \bm{Z}^t \bm{Y} ) ,\\
    \grad^\perp \cdot ( \grad \bm{Y}^t \bm{Y} ) &= 0 .
\end{align}
Using these identities as well as the fact the Lagrangian divergence and curl of the term \( \left( \grad \bigh_- |h|^2 \right)(\bX) \) is zero we then find
\begin{align*}
 \grad\cdot(\, \cof(\bM) \bV_*) &= \grad\cdot(\, \cof(\grad \bU^t) (\bU-\bX_t) - \cof(\grad \bB^t) (\bB-\bX_\theta) ) \\
  &= -\grad\cdot(\, \cof(\grad \bU^t)\bW + \cof(\grad \bB^t) (\bB-\bX_\theta) ) ,
\end{align*}
and
\begin{align*}
   \grad^\perp \cdot(\, \bM \bV_*) 
    &= \grad^\perp \cdot( \grad \bU^t \bX_t- \grad \bB^t \bX_\theta ) \\
    &= \grad^\perp \cdot( \grad \bU^t (\bX_t-\bU)+ \grad \bB^t (\bB-\bX_\theta) ) \\
    &= \grad^\perp \cdot( \grad \bU^t \bW + \grad \bB^t (\bB-\bX_\theta) ) .
\end{align*}
These result in expressions of the form
\begin{align*}
    \grad\cdot(\, \cof(\bM) 
    \bV_*)
    &= \bU^\perp_\psi \cdot \bW_\theta - \bU^\perp_\theta\cdot \bW_\psi + \bB^\perp_\psi \cdot (\bB_\theta-\bX_{\theta\theta}) - \bB^\perp_\theta \cdot (\bB_\psi-\bX_{\theta\psi}) , \\
   \grad^\perp \cdot(\, \bM \bV_*)
   &= \bU_\psi \cdot \bW_\theta - \bU_\theta \cdot \bW_\psi + \bB_\psi \cdot (\bB_\theta-\bX_{\theta\theta}) - \bB_\theta \cdot (\bB_\psi-\bX_{\theta\psi}) .
\end{align*}
Now we take the resulting div-curl system satisfied by $\bV_*$, and apply Lemma~\ref{GeneralDivCurlSysLem}, defining $\cT_*$ in the obvious way. Carefully keeping track of the various regularities, one obtains the claimed bound.
\end{proof}

Now we use the above result to get our evolution equation of the form \eqref{UdiffEvoEx} for $\bU-U$ at $\psi=0$.
\begin{corollary}\label{cT1Cor}
For each $t$ in $[0,T]$ we have a map $\cT_u(t)$ in $B(\sZ, H^2(S^1))$ such that
\begin{align}
& & \del_t(\bU - U) - \del_\theta(\bB - B) &= \cT_u \Xi    &   (\psi = 0) ,
\end{align}
and
\begin{align}
&&\lV \cT_u (t) \rV_{B(\sZ,H^2(S^1))} &\leq C .&
\end{align}
\end{corollary}
\begin{proof}
Recall that
\begin{align}
& & \bV_* &= \bU_t - \bB_\theta + P^-_\grad + \half (\grad \bigh_- |h|^2 )(\bX)  & (\psi = 0 ) ,
\end{align}
while, meanwhile, \eqref{MainSurfIdent} holds. Using these with \eqref{QDefn}, we find
\begin{align}
& & \del_t(\bU - U) - \del_\theta(\bB - B) &= \bV_* - \mathrm{Q} ( \bX - X ) & (\psi = 0 ) .
\end{align}
Now we obtain the statement of the proposition by applying Proposition~\ref{bVstarProp}.
\end{proof}

Using similar techniques, we get an evolution equation for $\bB-B$ at $\psi=0$ of the form \eqref{BdiffEvoEx}.
\begin{proposition}
\label{BtMinusUthetaLinearProp}
For each $t$ in $[0,T]$ we have a map $\cT^1_b(t)$ in $B(\sZ,H^{5/2}(\Sigma))$ and a map $\cT^2_b(t)$ in $B(\sZ,H^2(S^1))$ such that
\begin{align}
& & \bB_t - \bU_\theta &= \cT^1_b \Xi &   ( a \in \Sigma) , \label{cT2ident}\\
& & \del_t(\bB - B) - \del_\theta(\bU - U) &= \cT^2_b \Xi    &   (\psi = 0 ),\label{cT3ident}
\end{align}
and
\begin{align}
\sup_t\lV\cT^1_b\rV_{B(\sZ,H^{5/2}(\Sigma))}+\sup_t\lV \cT^2_b\rV_{B(\sZ,H^2(S^1))} &\leq C.
\end{align}
\end{proposition}
\begin{proof}
The proof is similar to the argument used to prove Proposition~\ref{bVstarProp}. For now, we just include some work calculating the interior div-curl equations for \(\bB_t - \bU_\theta\).

Using the identity \eqref{cofId} together with the divergence free conditions for $\bU$ and $\bB$ in Lagrangian coordinates, one can show
\begin{align}
\label{ProjIdent4}
\begin{aligned}
    \grad\cdot( \, \cof(\bM) (\bB_t-\bU_\theta) ) 
    &= \grad \cdot ( \, \cof (\bM_\theta)  \bU - \, \cof(\bM_t) \bB \, ) \\
    &= -\grad \cdot ( \, \cof (\bM_\theta) \bW+\cof(\bM_t) (\bB-\bX_\theta) ) .
\end{aligned}
\end{align}
Using the $\bom$ and $\bj$ evolution equations \eqref{divcurlEvoNice} together with the systems \eqref{AugmentedInteriorSystemL} for $\bU$ and $\bB$ one arrives at the pair of interior div-curl equations
\begin{align}
    \grad \cdot ( \, \cof(\bM) (\bB_t - \bU_\theta) ) &= -\grad \cdot ( \, \cof(\bM_t) (\bB-\bX_\theta) + \cof (\bM_\theta) \bW) \\
    &=\bX_{t\psi}^\perp \cdot (\bB_\theta-\bX_{\theta\theta})-\bX^\perp_{t\theta} \cdot(\bB_\psi-\bX_{\theta\psi}) +\bX^\perp_{\theta \psi} \cdot \bW_\theta  - \bX^\perp_{\theta\theta} \cdot \bW_\psi , \\
    \grad^\perp \cdot ( \, \bM (\bB_t - \bU_\theta ) ) &= \grad^\perp \cdot ( \grad\bU^t (\bB-\bX_\theta) + \grad \bB^t \bW) \\
    &= \bU_\psi \cdot (\bB_\theta-\bX_{\theta\theta}) - \bU_\theta \cdot (\bB_\psi-\bX_{\theta\psi})+ \bB_\psi \cdot \bW_\theta - \bB_\theta \cdot \bW_\psi.
\end{align}
Very similar to the proof of Proposition~\ref{bVstarProp}, one is able to verify certain boundary conditions hold for the vector field $\bB_t-\bU_\theta$ at $\psi=0$ and $\psi=-1$ which allow one to apply Lemma~\ref{GeneralDivCurlSysLem} when combined with the above div-curl equations. This results in a map $\cT^1_b(t)$ which satisfies \eqref{cT2ident} and the corresponding bound. Once one has these, \eqref{cT3ident} and the bound for $\cT^2_b$ follow easily from the identity $B_t=U_\theta$ of Lemma~\ref{variousIdents}.
\end{proof}

Finally, we have one more evolution equation to check, which is for the error term $\bB-\bX_\theta$.
\begin{corollary}\label{bBminbXCor}
For each $t$ in $[0,T]$, we have a map $\cT^3_b(t)$ in $B(\sZ,H^{5/2}(\Sigma))$ such that
\begin{align}
&&\del_t(\bB - \bX_\theta ) &= \cT^3_b \Xi &(a\in \Sigma),
\end{align}
and
\begin{align}
\sup_t\lV \cT^3_b \rV_{B(\sZ,H^{5/2}(\Sigma))} \leq C.
\end{align}
\end{corollary}
\begin{proof}
Observe from \eqref{XUWident} that
\begin{align}
\label{AuxiliaryEqn}
    \bB_t-\bX_{\theta t}=(\bB_t-\bU_\theta)-\bW_\theta .
\end{align}
Using $\cT^1_b$ from Proposition~\ref{BtMinusUthetaLinearProp} and combining this with the expression for $\bW_\theta$ given by Lemma~\ref{WPartialsLem}, we find an appropriate $\cT^3_b$ with the desired properties.
\end{proof}
All in all, by collecting the time derivatives of each component of $\Xi$, we are able to show that the ``error~vector'' $\Xi$ satisfies a system of the following form, which we shall easily verify has the unique solution $\Xi(t)=0$.
\begin{proposition}
For all $t$ in $[0,T]$ we have a map $\cT_\xi(t)$ in $B(\sZ,\sZ)$ such that $\Xi$ satisfies the system below, in which $\bm{\mathrm{J}}_0$ is defined by \eqref{bJ0def}:
\begin{align}\label{XiFinalSysProved}
\begin{aligned}
\frac{d\Xi}{dt} &= \bm{\mathrm{J}}_0 \del_\theta \Xi + \cT_\xi \Xi ,\\
\Xi(0) &= 0.
\end{aligned}
\end{align}
Moreover
\begin{align}
\sup_t\lV\cT_\xi\rV_{B(\sZ,\sZ)} \leq C.
\end{align}
\end{proposition}
\begin{proof}
As a consequence of the initial data $\xi_0$ chosen for the $\xi$ system and Definition~\ref{Xidefn}, one finds $\Xi(0)=0$. The result then follows by applying Lemmas~\ref{cTDeltaxLem},~\ref{DerivDiffLem1}, and~\ref{WPartialsLem}, Corollaries~\ref{cT1Cor} and~\ref{bBminbXCor}, and Proposition~\ref{BtMinusUthetaLinearProp}.
\end{proof}
Now we prove $\Xi(t)=0$ and record the resulting identities.
\begin{corollary}\label{identsResolved}
    For all $t$ in $[0,T]$, $\Xi(t)=0$, and with the various quantities below given by Definition~\ref{finaldefns}, for $a=(\theta,\psi)$ in $\Sigma$ and $x$ in $\Omega(t)$,
\begin{align*}
  X(t,\theta)       &= \bX(t,\theta,0),     &\quad
  \del_t \bX(t,a)   &= \bU(t,a),            &\quad
  w_\Omega(t,x)     &= 0,
  \\
  U(t,\theta)       &= \bU(t,\theta,0),     &\quad
  \del_\theta\bX(t,a)&= \bB(t,a),           &\quad
  \Phi(t)           &= 0.
  \\
  B(t,\theta)       &= \bB(t,\theta,0),     &\quad
  \del_t \bB(t,a)   &= \del_\theta\bU(t,a), &&
\end{align*}
\end{corollary}
\begin{proof}
Recalling the matrix $\bm{\mathrm{J}}_0$ from \eqref{bJ0def}, we observe that the semigroup operator $e^{t\bm{\mathrm{J}}_0 \del_\theta}$ is the solution operator to a constant coefficient, one-dimensional wave equation system. It trivially follows that $e^{t\bm{\mathrm{J}}_0 \del_\theta}$, in addition to the map $\cT_\xi(t)$, is uniformly bounded in $B(\sZ,\sZ)$. Meanwhile, \eqref{XiFinalSysProved} implies
\begin{align}\label{fixedpointInt}
\Xi(t) = \int^t_0 e^{(t-s)\bm{\mathrm{J}}_0 \del_\theta}\cT_\xi(s) \Xi (s) ds ,
\end{align}
and thus $\Xi(t)$ must be zero by a fixed point argument. In view of the definition of $\Xi(t)$, this implies the claimed identities.
\end{proof}

\vspace{1cm}
\begin{remark}
We note that $\bU$ and $\bB$ are now guaranteed to solve the original Lagrangian div-curl system \eqref{OriginalLagrDivCurlSys}, since the above shows the extra terms involving $w_\Omega$ that appear in \eqref{AugmentedInteriorSystemL} are zero.
\end{remark}
Now that the above result is proved, one can show the following. 
\begin{lemma}\label{gradharmonicRHS}
There is a function $\phi(t,x)$ defined for $x$ in $\Omega(t)$ which is harmonic in $\Omega(t)$, where
\begin{align*}
 & &   \bU_t - \bB_\theta + \bP^-_\grad = (\grad \phi)\circ \bX &    &(a \in \Sigma).
\end{align*}
\end{lemma}
\begin{proof}
Because $\Xi = 0$, Proposition~\ref{bVstarProp} implies $\bV_* = 0$. In view of \eqref{Vstardef}, the above assertion is then immediate.
\end{proof}
Given the identity above, it is not terribly difficult to justify that $\bU$ and $\bB$ satisfy the original Lagrangian evolution equations \eqref{OriginalLagrangianSystem}, and that the original ideal MHD system is indeed solved.
\begin{proposition}\label{maintheoremforward}
Taking $\xi$ to be the solution provided by Proposition~\ref{improvedxireg} to the Lagrangian wave system and defining the terms given in Definition~\ref{finaldefns} as well as 
\begin{align}
&&
\begin{aligned}
u(t,x)&=\bU(t,\bX^{-1}(t,x)), \\[.3em]
b(t,x)&=\bB(t,\bX^{-1}(t,x)), \\
p(t,x)&=p_-(t,x) + \half(\bigh_-|h|^2)(t,x),
\end{aligned}
&&(t\in [0,T], \ x\in \Omega(t)),
\end{align}
we obtain a solution to the original ideal MHD system, \eqref{IdealMHD1}--\eqref{IdealMHD6}. We have that $X$ is in $C^0([0,T];H^{k+2}(S^1))$, that $u$, $b$, and $p$ are in $C^0([0,T];H^{k+1}(\Omega(t)))$, and that $h$ is in $C^0([0,T];H^{k+1}(\cV(t)))$.

Additionally, the corresponding interface $\Gamma(t)$ exhibits a splash at time $t=0$, with splash point $p_\splash$, and is non-self-intersecting for $0<t\leq T$. The external magnetic field $h(t,x)$ is zero at time $t=0$ only at the splash point $x=p_\splash$, and nonzero in all of $\overline{\cV(t)}$ for $0<t\leq T$.
\end{proposition}
\begin{proof}
The following div-curl system is satisfied for reasons explained below.
\begin{align}
(a \in \Sigma) \qquad &
\left\{
    \begin{aligned}
        \grad \cdot (\, \cof \bM (\bU_t - \bB_\theta + \bP^-_\grad) ) &= 0 \\
        \grad^\perp\cdot (\, \bM (\bU_t - \bB_\theta + \bP^-_\grad) )   &= 0 \, \Lcm
    \end{aligned}
\right. \label{divcurlFinal}\\
(\psi = 0,-1)  \qquad & \ \Big\{ \, N \cdot (\bU_t - \bB_\theta + \bP^-_\grad) = -\half N\cdot (\grad \bigh_- |h|^2) \circ X ,
\end{align}
where $N=-e_2$ at $\psi=-1$. The div-curl equations of \eqref{divcurlFinal} can be deduced from Lemma~\ref{gradharmonicRHS}, and the boundary condition at $\psi=0$ follows from Lemma~\ref{SurfIdentLems} with the aid of Corollary~\ref{identsResolved}. It is not hard to check from the systems for $\bU$ and $\bB$ that $\bU_2=\bB_2=0$ at $\psi=-1$, and then the condition at $\psi=-1$ in the system above follows from the boundary conditions in the system \eqref{InteriorPressureGradSystem} for $\bP^-_\grad$.

Taking $\bP_\grad = (\grad p) \circ \bX$ for $p$ as defined in the statement of the theorem, it is then not hard to show that
\begin{align}
\begin{aligned}
(a \in \Sigma) \qquad &
\left\{
    \begin{aligned}
        \grad \cdot (\, \cof \bM (\bU_t - \bB_\theta + \bP_\grad) ) &= 0 \\
        \grad^\perp\cdot (\, \bM (\bU_t - \bB_\theta + \bP_\grad) )   &= 0 \, \Lcm
    \end{aligned}
\right. \\
(\psi = 0,-1)  \qquad & \ \big\{ \, N \cdot (\bU_t - \bB_\theta + \bP_\grad) = 0.
\end{aligned}
\end{align}
Using the same kind of reasoning as that used in the proof of Corollary~\ref{SurfIdentLems} to show \eqref{averageZeroEvo}, one then  finds
\begin{align}\label{surfacecondition}
& & \int^\pi_{-\pi} X_\theta \cdot (\bU_t - \bB_\theta + \bP_\grad) \, d\theta & = 0   &(\psi = 0) .
\end{align}
Using these together with the uniqueness implied by Lemma~\ref{GeneralDivCurlSysLem} (following a trivial modification to replace the integral condition over $\psi=-1$ with an integral condition at $\psi=0$, such as \eqref{surfacecondition}), we arrive at
\begin{align}
& & \bU_t - \bB_\theta + \bP_\grad &= 0 & &(a \in \Sigma) , \label{finalUEvo}\\ \shortintertext{and from Corollary~\ref{identsResolved},}
& & \bB_t - \bU_\theta &=0 &&(a\in\Sigma).
\end{align}
By backing out of Lagrangian coordinates, using the fact that $\bX_\theta=\bB$, from Corollary~\ref{identsResolved}, we thus find the evolution equations for $u$ and $b$ in \eqref{IdealMHD1} are satisfied. The divergence free conditions for $u$ and $b$ can be shown to hold by checking
\begin{align}\label{divsZero}
\grad \cdot (\, \cof(\bM) \bU ) =0 \qquad\qquad\mbox{and}\qquad\qquad \grad \cdot (\, \cof(\bM) \bB ) =0.
\end{align}
 The above identities can be justified by observing the quantities in the left-hand sides are zero at $t=0$, and that they solve a simple linear wave system for $t$ in $[0,T]$, namely
 \begin{align}
 \del_t[\grad \cdot (\, \cof(\bM) \bU )] &= \del_\theta  [\grad \cdot (\, \cof(\bM) \bB )], \\
 \del_t[\grad \cdot (\, \cof(\bM) \bB )] &= \del_\theta[\grad \cdot (\, \cof(\bM) \bU )].
 \end{align}
 For example, using \eqref{finalUEvo} and the system solved by $\bP^-_\grad$, the first equation follows from
 \begin{align}
 \del_t[\grad \cdot (\, \cof(\bM) \bU )] &= \grad \cdot (\, \cof (\grad \bU^t)\bU + \cof(\bM) \bU_t) \\
 &=\grad \cdot (\, \cof (\grad \bU^t)\bU + \cof(\bM)\bB_\theta) - \grad \cdot (\, \cof(\bM)\bP_\grad)\\
  &=\grad \cdot (\, \cof (\grad \bB^t)\bB + \cof(\bM)\bB_{\theta}),
 \end{align}
and the second equation is verified by using the identities \eqref{cofId} and $\bB_t=\bU_\theta$ in particular:
\begin{align}
\del_t [\grad \cdot (\,\cof (\bM) \bB)] &=\grad \cdot (\,\cof (\grad\bU^t)\bB) + \grad \cdot (\,\cof (\bM) \bB_t) \\
&=\grad \cdot (\,\cof (\grad\bB^t)\bU) +\grad \cdot (\,\cof (\bM) \bU_\theta) .
\end{align}
Thus, we do indeed get \eqref{divsZero}, which implies $u$ and $b$ are divergence free. The boundary conditions \eqref{IdealMHD3}--\eqref{IdealMHD4} are immediate from the system \eqref{OriginalLagrDivCurlSys}, and, by the construction of $h$, the vacuum system \eqref{IdealMHD2}, \eqref{IdealMHD5}, \eqref{IdealMHD6} is of course satisfied for $\eta_1=\eta_2=1$.

That the various quantities have the regularities claimed in the statement of the theorem is not difficult to check as a consequence of the fact that $\xi$ is in $\sB$. By Lemma~\ref{dcXlemma}, this fact also implies that the interface becomes non-self-intersecting for $t>0$, and by Lemma~\ref{hTanDirLemma}, that the external magnetic field $h$ only vanishes at the splash point at time $t=0$.
\end{proof}
\subsection{Sobolev data existence theorem}\label{maintheoremsubsection}
Finally, we are ready to prove our first main theorem, following from the time-reversibility property for the ideal MHD system together with the result of Proposition~\ref{maintheoremforward}, which provides a solution to the ideal MHD equations that starts with a splash--squeeze singularity at time $t=0$ and features an interface which opens up, becoming non-self-intersecting for $t$ positive.

Whereas $k$ is considered to be a fixed integer in the results stated above, we emphasize that the result below holds for arbitrary integer $k\geq 4$. This means that we are able to create splash--squeeze singularities for the ideal MHD system for which the interface has arbitrarily high degree of smoothness up to and including the moment at which the singularity has formed.
\begin{theorem}\label{maintheorem}
For any $k\geq 4$, for some positive $t_\splash$, there exists a solution $(u_\play,b_\play,h_\play,\Gamma_\play)$ to the free-boundary ideal MHD equations for which $\Gamma_\play(t)$ forms a single glancing self-intersection at time $t=t_\splash$ as in a standard splash singularity and the $h_\play$ is nonzero at all points inside the vacuum $\cV_\play(t)$ for all $t$ in $[0,t_\splash]$. In particular, the solution exhibits a splash--squeeze singularity at the time $t_\splash$.

Furthermore, the interface $\Gamma_\play(t)$ is parametrized by $X_\play$ in $C([0,t_\splash];H^{k+2}(S^1))$, $u_\play$ and $b_\play$ are in $C([0,t_\splash]; H^{k+1}(\Omega(t)))$, and $h_\play$ is in $C([0,t_\splash]; H^{k+1}(\cV(t)))$.
\end{theorem}
\begin{proof}
For any such $k$, we are able to apply Proposition~\ref{maintheoremforward} to get the corresponding solution to the ideal MHD equations, which we denote by $u_\revplay$, $b_\revplay$, and $h_\revplay$, with $u_\revplay$ and $b_\revplay$ defined in the plasma domain $\Omega_\revplay$ and $h_\revplay$ in the vacuum domain $\cV_\revplay$, all continuously taking values in the appropriate Sobolev spaces on the time interval $[0,T]$. Let us denote $t_\splash=T$. Appealing to the time-reversibility of the system, we construct a new solution to \eqref{IdealMHD1}--\eqref{IdealMHD6} for $t$ in $[0,t_\splash]$ by defining for $\Omega_\play(t)=\Omega_\revplay(t_\splash-t)$ and $\cV_\play(t)=\cV_\revplay(t_\splash-t)$ the vector fields
\begin{align}
&& u_\play(t,x) &= - u_\revplay(t_\splash-t,x),
&& (t\in[0,t_\splash), \ x\in \Omega_\play(t)), \\
&& b_\play(t,x) &=  b_\revplay(t_\splash-t,x),
&& (t\in[0,t_\splash), \ x\in \Omega_\play(t)), \\
&& h_\play(t,x) &= h_\revplay(t_\splash-t,x)
&& (t\in[0,t_\splash], \ x\in \cV_ \play(t)),
\end{align}
and correspondingly defining $X_\play$ and $\Gamma_\play(t)$. This produces a solution with the described Sobolev regularity which starts with a non-self-intersecting interface and culminates in a splash--squeeze singularity at the time $t=t_\splash$.
\end{proof}

\section{Supporting lemmas and estimates}\label{auxiliaryestimates}

\subsection{Extension lemmas}
Let us define
\begin{align}
\Sigma_- = \{(\theta,\psi):\theta\in S^1, \ \psi \leq 0\}
\end{align}
and think of $\Sigma_-$ as sitting below $S^1$. Whereas $\Sigma_+(\delta)$, defined in \eqref{SigmaPlusDefn}, is analogous $\Omega_+$, the domain $\Sigma_-$ is analogous to $\Omega_-$. The following is a basic lemma providing a linear extension operator for extending vector fields defined on just $S^1$ to the rest of $\Sigma_-$.
\begin{lemma}\label{vecLowerExtLem}
There exists a linear map $\cE_{\operatorname{lin}}$ acting on $\R^2$-valued vector fields $V$ in $H^s(S^1)$ for any $s\geq 0$, producing an $\R^2$-valued vector field $\cE_{\operatorname{lin}} V$ in $H^{j+1/2}(\Sigma_-)$ with
\begin{align}
\cE_{\operatorname{lin}} V(\theta,0) = V(\theta) \hspace{2cm}  (\theta \in S^1), \\
\supp (\cE_{\operatorname{lin}} V )\subset\{(\theta,\psi):\theta\in S^1, \ -2\leq \psi \leq 0 \},
\end{align}
and for some constant $C_s$ dependent on $s$ alone,
\begin{align}
& & \lV\cE_{\operatorname{lin}} V\rV_{H^{s+1/2}(\Sigma_-)} &\leq C_s \lV V \rV_{H^{s}(S^1)}. &
\end{align}
\end{lemma}
\begin{proof}
One can construct $\cE_{\operatorname{lin}}$ satisfying the stated properties by taking harmonic extensions of the components of $V$ to $\Sigma_-$ and multiplying by a smooth cutoff function.
\end{proof}

Now we discuss another extension operator, in this case nonlinear, which is designed specifically for parametrizations of interface curves. The point is to provide a mapping between $\Sigma_-$ and a domain $\Omega_-$ corresponding to some $X:S^1\to\Gamma$. This is primarily used to compare vector fields defined on a pair of domains $\Omega_-$ and $\underline \Omega_-$ associated to some $X$ and $\uX$ after pulling the two vector fields back to the common domain $\Sigma_-$.
\begin{proposition}\label{LowExtensionProp}
There is an extension map
\begin{align}
\cE_- : H^{k+2}_\gp  \to \{\mbox{\emph{Continuous maps from}} \ \Sigma_- \ \mbox{to} \ \Omega_-\}
\end{align}
where, for $X$ in $H^{k+2}_\gp$, we have $\cE_-(X)(\theta,0)=X(\theta)$. Moreover, $\cE_-$ satisfies the following additional properties:

\begin{enumerate}[(i)]
\item $\cE_-(X)$ is a bijection from $\Sigma_-$ to $\Omega_-$.

\item  We have
\begin{align}
& & \cE_-(X)(a) &= a  &   (a=(\theta,\psi)\in\Sigma_-, \  \psi \leq -2 ).
\end{align}

\item We have for integer $j$ with $0\leq j \leq k$ that
\begin{align}
\lV \cE_-(X)-\operatorname{id}\rV_{H^{j+5/2}(\Sigma_-)} \leq C\left( \lV X \rV_{H^{j+2}(S^1)}\right) .
\end{align}

\item Given $\xi$ and $\uxi$ in $\sB^\bbox$, for the corresponding pair $X$ and $\uX$ and $j$ as above,
 \begin{align}
 \lV \cE_-(X)-\cE_-(\uX) \rV_{H^{j+5/2}(\Sigma_-)} \leq C \lV X-\uX \rV_{H^{j+2}(S^1)} .
 \end{align}
 \end{enumerate}
\end{proposition}
\begin{proof}
One can use Lemma~\ref{vecLowerExtLem} and a partition of unity to construct an extension map satisfying (i)-(iv), where $\cE_-(x)$ maps the lines of constant $\psi$ below $\{\psi=0\}$ to ``progressively flatter versions'' of $\Gamma$ lying below $\Gamma$, which sweep out the region $\Omega_-$.
\end{proof}

Now, we give the counterpart to the above proposition for extensions to the domain $\Sigma_+(\delta)$. This is for comparing vector fields defined on domains $\Omega_+$ and $\underline \Omega_+$ associated to parametrizations $X$ and $\uX$ of interfaces $\Gamma$ and $\underline \Gamma$. In these situations, it helps to ensure the extension operator satisfies a few more specialized properties.
\begin{proposition}\label{extensionProp2}
For $0<\delta\leq \delta_0$ there is an extension operator
\begin{align}
\cE_\delta : \{X:\xi \in \sB^\bbox(\delta)\}  \to \{\mbox{Continuous maps from} \ \Sigma_+(\delta) \ \mbox{to} \ \Omega_+\cup\{i\infty\}\} ,
\end{align}
where, for $\xi$ in $\sB^\bbox(\delta)$ with corresponding $X$, $\cE_\delta(X)(\theta,0)=X(\theta)$. We also have that $\cE_\delta(X)$

\begin{enumerate}[(i)]

\item maps the upper boundary as well as the lower boundary of $\Sigma_+(\delta)$ onto $\Gamma$,

\item maps $a_\star = (0,(1+\delta)/2)$ to $z_\star=i$ and $a_\infty = (\pi,(1+\delta)/2)$ to $i\infty$,

\item provides a two-to-one covering map from $\Sigma_+(\delta)\setminus \{a_\star ,a_\infty\}$ to $\Omega_+\setminus\{z_\star\}$ and a two-to-one covering map from $\Sigma^\circ_+(\delta)$ to $\cV$,

\item maps $\tilde \cW_i$ (defined in \eqref{tildeW1} and \eqref{tildeW2}) onto $\cW_i$ for $i=1,2$, and

\item satisfies the following for an $m\geq 0$ dependent only on $k$, where we denote 
\[\Sigma^\star=\Sigma_+(\delta)\cup\{|a-a_\star|\leq 10^{-1}\},\]
with the corresponding weighted norm defined in the obvious way, and we take any integer $j$ with $0\leq j \leq k$:
\begin{align}
\lV \cE_\delta(X)\rV_{H^{j+5/2}(\Sigma^\star,\delta,-m)} +\lV (\grad\cE_\delta(X))^{-1}\rV_{H^{j+3/2}(\Sigma^\star,\delta,-m)}\leq C \left(\lV X \rV_{H^{j+2}(S^1)}\right).
\end{align}
Additionally, for a pair $X,\uX$ corresponding to $\xi,\uxi$ in $\sB^\bbox(\delta)$,
\item  we have
\begin{align}
& & \cE_\delta(X)(a) &= \cE_\delta(\uX)(a)  &   \left(|a-a_\star|\leq 10^{-1} \mbox{ or }|a-a_\infty|\leq 10^{-1} \right),
\end{align}

\item and for any integer $j$ with $0\leq j \leq k$ (and $m$ as before),
 \begin{align}
 \lV \cE_\delta(X)-\cE_\delta(\uX) \rV_{H^{j+5/2}(\Sigma_+,\delta,-m)} \leq C \lV X-\uX \rV_{H^{j+2}(S^1)} .
 \end{align}

\end{enumerate}
\end{proposition}
\begin{proof}
Consider positive $\delta\leq \delta_0$. First let us construct a nice, smooth map $\sP_\delta$ from the domain $\Sigma_+(\delta)$ to a dumbbell-shaped region we call $\Sigma^\sharp_+(\delta)$.

Observe that applying the map $\sS_\delta(\theta,\psi)=\theta+i(\psi-(\delta+\cos^2 \theta)/2)$ to $\Sigma_+(\delta)$ yields the symmetric region $\tilde \Sigma_+(\delta)$ sketched in Figure~\ref{tildeSigma}, where $a_\star$ has been mapped to $0$ and $a_\infty$ to $\pi$.

\begin{figure}[h!]
  \centering

  \begin{minipage}[t]{0.55\textwidth}
    \centering
\begin{tikzpicture}
  \begin{axis}[
    height=4cm,
    width=8cm,
    axis lines=none,
    domain=-pi:pi,
    samples=200,
    clip=false,
    enlargelimits=false
  ]

  \addplot [
    fill=gray!20,
    draw=none,
  ] coordinates {
    (-pi,1.3)
    (pi,1.3)
  }
  -- plot[domain=pi:-pi,samples=200] (\x,{-cos(deg(2*\x))})
  -- cycle;


  \addplot [very thick, black, domain=-pi:pi, samples=200] {-cos(deg(2*x))};
\addplot [
  fill=gray!20,
  draw=none,
  domain=-pi:pi,
  samples=200
] 
{1.2} -- plot[domain=pi:-pi, samples=200] (\x,{2.4 + cos(deg(2*\x))}) -- cycle;

\addplot [very thick, black, domain=-pi:pi, samples=200] {2.4 + cos(deg(2*x))};

  \addplot [
    only marks,
    mark=*,
    mark size=1.5pt,
    black
  ] coordinates {(0,1.2)};
  
  \addplot [
    only marks,
    mark=*,
    mark size=1.5pt,
    black
  ] coordinates {(-pi,1.2) (pi,1.2)};

  \node[below left] at (axis cs:-pi,1.2) {$-\pi$};
  \node[below right] at (axis cs:pi,1.2) {$\pi$};
  \node at (0,.8) {$0$};
  \end{axis}
\end{tikzpicture}
    \captionof{figure}{$\tilde \Sigma_+(\delta)$}
    \label{tildeSigma}
  \end{minipage}\hfill
  %
  \begin{minipage}[t]{0.4\textwidth}
    \centering
    \includegraphics[width=0.9\linewidth]{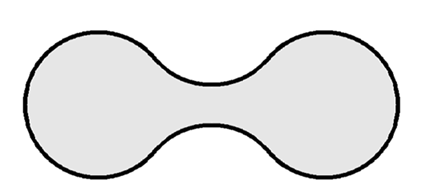}

    \begin{tikzpicture}[overlay, remember picture]
        \node at (-1.5,1.65) {$-1$};
        \node at (-1.15,1.65) {$\bullet$};
        \node at (1.4,1.65) {$1$};
        \node at (1.15,1.65) {$\bullet$};
    \end{tikzpicture}
    
    \captionof{figure}{$\Sigma^\sharp_+(\delta)$}
    \label{fig:intermediateVacuum}
  \end{minipage}

\end{figure}

We then have $z\mapsto e^{iz}$ gives a bijective map from $\tilde \Sigma_+(\delta)$ to a distorted, but symmetric, annular-type region, which has pinches of thickness approximately $\delta$ near $i$ and $-i$, the images of $\pi/2$ and $-\pi/2$, respectively. Applying the Joukowski map\footnote{The Joukowski map, which is two-to-one, is the reason that we end up with a double cover.} $\sJ(z)=(z+z^{-1})/2$ after applying $e^{iz}$ then yields the simply-connected region $\Sigma^\sharp_+(\delta)$ (also with pinch approximately $\delta$), which is sketched in Figure~\ref{fig:intermediateVacuum}, where $0$ is mapped to $1$ and $\pi$ is mapped to $-1$. Let us note $\sJ(e^{iz})=\cos z$. We thus get our desired map from $\Sigma_+(\delta)$ to $\Sigma^\sharp_+(\delta)$ (providing a two-to-one cover) by defining
\begin{align}\label{sPmapdefn}
&& \sP_\delta(a)&=\cos(\sS_\delta(a))& (a\in\Sigma_+(\delta)).
\end{align}

Applying the map $\Psi_\delta(x_1+ix_2)=(x_1,x_2/(x^2_1+\delta))$ to $\Sigma^\sharp_+(\delta)$ then results in a nice, simply-connected region $ \Sigma^\flat_+(\delta)$ whose pinch is uniformly bounded below. We define the map $Y_\delta:\Sigma_+(\delta)\to \Sigma^\flat_+(\delta)$ by $Y_\delta(a)=\Psi_\delta\circ\sP_\delta$.

Now consider $X$ corresponding to a fixed $\xi$ in $\sB^\bbox(\delta)$ for $0<\delta\leq\delta_0$. We proceed to define a bijection from $\Omega_+\cup\{i\infty\}$ to a region $\Omega^\flat_+$ which is similar to $\Sigma^\flat_+(\delta)$ in that it, too, is both simply-connected and of pinch uniformly bounded below. Let $p_\ell$ and $p_r$ be the unique pair of points on either side of the pinch of $\Gamma$ ($p_\ell$ to the left and $p_r$ to the right) such that $|p_\ell-p_r|=\delta_\Gamma$, and let $p_\odot$ be the midpoint of $p_\ell$ and $p_r$. Consider the map
\begin{align}
\sC_X(z)= 2\cos^2(p_r/2)(\tau(p_r))^{-1}(\tan(z/2)-\tan(p_\odot/2)),
\end{align}
where, above, $\tau(p_r)$ represents the tangent to $\Gamma$ at $p_r$ regarded as a complex number. It follows $\sC_X(z)$ maps $\Omega_+\cup\{i\infty\}$ to a bounded, roughly dumbbell-shaped region $\Omega^\sharp_+$ whose pinch is approximately $\delta$, with horizontal tangents at the new pinch points, i.e. $\sC_X(p_\ell)$ and $\sC_X(p_r)$. Thus $\Omega^\sharp_+$ is qualitatively similar to $\Sigma^\sharp_+(\delta)$. Moreover, $i\infty$ and $z_\star$ are mapped to points bounded away from the origin, just as $1$ and $-1$ are in Figure~\ref{fig:intermediateVacuum}. After this, applying the map $\Psi_\delta$ results in the region $\Omega^\flat_+=\Psi_\delta(\Omega^\sharp_+)$, which has pinch uniformly bounded below. Let us apply the composition of these maps to $X$ to get $X^\flat: S^1 \to \del\Omega^\flat_+$, with
\begin{align}
&& X^\flat(\theta) &=\Psi_\delta(\sC_X(X(\theta)))&(\theta\in S^1).
\end{align}

Now the idea is to define a diffeomorphism between $\Omega^\flat_+$ and $ \Sigma^\flat_+(\delta)$, say, $\sM_X:\Omega^\flat_+\to \Sigma^\flat_+(\delta)$. We choose a nice, smooth map $\sM_X$ so that, in particular,
\begin{align}\label{sMdiffeo}
&& \sM_X(X^\flat(\theta)) &= Y_\delta(\theta,0) & (\theta \in S^1).
\end{align}
Therefore we get a map $\cE_\delta(X):\Sigma_+(\delta)\to \Omega_+\cup\{i\infty\}$ by defining
\begin{align}
&& \cE_\delta(X)(a) &= \sC_X^{-1}\circ\Psi^{-1}_\delta\circ\sM_X^{-1} \circ Y_\delta(a) & (a\in \Sigma_+(\delta)),
\end{align}
resulting in $\cE_\delta(X)(\theta,0)=X(\theta)$, so that we have an extension of $X$ from $S^1$ to $\Sigma_+(\delta)$.
Property (i) is automatic from the construction, and with some minor additional specifications for the map $\sM_X$, it is not difficult to ensure that properties (ii)-(vii) are all satisfied. We comment that in order to guarantee (v) and (vii) specifically, we rely on the key fact that $\lV X-X_0\rV_{H^{k+1}(S^1)}\leq C\delta$ for $\xi$ in $\sB^\bbox(\delta)$.

\end{proof}


\subsection{Local elliptic estimates in a neighborhood of a smooth boundary}\label{basicEstSection}
Now we discuss some basic local estimates which we are able to stitch together to produce virtually all the other elliptic estimates, weighted and unweighted, that we need for our analysis.

\begin{lemma}\label{buildinglemma1}
Consider a bounded domain $\cD\subset \R^2$ whose boundary is parametrized by $Z$ in $H^{k+2}(S^1)$, with curvature bounded by $10^{-1}$ and length bounded by $10$. Take a pair of concentric discs $D_1$, $D_2$ with common center on $\del\cD$ and radii $r_1$ and $r_2$, respectively, where $r_1\geq 1$, $r_2\leq 2$, and $r_1/r_2\leq 9/10$.
Consider $\mathpzc{V} = D_1 \cap \cD$ and $\cU = D_2 \cap \cD$. Consider real $s$ with \( 0 \leq s \leq k+1/2\). Then we have the estimate
\begin{align}
& &     \lV f \rV_{H^{s+2}(\mathpzc V)} & \leq C\lV \Delta f \rV_{H^s(\cU)}    &   (f \in H^{s+2}_0(\cU ) ),
\end{align}
where $C=C\left(\lV Z \rV_{H^{k+2}(S^1)}\right)$ if $s>k-1/2$ and $C=C\left(\lV Z \rV_{H^{k+1}(S^1)}\right)$ if $s\leq k-1/2$.
\end{lemma}
\begin{proof}
    This is a very classical type of estimate whose proof is fairly standard. We simply comment that in the proof for the case of integer $s$, which would typically be carried out first, the Poincare inequality is useful for absorbing an $L^2$ norm of $f$ that appears in the right-hand side.
\end{proof}
\begin{lemma}\label{unscaledbasicest}
Let $\cD$ be a bounded domain in $\R^2$ with boundary $\gamma$ parametrized by $Z$ as in Lemma~\ref{buildinglemma1}. Consider real $s$ with \(0 \leq s \leq k + 1/2\). Suppose $\gamma$ passes through a pair of opposing sides for some unit square $Q$, cutting $Q$ into two pieces $P_1$ and $P_2$, each having area at least $10^{-1}$. Furthermore, suppose for each $i=1,2$ that $\del P_i$ is given by arclength parametrization $\zeta_i:S_i\to\R^2$ (which is Lipschitz) satisfying\footnote{Note \eqref{chordarclipbound} gives a lower bound on the interior angles formed at the corners of $\del P_i$ as well as the ``pinch'' of $\del P_i$.}
\begin{align}\label{chordarclipbound}
&& \frac{|\zeta_i(s_1)-\zeta_i(s_2)|}{|s_1-s_2|}&\ge 10^{-1} &(s_1,s_2\in S_i).
\end{align}

Let $1.1 Q$ denote the dilate of $Q$ by a factor of $1.1$ about its center. Define $\cU = 1.1 Q \cap \cD$ and \(\mathpzc{V} = Q \cap \cD\). Then the following holds:
\begin{align}
    \lV f \rV_{H^{s+2}(\mathpzc{V})} \leq C&\left( \lV \Delta f \rV_{H^s(\cU)} + \lV f \rV_{H^{s+3/2}( 1.1 Q \cap \gamma )}\right)  \\
    &\hspace{2cm}(f \in C^\infty(\overline{\cD}) \text{ \emph{with} } \supp ( f )\subset \overline{\cU} ),
\label{BasicSquareEst}
\end{align}
where $C=C\left(\lV Z \rV_{H^{k+2}(S^1)}\right)$ if $s>k-1/2$ and $C=C\left(\lV Z \rV_{H^{k+1}(S^1)}\right)$ if $s\leq k-1/2$.
\end{lemma}
\begin{proof}
    First, we verify the conclusion of Lemma~\ref{buildinglemma1} holds when we replace the $\mathpzc V$ and $\cU$ as described in the statement of Lemma~\ref{buildinglemma1} by $\mathpzc V$ and $\cU$ as in the statement of this lemma. Suppose we have $f$ in $H^{s+2}_0(\cU)$. Extend $f$ to the rest of $\R^2$ by taking it to be zero. Let us take a pair of concentric discs $D_1$ and $D_2$ as described in the statement of Lemma~\ref{buildinglemma1} such that $\mathpzc V\subset D_1$ and $\cU \subset D_2$, defining $\mathpzc V' = D_1\cap \cD$ and $\cU' = D_2\cap \cD$. Applying Lemma~\ref{buildinglemma1} and using the fact that $\supp (f) \subset \overline \cU$, we obtain
    \begin{align}\label{buildingblockest1}
    \lV f \rV_{H^{s+2}(\mathpzc V)} \leq \lV f \rV_{H^{s+2}(\mathpzc V')} \leq C \lV \Delta f \rV_{H^s(\cU')}\leq C' \lV \Delta f \rV_{H^s(\cU)}.
    \end{align}
    
    This gives us the desired bound in the case that $f$ is in $H^{s+2}(\cU)$.

    Now let us consider \( f \in C^\infty (\overline{\cD})\) with \(\supp ( f )  \subset \cU \). Taking a standard Sobolev extension operator $\cE_s$ for extending functions from the curve $\gamma$ to the rest of $\R^2$, we find
    \begin{align}
        \lV \cE_s (f|_\gamma) \rV_{H^{s+2}(\R^2)} \leq C \lV f \rV_{H^{s+3/2}(1.1 Q \cap \gamma)} .
    \end{align}
    Note $f - \cE_s (f|_\gamma)$ is in $H^{s+2}(\cU)$. By applying our estimate \eqref{buildingblockest1} to $f - \cE_s (f|_\gamma)$, we conclude \eqref{BasicSquareEst} holds. The claimed dependence of the bounding constant on the norm of $Z$ is not difficult to verify.
\end{proof}
\begin{corollary}\label{buildingcor}
    Suppose $\cD$, $\gamma$, $Q$, $\cU$, $\mathpzc{V}$, $Z$, and $s$ are all as in the statement of Lemma~\ref{unscaledbasicest} except that the sidelength of $Q$ is $\ell$ for some $\ell \leq 1$. Then for an $m\geq 0$ depending only on $s$, the following holds.
\begin{align}
     \lV f \rV_{H^{s+2}(\mathpzc{V})} &\leq C\left( \lV \Delta f \rV_{H^s(\cU)} + \lV f \rV_{H^{s+3/2}( 1.1 Q \cap \gamma )}\right) \ell^{-m} \\
    &\hspace{2cm}(f \in C^\infty(\overline{\cD}) \text{ \emph{with} } \supp ( f )\subset \overline{\cU} ),
\label{ScaledSquareEst}
\end{align}
where $C=C\left(\lV Z \rV_{H^{k+2}(S^1)}\right)$ if $s>k-1/2$ and $C=C\left(\lV Z \rV_{H^{k+1}(S^1)}\right)$ if $s\leq k-1/2$.
\end{corollary}
\begin{proof}
    The above bound is easily verified by rescaling the estimate of Lemma~\ref{unscaledbasicest}.
\end{proof}

\subsection{Below-interface elliptic estimates}\label{unpinchedestimates}
In this section, we apply our local estimates from the previous section to prove estimates for the domains in our problems which do not become pinched, such as the plasma domain $\Omega$, as opposed to the vacuum domain $\cV$. As a result, the proofs of many of these estimates are fairly elementary.

Recall the domain $\Sigma_- = \{(\theta,\psi):\theta\in S^1, \ \psi \leq 0\}$ we defined earlier. Consider an $X$ in $H^{k+2}_\gp$, with the region below the associated $\Gamma$ denoted by $\Omega_-$. Using the extension map $\cE_-$ provided by Proposition~\ref{LowExtensionProp}, we define $X_- = \cE_-(X):\Sigma\to\Omega_-$ and correspondingly, for $f:\Sigma_-\to\R$,
\begin{align}
&&
\begin{aligned}
 A(X) &= (\grad X_-)^{-1}\cof (\grad X_-),  \\
L(X) f  &= \sum^2_{i,j=1} \del_i(A(X)_{ij}\del_j f),
\end{aligned}
&&(a\in\Sigma_-).
\end{align}
Let us remark that for $\varphi_f:\Omega_-\to\R$ defined by $f = \varphi_f \circ X_-$, one has the identity
\begin{align}\label{varphiLagProb}
L(X)f = (\det \grad X_-)(\Delta \varphi_f) \circ X_- .
\end{align}
By stitching together estimates given by Lemma~\ref{buildinglemma1}, using \eqref{varphiLagProb}, and applying Proposition~\ref{LowExtensionProp}, we obtain the following.
\begin{lemma}\label{basicellipticboundLagr}
For $X$ in $H^{k+2}_\gp$, we have for $0\leq s \leq k+1/2$ the estimate
\begin{align}
&& \lV f \rV_{H^{s+2}(\Sigma_-)} &\leq C\left( \lV L(X) f \rV_{H^{s}(\Sigma_-)} + \lV f |_{\psi = 0} \rV_{H^{s+3/2}(S^1)}\right), &
\end{align}
where $C=C\left(\lV X \rV_{H^{k+2}(S^1)}\right)$ if $s=k+1/2$ and $C=C\left(\lV X \rV_{H^{k+1}(S^1)}\right)$ if $s\leq k-1/2$.
\end{lemma}
\subsubsection*{Div-curl system bounds}

We are able to formulate bounds for div-curl systems as a consequence of Lemma~\ref{basicellipticboundLagr}, for a vector field $V:\Sigma_-\to\R^2$ given by $V=(\grad \varphi_f)\circ X_-$ (or $(\grad^\perp \varphi_f)\circ X_-$), where $f =  \varphi_f \circ X_-$ solves a Poisson-type problem of the form below, or some minor variation thereof:
\begin{align}
&& L(X) f &= \rho &&(a\in\Sigma), \\
&& f &= g && (\psi = 0).
\end{align}
Below, we record several elementary results which can be proved in this way. 
\begin{lemma}
\label{normalDivCurlSys}
Take $X$ in $H^{k+2}_\gp(S^1)$. Let us denote $X_- = \cE_-(X)$, the extension of $X$ to $\Sigma_-$ given by Proposition~\ref{LowExtensionProp}. Consider the system
\begin{align}
\label{GeneralDivCurlSystem1}
\begin{aligned}
(a \in \Sigma_-) \qquad &
\left\{
    \begin{aligned}
        \grad \cdot (\, \cof (\grad X^t_-) V ) &= \rho \\
        \grad^\perp\cdot (\, \grad X^t_- V )   &= \varpi \, \Lcm
    \end{aligned}
\right. \\
(\psi = 0)  \qquad & \ \big\{ \, X^\perp_\theta \cdot V = \eta .
\end{aligned}
\end{align}
Suppose for half-integer $s$ with $0\leq s \leq k+1/2$ we have $\rho,\varpi \in H^s(\Sigma_-)$, $\eta \in H^{s+1/2}(S^1)$, and
\begin{align}
  \int_{\Sigma_-} \rho \, da = \int^\pi_{-\pi} \eta \, d\theta\, .
\end{align}
Then there exists a unique solution $V$ defined in $\Sigma_-$ to the above system. Moreover, we have the estimate 
\begin{align}\label{GenDivCurlSysBd}
\lV V \rV_{H^{s+1}(\Sigma_-)}
\leq C\left(\lV \rho \rV_{H^s(\Sigma_-)}+ \lV \varpi \rV_{H^s(\Sigma_-)} + \lV \eta \rV_{H^{s+1/2}(S^1)}\right),
\end{align}
where $C=C\left(\lV X \rV_{H^{k+2}(S^1)}\right)$ if $s=k+1/2$ and $C=C\left(\lV X \rV_{H^{k+1}(S^1)}\right)$ if $s\leq k-1/2$.
\end{lemma}
\begin{lemma}
\label{tanglDivCurlSys}
Let us take $X$ in $H^{k+2}_\gp(S^1)$ and $X_- = \cE_-(X)$. Consider the system
\begin{align}
\label{GeneralDivCurlSystem2}
\begin{aligned}
(a \in \Sigma_-) \qquad &
\left\{
    \begin{aligned}
        \grad \cdot (\, \cof (\grad X^t_-)  V ) &= \rho \\
        \grad^\perp\cdot (\, \grad X^t_- V )   &= \varpi \, \Lcm
    \end{aligned}
\right. \\
(\psi = 0)  \qquad & \ \big\{ \,  X_\theta \cdot V = \eta .
\end{aligned}
\end{align}
Suppose for half-integer $s$ with $0\leq s \leq k+1/2$ we have $\rho,\varpi \in H^s(\Sigma_-)$, $\eta \in H^{s+1/2}(S^1)$, and
\begin{align}
\int_{\Sigma_-} \varpi \, da = - \int^\pi_{-\pi} \eta \, d\theta\, .
\end{align}
Then there exists a unique solution $V$ defined in $\Sigma_-$ to the above system, and the estimate \eqref{GenDivCurlSysBd} holds.
%
\end{lemma}
We record one more result on existence and bounds for solutions to div-curl type problems proven in a similar manner to Lemmas~\ref{normalDivCurlSys} and~\ref{tanglDivCurlSys}, this time in the bounded domain $\Sigma$ as opposed to $\Sigma_-$.
\begin{lemma}
\label{GeneralDivCurlSysLem}
Consider $X$ in $H^{k+1}_\gp$ with associated plasma domain $\Omega$. Consider a bijective map $\bX:\Sigma\to\Omega$ in $H^{k+3/2}(\Sigma)$. For half integer $s$ with $0\leq s\leq k-1/2$, we define the vector space
\begin{align}
\mathpzc{F}^s = \{( \rho, \varpi, f, \gamma ) : \: \rho, \varpi \in H^s(\Sigma) , \ \eta \in H^{s+1/2}(S^1) , \ \gamma \in \R , \ \mbox{and} \ \int_\Sigma \rho \, da = \int^\pi_{-\pi} \eta \, d\theta \} .
\end{align}
There is a bounded map $\mathpzc T \in B(\mathpzc F^s, H^{s+1}(\Sigma) )$ defined by
\begin{align}
&& \mathpzc T(\rho, \varpi, \eta, \gamma)(a) &= V(a) & (a\in \Sigma),
\end{align}
where, for \( \bM = \grad \bX^t \) and $\bm{N} = \bX^\perp_\theta/|\bX_\theta|$, $V$ is the unique solution of the following system.
\begin{align}
\label{GeneralDivCurlSystem}
\begin{aligned}
(a \in \Sigma) \qquad &
\left\{
    \begin{aligned}
        \grad \cdot (\, \cof (\bM) V ) &= \rho \\
        \grad^\perp\cdot (\, \bM V )   &= \varpi \, \Lcm
    \end{aligned}
\right. \\
(\psi = 0)  \qquad & \ \big\{ \, \bm{N} \cdot V = \eta , \\
(\psi = -1) \qquad & \ \Big\{ \, \bm{N} \cdot V = 0 , \qquad\quad \int^\pi_{-\pi} \bX_\theta \cdot V \, d\theta = \lambda .
\end{aligned}
\end{align}
Moreover, we have the estimate
\begin{align}
\lV V \rV_{H^{s+1}(\Sigma)}
\leq C \left(\lV \rho \rV_{H^s(\Sigma)}+ \lV \varpi \rV_{H^s(\Sigma)} + \lV \eta \rV_{H^{s+1/2}(S^1)} + |\lambda|\right),
\end{align}
where $C=C\left(\lV X \rV_{H^{k+1}(S^1)}\right)+C\left(\lV \bX \rV_{H^{k+3/2}(\Sigma)}\right)$.
\end{lemma}

\subsubsection*{Basic applications to $\cN_-(\xi)$ and $\bE(\xi)$}

The bound of Lemma~\ref{tanglDivCurlSys} is relevant to the Dirichlet-to-Neumann operator $\cN_-(\xi)$, whose definition for a given $\xi$ in $\sB^\bbox$ (see Definition~\ref{DirichToNeumDefs}) involves extension into the unpinched region $\Omega_-$. Indeed, consider a function $f$ in $H^1(S^1)$. Then we may define $V$ to be the unique solution of \eqref{GeneralDivCurlSystem2} with $\rho=\varpi=0$ in $\Sigma_-$ and boundary condition $X_\theta \cdot V = \eta=\frac{df}{d\theta}$. It follows
\begin{align}
& & \cN_-(\xi) f &= N \cdot V &   (\theta \in S^1). 
\end{align}
Let us show how this relation with the system \eqref{GeneralDivCurlSystem2} is used to verify a Lipschitz bound for the map $\xi\mapsto\cN_-(\xi)$.
\begin{proposition}\label{LipDirichNeumBd}
For $\xi$, $\uxi$ in $\sB^\bbox$,
\begin{align}
\lV ( \cN_-(\xi) - \cN_-(\uxi) )f \rV_{H^2(S^1)} \leq C \lV \xi-\uxi \rV_{\sH^1} \lV f \rV_{H^3(S^1)}.
\end{align}
\end{proposition}
\begin{proof}
In the following, we use the obvious notation with underlines to denote various quantities corresponding to $\uxi$. Let us use Proposition~\ref{LowExtensionProp} to define extensions $X_-=\cE_-(X)$ and $\uX_- = \cE_-(\uX)$. Now consider $V$ and $\underline V$ satisfying the corresponding div-curl systems of \eqref{GeneralDivCurlSystem2} with $X_-$ and $\uX_-$, respectively, and taking $\rho=0$, $\varpi=0$, and boundary conditions $X_\theta \cdot V =\uX_\theta \cdot \underline V= \frac{df}{d\theta}$ at $\psi=0$. It follows
\begin{align}
\cN_- f - \underline \cN_- f = N \cdot V - \underline N\cdot \underline V.
\end{align}
We then observe
\begin{align}\label{cNdifference}
\lV (\cN_-  - \underline \cN_- )f\rV_{H^2(S^1)}
\leq C ( \lV N- \underline N \rV_{H^2(S^1)} \lV V \rV_{H^2(S^1)} + \lV \underline N \rV_{H^2(S^1)}\lV V - \underline V \rV_{H^2(S^1)}) .
\end{align}
Meanwhile, noting in particular \(X_\theta \cdot V = \underline X_\theta \cdot \underline V\), we find that when we apply Lemma~\ref{tanglDivCurlSys} to the system satisfied by $V-\underline V$, we get
\begin{align}
\lV V - \underline V \rV_{H^{5/2}(\Sigma_-)} &\leq C ( \lV \grad \cdot [(\cof(\grad \uX^t_-) - \cof(\grad X^t_-))V] \rV_{H^{3/2}(\Sigma_-)}) \\
&\qquad +\lV \grad^\perp \cdot [(\grad \uX^t_- - \grad X^t_-)V] \rV_{H^{3/2}(\Sigma_-)} + \lV (\uX_\theta -  X_\theta)\cdot V \rV_{H^2(S^1)} )\\
&\leq C'( \lV \uX_- - X_- \rV_{H^{5/2}(\Sigma_-)}\lV V \rV_{H^{5/2}(\Sigma_-)} +\lV \uX- X \rV_{H^3(S^1)}\lV V \rV_{H^2(S^1)}) \\
&\leq C'' \lV X- \uX\rV_{H^3(S^1)} \lV f\rV_{H^3(S^1)}.
\end{align}
Combining this with \eqref{cNdifference} and again using the bound \(\lV V \rV_{H^2(S^1)} \leq C \lV f \rV_{H^3(S^1)}\), the result follows.
\end{proof}

Using the estimate of Lemma~\ref{normalDivCurlSys}, we are able to show the following bounds hold for the operators $\cH(\xi)$ and $\bE(\xi)$ defined in Proposition~\ref{bEDefns}.
\begin{proposition}\label{bEbounds}
For $\xi$ in $\sB^\bbox$, for integer $s$ with $0\leq s \leq k+1$, we have the estimates
\begin{align}
\lV \cH(\xi) f \rV_{H^s(S^1)} &\leq C(M) \lV f \rV_{H^s(S^1)},\\
\lV \bE(\xi) V \rV_{H^s(S^1)} &\leq C(M) \lV V \rV_{H^s(S^1)} ,\\
\lV (\bE(\xi))^{-1} V \rV_{H^s(S^1)} &\leq C(M) \lV V \rV_{H^s(S^1)}.
\end{align}
In the case $s\leq k$, we can replace $M$ in the right-hand sides above by $\lV \xi \rV_{\sH^{k-1}}$, and we also have
\begin{align}
\lV [\del_\theta,\bE(\xi)] V \rV_{H^s(S^1)} &\leq C(M) \lV V \rV_{H^s(S^1)} .
\end{align}

Given $\uxi$ in $\sB^\bbox$ as well, we have for integer $s$ with $0\leq s \leq k+1$
\begin{align}
\lV (\cH(\xi) - \cH(\uxi))f \rV_{H^s(S^1)} &\leq C \lV X - \uX \rV_{H^{s+1}(S^1)} \lV f \rV_{H^s(S^1)} ,\\
\lV (\bE(\xi) - \bE(\uxi))V \rV_{H^s(S^1)} &\leq C \lV X - \uX \rV_{H^{s+1}(S^1)} \lV V \rV_{H^s(S^1)} ,\\
\lV ((\bE(\xi))^{-1} - (\bE(\uxi))^{-1})V \rV_{H^s(S^1)} &\leq C \lV X - \uX \rV_{H^{s+1}(S^1)} \lV V \rV_{H^s(S^1)} .
\end{align}
In the case $s \leq k$,
\begin{align}
\lV ([\del_\theta,\bE(\xi)] - [\del_\theta,\bE(\uxi)])V \rV_{H^s(S^1)} &\leq C \lV X - \uX \rV_{H^{s+2}(S^1)} \lV V \rV_{H^s(S^1)} .
\end{align}
\end{proposition}





\subsection{Time derivative estimates}\label{timederivdefnsection}

In this section, we give the definitions and bounds for the various time derivative type maps acting on $\xi$ in $\sB^\bbox$, a class of time-independent states. 

Below we begin with basic vacuum-centric quantities first, defining $\dot h(\xi)$ and $\dot H(\xi)$, for example. Following this we define quantities living in the plasma region, such as $\dot \bU(\xi)$ and $\dot \bB(\xi)$. Afterwards, we focus on defining time derivative commutator maps such as $[\del_t;\bE](\xi)$, which we arrange to agree with $[\del_t,\bE(\xi)]$.
\subsubsection{Vacuum quantities}\label{VacuumQuantitiesTime}

\begin{definition}\label{dotHdefns}
Consider $\xi$ in $\sB^\bbox$ with corresponding vacuum region $\cV$ and interface $\Gamma$.

\begin{enumerate}[(i)]
\item We then define $\dot h(\xi)$ to be the unique solution $\dot h:\overline \cV \to \R^2$ to the system below, in which $h=h(\Gamma)$.
\begin{align}
(x \in \cV ) \qquad
&   \left\{
        \left.
        \begin{aligned}
            \grad \cdot \dot h &= 0 \ \\
            \grad^\perp \cdot \dot h & = 0 \, ,
        \end{aligned}
        \right.
    \right. \label{ExternalFieldDeriv1} \\
(x \in \Gamma)  \qquad 
& \ \bigg\{ \,  n \cdot \dot h = \frac{(U_\theta\circ X^{-1}) \cdot n}{|X_\theta \circ X^{-1}|} |h|-\sum^2_{i,j=1}(U_i\circ X^{-1}) n_j  \del_i h_j, \label{ExternalFieldDeriv2} \\
(x \in \mathcal{W} )  \qquad & \ \big\{ \,  n \cdot \dot h = 0, \label{ExternalFieldDeriv3} \\
(i=1,2) \qquad & \int_{\cW_i} h \cdot d \vec r = 0.\label{ExternalFieldDeriv4}
\end{align}
In the directed line integral in \eqref{ExternalFieldDeriv4}, the path runs left to right along $\cW_1$ for $i=1$, and the path runs counter-clockwise around $\cW_2$ for $i=2$. Moreover, whenever we consider $\del_i h$ on $\Gamma$ for $i=1,2$, we refer to the limits taken from within $\cV$.

\item Additionally, we define
\begin{align}
&& \dot H(\xi)(\theta) &= \dot h(X(\theta)) + \sum^2_{i=1} U_i(\theta) (\del_i h)(X(\theta)) & (\theta \in S^1).
\end{align}
\end{enumerate}
\end{definition}

\begin{remark}
\hfill
\begin{enumerate}[(i)]
\item To derive the system above for defining $\dot h$, we used in particular that, if $X_t=U$, one has $N \cdot H_t = -N_t\cdot H$ and $N_t = -\bigtau (U_\theta \cdot N)/|X_\theta|$. 
 The definition of $\dot H$ is chosen to match with the identity $H_t= h_t\circ X + (u\cdot \grad h)\circ X$.
 
 \item Existence of a solution to the system \eqref{ExternalFieldDeriv1}--\eqref{ExternalFieldDeriv4} is straightforward. One verifies the compatibility condition $\int_\Gamma n\cdot \dot h \,dS=0$ without too much difficulty by using the identity $( v \cdot \grad h )\cdot w = - (w^\perp \cdot \grad h) \cdot v^\perp$ which holds for the divergence-free, curl-free field $h$. Indeed, one calculates
 \begin{align}
 \int_\Gamma n\cdot \dot h\, dS = \int^\pi_{-\pi}U_\theta \cdot H^\perp - (U \cdot \grad h \circ X) \cdot X^\perp_\theta\, d\theta = \int^\pi_{-\pi}U_\theta \cdot H^\perp - H_\theta \cdot U^\perp \, d\theta = 0 .
 \end{align}
 \end{enumerate}
\end{remark}

\begin{proposition}\label{dotHEstimateProp}
Consider integer $m\geq 0$. For $\xi$ and $\uxi$ in $\sB^\bbox(\delta)$ for $\delta\geq 0$, 
we have the estimates
\begin{align}
\lV \dot{h}(\xi) \rV_{H^k(\Gamma, \delta, m)} &\leq C(M) , \label{dotHest1}\\
\lV \dot{H}(\xi) \rV_{H^k(S^1,\delta,m)} &\leq C(M) , \label{dotHest2}\\ 
\lV \dot{H}(\xi)-\dot{H}(\uxi) \rV_{H^2(S^1,\delta,m)} &\leq C \lV \xi - \uxi \rV_{\sH^2} . \label{dotHest3}
\end{align}
\end{proposition}
\begin{proof}
Let us first prove \eqref{dotHest1} and \eqref{dotHest2} when $\delta>0$. 

We define $X_+=\cE_\delta(X)$ and $\dot H_+$ in $\Sigma^\circ_+(\delta)$ by $\dot H_+ = \dot h \circ X_+$. We then find that $\dot H_+$ solves the system of Proposition~\ref{LagrVacDivCurlEst} in the case that $\Sigma^*_+=\Sigma^\circ_+(\delta)$, $\rho=\varpi=0$, and $\eta = (U_\theta \cdot H^\perp -H_\theta \cdot U^\perp)/|X_\theta|$. Thus, by using Proposition~\ref{LagrVacDivCurlEst} and (iv) of Proposition~\ref{standardWeightedSobNormEstims}, for some $m'$ and $m''$ we find
\begin{align}
\lV \dot H_+ \rV_{H^{k+1/2}(\Sigma^\circ_+,\delta,m)} &\leq C(M) \left(\lV U_\theta \cdot H^\perp\rV_{H^k(S^1,\delta,m')}+\lV H_\theta \cdot U^\perp \rV_{H^k(S^1,\delta,m')} \right)\\
&\leq C(M)
\lV U \rV_{H^{k+1}(S^1)}\lV H \rV_{H^{k+1}(S^1,\delta,m'')} \\
&\leq C(M),
\end{align}
applying Proposition~\ref{hEstimateProp} in the last step. From this we easily conclude \eqref{dotHest1} and \eqref{dotHest2} in the case $\delta>0$.

Let us simply remark that a similar argument to the proof of (ii) of Proposition~\ref{hEstimateProp} verifies the Lipschitz estimate \eqref{dotHest3} for $\delta>0$. With this, using a method similar to the proof of Lemma~\ref{sNiExt}, one deduces the estimates \eqref{dotHest1}--\eqref{dotHest3} hold for $\delta=0$ as well.
\end{proof}

\begin{corollary}
We have a bounded map
\begin{align}
\dot{H} &: \sB^\bbox \to H^k(S^1)
.
\end{align}
\end{corollary}

Now we work our way towards the time derivative type map corresponding to the exterior pressure gradient trace $P^+_\grad$. To help clarify its definition, which we provide shortly, let us recall from Definition~\ref{ExtPressureDefnBd} the exterior pressure function $p_+ = \half(1-\bigh_+)|h|^2$ and observe that it satisfies $\Delta p_+ = |\grad h|^2$ in $\cV$.
Let us attempt to derive an expression for $\del_t p_+$. For a moving domain, one finds that (i)
\begin{align}
& & (\del_t + u \cdot \grad )\bigh_+ |h|^2&= 2 h_t \cdot h  &   (x \in \Gamma) ,
\end{align}
and (ii) that $\del_t \bigh_+ |h|^2$ is harmonic in $\Omega_+$. Thus
\begin{align}
& & \half \del_t \bigh_+ |h|^2 &=\bigh_+( h_t \cdot h - \half u \cdot \grad \bigh_+ |h|^2) &   (x \in \cV) .
\end{align}
Using this together with the fact that $\half \del_t |h|^2 = h_t\cdot h$ and the definition of $p_+$, we derive the expression for the time derivative of $p_+$ which is given in the expression for $\dot p_+$ below. The definition for $\dot P^+_\grad$, the time derivative of $P^+_\grad =( \grad p_+) \circ X$, can also be derived without much difficulty.
\begin{definition}\label{dotPplusdefn}
For $\xi$ in $\sB^\bbox$, consider $\dot{p}_+:\overline{\cV}\to\R$ as follows, where $h=h(\Gamma)$ and $\dot h = \dot h(\xi)$:
\begin{align}
& & \dot p_+(x) & = (1-\bigh_+)(h\cdot \dot h )(x) + \half \bigh_+\left(U( X^{-1}(x))\cdot \grad \bigh_+ |h|^2(x) \right)  &   (x\in \overline \cV).
\end{align}
We define $\dot{P}^+_\grad(\xi):S^1\to\R^2$ by the following, in which $p_+=\half (1-\bigh_+)|h|^2$:
\begin{align}
& & \big( \dot{P}^+_\grad(\xi)\big)_i (\theta) &= (\del_{x_i} \dot{p}_+)(X(\theta))+ \sum^2_{j=1} U_j(\theta) ( \del^2_{x_i x_j} p_+ )(X(\theta))   &   (\theta \in S^1 , \ i=1,2) .
\end{align}
\end{definition}

\begin{proposition}\label{dotPplusSysProp}
Consider integer $m\geq 0$. For $\xi$ in $\sB^\bbox$, 
we have the estimate
\begin{align}
\lV \dot{P}^+_\grad\rV_{H^k(S^1)} \leq C(M) .
\end{align}
If $\xi$ and $\uxi$ are in $\sB^\bbox(\delta)$ for $\delta>0$,
\begin{align}
\lV  \dot{P}^+_\grad(\xi) - \dot{P}^+_\grad(\uxi) \rV_{H^2(S^1)} &\leq C \lV \xi - \uxi \rV_{\sH^2} .
\end{align}
\end{proposition}
\begin{proof}
Similar to the proof of Proposition~\ref{PplusbdProp}, the estimates are proved with the use of Propositions~\ref{1stWeightedEstimate} and~\ref{LagrVacDivCurlEst}, in the end reducing to the upper and Lipschitz bounds we have proved for $h(\Gamma)$, $\dot h(\xi)$, $H(\xi)$, and $\dot H(\xi)$.
\end{proof}

\subsubsection{Estimates in the plasma region}\label{PlasmaQuantitiesTime}

Now we pivot to discuss time derivative type maps for the velocity and interior magnetic field.
We define $\dot{\bU}(\xi)$ and $\dot{\bB}(\xi)$ as the unique solutions, guaranteed to exist as a consequence of Lemma~\ref{GeneralDivCurlSysLem}, to the corresponding div-curl systems below, for $\bW$ as described in Proposition~\ref{dotUdotBdefns}:
\begin{align}
\begin{aligned}
(a \in \Sigma) \qquad &
\left\{
    \begin{aligned}
        \grad \cdot (\, \cof ({\bM}) \dot{\bU})  &= - \grad \cdot (\cof(\grad {\bU}^t+\grad \bW^t){\bU} ) \\
        \grad^\perp\cdot (\, {\bM} \dot{\bU})    &= \bm{\sigma} \bj_\theta - \grad^\perp \cdot ( \grad \bW^t {\bU}  )  \, \Lcm
    \end{aligned}
\right. \\
(\psi = 0)  \qquad & \ \Big\{ \, \bm{N} \cdot \dot{\bU} = {\bU} \cdot \frac{\bX_\theta}{|\bX_\theta|}({\bU} + \bW) \cdot \bm{N} , \\
(\psi = -1) \qquad & \ \Big\{ \, \bm{N} \cdot \dot{\bU} = 0 , \qquad\quad \int^\pi_{-\pi}  \dot{\bU} \cdot {\bX}_\theta \, d\theta = -\int^\pi_{-\pi}  |{\bU}|^2 + {\bU} \cdot \bW \, d\theta , \\
(a \in \Sigma) \qquad &
\left\{
    \begin{aligned}
        \grad \cdot (\, \cof (\bM) \dot{\bB})  &= - \grad \cdot (\cof(\grad {\bU}^t+\grad {\bB}^t){\bB} ) \\
        \grad^\perp\cdot (\, \bM \dot{\bB})    &= \bm{\sigma} \bom_\theta + \grad^\perp \cdot ( (\grad {\bU}^t - \grad \bW^t ) {\bB}  )  \, \Lcm
    \end{aligned}
\right. \\
(\psi = 0)  \qquad & \ \big\{ \, \bm{N} \cdot \dot{\bB} = ({\bU} + \bW) \cdot {\bB}^\perp , \\
(\psi = -1) \qquad & \ \Big\{ \, \bm{N} \cdot \dot{\bB} = 0 , \qquad\quad \int^\pi_{-\pi}  \dot{\bB} \cdot {\bX}_\theta \, d\theta = -\int^\pi_{-\pi} {\bB}\cdot({\bU} + \bW) \, d\theta .
\end{aligned}
\end{align}
\begin{proposition}\label{dotUdotBdefns}
For $\xi \in \sB$, at each time $t$ in $[0,T]$ we define $\dot{\bU}(\xi)(t):\Sigma\to\R^2$ and $\dot{\bB}(\xi)(t):\Sigma\to\R^2$ to be the unique solutions to the systems above with $\bU= \bU(\xi)$, $\bB= \bB(\xi)$, and $\bX= \bX(\xi)$ of Proposition~\ref{SolnMainDivCurlProp}, solving the system \eqref{AugmentedInteriorSystemL}, and we have $\bM=\grad \bX^t$, $\bm{N}=n\circ \bX$, and $\bW = w_\Omega\circ \bX$ for $w_\Omega:\Omega\to\R^2$ defined by \eqref{ulwformula}.

Then $\dot{\bU}(\xi)$ and $\dot{\bB}(\xi)$ obey the bound below.
\begin{align}
\sup_{t}\lV \dot{\bU}(\xi)\rV_{H^{k+1/2}(\Sigma)}+\sup_{t}\lV \dot{\bB}(\xi)\rV_{H^{k+1/2}(\Sigma)} &\leq C(M) .
\end{align}
\end{proposition}
\begin{proof}
The proof follows from Lemma~\ref{GeneralDivCurlSysLem} and Proposition~\ref{SolnMainDivCurlProp}.
\end{proof}
\begin{proposition}
For all $\xi$ and $\uxi$ in $\sB$ the following holds.
\begin{align}
\sup_{t}\lV \dot \bU(\xi) - \dot \bU(\uxi) \rV_{H^{5/2}(\Sigma)} + \sup_{t}\lV \dot \bB(\xi) - \dot \bB(\uxi) \rV_{H^{5/2}(\Sigma)}
\leq C \sup_{t \in [0,T] } \lV \xi(t) - \uxi(t) \rV_{\sH^2} .
\end{align}
\end{proposition}
\begin{proof}
These bounds are verified by writing out the div-curl systems for the differences $\dot \bU(\xi)-\dot\bU(\uxi) $ and $\dot \bB(\xi)-\dot\bB(\uxi)$ and applying Lemma~\ref{GeneralDivCurlSysLem}.
\end{proof}

Now we analogously define the map $\dot P^-_\grad$ in the proposition below, describing it as the solution to the system that $\del_t \bP^-_\grad$ must satisfy, and record related bounds for later use.
\begin{definition}\label{dotPintdefn}
For $\xi$ in $\sB$, taking the various quantities as in Proposition~\ref{dotUdotBdefns} as well as \(\bP^-_\grad=\bP^-_\grad(\xi)\), \( h =  h(\Gamma)\), and \(\dot h = \dot h(\xi)\), for each $t$ in $[0,T]$ we define $\dot{\bP}^-_\grad (\xi)(t):\Sigma\to\R^2$ as the solution, whose existence is implied by Lemma~\ref{GeneralDivCurlSysLem}, to the following system:
\begin{align}
\label{IntDotPressureGradSystem}
\begin{aligned}
(a \in \Sigma) \qquad &
\left\{
    \begin{aligned}
        \grad \cdot (\, \cof (\bM) \dot{\bP}^-_\grad )  &= 2 \grad \cdot ( \, \cof (\grad \bU^t) \dot{\bU} - \cof(\grad \bB^t) \dot{\bB} ) \\
                                                        & \quad -\grad \cdot (\, \cof (\grad \bU^t + \grad \bW^t) \bP^-_\grad )  \\
        \grad^\perp\cdot (\, \bM \dot{\bP}^-_\grad  )    &= - \grad^\perp\cdot (\, (\grad \bU^t + \grad \bW^t) \bP^-_\grad ) \, \Lcm
    \end{aligned}
\right. \\
(\psi = 0)  \qquad & \ \big\{ \, \bX_\theta \cdot \dot{\bP}^-_\grad  = -(\bU_\theta + \bW_\theta) \cdot \bP^-_\grad , \\
(\psi = -1) \qquad & \ \bigg\{ \, \bm{N} \cdot \dot{\bP}^-_\grad  =\left(\del_{x_2}  \left[ \bigh_-( \dot h \cdot h - \half u \cdot \grad \bigh_- |h|^2) \right]\right)(t,\bX).
\end{aligned}
\end{align}
We then define
\begin{align}
&& \dot P^-_\grad(\xi)(\theta) &= \dot \bP^-_\grad(\theta,0) &(\theta\in S^1).
\end{align}
\end{definition}

\begin{proposition}\label{PminusDotBds}
For $\xi$ in $\sB$, we have
\begin{align}\label{dotPmbound}
\sup_{t}\lV\dot P^-_\grad(\xi) \rV_{H^k(S^1)}+ \sup_{t}\lV\dot{\bP}^-_\grad(\xi) \rV_{H^{k+1/2}(\Sigma)}
\leq C(M) ,
\end{align}
and, given $\uxi$ also in $\sB$,
\begin{align}
\sup_{t}\lV \dot P^-_\grad(\xi) - \dot P^-_\grad(\uxi) \rV_{H^2(S^1)} 
+
\sup_{t}\lV \dot \bP^-_\grad(\xi) - \dot \bP^-_\grad(\uxi) \rV_{H^{5/2}(\Sigma)} 
&\leq C \sup_{t} \lV \xi - \uxi \rV_{\sH^2} .
\end{align}
\end{proposition}
\begin{proof}
Regarding \eqref{dotPmbound}, a suitable bound on $\dot \bP^-_\grad(\xi)$ can be shown by applying Lemma~\ref{GeneralDivCurlSysLem}. The Lipschitz bound is proved similarly by considering the system for $\dot \bP^-_\grad(\xi)-\dot \bP^-_\grad(\uxi)$.
\end{proof}

\subsubsection{Commutator maps}\label{commutatormapssection}

Now we take certain operator-valued maps, say $P(\xi)$, defined for $\xi$ in $\sB^\bbox$, and define a corresponding map $[\del_t;P](\xi)$ which will agree with $[\del_t, P(\xi(t))]$ for time-dependent $\xi$ in $\sB$.

{\subsubsection*{Defining $[\del_t;\cH](\xi)$ and $[\del_t;\bE](\xi)$}}

Recall the definitions of $\bE(\xi)$ and $\cH(\xi)$ given by Proposition~\ref{bEDefns}. In particular, $\cH(\xi)$ is essentially the Hilbert transform associated to the curve $\Gamma$. Using well-understood techniques, one is able to establish an integral kernel formula for this operator. We will see the use of integral kernel operators yield simple definitions for the commutator maps $[\del_t;\cH](\xi)$ and $[\del_t;\bE](\xi)$ (see Proposition~\ref{bEcommBounds}). We have the kernel
\begin{align}\label{HilbKernelDef}
\cK(x) &=\frac{1}{2\pi}\frac{(\sin (x_1) ,\sinh (x_2)) }{\cosh(x_2)-\cos(x_1)},\\ \intertext{which can also be represented in complex form by}
\cK(x)&=\frac{1}{2\pi}\overline{\cot\left(\frac{x_1+ix_2}{2}\right)}. \label{HilbKernelComplexVers}
\end{align}
One is able to show that in fact
\begin{align}
\cH(\xi)f(\theta)&=\pv\int^\pi_{-\pi} \del_\vartheta   X(\vartheta) \cdot \cK(  X(\theta)-  X(\vartheta))  \cP f(\vartheta)\,d\vartheta, \label{cHaltStdDefn}\\ \intertext{where}
  \cP f(\theta) &= f(\theta) - \frac{1}{|  \Gamma|}\int^\pi_{-\pi} f(\vartheta)|\del_\vartheta   X(\vartheta)|\,d\vartheta.
\end{align}
By using the above formula, one can use elementary techniques for bounding integral kernel operators to prove estimates such as those given in Proposition~\ref{bEbounds} for $\cH(\xi)$. To prove such bounds, this provides a fairly standard method for cases in which $\Gamma$ is uniformly chord-arc, alternative to the method of using div-curl estimates, as done in Section~\ref{unpinchedestimates}.

It is cleaner to define the commutator map $[\del_t; \cH](\xi)$ with an integral kernel formula as opposed to the solution to a div-curl problem, so we now explain how to adapt the integral kernel approach to our setting, where we must deal with splash curves as opposed to chord-arc curves.

In the following we employ the ``square root trick'' often used in splash constructions (such as in \cite{splash2}) to understand and control the corresponding $[\del_t;\cH](\xi)$ operator. In the trick, one employs a conformal map $z\mapsto \cO(z)$, which behaves morally like $z\mapsto \sqrt{z}$, to split open the vacuum domain and turn a given splash curve into a chord-arc curve. With this, one is able to convert the formula \eqref{cHaltStdDefn} to an analogous one associated to a uniformly chord-arc curve $\tilde \Gamma$ in place of $\Gamma$. After expressing the operator in this way, it can be estimated by elementary methods.

Recall the point $z_\star$ located inside the wall $\cW_2$. We define
\begin{align}\label{cOdefn}
\cO(z)=\sqrt{\tan\left(\frac{z-z_\star}{2}\right)},
\end{align}
where the branch cut is taken to emerge from $z_\star$, forming the ray which runs directly out of the cavity formed during the splash state, tangent to the two kissing arcs at the splash point and extending to infinity from there, so that it never enters $\Omega$.

\begin{remark}
For a given $\xi$ with corresponding $X$ and $\Gamma$, let us define
\begin{align}
&&\tilde X(\theta) &= \cO(X(\theta)) & (\theta\in S^1),\label{tildeXdefn}\\
&&\tilde \Gamma &= \cO(\Gamma). &\label{tildeGammadefn}
\end{align}
Since conformal maps preserve the property of being harmonic, in addition to the formula \eqref{cHaltStdDefn}, we find the definition of $\cH(\xi)$ given in Proposition~\ref{bEDefns} coincides with the formula below, for $\cK(x)$ as in \eqref{HilbKernelDef}:
\begin{align}
\cH(\xi)f(\theta)&=\pv\int^\pi_{-\pi} \del_\vartheta \tilde X(\vartheta) \cdot \cK(\tilde X(\theta)-\tilde X(\vartheta)) \tilde\cP f(\vartheta)\,d\vartheta, \label{cHaltDefn}\\ \intertext{where}
\tilde \cP f(\theta) &= f(\theta) - \frac{1}{|\tilde \Gamma|}\int^\pi_{-\pi} f(\vartheta)|\del_\vartheta \tilde X(\vartheta)|\,d\vartheta.
\end{align}
\end{remark}

In light of the above remark, we find the definitions for $[\del_t;\cH](\xi)$ and $[\del_t;\bE](\xi)$ in the following proposition are sufficient for our purposes.

\begin{proposition}\label{bEcommBounds}
Consider $\xi$ in $\sB^\bbox$ and $\cO(z)$ as in \eqref{cOdefn}. Let $\tilde X$ and  $\tilde\Gamma$ be as defined in \eqref{tildeXdefn} and \eqref{tildeGammadefn}, and take $\tilde U_i(\theta) =  U(\theta)\cdot \grad \cO_i(X(\theta))$ for $i=1,2$. We then define $[\del_t;\cH](\xi)$ and $[\del_t;\bE](\xi)$ by the following formulas:
\begin{align}\label{commcHform1}
& & [\del_t;\cH](\xi)f(\theta)&=\pv\int^\pi_{-\pi} \del_\vartheta \tilde U(\vartheta) \cdot \cK(\tilde X(\theta)-\tilde X(\vartheta))\tilde  \cP f(\vartheta)\,d\vartheta   &\\
&&& \quad+\pv\int^\pi_{-\pi} \del_\vartheta \tilde X_i(\vartheta) (\tilde U_j(\theta)-\tilde U_j(\vartheta) )(\del_i\cK_j)(\tilde X(\theta)-\tilde X(\vartheta)) \tilde \cP f(\vartheta)\,d\vartheta & (\theta \in S^1 ) ,\label{dtcHcomm}\\
&&[\del_t;\bE](\xi)V &= [\del_t;\cH](\xi)(N\cdot V) e_1 - \frac{U_\theta \cdot N}{|X_\theta|} \bE(\xi) V^\perp . &   (\theta \in S^1 ) . 
\label{CommDtbEDefn}
\end{align}
Then we have a bounded map for integer $s$ with $0\leq s \leq k$,
\begin{align}
[\del_t;\bE](\cdot) : \sB^\bbox \to B( H^s(S^1), H^s(S^1) ) ,
\end{align}
satisfying for $\xi$ in $\sB^\bbox$
\begin{align}\label{commUpperBd}
\lV [\del_t;\bE](\xi) V \rV_{H^s(S^1)}
&\leq C(M) \lV V \rV_{H^s(S^1)} , \\ \intertext{and for $\uxi$ in $\sB^\bbox$,}
\lV( [\del_t;\bE](\xi)-[\del_t;\bE](\uxi)) V \rV_{H^2(S^1)}
&\leq C \lV\xi-\uxi\rV_{\sH^2}\lV V \rV_{H^2(S^1)} .\label {commLipBd}
\end{align}
\end{proposition}
\begin{proof}
Since the curve $\tilde \Gamma$ is uniformly chord-arc, the above bounds can be proved with basic techniques. Rather than giving the complete proof, in the interest of brevity, we simply take the first term appearing the right-hand side of \eqref{commcHform1} and put it in a form which suggests standard, well-known methods estimate the $H^s(S^1)$ norm of the result in terms of the $H^s(S^1)$ norm of $f$. First, observe that the first term can be rewritten the following way, where we use the complex form \eqref{HilbKernelComplexVers} for the kernel and treat vectors in $\R^2$ as elements of $\C$:
\begin{align}
\cH_1(\xi) f(\theta)
=&\frac{1}{2\pi}\operatorname{Re}\left(\pv \int^\pi_{-\pi}  \cot \left(\frac{\tilde X(\theta)-\tilde X(\vartheta)}{2}
\right)
\del_\vartheta \tilde U(\vartheta)\tilde\cP f(\vartheta)\,d\vartheta
\right) \\
=&
\frac{1}{\pi}\operatorname{Re}\Bigg(\pv
\int^\pi_{-\pi} \frac{1}{\tilde X(\theta)-\tilde X(\vartheta)}\del_\vartheta \tilde U(\vartheta)\tilde\cP f(\vartheta) \, d\vartheta
 \\
&\qquad\qquad+\int^\pi_{-\pi}\cA(\tilde X(\theta)-\tilde X(\vartheta)) \del_\vartheta \tilde U(\vartheta) \tilde\cP f(\vartheta) \, d\vartheta \Bigg) ,
\end{align}
where $\cA(z)$ is analytic, thus leading to harmless contributions from the corresponding term above. This reduces most of the analysis to controlling the leading order part, for which we notice
\begin{align}
\frac{1}{\pi}\operatorname{Re}\left(\pv
\int^\pi_{-\pi} \frac{1}{\tilde X(\theta)-\tilde X(\vartheta)}\del_\vartheta \tilde U(\vartheta)\tilde\cP f(\vartheta) \, d\vartheta \right)
=
\frac{1}{\pi}\operatorname{Re}\left(\pv \int^\pi_{-\pi} \frac{1}{\theta-\vartheta}\frac{\del_\vartheta \tilde U(\vartheta)\tilde\cP f(\vartheta)}{\cI_{\tilde X}(\theta,\vartheta)}\,d\vartheta \right) ,
\end{align}
where
\begin{align}
\cI_{\tilde X}(\theta,\vartheta) = \int^1_0 (\del_\theta \tilde X)(\tau \theta + (1-\tau)\vartheta)\,d\tau.
\end{align}
Note that by using the square root trick to get a formula with $\tilde X$ in the kernel in place of $X$, where $\tilde X$ traces out a chord-arc curve, we have afforded ourselves the ability to divide by the harmless term $\cI_{\tilde X}(\theta,\vartheta)$ above, resulting in a leading order term resembling the basic Hilbert transform. The standard approach verifies
\begin{align}
\lV  \cH_1(\xi)f \rV_{H^s(S^1)} \leq C(M) \lV f \rV_{H^s(S^1)}.
\end{align}
Similar techniques yield a bound on the second term appearing in the right-hand side of \eqref{commcHform1}, thus giving us the bound \eqref{commUpperBd}. The Lipschitz bound \eqref{commLipBd} is proven similarly.
\end{proof}



{\subsubsection*{Defining $[\del_t;\cN_\pm](\xi)$ and $[\del_t;\Nres](\xi)$}}

Just as it is key that $\Nres(\xi)$ satisfies the cancellation property of Proposition~\ref{DtoNCancellation}, it is important that $[\del_t;\Nres](\xi)$ satisfies a cancellation property, which we verify in Proposition~\ref{CommutatorCancelBd}. To ensure we get this in a relatively clean way, we define the commutators $[\del_t;\cN_\pm](\xi)$ with the following method.

For many operators we handle that depend on $X$ alone, say $P(X)$ (this could represent $\bE(\xi)$, $\cH(\xi)$, or $\cN_\pm(\xi)$, for example), we find that $[\del_t,P(X)]$ can be realized as the Frech\'{e}t derivative with respect to $X$ in the direction of $U$, i.e.
\begin{align}
[\del_t,P(X)] = DP(X)[U] = \lim_{\epsilon\to 0} \frac{P(X+\epsilon U)-P(X)}{\epsilon} ,
\end{align}
taking the limit in an operator space. 

Though we typically denote the operators $\cN_\pm$ and $\Nres$ as functions of $\xi$, defined in Definition~\ref{DirichToNeumDefs}, let us note that they depend only on $X$. In this subsection, we will, with a temporary, slight inconsistency of notational choice, instead denote these operators by $\cN_\pm(X)$ and $\Nres(X)$. Our definition of the corresponding time derivative commutator operators can then be realized as follows.
\begin{definition}\label{DtNCommDefns}
For $\xi$ in $\sB^\wbox$, where $D\cN_\pm(X)[U]$ denotes the Frech\'{e}t derivative of $\cN_\pm(X)$ in the direction of $U$, we define
\begin{align}
[\del_t;\cN_\pm](\xi) &= D \cN_\pm(X)[U] ,\label{opfmla3}\\
[\del_t;\Nres](\xi) &= [\del_t;\cN_+](\xi)+[\del_t;\cN_-](\xi) . \label{DtNCommDef2}
\end{align}
\end{definition}

Proposition~\ref{existFrech} below clarifies the convergence of the derivatives $D \cN_\pm(X)[U]$, making the above definition rigorous.

\begin{remark}\label{NrescommDefnRemark}
\hfill
\begin{enumerate}[(i)]
\item 
In the proofs of Propositions~\ref{NresLipBd} and~\ref{LipDirichNeumBd}, we discuss how $\cN_+(\xi)f$ and $\cN_-(\xi)f$ can be represented in terms of solutions to corresponding div-curl systems, such as \eqref{DtNDivCurl}. Similarly, one can describe $[\del_t;\cN_+](\xi)f$ and $[\del_t;\cN_-](\xi)f$ in terms of solutions to certain div-curl systems. We use Definition~\ref{DtNCommDefns} to define $[\del_t;\cN_\pm](\xi)f$ instead simply to shorten our presentation.

\item Throughout the remainder of this section, since nearly all functions considered have domain $S^1$, we frequently drop the $S^1$ from our notation, using $H^s$ to refer to $H^s(S^1)$, for instance.
\end{enumerate}
\end{remark}

\begin{proposition}\label{existFrech}
For $\xi$ in $\sB^\wbox$ and integer $s$ with $0\leq s \leq k$, we have convergence of the Frech\'{e}t derivative
\begin{align}
D \cN_\pm(X)[U]\in B(H^{s+1},H^s),
\end{align}
that is, there exists an element of $B(H^{s+1},H^s)$, which we denote by $D\cN_\pm(X)[U]$, such that
\begin{align}
\lim_{\epsilon\to 0}\left\lV\frac{\cN_\pm(X+\epsilon U)-\cN_\pm(X)}{\epsilon}-D\cN_\pm(X)[U]\right\rV_{B(H^{s+1},H^s)} = 0.
\end{align}
\end{proposition}
\begin{proof}
The statement of the proposition follows from well-known results on the existence of the shape derivative for Dirichlet-to-Neumann operators for domains with non-self-intersecting boundaries of reasonable Sobolev regularity, such as $\Gamma$ associated to $\xi$ in $\sB^\wbox$. For examples of results on the existence of the shape derivative for Dirichlet-to-Neumann operators, we refer the reader to \cite{lanneswaterwaves}.
\end{proof}

\subsubsection*{The main cancellation bound for $[\del_t;\Nres](\xi)$}

Now we proceed with the task of proving the cancellation bound
\begin{align}\label{commutatorIdent}
            \lV [\del_t;\Nres](\xi) f \rV_{H^s(S^1)} \leq C(M) \lV f \rV_{H^s (S^1,\delta,m)} .
        \end{align}
Before we proceed, recall how the identity
\begin{align}
\slp \nres = 2 \dcal T
\end{align}
played a key role in the proof of Proposition~\ref{DtoNCancellation}, the cancellation bound for $\nres$. We the analogous Lagrangian identity is given in the lemma below, with \eqref{keynresIdent1copy}, where $\cS(X)$ and $\cT(X)$ are the Lagrangian versions of $\slp$ and $\dcal T$, respectively, defined by
    \begin{align}
    \cS(X) f (\theta) &= \int^\pi_{-\pi} G(X(\theta)-X(\vartheta)) f (\vartheta) \left| \del_\vartheta X(\vartheta) \right| \, d\vartheta &   (\theta &\in S^1 ) ,\\
    \cT(X) f (\theta) &= -\int^\pi_{-\pi} \left( \del_\vartheta X(\vartheta)\right)^\perp \cdot (\grad G)(X(\theta)-X(\vartheta)) f (\vartheta) \,d\vartheta  &   (\theta &\in S^1 ) .
    \end{align}
Above, $G(x)$ is the Newtonian potential for $S^1 \times \R$, given in \eqref{newtpotdefn}. The proof identity of the lemma below follows from Lemma~\ref{keycancellationident}.
\begin{lemma}\label{LagrangianIdentVers}
For $\xi$ in $\sB^\wbox$ and $\alpha>0$,
\begin{align}\label{keynresIdent1copy}
&& \cS(X) \Nres(X)f &= 2 \cT(X)f &(f\in C^{0,\alpha}).
\end{align}
\end{lemma}

Ultimately, to get our target bound \eqref{commutatorIdent}, we want to establish an identity analogous to \eqref{keynresIdent1copy} for $[\del_t;\Nres](\xi)$. For the true time derivative commutators $[\del_t,\Nres(X)]$, $[\del_t,\cS(X)]$, and $[\del_t,\cT(X)]$ corresponding to time-dependent $\xi$ in $\sB$, for which we have $X_t=U$, the identity \eqref{keynresIdent1copy} implies
\begin{align}\label{commcancelident}
\cS(X)[\del_t,\Nres(X)] &= -[\del_t,\cS(X)]\Nres(X)+2[\del_t,\cT(X)] .
\end{align}
In Corollary~\ref{commcancelidentCor}, we will verify the analogous identity for $\xi$ in $\sB^\wbox$, which is
\begin{align}\label{indirectCancFmla}
\cS(X)[\del_t;\Nres](\xi)=-[\del_t;\cS](\xi)\Nres(X) +  2[\del_t;\cT](\xi),
\end{align}
where $[\del_t;\cS](\xi)$ and $[\del_t;\cT](\xi)$ are defined below.

\begin{definition}\label{cScTcommDefns}
For $\xi$ in $\sB^\wbox$, where $G(x)$ is given by \eqref{newtpotdefn}, we define
    \begin{align}
   [\del_t;\cS](\xi) f(\theta) &= \cS (U_\theta \cdot X_\theta f)(\theta) + \int^\pi_{-\pi} ( U(\theta)-U(\vartheta)) \cdot (\grad G)(X(\theta)-X(\vartheta))\left| \del_\vartheta X(\vartheta) \right| f(\vartheta) \, d\vartheta & (\theta\in S^1), \\
[\del_t;\cT](\xi) f(\theta) &= -\frac{1}{2\pi} \operatorname{Im} \left(
\int \frac{d}{d\vartheta} \left( (U(\theta)-U(\vartheta))\cot \left( \frac{X(\theta)-X(\vartheta)}{2} \right) \right) f(\vartheta) \, d\vartheta
\right)  & (\theta\in S^1).
\end{align}
\end{definition}
\begin{remark}
The definitions above coincide with formal computations of the Frech\'{e}t derivatives of $\cS(X)$ and $\cT(X)$ in the direction of $U$. To see this for $[\del_t;\cT](\xi)f(\theta)$, observe that $\cT(X) f(\theta)$ can be written in the form
\begin{align}\label{altcTident}
\cT(X) f(\theta) &= -\frac{1}{\pi} \operatorname{Im} \left(
\int \frac{d}{d\vartheta} \left( \log \sin \left( \frac{X(\theta)-X(\vartheta)}{2} \right) \right) f(\vartheta) \, d\vartheta
\right) .
\end{align}
\end{remark}

The following is easy to prove by basic, well-known methods, such as those discussed in the proof of Propsition~\ref{bEcommBounds}.
\begin{proposition}
For $\xi$ in $\sB^\wbox$ and integer $s$ with $0\leq s \leq k+1$,
\begin{align}
[\del_t;\cS](\xi) & \in B(H^s,H^{s+1}), \\
[\del_t;\cT](\xi) & \in B(H^s,H^{s+1}).
\end{align}
\end{proposition}

Let us collect additional statements regarding the boundedness of various operators, all for non-self-intersecting interfaces. To prove the identity \eqref{indirectCancFmla}, we do not need explicit bounds on the operator norms or that they remain uniformly bounded as the pinch tends to zero.
\begin{lemma}\label{assortedBddOps}
For $\xi$ in $\sB^\wbox$ and integer $s$ with $0\leq s \leq k+1$, the following statements hold:
\begin{align}
\cS(X) &\in B(H^s,H^{s+1} ) ,\label{op1}\\
\cT(X) &\in B(H^s,H^{s+1}) ,\label{op2}\\
(\cS(X))^{-1} &\in B(H^{s+1},H^s) ,\label{op3}\\
\Nres(X) &\in B(H^s,H^s ) .\label{op4}\\
\end{align}
Moreover, the difference quotients for $D \cS(X)[U]$ and $D \cT(X)[U]$ converge in the $B(H^s,H^{s+1})$ norm, and those for $D\Nres(X)[U]$ converge in the $B(H^{s+1},H^s)$ norm, with
\begin{align}
D \cS(X)[U] &= [\del_t;\cS](\xi)\in B(H^s,H^{s+1}), \label{opfmla1}\\
D \cT(X)[U] &= [\del_t;\cT](\xi)\in B(H^s,H^{s+1}),\label{opfmla2} \\
D\Nres(X)[U] &=[\del_t;\Nres](\xi)\in B(H^{s+1},H^s). \label{op7}
\end{align}
\end{lemma}
\begin{proof}
Assertions \eqref{op1}--\eqref{op3} follow from classical results for non-self-intersecting interfaces, similar to Lemma~\ref{ChordArcOpBdLem}. The statement \eqref{op4} follows trivially from Proposition~\ref{DtoNCancellation}.

Standard computations of Frech\'{e}t derivatives yield \eqref{opfmla1}, \eqref{opfmla2}, and \eqref{op7} from the definitions of $\cS(X)$, $\cT(X)$, and $\Nres(X)$, the identity \eqref{altcTident}, and Proposition~\ref{existFrech}.
\end{proof}
We now calculate the Frech\'{e}t derivative of $(\cS\Nres)(X)$.
\begin{corollary}
For $\xi$ in $\sB^\wbox$ and integer $s$ with $0\leq s \leq k+1$, we have convergence of the Frech\'{e}t derivative
\begin{align}
D(\cS\Nres)(X)[U]\in B(H^s,H^{s+1}),
\end{align}
which satisfies the identity below, given $\alpha>0$:
\begin{align}\label{DSNIdent2}
& & D(\cS\Nres)(X)[U]f &= 2 [\del_t;\cT](\xi)f & (f\in C^{0,\alpha}) .
\end{align}
\end{corollary}
\begin{proof}
This follows from Lemma~\ref{LagrangianIdentVers} and \eqref{opfmla2}, where we use a density argument for the case $s=0$.
\end{proof}

Now we prove our cancellation identity, verifying an improvement over \eqref{op7}, which says that $D \Nres(X)[U]$ is in fact in $B(H^s,H^s)$.
\begin{corollary}\label{commcancelidentCor}
For $\xi$ in $\sB^\wbox$ and integer $s$ with $0\leq s \leq k+1$, we have 
$[\del_t;\Nres](\xi)$ in $B(H^s,H^s)$. Furthermore, we have the formula below, given $\alpha>0$:
\begin{align}\label{indirectCancFmla2}
&& \cS(X)[\del_t;\Nres](\xi)f&=-[\del_t;\cS](\xi)\Nres(\xi) f+  2[\del_t;\cT](\xi) f & (f \in C^{0,\alpha}).
\end{align}
\end{corollary}
\begin{proof}
Fix $f$ in $C^{0,\alpha}$. From standard rules for computing Frech\'{e}t derivatives of the composition of linear operators, we find difference quotients approximating $D(\cS\Nres)(X)[U]f$ converge pointwise to
\begin{align}
D(\cS\Nres)(X)[U]f = D\cS(X)[U] \Nres(X)f + \cS(X) D\Nres(X)[U]f ,\label{frechetcompositionderiv}
\end{align}
thanks to \eqref{op4} and \eqref{op1} as well as the convergence of the derivatives \eqref{opfmla1} and \eqref{op7}. By using the identities \eqref{DSNIdent2}, \eqref{opfmla1}, and \eqref{op7} in \eqref{frechetcompositionderiv} and rearranging terms, we get \eqref{indirectCancFmla2}.

Now note that if $f$ is in $H^s$ for $s\geq 1$, the right-hand side of \eqref{indirectCancFmla2} is in $H^{s+1}$. Thus we may apply $(\cS(X))^{-1}$ to both sides, giving us a formula for $[\del_t;\Nres](\xi)f$ which, when viewed with the help of Lemma~\ref{assortedBddOps}, justifies that $[\del_t;\Nres](\xi)$ is in $B(H^s,H^s)$ for $s\geq 1$. The case $s=0$ follows from a density argument.

\end{proof}



Now that the above corollary has given us that $[\del_t;\Nres](\xi)$ maintains $H^s$ regularity, it still remains to give an explicit estimate involving the weighted space $H^s(S^1,\delta,m)$ with a uniform constant in the right-hand side. This is the purpose of Proposition~\ref{CommutatorCancelBd}. Our proof requires the following lemma as the last ingredient.
\begin{lemma}\label{suppBddAwayCommLemma}
    For an integer $m\geq0$ the following holds. Consider $\xi$ in $\sB^\bbox(\delta)$ with $0<\delta\leq \delta_0$ and integer $s$ with $0\leq s \leq k+1$. For $I\subset S^1$ and $f:S^1 \to \R$ with $\operatorname{dist}(\supp (f),I)\geq d_0$, we have
    \begin{align}\label{commcancelbd0}
        \lV[\del_t;\cN_{\pm}](\xi)f\rV_{H^s(I)}\leq C(M) \lV f\rV_{H^s (S^1,\delta,m)}.
    \end{align}
\end{lemma}
\begin{proof}
Recall Remark~\ref{NrescommDefnRemark} (i). By deriving the expressions for $[\del_t;\cN_{\pm}](\xi)f$ in terms of solutions to associated div-curl systems, one can prove the above bound in a similar manner to the proof of Lemma~\ref{SuppsBddAwayLem}.
\end{proof}
Now we give the proof of the main cancellation bound for $[\del_t;\Nres](\xi)$, which is in the same spirit as the proof of Proposition~\ref{DtoNCancellation}.
\begin{proposition}\label{CommutatorCancelBd}
For an integer $m\geq0$ the following holds. Consider $\xi$ in $\sB^\bbox(\delta)$ with $0<\delta\leq \delta_0$ and integer $s$ with $0\leq s \leq k+1$. Then for the corresponding operator \([\del_t;\Nres](\xi)\) we have the estimate
\begin{align}
            \lV [\del_t;\Nres](\xi) f \rV_{H^s(S^1)} \leq C(M) \lV f \rV_{H^s (S^1,\delta,m)} .
        \end{align}
\end{proposition}
\begin{proof}
    We begin by assuming we are in the case that $f\in C^\infty_0([-\pi,0])$. For $i=1,2,3$, let us define $a_i, b_i, \Gamma_i, \Gamma^c_i$ exactly as we did in the proof of Proposition~\ref{DtoNCancellation} and define $I_i= X^{-1}(\Gamma_i)$. 
    
As a minor abuse of notation, we drop the arguments $\xi$ and $X$ from maps such as \([\del_t;\Nres](\xi)\), $\cS(X)$, etc. We then have
\begin{align}\label{commcancelbd1}
    \lV [\del_t;\Nres]f\rV_{H^s(S^1)}\leq \lV [\del_t;\Nres]f \rV_{H^s(I_2)}+\lV [\del_t;\cN_+] f \rV_{H^s(I^c_2)}+\lV [\del_t;\cN_-] f \rV_{H^s(I^c_2)}.
\end{align}
By Lemma~\ref{suppBddAwayCommLemma}, for some $m_1$ we get
\begin{align}\label{commcancelbd2}
    \lV [\del_t;\cN_\pm]f\rV_{H^s(I^c_2)}\leq C(M)\lV f\rV_{H^s (S^1,\delta,m_1)},
\end{align}
leaving us just with the task of bounding the term $\lV [\del_t;\Nres]f \rV_{H^s(I_2)}$.

We now use the same tailored chord-arc curve $\Gamma_\vardiamond$ as in the proof of Proposition~\ref{DtoNCancellation}, taking the same cutoff $\chi:\Gamma\to \R$ equal to one on $\Gamma_2$ with $\supp(\chi)\subset \Gamma_3$. For the single layer potential $\slp_\vardiamond$ associated to $\Gamma_\vardiamond$, using Lemma~\ref{curveSwapIdent}, we find
\begin{align}
    ([\del_t;\Nres] f)\circ X^{-1} =&
    \slp^{-1}_\vardiamond\bigg(\left.\Big(\Slp \big\{([\del_t;\Nres]f)\circ X^{-1}\big\}\Big)\right|_{\Gamma_\vardiamond}\bigg)
    \\
    & +
    \slp^{-1}_\vardiamond\bigg(\left.\Big(\Slp\big\{(\chi-1)\left(([\del_t;\Nres]f)\circ X^{-1}\right)\big\}\Big)\right|_{\Gamma_\vardiamond}\bigg) &(x\in \Gamma_2), \\
    =& \sS^1_\res f+ \sS^2_\res f, &
\end{align}
Note then from \eqref{commcancelbd1} and \eqref{commcancelbd2},
\begin{align}\label{commcancelbd3}
\lV [\del_t;\Nres]f\rV_{H^s(S^1)} \leq C(M)\left(\lV\sS^1_\res f\rV_{H^s(\Gamma_2)} +  \lV\sS^2_\res f\rV_{H^s(\Gamma_2)} + \lV f\rV_{H^s (S^1,\delta,m)}\right).
\end{align}
First, we handle $\lV\sS^2_\res f\rV_{H^s(\Gamma_2)}$. Let us denote $\tilde \chi = \chi \circ X$, noting $\tilde \chi=1$ on $I_2$ and has $\supp (\tilde \chi) \subset I_3$. We apply Lemma~\ref{ChordArcOpBdLem} to get a bound for $\slp^{-1}_\vardiamond$ followed by Lemma~\ref{EvalSLPBdLem} to bound the evaluation of $\Slp \{\ldots\}$ along $\Gamma_\vardiamond$, finding
\begin{align}
    \lV \sS^2_\res f\rV_{H^s(\Gamma_2)}&\leq C(M) \left\lV \slp^{-1}_\vardiamond
\left.\Big(\Slp\big\{(\chi-1)([\del_t;\Nres] f)\circ X^{-1}\big\}\Big)\right|_{\Gamma_\vardiamond}
    \right\rV_{H^{s}(\Gamma_\vardiamond)}\\
    &\leq C(M) \left\lV 
    \left.\big(\Slp\big\{(\chi-1)([\del_t;\Nres ]f)\circ X^{-1}\big\}\big)\right|_{\Gamma_\vardiamond}
    \right\rV_{H^{s+1}(\Gamma_\vardiamond)} \label{commcancelbd4}\\
    &\leq C(M)
    \lV (\tilde\chi-1)[\del_t;\Nres] f\rV_{H^s(S^1)} \\
    &\leq C(M)
    \lV [\del_t;\Nres] f\rV_{H^s(I^c_2)}\\
    &\leq C(M)\left(\lV [\del_t;\cN_+] f \rV_{H^s(I^c_2)}+\lV [\del_t;\cN_-]f \rV_{H^s(I^c_2)}\right)\\
    &\leq C(M)\lV f\rV_{H^s(S^1,\delta,m_1)}, 
\end{align}
using \eqref{commcancelbd2} in the last step above.

Now we take care of $\lV\sS^1_\res f\rV_{H^s(\Gamma_2)}$. Using Corollary~\ref{commcancelidentCor} together with similar techniques to those in the proof of Proposition~\ref{DtoNCancellation} and the facts that $\supp (\chi)\subset\Gamma\cap\Gamma_\vardiamond$ and $\supp (f\circ X^{-1})\subset \Gamma \cap \Gamma_\vardiamond$, we find the following holds for $x$ in $\Gamma_\vardiamond$:
\begin{align}
    \Slp \big\{([\del_t;\Nres]f) \circ X^{-1}\big\} &= 
    \bigh_- \Big\{\big( \cS[\del_t;\Nres]f\big)\circ X^{-1}\Big\} \\
    &=
    -\bigh_- \Big\{\big( [\del_t;\cS] \Nres f\big)\circ X^{-1}\Big\}
    +
    2\bigh_- \Big\{\big( [\del_t;\cT]f\big)\circ X^{-1}\Big\}\\
    &=
    -\bigh_- \Big\{\big( [\del_t;\cS] \tilde\chi\Nres f\big)\circ X^{-1}\Big\}
    +2\bigh_- \Big\{\big( [\del_t;\cT]f\big)\circ X^{-1}\Big\}\\
    & \quad
    +\bigh_- \Big\{\big( [\del_t;\cS](\tilde \chi-1)\Nres f\big)\circ X^{-1}\Big\}\\
    &=
    -\big( [\del_t;\cS_\vardiamond] \tilde\chi\Nres f\big)\circ X^{-1}_\diamond
    +2\big( [\del_t;\cT_\vardiamond]f\big)\circ X^{-1}_\diamond \\
    &\quad +\bigh_- \Big\{\big( [\del_t;\cS](\tilde \chi-1)\Nres f\big)\circ X^{-1}\Big\},
%
\end{align}
where $X_\diamond$ is as in Definition~\ref{tailoredCA}, and by $[\del_t;\cS_\vardiamond]$ and $[\del_t;\cT_\vardiamond]$, we mean the analogous expressions to those in Definition~\ref{cScTcommDefns} where we replace $X$ by $X_\diamond$.

Recalling the expression for $\sS^1_\res f$ and applying Lemma~\ref{ChordArcOpBdLem}, we thus find
\begin{align}
    \lV \sS^1_\res f\rV_{H^s(\Gamma_2)} 
        \leq&
     C(M)\bigg(
     \left\lV
     \big( [\del_t;\cS_\vardiamond] \tilde\chi\Nres f\big)\circ X^{-1}_\diamond
     \right\rV_{H^{s+1}(\Gamma_\vardiamond)}
     +
     \left\lV
     \big( [\del_t;\cT_\vardiamond]f\big)\circ X^{-1}_\diamond
     \right\rV_{H^{s+1}(\Gamma_\vardiamond)} \\
     &\qquad\qquad+
     \left\lV
     \bigh_- \Big\{\big( [\del_t;\cS](\tilde \chi-1)\Nres f\big)\circ X^{-1}\Big\}
     \right\rV_{H^{s+1}(\Gamma_\vardiamond)} \bigg)\\
     =&
     C(M)\left(\cI^{11}_\res+\cI^{12}_\res+\cI^{13}_\res\right).
\end{align}
Analogous to the way that one justifies the classical uniform bounds of Lemma~\ref{ChordArcOpBdLem} for admissible chord-arc curves such as $\Gamma_\diamond$, one verifies uniform bounds of the form
\begin{align}
&&\lV[\del_t;\cS_\vardiamond]g\rV_{H^{s+1}(S^1)}
+
\lV[\del_t;\cT_\vardiamond]g\rV_{H^{s+1}(S^1)}
&\leq C(M) \lV g \rV_{H^s(S^1)} &(g :S^1\to \R).
\end{align}
Using this and Proposition~\ref{DtoNCancellation}, for some $m_2$ we obtain
\begin{align}
\cI^{11}_\res+\cI^{12}_\res \leq
C(M)\left(\lV \Nres f \rV_{H^s(S^1)}+\lV f \rV_{H^s(S^1)}
\right)\leq C(M)\lV f \rV_{H^s(S^1,\delta,m_2)}.
\end{align}
Similar to the way we arrived at the bound \eqref{commcancelbd4}, we find the following, using the facts that $f$ is supported in $I_1$ and $\tilde \chi=1$ on $I_2$ together with Lemma~\ref{SuppsBddAwayLem}.
\begin{align}
\cI^{13}_\res &\leq C(M)
\lV
(\tilde \chi-1)\Nres f
\rV_{H^s(S^1)} \\
&\leq C(M) \lV\Nres f\rV_{H^s(I^c_2)} \\
&\leq C(M) \lV f \rV_{H^s(S^1,\delta,m_3)} .
\end{align}
Therefore, taking $m$ to be the largest of the $m_i$ for $i=1,2,3$, we get
\begin{align}
    \lV \sS^1_\res f \rV_{H^s(\Gamma_2)} &\leq C(M) \lV f \rV_{H^s(S^1,\delta,m)},
\end{align}
which, when combined with our bound for $\lV \sS^2_\res f \rV_{H^s(\Gamma_2)}$ and \eqref{commcancelbd3}, gives us \eqref{commcancelbd0}. This concludes the proof in the case that $f$ is supported in $[-\pi,0]$. By using a partition of unity and a similar argument for the case of smooth $f$ with support contained in $S^1\setminus[-3\pi/4,-\pi/4]$, we get the general estimate desired.
\end{proof}
\begin{lemma}\label{commcNpmLipBounds}
For an integer $m\geq0$ the following holds for $\xi$ and $\uxi$ in $\sB^\bbox(\delta)$ with $0<\delta\leq \delta_0$.
\begin{align}\label{commlipbd}
\lV ([\del_t;\cN_\pm](\xi)-[\del_t;\cN_\pm](\uxi) )f\rV_{H^2(S^1)} \leq C \lV \xi-\uxi \rV_{\sH^2} \lV f\rV_{H^3(S^1,\delta,m)}.
\end{align}
In the case $\pm=-$, the weighted norm in the right-hand side can be replaced by the $H^3(S^1)$ norm.
\end{lemma}
\begin{proof}
Let us recall Remark~\ref{NrescommDefnRemark} (i). By describing $[\del_t;\cN_\pm](\xi)-[\del_t;\cN_\pm](\uxi)$ in terms of the solution to an associated div-curl system and following along with an argument analogous to the proof of the Lipschitz bound for the map $\cN_+(\xi)$ in Proposition~\ref{NresLipBd}, one verifies the Lipschitz bound \eqref{commlipbd}.
\end{proof}
Finally, let us record the ensuing Lipschitz bound for the map $\xi\mapsto [\del_t;\Nres](\xi)$.
\begin{proposition}\label{commNresLipBound}
For an integer $m\geq0$ the following holds for $\xi$ and $\uxi$ in $\sB^\bbox(\delta)$ with $0<\delta\leq \delta_0$.
\begin{align}
\lV ([\del_t;\Nres](\xi)-[\del_t;\Nres](\uxi) )f\rV_{H^2(S^1)} \leq C \lV \xi-\uxi \rV_{\sH^2} \lV f \rV_{H^3(S^1,\delta,m)}.
\end{align}
\end{proposition}
\begin{proof}
This follows from Lemma~\ref{commcNpmLipBounds} and the fact that $[\del_t;\Nres](\xi)=[\del_t;\cN_+](\xi)+[\del_t;\cN_-](\xi)$.
\end{proof}

\section{Breakdown of analyticity}\label{analyticbreakdownsection}

\subsection{Nonexistence of analytic splash--squeezes}

In 2D, we are able to prove that analytic splash--squeeze singularities do not exist by using a well-known result on the level sets of harmonic functions of two variables near a critical point, namely that they are given by finitely many curves intersecting transversally at equal angles. The classical proof of this fact relies on the general existence of conformal maps in 2D which are essentially able to convert the level sets of these functions into the level curves of the real part of $z^n$ for some $n\geq 2$.

\begin{proposition}\label{nonanalytsupportingprop}
Consider a pair of smooth arcs $\gamma_1$ and $\gamma_2$ forming a single glancing intersection at the origin. Fix any $r>0$ such that the disc $\{|x|<r\}$ is divided into four regions by $\gamma=\gamma_1\cup\gamma_2$, and define $\sR$ to be the union of the two regions with cusps. Suppose $\phi$ satisfies
\begin{align}
& & \Delta \phi &= 0 & (x \in \sR) , \\
& & \phi&= 0 & (x \in \gamma) .
\end{align}
Then either $\phi$ is identically zero in $\sR$ or $\phi$ is not analytic at the origin.\footnote{Here, to be analytic at a point means that there exists an extension of the function to a neighborhood of this point such that the function is real analytic on said neighborhood.}
\end{proposition}
\begin{proof}
Suppose we have such a situation, and that $\phi$ is analytic at the origin. There is then a small disc $D$ about this point onto which $\phi$ can be extended such that $\phi$ is analytic with respect to each variable $x_1$ and $x_2$ independently throughout $D$. Since $\phi$ is analytic in $\sR \cup D$, so is $\Delta \phi$, which is zero in $\sR$. It follows $\Delta\phi = 0$ in $D$, and thus $\phi$ is harmonic on the extended domain $\sR\cup D$.

Now let us note
\begin{align}
\gamma \subset \{x:\phi(x) = 0\}
\end{align}
contains a pair of distinct arcs $\gamma_1'$ and $\gamma_2'$ inside $D$ (on which $\phi$ is harmonic) which make angle zero with one another. It is well-known that this is only possible for such a harmonic function $\phi$ if it is identically zero in $D$. A quick sketch of the proof of this fact proceeds as follows. Let $f(z)$ be holomorphic in $D$ with real part given by $\phi$ and $f(0)=0$. Assuming $f(z)$ is not identically zero, the first nonzero term in its power series is given by $a_n z^n$ for some $n\geq 1$. It follows there is a conformal change of coordinates $w=\sC(z)$ such that $f(w)=z^n$ near the origin. This implies the zero set of $\phi$ consists of $n$ analytic arcs meeting at the origin at equal angles of $2\pi/n$.
\end{proof}

\begin{theorem}\label{nonexist2Dtheorem}
Suppose one has a classical solution to the 2D ideal MHD system \eqref{IdealMHD1}--\eqref{IdealMHD6} such that at time $t_\splash$, the interface forms a single self-intersection point, and that at this point it is self-glancing.\footnote{By self-glancing, we mean it is locally given by a pair of tangent curves which do not cross.}  Then at time $t_\splash$, either the external magnetic field $h$ is identically zero in the vacuum $\cV$ or it is not analytic at the point of intersection.
\end{theorem}
\begin{proof}
Following the calculations discussed in Section \ref{vacuumchamberwallssection}, we find that the external magnetic field is necessarily given by $h=\grad^\perp \phi$ for some $\phi$ defined in $\cV$ with $\phi=0$ on $\Gamma$. Given this, the result follows from Proposition \ref{nonanalytsupportingprop}.
\end{proof}

\begin{remark}
\hfill
\begin{enumerate}[(i)]
\item
Note that in a hypothetical splash--squeeze formed in a situation where instead the plasma region consists of more than one connected component, say with free boundary $\Gamma$ containing a pair of distinct closed curves $\Gamma_1$ and $\Gamma_2$ which collide to form the splash, we still find $h=\grad^\perp \phi$, but $\phi$ may take on a different constant value on each $\Gamma_i$. While the proof above does not apply in this scenario, note it is trivial to rule out analyticity of $\phi$ at the splash point in such a scenario.
\item
It appears likely that analytic splash--squeezes are impossible in the 3D ideal MHD system as well. It is worth noting that one can show that the statement analogous to that in Proposition \ref{nonanalytsupportingprop} holds for harmonic functions in three dimensions, but this alone is not enough to prove the analogue of Theorem \ref{nonexist2Dtheorem} for 3D ideal MHD.
\end{enumerate}
\end{remark}

\subsection{Forward-in-time splash--squeeze construction in Sobolev spaces}\label{analyticbreakdownsubsection}

In Theorem~\ref{maintheorem}, we give solutions which which have high Sobolev regularity and end in a splash--squeeze singularity at a later time. Now we will show that there exist solutions which are analytic at time $t=0$ and end in a splash--squeeze singularity at time $t=t_\splash$. Since the splash--squeeze is a fundamental obstruction to the analyticity of the solution at time $t_\splash$, it follows that at some time $t_\star$ in $[0,t_\splash]$, analyticity is lost. Sobolev regularity, however, will persist from the initial moment up to $t_\splash$.

Let us reflect on the idea of adapting the strategy of Theorem~\ref{maintheorem} to demonstrate this without significantly modifying the bulk of our machinery.

Regardless of whether we ``run time forward or backward'', to show analytic breakdown, we must produce a solution for which there exists a time $t_a$ at which the solution is analytic and a time $t_\splash$ at which it forms a splash--squeeze. Since our local existence framework is built only for Sobolev spaces, by itself it does nothing to suggest the existence of such a time $t_a$. The only way to guarantee analyticity at some time would be to argue local existence \emph{starting from some analytic initial datum}. On the other hand, the whole point of a time-reversal argument is essentially to argue local existence starting from $t_\splash$, and yet Theorem~\ref{nonexist2Dtheorem} implies $t_a \neq t_\splash$. Therefore, a naive time-reversal argument using only Sobolev-based local existence machinery is destined to fall short of establishing analytic breakdown.

Instead, without revamping our entire Sobolev-based local existence framework to involve analytic norms, we modify details of our iteration scheme to construct solutions starting from analytic non-splash initial data which terminate in a splash--squeeze at a later time $t_\splash$, losing analyticity at some time $t_\star \leq t_\splash$.

Let us outline a few key points regarding the alterations to our strategy.

\begin{itemize}
\item
We discard $\xi_0$, starting with analytic data $\xi_\init$ with a small but positive pinch $\delta_\init$ at time $t=0$. The initial velocities are selected to ensure the two arcs of the interface come closer in positive time. In fact, we show that if $\xi(t)$ starts out as $\xi_\init$ and its acceleration remains bounded, a splash forms at some time $t_\splash$ which has explicit upper bounds.

\item
We devise an artificial interpretation of quantities appearing in the evolution equation for short times past a moment of splash, i.e. when the domain $\Omega(t)$ starts to overlap itself.

\item We prove Sobolev local existence starting from $\xi_\init$ for times $0\leq t \leq T$, where $T> t_\splash$, meaning the solution achieves the splash--squeeze during its interval of existence.

\end{itemize}
Adapting our local existence argument to be able to construct solutions forward in time which \emph{pass through} a splash state\footnote{We sometimes use the term \emph{splash state} to refer to a state vector $\xi_\splash$ (not time-dependent) for which the corresponding interface exhibits a splash-type self-intersection.} requires the most new ideas. The main technical difficulty is proving the iterates produced by a scheme converge, despite having different splash states and mismatched ``splash parameters'', such as differing splash times. This means we must deal with pairs $\xi(t)$ and $\uxi(t)$ at certain times $t$ when their pinches are of different orders of magnitude, though many of our weighted estimates are only designed to compare objects like $\Nres(\xi)|h(\Gamma)|^2$ and $\Nres(\uxi)|h(\underline\Gamma)|^2$ when $\xi$ and $\uxi$ are in the same $\sB^\bbox(\delta)$ class. We will get around this by applying slight time-shifts to the more delicate quantities before making such comparisons.

\subsubsection{Analytic initial data}
Since the terms in our Lagrangian wave system are currently defined for $\xi$ in the state classes $\sB$, $\sB_\dagger$, etc., which are centered around $\xi_0$, rather than the analytic near-splash initial data $\xi_\init$, we will need to define variants on these classes shortly. First, we select $\xi_\init$ for our forward-in-time construction.

\begin{definition}\label{analytInitdata}
\hfill
\begin{enumerate}[(i)]
\item For each $\delta_\init\in (0,T]$, we produce a corresponding $\bX_\init=\bX_\init(\delta_\init)$, $\Omega_\init=\Omega_\init(\delta_\init)$, and $X_\init=X_\init(\delta_\init)$ by following the construction of Definition~\ref{initialData}, except that the associated $\Gamma_\init$ does not intersect itself. Instead we arrange that it forms a pinch of width $\delta_\init$, where the two pinch points $p_\ell$ and $p_r$, satisfying $|p_\ell-p_r|=\delta_\init$, are given by $X_\init(-\pi/2)=p_\ell$ (to the left) and $X_\init(\pi/2)=p_r$ (to the right), both $p_\ell$ and $p_r$ are contained in the disc $\{|x-(0,2)|<10^{-3}\}$, and the tangents at $p_\ell$ and $p_r$ are vertical. Clearly, we may also arrange that $\bX_\init$ is analytic in $\Sigma$, and
\begin{align}\label{bXdelta0cap}
\sup_{\delta_\init\in(0, T]} \left(\lV \bX_\init(\delta_\init) \rV_{C^2(\Sigma)} +\lV (\grad\bX_\init(\delta_\init))^{-1}\rV_{C^1(\Sigma)}\right)<\infty .
\end{align}
\item We define the Lagrangian initial stream function below, for a constant $\underline \nu_a$ (independent of $\delta_\init$):
\begin{align}\label{vbarphidef}
&& \bvarphi_\init(\theta,\psi) &= -\underline \nu_a(1+\psi)\sin(2\theta) &( (\theta,\psi) \in \Sigma).
\end{align}
By using $\bX_\init$ we then define corresponding $\varphi_\init$, $u_\init$, $\bU_\init$, $U_\init$, $b_\init$, $\bB_\init$, and $B_\init$ analogously to our definitions of $\varphi_0$, $u_0$, $\bU_0$, etc. in Definition~\ref{initialData}.

We define $\bom_\init$, $\bj_\init$, $\dot U^*_\init$, $\dot B^*_\init$, $\xi_\init$, $\xi_{\dagger,\init}$, etc. in the same way the corresponding quantities were defined in Definition~\ref{initialxiData}.

\item Let us define $C_{\max}$ below, finite as a consequence of \eqref{bXdelta0cap}:
\begin{align}\label{Cmaxconstdefn}
C_{\max} = \sup_{\delta_\init\in(0, T]}\left(\lV X_\init(\delta_\init)\rV_{C^2(S^1)}+\lV U_\init(\delta_\init) \rV_{C^1(S^1)}\right).
\end{align}

\end{enumerate}
\end{definition}

\begin{remark}
\hfill
\begin{enumerate}[(i)]

\item To produce satisfactory initial data for our main result, Theorem~\ref{maintheorem2}, we may fix any $\delta_\init$ in $(0,T]$ in the above definition. For example, we may take $\delta_\init=T$. Our need to select $\delta_\init$ bounded in terms of $T$ reflects our need to force a splash to happen before the interval of existence is over. In the proofs of Propositions~\ref{solutionOpBds},~\ref{iterationstep1}, and~\ref{ISclaim8} in particular, we explain how to select $T$ so that the argument closes.

\item 
Similarly as in the proof of Lemma~\ref{separationPrepLem}, we select the universal constant $\underline \nu_a$ in such a way that we ensure $|U_\init(\pm\pi/2)|\geq 2$. Note that as a consequence of the definition of $\bvarphi_\init$ and the fact that $\bX_\init$ is conformal, we have that the velocities at $p_\ell$ and $p_r$ point toward the future formation of a splash, with $U_\init(-\pi/2)=(|U_\init(-\pi/2)|,0)$ and $U_\init(\pi/2)=(-|U_\init(\pi/2)|,0)$. This implies the pinch must decrease, with initial rate $\frac{d\delta_\Gamma}{dt}(0)\leq-4$.

\end{enumerate}
\end{remark}

\begin{definition}
We define the class of admissible splash states $\sH^k_\splash$ by
\begin{align}
      \sH^k_\splash
        = \{
            \xi \in\sH^k :
            \Gamma \ \mbox{has a splash point at } \ p_\splash, \ |p_\splash-(0,2)|<10^{-3},\ \lV \xi - \xi_\init \rV_{\sH^{k-1}}\leq r_1
          \},
\end{align}
\end{definition}

Due to our need to compare iterates with mismatched splash parameters, as discussed earlier, we will sometimes use a reparametrization map $\Theta$, with the intent of ``recalibrating'' a given $X(\theta)$ so that the values of $\theta$ at which it realizes a splash point, say, $\theta_\splash\approx \pi/2$ and $\vartheta_\splash\approx-\pi/2$ with $X(\theta_\splash)=X(\vartheta_\splash)$, are mapped to $\pi/2$ and $-\pi/2$. Additionally, we define a class of equivalent initial data produced by reparametrizing $\xi_\init$ in the $\theta$ argument.

\begin{definition}\label{relabelingDefn}
\begin{enumerate}[(i)]
\hfill
\item
Let us define
\begin{align}\label{IpmDefn}
I_\pm=[\pm\pi/2-10^{-2},\pm\pi/2+10^{-2}].
\end{align}
We fix a map $\Theta$ in $C^\infty(S^1\times I_+ \times I_-)$ taking values in $S^1$ such that for each $\theta$ in $I_+$ and $\vartheta$ in $I_-$, $\Theta(\,\cdot\,;\theta,\vartheta)$ is a diffeomorphism on $S^1$ satisfying
\begin{align}\label{ThetaClosetoId}
&& \lV\Theta(\,\cdot\,;\theta,\vartheta)- \operatorname{id}\rV_{H^{k+2}(S^1)} &\leq \epsilon_2 &(\theta \in I_+, \ \vartheta \in I_-),
\end{align}
for a small constant $\epsilon_2$ and the property
\begin{align}
&&&
\begin{aligned}
\Theta(\pi/2;\theta,\vartheta)&=\theta,\\
\Theta(-\pi/2;\theta,\vartheta)&=\vartheta,
\end{aligned}
&(\theta \in I_+, \ \vartheta \in I_-).
\end{align}
\item For our class of reparametrizations of $\xi_\init$ with respect to $\theta$ we define
\begin{align}
\bm{\Xi}_\init = \{(\theta,\psi)\mapsto\xi_\init(\Theta(\theta;\theta_\splash,\vartheta_\splash),\psi) : \theta_\splash\in I_+, \ \vartheta_\splash\in I_-\}. 
\end{align}
\end{enumerate}
\end{definition}

\subsubsection{Generalization of classes to alternate splash states}

We proceed to define several different classes of state vectors $\xi$, some time-dependent, others not. These are variants on the $\sB$, $\sB^\bbox(\delta)$ type classes defined earlier.

In contrast to our class $\sB$ of time-evolving states, featuring splashes which open up as time goes forward, we now define a collection $\sB^\tosplash$ of states \emph{headed toward} eventual splashes, which we refer to as the \emph{to-splash class.} To deal with time-shifts later, it is helpful to consider negative times as well.
\begin{align}
\sB^\tosplash_\dagger &= \{\xi_\dagger\in C^0([-T,T];\sH^k_\dagger): \sup_{t}\lV \xi_\dagger\rV_{\sH^k_\dagger}\leq M, \ \sup_{t}\lV \xi_\dagger-\xi_{\dagger}(0) \rV_{\sH^{k-1}_\dagger} \leq r_0\}, \\
\sB^\tosplash &= \{\xi\in C^0([-T,T];\sH^k):  \sup_{t}\lV \xi\rV_{\sH^k}\leq M, \ \xi_\dagger \in \sB^\tosplash_\dagger,\ \xi(0)\in\bm{\Xi}_\init, \  X_t=U, \  \lV U\rV_{C^1_{t,\theta}}\leq M\}.
\end{align}

Note that the elements of $\sB^\tosplash$ are initially equal to some reparametrization of $\xi_\init$.
\begin{lemma}\label{tosplashXbd1}
For $\xi$ in $\sB^\tosplash$, for sufficiently small $T$ we have for $\epsilon_1$ of \eqref{Xzerobds} and $C_{\max}$ of \eqref{Cmaxconstdefn}
\begin{align}
&&|\del_\theta X(t,\theta) |& \geq \epsilon_1 /2&& (t\in[-T,T],\ \theta\in S^1),\\
&&\lV X(t)\rV_{C^2(S^1)}+\lV U(t)\rV_{C^1(S^1)}&\leq 2 C_{\max}&&(t\in[-T,T]).%
\end{align}
\end{lemma}
\begin{proof}
Recalling our definition of $\xi_\init$, this is guaranteed with the fundamental theorem of calculus by ensuring $T$ is sufficiently small, dependent on $\epsilon_1$, $C_{\max}$, and $M$ alone.
\end{proof}
\begin{remark}
Note at first glance it seems there may be some redundancy between the bound above on $\lV U\rV_{C^1(S^1)}$ and the bound $\lV U \rV_{C^1_{t,\theta}}\leq M$ for $\xi$ in $\sB^\tosplash$. The difference is $C_{\max}$ has already been defined, whereas $M$, which bounds $\lV U_t \rV_{L^\infty(S^1)}$, will later be chosen sufficiently large to close the local existence argument.
\end{remark}

Now seek to construct another class, which, like our original class $\sB$, involves states experiencing the splash at time $t=0$, but in a manner consistent with our forward-in-time argument. Let us define the following generalization of $\sB$ which in particular removes the criterion of having a specified initial state:
\begin{align}\label{Bgendefn}
\sB_{gen} = \{\xi\in C^0([-T,T];\sH^k):  \sup_{t}\lV \xi\rV_{\sH^k}\leq 1.1 M, \ \lV\xi_\dagger-\xi_\dagger(0)\rV_{\sH^{k-1}_\dagger}\leq 1.1 r_0,\  X_t=U, \  \lV U\rV_{C^1_{t,\theta}}\leq 1.1 M\}.
\end{align}

After we show with Proposition~\ref{splashtimeprop} that members of $\sB^\tosplash$ indeed necessarily lead to a splash at a time $t_\splash$, we will sometimes work with an associated time-shifted version of a given $\xi\in\sB^\tosplash$, the \emph{from-splash shift} of $\xi$, i.e. the map $t\mapsto \xi(t+t_\splash)$. This time-evolving state instead experiences the splash at $t=0$, with the lead-in to the splash taking place over negative times.

\begin{definition}\label{startatsplashclassdefn}
For a splash state $\xi_\splash$ in $\sH^k_\splash$ and parameter\footnote{One should think of $s$ as a splash time $t_\splash>0$ for an evolving state $\xi$ which eventually forms a splash.} $s$ in $(0,T]$ we define the \emph{start-at-splash classes} $\sB^\skipto(\xi_\splash;s)$ and $\sB^\skipto(\xi_\splash)$ by the following:
\begin{align}
\sB^\skipto(\xi_\splash;s) &= \{ t\mapsto \xi(t+s) :  \ \xi\in \sB_{gen}, \ \xi(s)=\xi_\splash\},\\
\sB^\skipto(\xi_\splash) &= \bigcup_{s\in(0,T]}\sB^\skipto(\xi_\splash;s) .
\end{align}
\end{definition}
\begin{remark}
Note that for a given element in $\xi$ in $\sB^\skipto(\xi_\splash)$, $\xi(t)$ is defined for times $t$ in $[-T-s,T-s]$ for an associated parameter $s$ in $(0,T]$.
\end{remark}

To provide intuition, let us explain the precise connection between our old class $\sB$ and the start-at-splash class defined above. Consider an evolving state $\xi^\revplay=(\dot U^*,\dot B^*, \bom, \bj, X, U)$. If $\xi^\revplay$ is in $\sB$, it starts in the splash state $\xi_0$ at time $t=0$ and opens up as time advances over the interval $[0,T]$. Now let us define the time-reversed version of $\xi^\revplay$ by
\begin{align}
\xi^\play(t) = (\dot U^*(-t), \dot B^*(-t), \bom(-t), \bj(-t), X(-t), -U(-t)).
\end{align}
From the point of view we take in this section, which uses forward-in-time arguments, $\xi^\play$ starts with a separated interface at initial time $t=-T$ and ends up in the splash state $\ol \xi_0$ at $t=0$, where we define
\begin{align}
\ol \xi_0 = (\dot U^*_0, \dot B^*_0, \bom_0,\bj_0, X_0, -U_0).
\end{align}
In fact, if we define $\tilde\sB$ exactly as we define $\sB$, except with $1.1M$ and $1.1r_0$ in place of $M$ and $r_0$, we find
\begin{align}
\xi^\play \in \sB^\play(\ol \xi_0;T) \implies \xi^\revplay|_{[0,T]}\in \tilde \sB .
\end{align}
All the bounds of the previous sections, proved for $\xi$ in $\sB$, are thus essentially trivial to extend to $\sB^\skipto(\ol \xi_0;T)$, as well as to more general $\sB^\skipto(\xi_\splash)$, which we explain later. Let us emphasize here that in the case of defining maps for the class $\sB^\skipto( \xi_\splash)$, we only apply the definitions of the previous sections up to and including the moment of splash. After the splash, when the domain $\Omega(t)=\bX(t,\Sigma)$ begins to overlap itself, we choose artificial definitions for these maps, to be explained later.


Now we make the following definitions of classes  to replace the classes for fixed-time states we defined in Section~\ref{vacandopmapssection}, namely $\sB^\bbox(\delta)$, $\sB^\bbox$, and $\sB^\wbox$. The definitions are the same aside from the replacement of $\xi_0$ by a more general splash state $\xi_\splash$.

\begin{definition}\label{fixedtimeclasses}
For a splash state $\xi_\splash=(\ldots,X_\splash,U_\splash)$ in $\sH^k_\splash$, we define
\begin{align}
    \sB^\bbox(\delta,\xi_\splash) &= \{\xi \in \sH^k_\gp :\lV \xi \rV_{\sH^k} \leq M, \  \lV X - X_\splash \rV_{H^{k+1}(S^1)} \leq C\delta , \ c\delta \leq \delta_\Gamma \leq C\delta \}, \\
    \sB^\bbox(\xi_\splash) &= \bigcup_{\delta \in [0,\delta_0]}\sB^\bbox(\delta,\xi_\splash) ,\\
    \sB^\wbox(\xi_\splash) &= \bigcup_{\delta \in (0,\delta_0]}\sB^\bbox(\delta,\xi_\splash) .
\end{align}
\end{definition}
\subsubsection{Forward-in-time splash dynamics}

\subsubsection*{Proof of eventual splash in the to-splash class}

We now show that if a certain relation holds between the initial pinch $\delta_0$ and the other constants, which we are easily able to arrange, if a state $\xi$ is in the to-splash class $\sB^\tosplash$, it experiences a splash at a time $t_\splash$, with $0<t_\splash\leq T/2$. Let us note that for such functions, even as $t$ passes through $t_\splash$, and the curve becomes self-intersecting, $\xi(t)$ continues to evolve for a short time.

The following lemma gives us some tools to track the evolution of the pinch $\delta_\Gamma(t)$. 
\begin{lemma}\label{pinchlabellemma}
Let $\xi$ be in $\sB^\tosplash$. Define $\underline T= \underline T(\xi)$ by
\begin{align}\label{uTdefn}
\underline T = \min(\{T\}\cup\{t\in[0,T]:\Gamma(t)\ \textrm{\emph{is self-intersecting}}\}).
\end{align}
For $t$ in the interval $[0,\underline T]$, there is a unique pair $(\theta_*(t),\vartheta_*(t))$ in $I_+\times I_-$ (see \eqref{IpmDefn}) which minimizes
\begin{align}
\cQ(t,\theta,\vartheta)&=|X(t,\theta)-X(t,\vartheta)| .
\end{align}
Furthermore, for $t$ in $[0,\underline T]$,
\begin{align}\label{pinchfmla}
\delta_\Gamma(t) =|X(t,\theta_*(t))-X(t,\vartheta_*(t))|,
\end{align}
and the functions $\theta_*(t)$ and $\vartheta_*(t)$ are in $C^1([0,\underline T])$, satisfying the following for some constant $C$ dependent on $C_{\max}$ and $\epsilon_1$, as in Lemma~\ref{tosplashXbd1}, alone:
\begin{align}
\lV \theta_* \rV_{C^1([0,\underline T])}+\lV \vartheta_* \rV_{C^1([0,\underline T])} \leq C.
\end{align}
\end{lemma}
\begin{proof}
The existence and uniqueness of $\theta_*(t)$ and $\vartheta_*(t)$ can be observed by verifying the fact that the curvature of $\Gamma(t)$ remains bounded below, due to closeness of the parametrization $X$ to $X_\init$, by choice of $T$. The identity \eqref{pinchfmla} is immediate, and one verifies $\theta_*(t)$ and $\vartheta_*(t)$ are in $C^1([0,\underline T])$ with the claimed bound by using the implicit function theorem.

\end{proof}

\begin{corollary}\label{kingeolemma1}
Given $\xi$ in $\sB^\tosplash$, for $\underline T$ and $\vartheta_*(t)$ as in Lemma~\ref{pinchlabellemma}, for $t$ in $[0,\underline T]$ let us define
\begin{align}
\bn(t)&=n(t,X(t,\vartheta_*(t)).
\end{align}
Then
\begin{align}\label{deltaspeedfmla}
\frac{d\delta_\Gamma}{dt}(t) = \big(U(t,\theta_*(t))-U(t,\vartheta_*(t))\big) \cdot \bn(t),
\end{align}
Moreover, for a constant $C$ dependent on $C_{\max}$ and $\epsilon_1$ alone,
\begin{align}
\lV \delta_\Gamma \rV_{C^1([0,\underline T])}+\lV\bn\rV_{C^1([0,\underline T])} \leq C,& \label{deltaspeedbd1}\\
\sup_{t\in[0,\underline T]}\left |\frac{d^2 \delta_\Gamma}{dt^2}(t)\right | \leq C M .&\label{deltaspeedbd2}
\end{align}
Additionally,
\begin{align}\label{deltaspeedbd3}
\frac{d\delta_\Gamma}{dt}(0) \leq -4.
\end{align}
\end{corollary}
\begin{proof}
By using Lemma~\ref{pinchlabellemma}, we obtain the formula \eqref{deltaspeedfmla} from \eqref{pinchfmla}, followed by the bounds \eqref{deltaspeedbd1} and \eqref{deltaspeedbd2}. The bound \eqref{deltaspeedbd3} follows from Definition~\ref{analytInitdata}.
\end{proof}
Now, with Proposition~\ref{splashtimeprop}, for $\xi$ in $\sB^\tosplash$ we give explicit estimates for the pinch $\delta_\Gamma(t)$ as it changes in time and deduce the interface must form a splash at a time $t_\splash$ lying in a certain window.




\begin{proposition}\label{splashtimeprop}
Consider $\xi$ in $\sB^\tosplash$. For $T$ sufficiently small, the following holds.

\begin{enumerate}[(i)]
\item Setting $\nu_\init=-\frac{d\delta_\Gamma}{dt}(0)$, we have the following, for $t$ in $[0,\underline T]$, $\underline T$ as in Lemma~\ref{pinchlabellemma}, and a constant $C$ depending on $C_{\max}$ and $\epsilon_1$ alone:
\begin{align}\label{pinchevolutionbound}
 \delta_\init-\nu_\init t-C M t^2\leq \delta_\Gamma(t)
\leq \delta_\init-\nu_\init t + C M t^2 .
\end{align}

\item There is a time $t_\splash$ with $0<t_\splash \leq T/2$ such that $\delta_\Gamma(t_\splash)=0$, so that in fact the state $\xi(t)$ exhibits a splash at $t=t_\splash$.
\end{enumerate}
\end{proposition}
\begin{proof}
Corollary~\ref{kingeolemma1} readily implies (i). To get (ii), note that if the quadratic in $t$ in the upper bound from \eqref{pinchevolutionbound} realizes a zero on the interval $[0,T]$, so must $\delta_\Gamma(t)$. We now proceed to bound this zero above.

Taking the constant $C$ from \eqref{pinchevolutionbound}, assuming $T\leq (C_{\max})^2/(2CM)$, so that $\delta_\init\leq T$ implies
\begin{align}
\varepsilon=\frac{4 CM\delta_\init }{\nu^2_{\init}} \leq \half,
\end{align}
we find the first zero of the quadratic occurs at
\begin{align}\label{splashtimebounds}
\frac{ \nu_\init-\sqrt{\nu^2_\init- 4C M\delta_\init}}{2C M}
=\frac{\nu_\init}{2 CM} \left(1-\sqrt{1-\varepsilon} \right)
\leq \frac{\nu_\init}{2 CM} \varepsilon
= 2\frac{\delta_\init}{\nu_\init}
\leq \frac{T}{2}.
\end{align}
In the last step above, we used that $\delta_\init\leq T$ and $\nu_\init\geq 4$, by \eqref{deltaspeedbd3}. Thus, we get (ii).
\end{proof}

As a consequence of Proposition~\ref{splashtimeprop} and our choice of initial data we get the following.
\begin{corollary}\label{splashguaranteeprop}
Suppose $\xi$ is in $\sB^\tosplash$. For $T$ sufficiently small, at the time $t_\splash$ at which $\xi$ splashes, given by Proposition~\ref{splashtimeprop}, we have the splash state $\xi(t_\splash)=\xi_\splash$ is in $\sH^k_\splash$, and
\begin{align}\label{SplashClassEvo}
&& \xi(t)&\in \sB^\bbox(t_\splash-t,\xi_\splash) & (t\in [0,t_\splash]).
\end{align}
\end{corollary}
\begin{proof}
Using Proposition~\ref{splashtimeprop} to guarantee the splash state $\xi_\splash=(\ldots,X_\splash,U_\splash)$ occurs before time $T/2$, as long as $T$ is sufficiently small, it is not hard to verify $\xi_\splash\in \sH^k_\splash$. One uses a similar argument to the proof of \eqref{Xshapetbound} of Proposition~\ref{dcXlemma} to verify the closeness of $X_\splash$ to $X$ relative to $t_\splash-t$.
\end{proof}
\begin{remark}
Corollary~\ref{splashguaranteeprop} plays the role of Proposition~\ref{xiInPinchRanget} in our new setting for the forward-in-time existence argument.
\end{remark}
\subsubsection{Splash-centering for the to-splash class}

For $\xi$ in $\sB^\tosplash$ with a corresponding splash time $t_\splash$, we previously mentioned the corresponding from-splash shift given by the map $t\mapsto \xi(t+t_\splash)$, which experiences the splash at $t=0$. A pair of time-shifted evolving states which both experience the splash at time $t=0$ is much more natural to compare than a pair of objects experiencing splashes at different times.

The same is true regarding values of $\theta$ corresponding to splash points. 
Using the map $\Theta$ from Definition~\ref{relabelingDefn}, we may reparametrize $X$ so that the splash point is realized at $\theta=\pi/2$ and $\theta=-\pi/2$ instead of some arbitrary pair of labels $\theta_\splash$ and $\vartheta_\splash$ in $S^1$. To do so requires the lemma below, which follows easily from Lemma~\ref{pinchlabellemma}.
\begin{lemma}
Consider $\xi$ in $\sB^\tosplash$ with splash time $t_\splash$ and splash point $p_\splash$. For the corresponding pair\footnote{We generally use $\theta_\splash$ to refer to the $\theta$ value closer to $\pi/2$ and $\vartheta_\splash$ to refer to the one closer to $-\pi/2$.} $\theta_\splash,\vartheta_\splash$ in $S^1$ such that $X(t_\splash,\theta_\splash)=X(t_\splash,\vartheta_\splash)=p_\splash$, we have $|\theta_\splash-\pi/2|\leq 10^{-2}$ and $|\vartheta_\splash+\pi/2|\leq 10^{-2}$.
\end{lemma}

By applying the combined time and label ``realignment'' discussed above to a given $\xi$ in $\sB^\tosplash$, we produce the \emph{splash-centered version} of $\xi$.
\begin{definition}\label{splashcenteringDefn}
Consider $\xi$ in $\sB^\tosplash$, which has a splash time $t_\splash$ in $(0,T/2]$, by Proposition~\ref{splashtimeprop}. Let $p_\splash$ be the corresponding splash point.
\begin{enumerate}[(i)]

\item Let $\theta_\splash,\vartheta_\splash$ be the pair of distinct elements of $S^1$ such that $X(\theta_\splash)=X(\vartheta_\splash)=p_\splash$. We use the term \emph{splash labels} to refer to $\theta_\splash$ and $\vartheta_\splash$ and the term \emph{splash parameters} to refer to the splash time $t_\splash$ together with the splash labels $\theta_\splash,\vartheta_\splash$.

\item Letting $\tilde\theta(\theta)=\Theta(\theta;\theta_\splash,\vartheta_\splash)$ ($\Theta$ given in Definition~\ref{relabelingDefn}), we define $\xi^\odot:[-T-t_\splash,T-t_\splash]\to\sH^k$, the \emph{splash-centered version} of $\xi$, by
\begin{align}
&& \xi^\odot(t,\theta,\psi) &= \xi(t+t_\splash,\tilde\theta(\theta),\psi)  &(t\in [-T-t_\splash,T-t_\splash], \ (\theta,\psi)\in\Sigma).\label{splashcenteredref}
\end{align}
For the components of $\xi^\odot$, we use the notation $\xi^\odot=(\dot U^{*\odot},\dot B^{*\odot},\bom^\odot,\bj^\odot,X^\odot,U^\odot)$.
\end{enumerate}
\end{definition}
\begin{remark}
For $\xi$ in $\sB^\tosplash$ with splash state $\xi_\splash$ and splash time $t_\splash$, 
we have for the splash-centered version that $\xi^\odot\in\sB^\play(\xi_\splash;t_\splash)\subset\sB^\play(\xi_\splash)$. To verify this, in particular we use the bound \eqref{ThetaClosetoId}.
\end{remark}
\subsubsection*{Stability of splash parameters}

In Proposition~\ref{ISclaim8}, which justifies convergence in the iteration scheme for our forward-in-time existence argument, we consider solutions $\xi$ and $\uxi$ in $\sB^\tosplash$ to a pair of nearly identical linear wave systems. In the final analysis, we need to verify not only that $\xi$ and $\uxi$ are correspondingly close, but also that their splash times and splash labels are.

For this reason we prove a stability result, Proposition~\ref{splashparamStabilityProp}, which shows that when $\xi$ and $\uxi$ are close to each other, so are the corresponding triples of splash parameters $(t_\splash,\theta_\splash,\vartheta_\splash)$ and $(\underline{t_\splash},\underline{\theta_\splash},\underline{\vartheta_\splash})$. We now state a topological lemma to be used in the proof of this result, guaranteeing the existence of a zero of an $\R^3$-valued vector field under suitable hypotheses.
\begin{lemma}\label{zerolemma}
Fix $q\in\R^3$ and $R>0$. Consider the ball $B_R(q)=\{|y-q|<R\}\subset\R^3$ and a pair of $C^1$ vector fields $Y$, $\underline Y:\overline{B_R(q)}\to\R^3$. Suppose for some $r>0$ that we have
\begin{align}
&&|Y(y)|&\geq r & (y\in \del B_R(q)),\\ \shortintertext{
and we have}
&&| Y(y) - \underline Y(y)|&< r &(y\in \del B_R(q)).
\end{align}
Suppose also that the point $q$ is the unique zero in $B_R(q)$ of the vector field $Y$, which has topological degree $\deg (Y,B_R(q),0)=1$.
\\
Then $\underline Y$ has topological degree $\deg (\underline Y,B_R(q),0)=1$. In particular, $\underline Y$ has a zero in $B_R(q)$.
\end{lemma}
\begin{proof}
The proof of the above amounts to a standard argument involving homotopy invariance of the degree.
\end{proof}
\begin{proposition}\label{splashparamStabilityProp}
Consider $\xi$ and $\uxi$ in $\sB^\tosplash$, with respective splash times $t_\splash$ and $\underline{t_\splash}$ and respective splash labels $\theta_\splash,\vartheta_\splash$ and $\underline{\theta_\splash},\underline{\vartheta_\splash}$. Set
\begin{align}
\epsilon = \sup_{t}\left(\lV U-\underline U\rV_{L^\infty(S^1)}+ \lV X -\uX \rV_{C^1(S^1)} \right).
\end{align}
Then for some universal constants $\tilde \epsilon$ and $\tilde C$, as long as $\epsilon \leq \tilde\epsilon$, we have
\begin{align}
|t_\splash-\underline{t_\splash}| + |\theta_\splash - \underline{\theta_\splash}| + |\vartheta_\splash-\underline{\vartheta_\splash}| \leq \tilde C \epsilon.
\end{align}
\end{proposition}
\begin{proof}
For $X$ corresponding to $\xi$, let us define the function $Z:[0,T]\times S^2\to\R^3$ by
\begin{align}
Z(t,\theta,\vartheta) = \col{X(t,\theta)-X(t,\vartheta)}{\del_\theta X(t,\theta) \cdot \del_\vartheta X^\perp(t,\vartheta)}.
\end{align}
Similarly, we define the function $\underline Z(t,\theta,\vartheta)$ associated to $\uxi$. Observe that $Z(t,\theta,\vartheta)$ vanishes precisely when one plugs in the triple of splash parameters $q_\splash = (t_\splash,\theta_\splash,\vartheta_\splash)$. Likewise, we denote $\underline{q_\splash}=(\underline{t_\splash},\underline{\theta_\splash},\underline{\vartheta_\splash})$.

Let us use $y$ to denote a general point $(t,\theta,\vartheta)\in[0,T]\times S^2$ and consider the set
\begin{align}
\cY=\{(t,\theta,\vartheta): t\in[0,T], \ |\theta-\pi/2|\leq 10^{-2}, \ |\vartheta+\pi/2|\leq 10^{-2}\}.
\end{align}
Note that for a universal constant $\tilde c_1$,
\begin{align}\label{Zdiffbound}
&& |Z(y)-\underline Z(y)| &< \tilde c_1 \epsilon &(y\in\cY).
\end{align}
We proceed with the strategy of using Lemma~\ref{zerolemma} to bound the location of the unique zero of $\underline Z$, i.e. $\underline{q_\splash}$, in terms of the location of $q_\splash$.

Let us use the notation $D_y=(\del_t,\del_\theta,\del_\vartheta)$. Computing the Jacobian matrix $D_y Z(y)$, we find that the smallest eigenvalue is bounded below by some universal constant $\tilde c_2$. Thus
\begin{align}\label{Zpositivedefbound}
&& |D_y Z(y) v | &\geq \tilde c_2 |v| &(y\in\cY, \ v\in\R^3).
\end{align}
Let us take $\tilde \epsilon$ small enough that $\epsilon\leq \tilde \epsilon$ implies the sphere $\{y:|y-q_\splash|= \frac{\tilde c_1}{\tilde c_2} \epsilon\}$ is contained in $\cY$. Thus from \eqref{Zpositivedefbound} it follows
\begin{align}
&& |Z(y)| & \geq \tilde c_1 \epsilon &\left(|y-q_\splash|= \frac{\tilde c_1}{\tilde c_2} \epsilon \right).
\end{align}
Now we use this in combination with the bound \eqref{Zdiffbound} (which of course holds for $|y-q_\splash|=\frac{\tilde c_1}{\tilde c_2} \epsilon$), and apply Lemma~\ref{zerolemma} to $Z$ and $\underline Z$, taking $q=q_\splash$, $r=\tilde c_1\epsilon$, and $R=\frac{\tilde c_1}{\tilde c_2} \epsilon$. Since $\underline{q_\splash}$ is the unique zero of $\underline Z$, we conclude
\begin{align}
|q_\splash-\underline{q_\splash}|<\frac{\tilde c_1}{\tilde c_2} \epsilon,
\end{align}
from which the claim follows.
\end{proof}

\subsubsection*{Definitions of terms in the splash-centered Lagrangian wave system}

Now we discuss how to extend the definitions of our various maps $H(\xi),\dot H(\xi), P^\pm_\grad(\xi), \dot P^\pm_\grad(\xi)$, etc. to appropriate $\xi$ for our forward-in-time argument. This includes defining most of them for $\xi$ past its splash time, when the corresponding domain $\Omega$ begins to overlap itself. The solution map $S_+(t)$ of the following lemma, whose proof is elementary, is used to produce artificial extensions with appropriate regularity beyond the splash time.

\begin{lemma}\label{Sopdefn}
For $s\geq 0$, and a pair of real-valued functions $f_0,\dot f_0$, with $f_0$ in $H^{s+1}(S^1)$ and $\dot f_0$ in $H^s(S^1)$, consider the system below, which has a unique solution pair $f(t,\theta)$ and $\dot f(t,\theta)$, defined for all $t\geq 0$, $\theta$ in $S^1$:
\begin{align}
f_t &= \dot f ,\\
\dot f_t &= f_{\theta\theta} ,\\
(f,\dot f)(0) &= (f_0 ,\dot f_0).
\end{align}
It follows $(f,\dot f)$ is in $C^0([0,\infty); H^{s+1}(S^1)\times H^s(S^1))$. Moreover, defining the solution map $S_+(t)$ by
\begin{align}
&& S_+(t)(f_0,\dot f_0) &= (f(t),\dot f(t)) &(t\geq 0 ), \\
\intertext{we have}
&&\lV S_+(t)(f_0,\dot f_0) \rV_{H^{s+1}(S^1)\times H^s(S^1)} &= \lV (f_0,\dot f_0) \rV_{H^{s+1}(S^1)\times H^s(S^1)} & (t\geq 0 ).
\end{align}
\end{lemma}

The key maps are first defined up to, but not past the splash state, with essentially the same definitions as before.
\begin{definition}\label{keyPinchSensMapsExtDefns}

\begin{enumerate}[(i)]
Consider a splash state $\xi_\splash$ in $\sH^k_\splash$.

\item We define the maps below with the designated domains in the same way as we do in the listed corresponding definitions and propositions, but where we replace the class $\sB^\bbox$ by $\sB^\bbox(\xi_\splash)$. 
Doing so results in maps taking values in one of the spaces $H^k(S^1)$ or $H^{k+1}(S^1)$, as listed below.
\begin{align}
 H &: \sB^\bbox(\xi_\splash) \to H^{k+1}(S^1) & (\mbox{Definition~\ref{prelimHmapDefn}}), \\
\dot H &: \sB^\bbox(\xi_\splash) \to H^k(S^1) & (\mbox{Definition~\ref{dotHdefns}}), \\
 P^+_\grad &: \sB^\bbox(\xi_\splash) \to H^{k+1}(S^1) & (\mbox{Definition~\ref{ExtPressureDefnBd}}), \\
\dot P^+_\grad &: \sB^\bbox(\xi_\splash) \to H^k(S^1) & (\mbox{Proposition~\ref{dotPplusdefn}}), \\
\sN &: \sB^\bbox(\xi_\splash) \to H^{k+1}(S^1) & (\mbox{Remark~\ref{dcF1forBcircDefn}, Definition~\ref{finalDefdcF1}}), \\
\dot \sN &: \sB^\bbox(\xi_\splash) \to H^k(S^1) & (\mbox{Remark~\ref{dcF1forBcircDefn}, Definition~\ref{finalDefdcF1}}).
\end{align}
To define the maps $\sN$ and $\dot \sN$ above, we note that one first must define them on $\sB^\wbox(\xi_\splash)$, and, by density, extend to $\sB^\bbox(\xi_\splash)$, as done for $\mathring \cF_\sN$ in Definition~\ref{finalDefdcF1} via Lemma~\ref{sNiExt}.

\item For $\xi$ in the start-at-splash class $\sB^\play(\xi_\splash)$ with associated shift parameter $t_\splash$ (i.e. $\xi$ in $\sB^\play(\xi_\splash;t_\splash)$) and $t\in[-T-t_\splash, 0]$ (i.e. up to the moment of splash at $t=0$, but not past it) we define
\begin{align}
H(\xi)(t) &= H(\xi(t)), & 
P^+_\grad(\xi)(t) &= P^+_\grad(\xi(t)), &
\sN(\xi)(t) &= \sN(\xi(t)), \\
\dot H(\xi)(t) &= \dot H(\xi(t)), & 
\dot P^+_\grad(\xi)(t) &= \dot P^+_\grad(\xi(t)), &
\dot \sN(\xi)(t) &= \dot \sN(\xi(t)).
\end{align}

\item 
For $\xi$ in $\sB^\play(\xi_\splash;t_\splash)$ and $t\in[-T-t_\splash, 0]$, we define $\bU(\xi)(t)$, $\bB(\xi)(t)$, and $\bX(\xi)(t)$ by constructing the solution to \eqref{AugmentedInteriorSystemL} in the same way we did for $\xi$ in $\sB$ for Definition~\ref{UBXdefn}. Likewise, we define $P^-_\grad(\xi)(t)$ and $\dot P^-_\grad(\xi)(t)$ for $t\in[-T-t_\splash, 0]$ just as we do for $\xi$ in $\sB$ (see Definitions~\ref{PminusDefns} and~\ref{dotPintdefn}).

\item
For $\xi$ in $\sB^\play(\xi_\splash;t_\splash)$ and $t\in (0, T-t_\splash]$, using $S_+(t)$ of Lemma~\ref{Sopdefn}, we define for $i=1,2$,
\begin{align}
( H_i(\xi)(t), \dot H_i(\xi)(t)) &= S_+(t)( H_i(\xi)(0), \dot H_i(\xi)(0)),\\
(  (P^\pm_\grad)_i(\xi)(t), (\dot P^\pm_\grad)_i(\xi)(t)) &= S_+(t)(  (P^\pm_\grad)_i(\xi)(0), (\dot P^\pm_\grad)_i(\xi)(0)),\\
(\sN(\xi)(t), \dot \sN(\xi)(t)) &= S_+(t)(\sN(\xi)(0), \dot \sN(\xi)(0)).
\end{align}


\item For $\xi$ in $\sB^\play(\xi_\splash;t_\splash)$ and $t\in[-T-t_\splash,T-t_\splash]$, we define $\cH(\xi)(t)$ by using \eqref{cHaltDefn}. This involves making a slight alteration to the conformal map $\cO(z)$ defined by \eqref{cOdefn}, in which the ray from $z_\star$ taken as the branch cut is replaced by a curve emerging from $z_\star$ which is tangent to the corresponding splash interface $\Gamma_\splash$ at the splash point. After the moment of self-intersection at $t=0$, smooth $\tilde X(t,\theta)$ is defined in the appropriate way by taking $\tilde X(t,\theta)=\cO(X(t,\theta);t,\theta)$ for the appropriate alternate branch $\cO(z;t,\theta)$ of $\cO(z)$ for each $(t,\theta)$ corresponding to a trajectory $X(t,\theta)$ which has crossed the branch cut by time $t$.

For $t\in[-T-t_\splash,T-t_\splash]$ we then define $\bE(\xi)(t)$ by \eqref{bEdefnFmla}, using $\cH(\xi)(t)$ as above. Similarly, we take the corresponding definitions for $[\del_t;\cH](\xi)(t)$ and $[\del_t;\bE](\xi)(t)$ (see Proposition~\ref{bEcommBounds}).

\item For $\xi$ in $\sB^\play(\xi_\splash;t_\splash)$ and $t\in[-T-t_\splash,0]$, we define $\dcF_2(\xi)(t)$ as in Definition~\ref{dcF2Defn}.

For $t\in(0,T-t_\splash]$ we define $\dcF_2(\xi)(t)=\dcF_2(\xi)(0)$.

\item
 For $\xi$ in $\sB^\play(\xi_\splash;t_\splash)$ and $t\in[-T-t_\splash,T-t_\splash]$, we define all the quantities $\bD(\xi)(t)$, $\bJ(\xi)(t)$, $\bR^{ij}(\xi)(t)$, etc. in Definition~\ref{DEcommDefns} in the same way, but in terms of the above maps. 
 
 We do the same to define $\mathring \cF(\xi)$ and $\dcF_1(\xi)$ (see Remark~\ref{dcF1forBcircDefn}) on the time interval $[-T-t_\splash,T-t_\splash]$, and define $\dcF(\xi)=(\dcF_1(\xi),\dcF_2(\xi),0,0,0,0)$, as before. For such $\xi$ we also define
\begin{align}
&&
\begin{aligned}
\cF(\xi) &= -\bE(\xi) P^-_\grad(\xi) - (N(\xi)\cdot P^+_\grad(\xi)+\sN(\xi))e_2,\\
\mathrm{F}(\xi) &= \col{\dcF_1(\xi)}{\dcF_2(\xi)},
\end{aligned}&&(t\in[-T-t_\splash,T-t_\splash]).
\end{align}

\item
For $\xi$ in $\sB^\play(\xi_\splash;t_\splash)$ we define the following operator-valued maps for $t\in[-T-t_\splash,T-t_\splash]$:
\begin{align}
\cA(\xi)(t) &= \bJ(\xi)(t)\del_\theta +\cR(\xi)(t) ,\\
A(\xi)(t) &= J(\xi)(t)\del_\theta + R(\xi)(t),
\end{align}
\end{enumerate}
where $J(\xi)(t)=\begin{pmatrix} \bJ_1(\xi)(t) & 0 \\ 0 & \bJ_2\end{pmatrix}$.
\end{definition}
\begin{remark}
\begin{enumerate}[(i)]
\hfill
\item
Note that above we do not define $\bU(\xi)(t)$, $\bB(\xi)(t)$, and $\bX(\xi)(t)$ themselves for $t>0$, instead simply defining the combination represented by $\dcF_2(\xi)(t)$ for such $t$, since we only need to define objects that appear directly in the Lagrangian wave system.

\item
Using the wave solution operator to extend many of the quantities above to $t>0$ is just one straightforward way to ensure we maintain the crucial identities $\dot H(\xi) = \del_t ( H(\xi) )$, $\dot \sN(\xi) = \del_t ( \sN(\xi) )$, etc., all while preserving the appropriate Sobolev regularity levels. Note that, unlike the method above, simply using Taylor expansions in time to define the extensions to $t>0$ would not result in the correct Sobolev regularities in space.
\end{enumerate}
\end{remark} 

For the forward-in-time argument, we actually derive a system for states of the form $\tilde \xi = \xi \circ \tilde \theta$, instead of only dealing with unprepared $\xi$ taken directly from $\sB^\tosplash$. Many of the objects of Definition~\ref{keyPinchSensMapsExtDefns} appear directly in the equivalent wave system for unmodified $\xi$ in $\sB^\tosplash$, as opposed to the system \eqref{splashcenteredsystem1} for $\tilde \xi$. By using the quantities defined below we get the corresponding objects appearing directly in the ``prepared'' system.

\begin{definition}\label{conjugatedmapdefns}
For $\xi_\splash$ in $\sH^k_\splash$ and any smooth diffeomorphism $\tilde\theta(\theta)$ of $S^1$, we define the following for $\xi$ in $\sB^\play(\xi_\splash)$, $\uxi$ in $\sH^k$, and $V:S^1\to\R^2$, where composition with $\tilde\theta$ or $\tilde\theta^{-1}$ refers specifically to that in the $\theta$ argument:
\[
\begin{aligned}
\tilde \bD(\xi,\tilde\theta)
  &= \big( \bD(\xi\circ \tilde\theta^{-1})\big)\circ \tilde\theta , \\
\tilde \dcF(\xi,\tilde\theta)
  &= \big(\dcF(\xi \circ \tilde\theta^{-1}) \big)\circ \tilde\theta ,\\
\tilde \cF(\xi,\tilde\theta)
  &= \big(\cF(\xi\circ\tilde\theta^{-1})\big)\circ\tilde\theta,\\
\tilde F(\xi,\tilde\theta)
  &= \big(\mathrm{F}(\xi \circ \tilde\theta^{-1}) \big)\circ \tilde\theta ,
\end{aligned}
\qquad
\begin{aligned}
\tilde \cR(\xi,\tilde\theta)\uxi
  &= \big(\cR(\xi \circ \tilde\theta^{-1})[\uxi \circ \tilde\theta^{-1}]  \big)\circ \tilde\theta ,\\
\tilde \cA(\xi,\tilde\theta)\uxi
  &= \big(\cA(\xi \circ \tilde\theta^{-1})[\uxi \circ \tilde\theta^{-1}] \big)\circ \tilde\theta ,\\
\tilde R(\xi,\tilde\theta)\uxi_\dagger
  &= \big(R(\xi \circ \tilde\theta^{-1})[\uxi_\dagger \circ \tilde\theta^{-1}]  \big)\circ \tilde\theta ,\\
\tilde A(\xi,\tilde\theta)\uxi_\dagger
  &= \big(A(\xi \circ \tilde\theta^{-1})[\uxi_\dagger \circ \tilde\theta^{-1}]  \big)\circ \tilde\theta, \\
  \tilde \bE(\xi,\tilde\theta)V
  &= \big(\bE(\xi \circ \tilde\theta^{-1})[V \circ \tilde\theta^{-1}]  \big)\circ \tilde\theta .
\end{aligned}
\]

\end{definition}
\begin{remark}\label{conjugfmlas}
The effect of the above transformations involving the conjugation of certain maps with the diffeomorphism $\tilde \theta^{-1}$ is essentially just that everywhere a $\del_\theta$ appeared in the formula for the original version of a map, it is now replaced by $(\tilde \theta')^{-1}\del_\theta$ in the formula for the conjugated map. For example, in view of the definitions, we find our old expression for $\bD(\xi)$ leads to $\tilde \bD(\xi,\tilde \theta)$ as follows:
\begin{align}
\bD(\xi) = \begin{pmatrix} 1 & 0 \\ 0 & 1 + \frac{|H(\xi)|^2}{|X_\theta|^2}\end{pmatrix}
\quad\rightsquigarrow\quad
\tilde\bD(\xi,\tilde\theta) = \begin{pmatrix} 1 & 0 \\ 0 & 1 + \frac{|H(\xi)|^2}{|(\tilde\theta')^{-1}X_\theta|^2}\end{pmatrix}.
\end{align}
Such a change is harmless in our setting, where we plug in $\tilde \theta(\theta) = \Theta(\theta;\theta_\splash,\vartheta_\splash)$, for $\theta_\splash\in I_+$ and $\vartheta_\splash\in I_-$ (see Definition~\ref{relabelingDefn}), as it is then easily verified $\tilde \theta'$ is uniformly bounded away from zero and the smoothness norms of $\tilde \theta$ are uniformly bounded, independently of the $\xi$ in question. With Proposition~\ref{preAndPostSplashUpperBds} we verify the old bounds carry over to the newly defined maps of Definition~\ref{conjugatedmapdefns}.
\end{remark}

By conjugating with $\tilde \theta^{-1}$, we prepare maps to be evaluated at splash-centered states. This results in maps which have virtually the same Lipschitz bounds as before, such as Proposition~\ref{hEstimateProp} (ii), Proposition~\ref{NresLipBd}, etc., given that we are comparing maps evaluated at splash-centered pairs which, by design, have pinches the same order of magnitude at each time, and for which the interface parametrizations $X(t,\theta)$ and $\uX(t,\theta)$ pinch at the same values of $\theta$.

\subsubsection{Derivation of bounds for the splash-centered system}
Now we discuss bounds for the maps prepared for the splash-centered system.

\begin{proposition}\label{preAndPostSplashUpperBds}
Given $\xi$ in $\sB^\tosplash$ with splash state $\xi_\splash$ in $\sH^k_\splash$ and splash parameters $(t_\splash,\theta_\splash,\vartheta_\splash)$, let us set $\tilde \theta(\theta)= \Theta(\theta;\theta_\splash,\vartheta_\splash)$. Then for $\xi^\odot$ defined by \eqref{splashcenteredref}, we have the following estimates for $t$ in $[-T-t_\splash , T-t_\splash]$:
\begin{align}
\lV\tilde\bD(\xi^\odot,\tilde \theta)(t)\rV_{H^k(S^1)}
+
\lV \tilde\bE(\xi^\odot,\tilde \theta)(t) \rV_{B(H^k(S^1),H^k(S^1))}
+
\lV \tilde\cF(\xi^\odot,\tilde\theta)(t) \rV_{H^k(S^1)}
&\leq C \left(\sup_{[-T-t_\splash,T+t_\splash]}
\lV \xi^\odot \rV_{\sH^{k-1}}\right) ,\\
\lV \tilde R(\xi^\odot,\tilde\theta)(t) \rV_{B(\sH^s_\dagger,\sH^s_\dagger)} + \lV \tilde\dcF(\xi^\odot,\tilde\theta)(t) \rV_{\sH^k} &\leq C(M) ,
\end{align}
for integer $s$ with $0\leq s \leq k$.
\end{proposition}
\begin{proof}
First let us comment that due to the usage of the formula \eqref{cHaltDefn} to define $\cH(\xi^\odot)$, the proof of the analogous bound for $\bE(\xi^\odot)$, and also for $\tilde \bE(\xi^\odot,\tilde\theta)$, is done in a similar manner to the proof of Proposition~\ref{bEcommBounds}, using standard techniques for bounding integral kernel operators for chord-arc curves, such as the associated curve $\tilde\Gamma(t)$ produced from $\Gamma(t)$ with the use of the conformal map $\cO(z)$. The bound \eqref{commUpperBd} for $[\del_t;\bE](\xi^\odot)$ is proved just as the corresponding bound is in the proof of Proposition~\ref{bEcommBounds}.

To bound the maps defined in (i)-(iv) of Definition~\ref{keyPinchSensMapsExtDefns}, one splits the bound into the cases of $t<0$ and $t\geq0$. The proofs for $t<0$ follows the same as before. The proofs for $t\geq 0$ uses the boundedness of the wave solution operator $S_+(t)$ (see Lemma~\ref{Sopdefn}). Using these, we verify the necessary bounds on $\tilde D(\xi^\odot,\tilde\theta)$, $\tilde \cF(\xi^\odot,\tilde\theta)$, $\tilde R(\xi^\odot,\tilde\theta)$, and $\tilde \dcF(\xi^\odot,\tilde\theta)$ without excessive additional work.
\end{proof}
Now we give the main Lipschitz bounds for our new setting. In what follows, we use the notation $[a,b]_*$ to refer to an interval without specifying which endpoint is which, that is,
\begin{align}
[a,b]_* = [\min(a,b),\max(a,b)].
\end{align}
\begin{proposition}\label{centeredLipBds}
Given $\xi,\uxi$ in $\sB^\tosplash$ with respective splash states $\xi_\splash,\uxi_\splash$ in $\sH^k_\splash$ and splash parameter triples $q_\splash=(t_\splash,\theta_\splash,\vartheta_\splash)$, $\underline q= (\underline{t_\splash},\underline{\theta_\splash},\underline{\vartheta_\splash})$, let us set $\tilde \theta(\theta)= \Theta(\theta;\theta_\splash,\vartheta_\splash)$, $\underline{\tilde\theta}(\theta)= \Theta(\theta;\underline{\theta_\splash},\underline{\vartheta_\splash})$, and $t^*_\splash=\min(t_\splash,\underline{t_\splash})$. Additionally, we denote the splash-centered versions $\xi^\odot=(\ldots,X^\odot,U^\odot)$ and $\uxi^\odot=(\ldots,\uX^\odot,\underline{U}^\odot)$ as in~\eqref{splashcenteredref}. Then for integer $s$ with $0\leq s \leq k-1$, the bound below holds for all $t$ in $[-T/2,T/2]$.
\begin{align}
\lV \tilde \bD(\xi^\odot,\tilde\theta)(t)-\tilde \bD(\uxi^\odot,\underline{\tilde\theta})(t)\rV_{H^2(S^1)}
+\lV \tilde \cF(\xi^\odot,\tilde\theta)(t)-\tilde \cF(\uxi^\odot,\underline{\tilde\theta})(t)\rV_{H^2(S^1)}
&\leq C\left(\sup_{[-t^*_\splash,t]_*} \lV\xi^\odot-\uxi^\odot \rV_{\sH^1}+|q-\underline q|\right),\\
\lV \tilde R(\xi^\odot,\tilde\theta)(t)-\tilde R(\uxi^\odot,\underline{\tilde\theta})(t)\rV_{B(\sH^2_\dagger,\sH^2_\dagger)}
+
\lV \tilde \dcF(\xi^\odot,\tilde\theta)-\tilde \dcF(\uxi^\odot,\underline{\tilde\theta})\rV_{\sH^2}
&\leq C\left(\sup_{[-t^*_\splash,t]_*} \lV\xi^\odot-\uxi^\odot \rV_{\sH^2}+|q-\underline q|\right),\\
\lV \tilde \bE(\xi^\odot,\tilde\theta)(t)-\tilde \bE(\uxi^\odot,\underline{\tilde\theta})(t)\rV_{B(H^s(S^1),H^s(S^1))} \leq C\Big(\lV& X^\odot(t)-\uX^\odot(t) \rV_{H^{s+1}(S^1)}+|q-\underline q|\Big) .
\end{align}
Furthermore, we have
\begin{align}
\lV \tilde A(\xi^\odot,\tilde \theta)(t)-\tilde A(\uxi^\odot,\underline{\tilde\theta})(t) \rV_{B(\sH^k_\dagger,\sH^2_\dagger)}&\leq C\left(\sup_{[-t^*_\splash,t]_*}\lV \xi^\odot - \uxi^\odot\rV_{\sH^2} + |q-\underline q|\right).
\end{align}
\end{proposition}
\begin{proof}
The comments in the proof of Proposition~\ref{preAndPostSplashUpperBds} apply almost verbatim here. Let us add that to prove the Lipschitz bounds above, first one may prove the analogous bounds in which $|q-\underline q|$ is replaced by $\lV \tilde\theta-\underline{\tilde \theta}\rV_{H^s(S^1)}$ for appropriate $s$ in the right-hand side. Trivial bounds on the associated $\tilde \theta$ and $\underline{\tilde\theta}$ are then used to complete the result.

Consider the Lipschitz bound for $\tilde \bD$. First of all, let us note that the maximal interval of time on which both $\xi^\odot$ and $\uxi^\odot$ are defined is $[-T-t_\splash,T-t_\splash]\cap [-T-\underline{t_\splash},T-\underline{t_\splash]}$, which contains the interval $[-T/2,T/2]$ in particular. For such times we find
\begin{align}
\tilde \bD(\xi^\odot,\tilde\theta)-\tilde \bD(\uxi^\odot,\underline{\tilde\theta}) =
\begin{pmatrix}
 0& 0 \\ 0 & \frac{|H(\xi^\odot)|^2}{(\tilde\theta')^{-2}|X^\odot_\theta|^2} - \frac{|H(\uxi^\odot)|^2}{(\underline{\tilde\theta}')^{-2}|\underline X^\odot_\theta|^2}
\end{pmatrix} ,
\end{align}
and thus
\begin{align}
\lV\tilde \bD(\xi^\odot,\tilde\theta)-\tilde \bD(\uxi^\odot,\underline{\tilde\theta})\rV_{H^2(S^1)}
&\leq C\left(\lV X^\odot - \uX^\odot \rV_{H^3(S^1)} + \lV H(\xi^\odot)-H(\uxi^\odot) \rV_{H^2(S^1)} + \lV \tilde \theta - \underline {\tilde \theta }\rV_{H^3(S^1)}\right).
\end{align}
Suppose first $t\leq 0$. Since we are dealing with splash-centered $\xi^\odot$ and $\uxi^\odot$, the pinches of the two interfaces are the same order of magnitude at each such time (by Corollary~\ref{splashguaranteeprop}). Due to this, we find the original proof of the Lipschitz bound for $H$ (Proposition~\ref{hEstimateProp} (ii)) can be carried out essentially as before to deduce
\begin{align}\label{finalHlipbd}
\lV H(\xi^\odot)-H(\uxi^\odot) \rV_{H^2(S^1)}\leq C\left(
    \lV X^\odot - \uX^\odot \rV_{H^3(S^1)}
    +|q-\underline q|
\right) .
\end{align}
We make the comment that in adapting our old arguments to verify \eqref{finalHlipbd}, it is vital that the splash states of $\xi^\odot$ and $\uxi^\odot$, say $\tilde \xi_\splash$ and $\tilde{\uxi}^\odot$, respectively, both have splash labels $\pi/2$ and $-\pi/2$, which we have intentionally arranged. This specific point arises in the appropriate analogue of (vii) of Proposition~\ref{extensionProp2}, involving the construction of the extension operator $\cE_\delta$.

Combining the above estimates with trivial bounds for $\tilde \theta - \underline{\tilde\theta}$ yields
\begin{align}\label{DexampleBd}
\lV\tilde \bD(\xi^\odot,\tilde\theta)-\tilde \bD(\uxi^\odot,\underline{\tilde\theta})\rV_{H^2(S^1)}
\leq C\left( \lV \xi^\odot-\uxi^\odot \rV_{\sH^1}+|q-\underline q|
\right).
\end{align}
For $t>0$, one uses Lemma~\ref{Sopdefn} to bound the differences in the components of $H(\xi^\odot)$ and $H(\uxi^\odot)$ in terms of their differences at $t=0$, reducing this case to our proof of the bound for $t\leq 0$. This gives us \eqref{DexampleBd} for $-T/2\leq t \leq T/2$. 

Let us also consider the example of the Lipschitz bound for the $\tilde \dcF_2$ component of $\tilde \dcF$, starting with the case $t\le 0$. From the construction in the proof of Proposition~\ref{SolnMainDivCurlProp}, we find
\begin{align}
\lV \tilde \dcF_2(\xi^\odot,\tilde\theta)
   -\tilde \dcF_2(\uxi^\odot,\underline{\tilde \theta})
\rV_{H^{5/2}(\Sigma)}
&\le
C\left(
    \lV \bU(\xi^\odot) - \bU(\uxi^\odot)
    \rV_{H^{7/2}(\Sigma)}
    +
    \lV \bB(\xi^\odot) - \bB(\uxi^\odot)
    \rV_{H^{7/2}(\Sigma)}
    + |q - \underline{q}|
\right) \\
&\le
C'\big(
    \lV \bX(\xi^\odot)-\bX(\uxi^\odot)
    \rV_{H^{7/2}(\Sigma)}
    +\lV \bom^\odot - \underline \bom^\odot
    \rV_{H^{5/2}(\Sigma)}
    \\
    &\hspace{1cm}
    +\lV X^\odot-\uX^\odot 
    \rV_{H^4(S^1)}
    +\lV U^\odot - \underline U^\odot
    \rV_{H^3(S^1)}
    +|q-\underline q|
\big)\\
&\le
C''\left(
    \lV \bX(\xi^\odot)-\bX(\uxi^\odot)
    \rV_{H^{7/2}(\Sigma)}
    +\lV \xi^\odot - \uxi^\odot \rV_{\sH^2}
    +|q-\underline q|
\right). \label{dcFtwoLipExample}
\end{align}
Let us suppose without loss of generality $t^*_\splash = t_\splash$. Again, following the proof of Proposition~\ref{SolnMainDivCurlProp}, one establishes
\begin{align}
\lV \bX(\xi^\odot) - \bX(\uxi^\odot) \rV_{H^{7/2}(\Sigma)}
&\le
\lV \bX(\xi^\odot)(-t_\splash)-\bX(\uxi^\odot)(-t_\splash)
\rV_{H^{7/2}(\Sigma)}
+C \int^t_{-t_\splash} 
\lV \xi^\odot(s) -\uxi^\odot(s) \rV_{\sH^2}\, ds , \\
&\le
C'\left(
    \int^{-t_\splash}_{-\underline{t_\splash}}
        \lV \uxi^\odot(s)
        \rV_{\sH^2}
    +\int^t_{-t_\splash} 
        \lV \xi^\odot(s) -\uxi^\odot(s) \rV_{\sH^2}\, ds
\right) \\
&\le
C''\left(
    |q-\underline q|
    +
    \sup_{[-t_\splash,t]_*}
     \lV \xi^\odot -\uxi^\odot \rV_{\sH^2}
\right).
\end{align}
Using this in \eqref{dcFtwoLipExample}, we get
\begin{align}
\lV \tilde \dcF_2(\xi^\odot,\tilde\theta)
   -\tilde \dcF_2(\uxi^\odot,\underline{\tilde \theta})
\rV_{H^{5/2}(\Sigma)}
\le
C\left(
    \sup_{[-t_\splash,t]_*}
    \lV \xi^\odot -\uxi^\odot \rV_{\sH^2}
    +|q-\underline q|
\right) .
\end{align}
For $t> 0$ the bound then trivially holds due to the definition of $\dcF_2$ given above. This results in a satisfactory Lipschitz bound for the $\dcF_2$ component of the $\dcF$ map.

The proofs of Lipschitz bounds for various other maps in our original setting of $\xi$ in $\sB$ (such as in Proposition~\ref{finalLipBdSummary}) are ported over to the new setting in a similar manner. The main exception is that when establishing bounds for maps like $\tilde E(\xi^\odot,\tilde\theta)$, we now use techniques from the proof of Proposition~\ref{bEcommBounds}, as we did in the proof of Proposition~\ref{preAndPostSplashUpperBds}.
\end{proof}

\subsubsection{Forward-in-time local existence in Sobolev spaces}

\subsubsection*{Linearized splash-centered system}
Let us consider a fixed $\uxi$ in $\sB^\tosplash$ with corresponding splash parameters $(\underline{t_\splash},\underline{\theta_\splash},\underline{\vartheta_\splash})$. In addition we denote $\underline{\tilde \theta}(\theta)=\Theta(\theta;\underline{\theta_\splash},\underline{\vartheta_\splash})$, and, for another evolving state $\xi=(\dot U^{*},\dot B^{*}, \bom, \bj,X,U)$ in $\sB^\tosplash$, set
\begin{align}
\tilde \xi &= \xi \circ \underline{\tilde \theta}, \\
\tilde \xi_\dagger &= \xi_\dagger \circ \underline{\tilde \theta} ,\\
(\dot U^{*\sim},\dot B^{*\sim},\tilde \bom, \tilde \bj,\tilde X, \tilde U) &=(\dot U^{*}\circ \underline{\tilde \theta},\dot B^{*}\circ \underline{\tilde \theta}, \bom\circ \underline{\tilde \theta}, \bj\circ \underline{\tilde \theta},X \circ \underline{\tilde \theta},U \circ \underline{\tilde \theta}).
\end{align}
Linearizing the original Lagrangian wave system \eqref{xiSystemRep} and writing it in such a way that $\tilde\xi=(\tilde\xi_\dagger,\tilde X,\tilde U)$ becomes the unknown, by recalling Definition~\ref{conjugatedmapdefns} one arrives at the following, in which $\underline \xi^\odot \in \sB^\play(\underline{\xi_\splash};\underline{t_\splash})$ denotes the splash-centered version of $\uxi$.
\begin{align}
\begin{aligned}
\del_t \tilde \xi_\dagger &= \tilde A( \uxi^\odot,\underline{\tilde\theta})(t-\underline{t_\splash})  \tilde \xi_\dagger +\tilde F(\uxi^\odot,\underline{\tilde\theta})(t-\underline{t_\splash}),\\
\del_t \tilde X &= \tilde U, \\
\del_t \tilde U &=  \big(\tilde\bE( \uxi^\odot,\tilde\utheta)(t-\underline{t_\splash})\big)^{-1}\dot U^{*\sim} ,\\
\tilde\xi(0)&= \xi_\init\circ\underline{\tilde\theta}.
\end{aligned}\label{splashcenteredsystem1}
\end{align}
In each iteration step we solve such a system and verify that the resulting $\tilde\xi$ realizes a splash during its interval of existence. We then produce the new $\xi=\tilde \xi \circ \underline{ \tilde \theta}$, and then use the corresponding $\xi^\odot$ as the replacement for the term $\underline \xi^\odot$ in the next iteration of the system.

\subsubsection*{Existence of solutions}

As in Section~\ref{katosmethodsection}, we use Kato's theory of semigroups to verify there is a well behaved solution operator for the linearized $\tilde\xi_\dagger$ evolution equation, given just the assumption that $\uxi$ is in $\sB^\tosplash$. 
We then integrate to compute the corresponding $\tilde X$ and $\tilde U$ and use the result to get $\xi$ on the time interval $[-T,T]$. We show below without much difficulty that the $\xi$ produced is again in $\sB^\tosplash$.
\subsubsection*{Iteration scheme}
Suppose we have already constructed an iterate $\xi_n$. Below we explain how to define various terms to appear or to play a role in the iterated system for the $(n+1)^\text{st}$ iterate.
\begin{definition}\label{iterationsetupdefn}
Given $\xi_n$ in $\sB^\tosplash$ with splash parameters $q_n=(t_{n},\theta_{n},\vartheta_{n})$, for $\tilde\theta_n(\theta)=\Phi(\theta;\theta_{n},\vartheta_{n})$, consider $\xi^\odot_n(t,\theta,\psi)=\xi_n(t+t_n,\tilde\theta_n(\theta),\psi)$, noting $\xi^\odot_n$ is in $\sB^\play(\tilde\xi_{\splash,n})$ for some splash state $\tilde\xi_{\splash,n}$. Let us then define
\begin{align}
\bD^\odot_n &= \tilde \bD(\xi^\odot_n,\tilde\theta_n), &
\bE^\odot_n &= \tilde \bE(\xi^\odot_n,\tilde \theta_n), &
\dcF^\odot_n &= \tilde \dcF(\xi^\odot_n,\tilde\theta_n),
\end{align}
and analogously define $\cF^\odot_n$, $F^\odot_n$, $A^\odot_n$, $\cA^\odot_n$, etc. in terms of the corresponding maps defined in Definition~\ref{conjugatedmapdefns}, evaluated at $\xi^\odot_n$ and $\tilde\theta_n$.
\end{definition}
\begin{proposition}\label{solutionOpBds}
Under the conditions stated in Definition~\ref{iterationsetupdefn}, there is a solution operator $S^\odot_n(t,s)$ which, for each $s$ in $[-T-t_n,T-t_n]$, maps data $\tilde\xi_{\dagger;s}\in\sH^k_\dagger$ at time $t=s$ to the solution $S^\odot_n(t,s)\tilde\xi_{\dagger;s}=\tilde\xi_\dagger(t)$ to the evolution equation
\begin{align}
&& \frac{d\tilde\xi_{\dagger}}{dt}(t)
&= A^\odot_n(t) \tilde\xi_{\dagger}(t) &(t\in[-T-t_n,T-t_n]),\\
&&\tilde\xi_{\dagger}(t)&=\tilde\xi_{\dagger;s}&(t=s),
\end{align}
and it satisfies the bound for $0\leq j \leq k$
\begin{align}
&& \lV S^\odot_n(t,s)\rV_{B(\sH^j_\dagger,\sH^j_\dagger)} &\leq Ce^{C(M)(t-s)} & (t,s\in[-T-t_n,T-t_n]).
\end{align}
\end{proposition}
\begin{proof}
This is verified by following the proof of Proposition~\ref{KatoProp} and using Propositions~\ref{preAndPostSplashUpperBds} and~\ref{centeredLipBds}.
\end{proof}
By using the above, we can get a solution to the problem below, which represents the main piece of our linearized splash-centered system \eqref{splashcenteredsystem1}:
\begin{align}
&& \frac{d\tilde\xi_{\dagger}}{dt}
&= A^\odot_n(t-t_n) \tilde\xi_{\dagger}+F^\odot_n(t-t_n) &(t\in[-T,T]),\\
&&\tilde\xi_{\dagger}(0)&=\tilde\xi_{\dagger,\init}.&
\end{align}
Indeed, we may express the solution as
\begin{align}
\tilde\xi_{\dagger}(t)=S^\odot_n(t-t_n,-t_n)\tilde\xi_{\dagger,\init}+\int^{t-t_n}_{-t_n}S^\odot_n(t-t_n,s) F^\odot_n(s)\,ds.
\end{align}
Now we give the construction of the $(n+1)^\text{st}$ iterate $\xi_{n+1}$, given the previously defined iterate $\xi_n$ together with the corresponding terms $A^\odot_n$, $F^\odot_n$, etc. of Definition~\ref{iterationsetupdefn}.
\begin{definition}\label{subseqIterate}
Suppose we have $\xi_n$ in $\sB^\tosplash$ with splash state $\xi_{\splash,n}$ in $\sH^k_\splash$, splash parameters $(t_n,\theta_n,\vartheta_n)$, and associated $\xi^\odot_n$ in $\sB^\play(\tilde\xi_{\splash,n})$ for corresponding $\tilde\xi_{\splash,n}=\xi_{\splash,n}\circ \tilde \theta_n$. We define $\tilde \xi_{\dagger,n+1}$ to be the unique solution on the interval $[-T,T]$ to the problem
\begin{align}
\frac{d\tilde \xi_{\dagger,n+1}}{dt}
&= A^\odot_n(t-t_{n}) \tilde \xi_{\dagger,n+1} + F^\odot_n(t-t_{n}), \\
\tilde \xi_{\dagger,n+1}(0)&=\xi_{\dagger,\init}\circ \tilde\theta_n.
\end{align}
We then define $\dot U^{*\sim}_{n+1}$, $\dot B^{*\sim}_{n+1}$, $\tilde\bom_\npo$, and $\tilde\bj_\npo$ by $(\dot U^{*\sim}_{n+1},\dot B^{*\sim}_{n+1},\tilde\bom_\npo,\tilde\bj_\npo)=\tilde\xi_{\dagger,n+1}$ and take
\begin{align}
\tilde\xi_{n+1} &=(\tilde\xi_{\dagger,n+1},\tilde X_{n+1},\tilde U_{n+1}), \\ \intertext{
where}
\tilde U_{n+1} &= U_\init \circ \tilde\theta_n + \int^t_0 (\bE^\odot_n(\tau-t_n))^{-1} \dot U^{*\sim}_{n+1}(\tau)\,d\tau,\\
\tilde X_{n+1} &= X_\init \circ \tilde \theta_n + \int^t_0 \tilde U_{n+1}(\tau)\,d\tau.
\end{align}
We then define
\begin{align}
\xi_{n+1}=(\dot U^*_\npo, \dot B^*_\npo, \bom_\npo,\bj_\npo,X_\npo,U_\npo)=\tilde\xi_{n+1}\circ \tilde\theta^{-1}_n.
\end{align}
\end{definition}
\begin{proposition}\label{iterationstep1}
Given $n=0$ or for $n\geq 1$ that $\xi_n$ is in $\sB^\tosplash$ with splash state 
$\xi_{\splash,n}$ in $\sH^k_\splash$, we have the following.

Then for the resulting $\xi_{n+1}$ given by Definition~\ref{subseqIterate}, we have
\begin{align}
\xi_\npo\in\sB^\tosplash , \label{xisplashes}
\end{align}
with $\xi_\npo$ arriving at a splash state $\xi_{\splash,\npo}$ in $\sH^k_\splash$ at a time in $(0,T/2]$. Denote the corresponding splash parameters by $q_\npo=(t_{\npo},\theta_{\npo},\vartheta_{\npo})$. Using this, let us define $\tilde\theta_{n+1}(\theta)=\Theta(\theta,\theta_{\npo},\vartheta_{\npo})$ and
\begin{align}
&& \xi^\odot_\npo(t,\theta,\psi) &= \xi_\npo(t+t_\npo,\tilde\theta_\npo(\theta),\psi)
&&(t\in[-T-t_\npo,T-t_\npo]), \\[0.5em]
&&\xi^{\odot n}_\npo(t,\theta,\psi) &= \xi_\npo(t+t_{n},\tilde\theta_\npo(\theta),\psi)
&&( t\in [-T-t_n,T-t_n]). \label{staggeredshiftntn}
\end{align}
Then, for $\tilde\xi_{\splash,\npo}=\xi_{\splash,\npo}\circ \tilde\theta_{n+1}$, we have
\(
\xi^\odot_\npo\in\sB^\play(\tilde\xi_{\splash,\npo};t_\npo)
\). Moreover,
\begin{align}
&& \sup_t\lV \xi^{\odot}_\npo \rV_{\sH^k}
=\sup_t\lV \xi^{\odot n}_\npo \rV_{\sH^k} &\leq CM . &\label{staggeredbound}
\end{align}
\end{proposition}
\begin{proof}
The fact that $\xi_\npo$ is in $\sB^\tosplash$ is proved in a manner almost identical to the proof of Proposition~\ref{MainInductiveBoundProp}.

We then apply Proposition~\ref{splashtimeprop} to deduce that $\xi_\npo$ realizes a splash state at a time in $(0,T/2]$. By choosing $T$ sufficiently small, we guarantee the splash state is in $\sH^k_\splash$.

The bound \eqref{staggeredbound} follows easily from \eqref{xisplashes} and the definitions of $\xi^\odot_\npo$ and $\xi^{\odot n}_\npo$.
\end{proof}
\begin{corollary}\label{unifbddness}
We have $\xi_n$ is in $\sB^\tosplash$ for all $n\geq 1$, and \eqref{staggeredbound} holds for all $n\geq 0$.
\end{corollary}

Just as we define $\xi^{\odot n}_\npo$ in \eqref{staggeredshiftntn} Let us similarly define expressions representing the various components of $\xi^{\odot n}_\npo$, for example,
\begin{align}
&& X^{\odot n}_\npo(t,\theta) &= X_\npo(t+t_{n},\tilde\theta_\npo(\theta))  &(t\in[-T-t_\npo,T-t_\npo]),
\end{align}
and analogously we define $U^{\odot n}_\npo$, $\dot U^{*\odot n}_\npo$, $\xi^{\odot n}_{\dagger,\npo}$, etc.


Now we begin the task of verifying that the sequence $\{\tilde \xi_n\}_{n\geq 0 }$ produced above converges.

\begin{lemma}\label{ISclaim1}
For $j=1,2$, the following hold for each $n\geq1$ and $t$ in $[-T/2,T/2]$:
\begin{align}
\mbox{{(i)}}&& \lV (\xi^\odot_\npo-\xi^\odot_n)(t)\rV_{\sH^j}
&\leq
C(\lV(\tilde\xi_\npo-\tilde\xi_n)(t+t^*_n)\rV_{\sH^j}+|q_\npo-q_n|) &(t^*_n=t_n\mbox{ \emph{or} }t_\npo),\\
\mbox{{(ii)}}&&\lV (\xi^\odot_\npo-\xi^\odot_n)(t)\rV_{\sH^j}
&\leq
C(\lV(\xi^{\odot n}_\npo-\xi^{\odot\nmo}_n)(t)\rV_{\sH^j}+|q_\npo-q_n|+|q_n-q_\nmo|),&\phantom{()}\\
\mbox{{(iii)}}&&\lV (\tilde \xi_\npo-\tilde\xi_n)(t) \rV_{\sH^j}
&\leq
C(\lV (\xi^{\odot n}_\npo-\xi^{\odot \nmo}_n)(t-t^*_{n-1})\rV_{\sH^j}+|q_n-q_\nmo|) &(t^*_\nmo = t_\nmo\mbox{ \emph{or} }t_n).
\end{align}
\end{lemma}
\begin{proof}
The bound (i) in the case $t^*_n=t_n$ (the case $t^*_n=t_\npo$ is similar) can be checked by writing the following and using boundedness of $\del_t\tilde\xi_{n+1}$ and $\del_t\tilde\xi_n$:
\begin{align}
\xi^\odot_{n+1}(t)-\xi^\odot_n(t) = \tilde \xi_{n+1}(t+t_{n})-\tilde\xi_n(t+t_{n})
+
 \tilde \xi_{n+1}(t+t_{\npo})-\tilde\xi_{n+1}(t+t_{n}).
\end{align}
Likewise, the bound (iii) in the case $t^*_\nmo=t_\nmo$ (the case $t^*_\nmo=t_n$ is similar) can be shown from
\begin{align}
\tilde\xi_{n+1}(t)-\tilde\xi_n (t)&= \xi^{\odot n}_\npo(t-t_{\nmo}) -\xi^{\odot\nmo}_n(t-t_{\nmo}) + \tilde\xi_{n+1}(t)-\tilde\xi_{n+1}(t+t_{n}-t_{\nmo}) .
\end{align}
The bound (ii) can be shown by combining (i) in the case that $t^*_n=t_n$ and (iii) in the case that $t^*_\nmo=t_n$.
\end{proof}
\begin{lemma}\label{ISclaim2}
For each $n\geq 1$,
\begin{align}
\lV \xi^{\odot n}_\npo(-t_{n})-\xi^{\odot\nmo}_n(-t_{n})\rV_{\sH^2}\leq C|q_n-q_\nmo|.
\end{align}
\end{lemma}
\begin{proof}
This can be verified by noting
\begin{align}
\xi^{\odot n}_\npo(-t_{n})-\xi^{\odot\nmo}_n(-t_{n}) = \xi_{\init}\circ \tilde\theta_n - \xi_{\init} \circ \tilde\theta_{n-1} + \tilde \xi_{n}(0)-\tilde\xi_{n}(t_{\nmo}-t_{n}) ,
\end{align}
and then using bounds on the derivatives of $\xi_\init$ and $\tilde\xi_n$.
\end{proof}

Recall our notation $[a,b]_* = [\min(a,b),\max(a,b)]$ used again below.

\begin{lemma}\label{ISclaim3}
For each $n\geq 1$,
\begin{align}
&& \lV \cF^\odot_\npo(t) -
 \cF^\odot_n(t) \rV_\HTW
 &\leq C\left(\sup_{[-t_n,t]_*}\lV \xi^\odot_\npo-\xi^\odot_n\rV_\sHTO+|q_\npo-q_n|\right) & (t \in[-T/2,T/2]).
\end{align}
\end{lemma}
\begin{proof}
This follows from Proposition~\ref{centeredLipBds}.
\end{proof}

\begin{lemma}\label{ISclaim4}
For each $n\geq 1$ and $t$ in $[-T/2,T/2]$,
\begin{align}
\lV (\xi^{\odot n}_\npo-\xi^{\odot \nmo}_n)(t) \rV_{\sHTO}
\leq
C\left(T\left(\sup_{[-t_n,t]_*}\lV \xi^\odot_n-\xi^\odot_\nmo\rV_\sHTT+\sup_{[-t_n,t]_*}\lV \xi^{\odot n}_\npo-\xi^{\odot \nmo}_n \rV_{\sHTT}\right)+|q_n-q_\nmo|\right) .
\end{align}
\end{lemma}
\begin{proof}
Deriving an evolution equation satisfied by $\xi^{\odot n}_\npo-\xi^{\odot \nmo}_n$, we find
\begin{align}
\xi^{\odot n}_\npo-\xi^{\odot \nmo}_n = (\xi^{\odot n}_\npo -\xi^{\odot\nmo}_n)(-t_{n}) +&\int^t_{-t_{n}}\big((\cA^\odot_n - \cA^\odot_{n-1})(s)\xi^{\odot\nmo}_{n}(s) \\
& +\cA^\odot_{n}(s)(\xi^{\odot n}_\npo-\xi^{\odot\nmo}_n)(s)+ \dcF^\odot_n(s)-\dcF^\odot_{n-1}(s) \big)\,ds.
\end{align}
By using Lemma~\ref{ISclaim2} and the bounds of Propositions~\ref{preAndPostSplashUpperBds} and~\ref{centeredLipBds}, the claim follows.
\end{proof}

\begin{lemma}\label{ISclaim5}
For each $n\geq 1$,
\begin{align}
&& \lV\xi^{\odot n}_{\dagger,\npo}(t)-\xi^{\odot n-1}_{\dagger,n}(t)\rV_{\sH^2_\dagger}
&\leq C\left(T\sup_{[-t_n,t]_*}\lV \xi^\odot_n-\xi^\odot_\nmo\rV_\sHTT+|q_n-q_\nmo|\right) & (t\in[-T/2,T/2]).
\end{align}
\end{lemma}
\begin{proof}
From an evolution equation satisfied by $\xi^{\odot n}_{\dagger,\npo}-\xi^{\odot \nmo}_{\dagger,n}$ we find
\begin{align}
\xi^{\odot n}_{\dagger,\npo}-\xi^{\odot \nmo}_{\dagger,n} =& S^\odot_n(t,-t_{n})(\xi^{\odot n}_{\dagger,\npo}-\xi^{\odot \nmo}_{\dagger,n})(-t_{n}) \\&+\int^t_{-t_{n}}S^\odot_n(t,s)\big((A^\odot_n - A^\odot_{n-1})(s)\xi^{\odot\nmo}_{\dagger,n}(s)  + F^\odot_n(s)-F^\odot_{n-1}(s) \big)\,ds.
\end{align}
The result is then shown from Proposition~\ref{solutionOpBds}, Proposition~\ref{centeredLipBds}, and Lemma~\ref{ISclaim2}.
\end{proof}
\begin{lemma}\label{ISclaim6}
For each $n\geq 1$ and $t$ in $[-T/2,T/2]$,
\begin{align}
\lV (X^{\odot n}_\npo- X^{\odot\nmo}_n)(t)\rV_{H^4(S^1)}
+\lV (U^{\odot n}_\npo- U^{\odot\nmo}_n)(t)\rV_{H^3(S^1)}
\leq &C
\Bigg(\sup_{[-t_n,t]_*}\lV\xi^\odot_n-\xi^\odot_\nmo\rV_\sHTO \\
&+T\sup_{[-t_n,t]_*}\lV\xi^\odot_n-\xi^\odot_\nmo\rV_\sHTT
+|q_n-q_\nmo|
\Bigg).
\end{align}
\end{lemma}
\begin{proof}
Writing out expressions for $X^{\odot n}_\npo - X^{\odot\nmo}_n$ and $U^{\odot n}_\npo- U^{\odot\nmo}_n$, we find
\begin{align}
X^{\odot n}_\npo-X^{\odot\nmo}_n  &=X_\init\circ \tilde\theta_n-X_\init\circ\tilde\theta_\nmo + \int^t_{-t_{n}}(U^{\odot n}_\npo-U^{\odot\nmo}_n)(\tau)\,d\tau
+\int^{-t_{n}}_{-t_{\nmo}}U^{\odot n-1}_n(\tau)\,d\tau.\\
U^{\odot n}_\npo-U^{\odot\nmo}_n &=U_\init\circ \tilde\theta_n-U_\init\circ\tilde\theta_\nmo+\int^t_{-t_{n}}((\bE^\odot_n)^{-1} \dot U^{*\odot n}_\npo  -  (\bE^\odot_\nmo)^{-1}\dot U^{*\odot n-1}_n)(\tau)\,d\tau\\
&\qquad+\int^{-t_{n}}_{-t_{\nmo}}((\bE^\odot_\nmo)^{-1}\dot U^{*\odot n-1}_n)(\tau)\,d\tau.
\end{align}
Now it follows from Lemma~\ref{ISclaim2} and Proposition~\ref{centeredLipBds} that
\begin{align}\label{ISclaim6pfbound}
\lV X^{\odot n}_\npo - X^{\odot\nmo}_n \rV_{H^3(S^1)}
+\lV U^{\odot n}_\npo- U^{\odot\nmo}_n\rV_{H^2(S^1)}
\leq C&
\Bigg(T\bigg(\sup_{[-t_n,t]_*}\lV\xi^{\odot n}_{\npo}-\xi^{\odot\nmo}_{n}\rV_{\sH^2}\\
&\qquad+\sup_{[-t_n,t]_*}\lV\xi^\odot_n-\xi^\odot_\nmo\rV_\sHTT\bigg) +|q_n-q_\nmo|
\Bigg).
\end{align}
One can verify the following identities, which are analogous\footnote{Regarding the appearance of $\tilde\theta'_n$, recall Remark~\ref{conjugfmlas}.} to the identities \eqref{DthetaBformula} and \eqref{DthetaUformula} in the proof of Proposition~\ref{LimitExistence}.
\begin{align}
\del^2_\theta X^{\odot n}_\npo &= (\tilde\theta'_n)^2(\bD^\odot_n \bE^\odot_n)^{-1}(\dot U^{*\odot n}_\npo-\cF^\odot_n),\\
\del_\theta U^{\odot n}_\npo&=\tilde\theta'_n(\bE^\odot_n)^{-1}\dot B^{*\odot n}_\npo.
\end{align}
Using the above and the corresponding equalities in which $n$ is replaced by $n-1$ everywhere, taking differences, and applying two derivatives in $\theta$, we find
\begin{align}
\lV  X^{\odot n}_\npo -  X^{\odot\nmo}_n \rV_{\dot H^4(S^1)}
+\lV  U^{\odot n}_\npo- U^{\odot\nmo}_n\rV_{\dot H^3(S^1)}
\leq C 
&\Bigg(\sup_{[-t_n,t]_*}\lV\xi^\odot_n-\xi^\odot_\nmo\rV_\sHTO+\lV\xi^{\odot n}_{\dagger,\npo}-\xi^{\odot\nmo}_{\dagger,n}\rV_{\sH^2_\dagger}\\
&\quad+\lV \cF^\odot_n-\cF^\odot_{n-1}\rV_\HTW 
+\lV X^\odot_n - X^\odot_{n-1}\rV_{H^3(S^1)}\\
&\quad+|q_n-q_\nmo|
\Bigg).
\end{align}
Using Lemmas~\ref{ISclaim5} and~\ref{ISclaim3} in the right-hand side and combining with \eqref{ISclaim6pfbound}, we get the claim.
\end{proof}

\begin{corollary}\label{ISclaim7}
For each $n\geq 1$ and $t$ in $[-T/2,T/2]$,
\begin{align}
\lV(\xi^{\odot n}_\npo - \xi^{\odot \nmo}_n)(t)\rV_\sHTT
\leq C
\left(\sup_{[-t_n,t]_*}\lV\xi^\odot_n-\xi^\odot_\nmo\rV_\sHTO +T\sup_{[-t_n,t]_*}\lV\xi^\odot_n-\xi^\odot_\nmo\rV_\sHTT +|q_n-q_\nmo|
\right).
\end{align}
\end{corollary}
\begin{proof}
To get the above, we combine the bounds of Lemmas~\ref{ISclaim5} and~\ref{ISclaim6}.
\end{proof}

\begin{lemma}\label{splashparamdynamicbounds}
For each $n\geq 1$,
\begin{align}
|q_\npo-q_n|\leq C T \left(\sup_{t\in[0,T/2]}\lV \tilde \xi_{n+1}-\tilde \xi_{n}\rV_{\sH^2}+\sup_{t\in[0,T/2]}\lV \tilde \xi_n-\tilde \xi_{n-1}\rV_{\sH^2}+|q_n-q_\nmo|\right).
\end{align}
\end{lemma}
\begin{proof}
We use Proposition~\ref{splashparamStabilityProp} to get the bound
\begin{align}
|q_\npo-q_n| \leq  C \sup_{t\in[0,T/2]}(\lV X_{n+1}-X_n\rV_{H^2(S^1)}+\lV U_{n+1}-U_n\rV_{H^1(S^1)}).
\end{align}
Now, since $X_{n+1}=X_n$ and $U_{n+1}=U_n$ at $t=0$,
\begin{align}
\lV X_{n+1}-X_n\rV_{H^2(S^1)}+\lV U_{n+1}-U_n\rV_{H^1(S^1)} \leq C T \sup_{t\in[0,T/2]} \lV \del_t U_{n+1}-\del_t U_n \rV_{H^2(S^1)}.
\end{align}
Meanwhile, 
by writing out expressions for $\del_t U_{n+1}$ and $\del_t U_n$ in terms of $U^{*\sim}_{n+1}$ and $U^{*\sim}_n$ and applying Proposition~\ref{centeredLipBds} and Lemma~\ref{ISclaim1} (i), one verifies
\begin{align}
\sup_{t\in[0,T/2]}\lV \del_t U_{n+1} - \del_t U_n \rV_{H^2(S^1)} 
&\leq C\left(\sup_{t\in[0,T/2]}\lV \tilde \xi_{n+1}-\tilde\xi_{n}\rV_{\sH^2} + \sup_{t\in[0,T/2]}\lV \tilde \xi_n-\tilde \xi_{n-1}\rV_{\sH^2} + |q_n-q_\nmo|\right).
\end{align}
Combining the above bounds gives the claim.
\end{proof}

\begin{proposition}\label{ISclaim8}
For each $n\geq 3$,
\begin{align}\label{ISclaim8bound}
\sup_{[0,T/2]}\lV\tilde\xi_\npo-\tilde\xi_n\rV_\sHTT +|q_\npo-q_n|
\leq&
CT\Bigg(\sup_{[0,T/2]}\lV \tilde\xi_n-\tilde\xi_\nmo\rV_\sHTT+\sup_{[0,T/2]}\lV\tilde\xi_\nmo-\tilde\xi_{n-2}\rV_\sHTT\\
&\qquad+|q_n-q_\nmo|+|q_\nmo-q_{n-2}|+|q_{n-2}-q_{n-3}|\Bigg),
\end{align}
and the sequences $\{\tilde\xi_n\}_{n\geq 0}$ and $\{q_n\}_{n\geq 0}$ converge in $C^0([0,T/2];\sH^k)$ and $\R^3$, respectively. We denote the limit of $\{\tilde\xi_n\}_{n\geq 0}$ by $\tilde \xi$, and the limit of $\{q_n\}_{n\geq 0}$ by $q=(t_\splash,\theta_\splash,\vartheta_\splash)$.
\end{proposition}
\begin{proof}
Consider $t$ in $[0,T/2]$. We use (iii) of Lemma~\ref{ISclaim1}, followed by Corollary~\ref{ISclaim7} to bound
\begin{align}
\lV\tilde\xi_\npo-\tilde\xi_n\rV_\sHTT &\leq
C
\left(\sup_{[-t_n,t-t_n]}\lV\xi^\odot_n-\xi^\odot_\nmo\rV_\sHTO +T\sup_{[-t_n,t-t_n]}\lV\xi^\odot_n-\xi^\odot_\nmo\rV_\sHTT+|q_n-q_\nmo|
\right).
\end{align}
Applying (ii) of Lemma~\ref{ISclaim1}, we get
\begin{align}
\lV\tilde\xi_\npo-\tilde\xi_n\rV_\sHTT &\leq
C
\left(\sup_{[-t_n,t-t_n]}\lV\xi^{\odot n-1}_n-\xi^{\odot n-2}_{n-1}\rV_\sHTO +T\sup_{[-t_n,t-t_n]}\lV\xi^\odot_n-\xi^\odot_\nmo\rV_\sHTT+|q_n-q_\nmo|+|q_\nmo-q_{n-2}|
\right).
\end{align}
With the use of Lemma~\ref{ISclaim4} and Corollary~\ref{ISclaim7} one can then show
\begin{align}
\lV\tilde\xi_\npo-\tilde\xi_n\rV_\sHTT 
&\leq
C
\left(T\left(\sup_{[-t_n,t-t_n]}\lV\xi^\odot_n-\xi^\odot_\nmo\rV_\sHTT + \sup_{[-t_n,t-t_n]}\lV\xi^\odot_\nmo-\xi^\odot_{n-2}\rV_\sHTT\right)+|q_n-q_\nmo|+|q_\nmo-q_{n-2}|
\right). \\ \intertext{Now, by applying Lemma~\ref{ISclaim1} and recalling $t\leq T/2$, we find}
\lV\tilde\xi_\npo-\tilde\xi_n\rV_\sHTT&\leq C
\Bigg(T\left(\sup_{[0,T/2]}\lV\tilde\xi_n-\tilde\xi_\nmo\rV_\sHTT + \sup_{\tau\in[0,T/2]}\lV(\tilde\xi_\nmo-\tilde\xi_{n-2})(\tau+t_{n-1}-t_n)\rV_\sHTT \right)\\
&\qquad\qquad+|q_n-q_\nmo|+|q_\nmo-q_{n-2}|
\Bigg).
\end{align}
Now, in the supremum in $\tau$ over $[0,T/2]$, by using bounds on the time derivatives of $\tilde \xi_\nmo$ and $\tilde \xi_{n-2}$, we essentially replace $\tau+t_{n-1}-t_n$ by $\tau$ at the price of an $O(|t_n-t_\nmo|)$ error. This yields
\begin{align}
\lV\tilde\xi_\npo-\tilde\xi_n\rV_\sHTT&\leq C'
\Bigg(T\left(\sup_{[0,T/2]}\lV\tilde\xi_n-\tilde\xi_\nmo\rV_\sHTT + \sup_{[0,T/2]}\lV\tilde\xi_\nmo-\tilde\xi_{n-2}\rV_\sHTT \right)+|q_n-q_\nmo|+|q_\nmo-q_{n-2}|
\Bigg),
\end{align}

One then uses Lemma~\ref{splashparamdynamicbounds} to conclude that \eqref{ISclaim8bound} holds. It easily follows that, for a good choice of $T$, the sequence $\{(\tilde\xi_n,q_n)\}_{n\geq 1}$ converges in $C^0([0,T/2];\sH^2)\times \R^3$. To deduce that the $\tilde\xi_n$ in fact converge in $C^0([0,T/2];\sH^k)$, one uses the uniform boundedness implied by Corollary~\ref{unifbddness} of the sequence in a stronger norm with the help of a similar argument as in the proof of Proposition~\ref{improvedxireg}.
\end{proof}
\begin{proposition}\label{mainpropbreakdown}
For $t_\splash$, $\theta_\splash$, and $\vartheta_\splash$ as defined by Proposition~\ref{ISclaim8}, we take $\tilde\theta(\theta)=\Phi(\theta,\theta_\splash,\vartheta_\splash)$. Let us define for $t\in[0,t_\splash]$
\begin{align}
\xi_\star = \tilde\xi\circ \tilde\theta^{-1}.
\end{align}
It follows that $\xi_\star$ solves the system \eqref{xiSystemRep}, but where we replace $T$ by $t_\splash$ and $\xi_0$ by $\xi_\init$. Moreover, we have $\xi_\star(t_\splash)$ in $\sH^k_\splash$, so that $\xi_\star$ exhibits a splash state at time $t=t_\splash$.
\end{proposition}
\begin{proof}
Similar to the proof of the assertion in Proposition~\ref{improvedxireg} that the $\xi$ produced by Proposition~\ref{LimitExistence} solves the system \eqref{xiSystemRep}, one verifies that the $\tilde \xi$ produced by Proposition~\ref{ISclaim8} solves the system \eqref{splashcenteredsystem1} in which $\uxi$ and $\underline{\tilde\theta}$ are replaced by $\xi$ and $\tilde\theta$, respectively. One can then verify from the form of the terms in Definition~\ref{conjugatedmapdefns} that $\xi_\star$ solves \eqref{xiSystemRep}. It is straightforward to show $\xi_\star(t_\splash)$ is in $\sH^k_\splash$.
\end{proof}
\begin{proposition}\label{mainproposition}
Let us produce the corresponding solution to the ideal MHD equations from $\xi_\star$ analogously to the way we produce a solution from $\xi$ in Proposition~\ref{maintheoremforward}. We denote the new solution by $(u_\star(t,x),b_\star(t,x),p_\star(t,x),h_\star(t,y),\Gamma_\star(t))$, with corresponding plasma region $\Omega_\star(t)$ and vacuum region $\cV_\star(t)$, defined for $t$ in $[0,t_\splash]$, $x$ in $\overline{\Omega_\star(t)}$, and $y$ in $\overline{\cV_\star(t)}$. The solution at time $t=0$, which we denote by $(u_\init(x),b_\init(x),p_\init(x),h_\init(y),\Gamma_\init)$, is analytic in the associated domain $\overline{\Omega_\init}\times\overline{\cV_\init}$. Moreover, the solution exhibits a splash--squeeze singularity at $t=t_\splash$ and retains Sobolev regularity for all times $t$ in $[0,t_\splash]$.
\end{proposition}
\begin{proof}
The proof that the solution produced from $\xi_\star$ solves ideal MHD is the same as that of the corresponding statement in Proposition~\ref{maintheoremforward}. Analyticity at time $t=0$ is a direct consequence from our construction of $\xi_\init$.
\end{proof}

\subsection{Analytic local existence and singularity formation}
Observe that by Theorem~\ref{nonexist2Dtheorem}, at the moment of splash, $h_\star(t_\splash,x)$ is not analytic at the splash point. 
Let us now prove that the solution produced by the above proposition remains analytic for a short time, before necessarily losing analyticity at some positive time~$t_\star$ not larger than~$t_\splash$. This amounts to proving a local well-posedness result for short times in analytic norms, in particular when the initial interface is not self-intersecting.

\begin{proposition}\label{analyticlwp}
Take the solution $(u_\star(t,x),b_\star(t,x),p_\star(t,x),h_\star(t,y),\Gamma_\star(t))$ as defined in Proposition~\ref{mainproposition}. Then for some $t_a>0$, this solution is analytic with respect to $t$, $x$, and $y$ for each $t\in[0,t_a)$ and each $x$ and $y$ in the corresponding domains $\overline{\Omega_\star(t)}$ and $\overline{\cV_\star(t)}$.
\end{proposition}
\begin{proof}
Consider the ideal MHD equations taken with the initial data described as above together with associated initial Lagrangian labeling system $\bX_\init(a)$ and corresponding initial interior velocity $\bU_\init(a)$ and magnetic field $\bB_\init(a)$, each defined for $a$ in $\Sigma$. Let us note that $\bX_\init(a)$, $\bU_\init(a)$, and $\bB_\init(a)$ are analytic in $\Sigma$.

We now return to the first Lagrangian formulation of the ideal MHD system we derived in Section \ref{maglagvariablessection} (as opposed to the Lagrangian wave system \eqref{XiSystem}), which we show is amenable to an application of a Cauchy-Kowalevskaya-type theorem for such initial data. This formulation is given as follows:
\begin{align}\label{finalLagrSystem}
&&
\begin{aligned}
\bU_t &= \bB_\theta + \bP_\grad , \\
\bB_t &= \bU_\theta , \\
\bX_t &= \bU ,
\end{aligned}
&& (a\in\Sigma),
\end{align}
where we take initial data $(\bU_\init,\bB_\init,\bX_\init)$, and $\bP_\grad(t,a)=(\grad p)(t,\bX(t,a))$, in which $p(t,x)$ is the solution to the system \eqref{PressureSystem}. The equivalent div-curl system for $\bP_\grad$ in Lagrangian variables is given by
\begin{align}
\label{PressureGradSystem}
\begin{aligned}
(a \in \Sigma) \qquad &
\left\{
    \begin{aligned}
        \grad \cdot (\, \cof (\grad \bX^t) \bP_\grad ) &= 2\bU_\psi \cdot \bU^\perp_\theta - 2 \bB_\psi \cdot \bB^\perp_\theta \\
        \grad^\perp\cdot (\, \grad \bX^t \bP_\grad )  &= 0 \, \Lcm
    \end{aligned}
\right. \\
(\psi = 0)  \qquad & \ \big\{ \, \bX_\theta \cdot \bP_\grad = H\cdot H_\theta , \\
(\psi = -1) \qquad & \ \big\{ \, \bX^\perp_\theta \cdot \bP_\grad = 0,
\end{aligned}
\end{align}
where (dropping the $t$ variable momentarily) $X(\theta)=\bX(\theta,0)$ and $H=H(X)$ is defined by
\begin{align}
&& H(X)(\theta) &= h(X)(X(\theta)) & (\theta\in S^1),
\end{align}
in which $h(X)$ solves the problem \eqref{ExternalFieldEqDef1}--\eqref{ExternalFieldEqDef3} in the corresponding vacuum region $\cV(X)$.

Aiming to apply Theorem 1.1 of \cite{safonov}, we define the analytic extension domain $\Sigma_r\subset (\C/(2\pi\Z))\times\C$ below, for $r>0$:
\begin{align}
\Sigma_r &= \{a+z:a\in \Sigma, \ z=(z_1,z_2), \ |z_1|,|z_1|<r\}.
\end{align}

We then define the following norms for $f$ defined on $\Sigma_r$ and integer $j\geq 0$:
\begin{align}
\lV f \rV_{H^j_r(\Sigma)} = \sup_{|z_1|,|z_2|<r}\lV f((\cdot) + z) \rV_{H^j(\Sigma)}.
\end{align}
Now we define the spaces for $r>0$
\begin{align}
H^j_r(\Sigma) = \{  f:\overline{\Sigma_r}\to\C^2 \ \mbox{such that} \ \lV f \rV_{H^j_r(\Sigma)}<\infty \ \mbox{and} \ f \ \mbox{is analytic in} \ \Sigma_r \}.
\end{align}

Note that $H^j_r(\Sigma)$ fulfills the following property, which is required for the one-parameter family of Banach spaces in Theorem 1.1 of \cite{safonov}. 
\begin{align}\label{BanachScaleProp}
&& H^j_r(\Sigma) &\subset H^j_{r'}(\Sigma),\hspace{1.5cm} \lV \cdot \rV_{H^j_{r'}(\Sigma)} \leq \lV \cdot \rV_{H^j_r(\Sigma)} & (0< r'\leq r).
\end{align}
To consider triples $(\bU,\bB,\bX)$, let us define the associated analytic spaces and norms by the following, for positive integer $j$:
\begin{align}
\mathpzc{H}^j_r &= (H^j_r(\Sigma))^2\times H^{j+1}_r(\Sigma) , \\
\lV (\bU,\bB,\bX) \rV_{\mathpzc{H}^j_r}
& = \lV \bU \rV_{H^j_r(\Sigma)}
+\lV \bB \rV_{H^j_r(\Sigma)}
+\lV \bX \rV_{H^{j+1}_r(\Sigma)}.
\end{align}
Clearly the spaces $\mathpzc{H}^j_r$ satisfy the analogous property stated for $H^j_r(\Sigma)$ in \eqref{BanachScaleProp}.

Now we verify that in the evolution equation \eqref{finalLagrSystem} the right-hand side can essentially be thought of as a function of $(\bU,\bB,\bX)$ which loses no more than one derivative.

We remark that the inequality below holds for real-valued $\bX$:
\begin{align}
\lV H(X) \rV_{H^{9/2}(S^1)}\leq C \left(\lV \bX \rV_{H^5(\Sigma)}\right) .
\end{align}
It is not hard to verify the analytic dependence of $H(X)$ on $X$, meaning that we may extend $H$ analytically to be defined for certain $X$ taking values in $\C^2$. In particular, where $\tilde X(\theta;z) =  \bX(\theta+z_1,z_2)$, we have that for each $\theta$ in $S^1$, $H(\tilde X(\cdot;z))(\theta)$ is well-defined and analytic in $z=(z_1,z_2)$ in $\{|z_1|,|z_2|<r\}$, given $\bX$ in $H^5_r(\Sigma)$ and $\lV \bX-\bX_\init \rV_{H^5_r(\Sigma)}\leq \epsilon_1$ for $r\leq R_0$. We then find
\begin{align}
\sup_{|z_1|,|z_2|<r}\lV H(\tilde X(\cdot;z))\rV_{H^{9/2}(S^1)}\leq C \left(\lV \bX \rV_{H^5_r(\Sigma)}\right) .
\end{align}

Now let us consider $\bP_\grad$ solving the system \eqref{PressureGradSystem} as a map acting on $\bU$, $\bB$, and $\bX$. Suppose $(\bU,\bB,\bX)$ is in $\mathpzc H^4_r$ for $r\leq R_0$, with $\lV(\bU,\bB,\bX)\rV_{\mathpzc H^4_r}\leq M_0$ and $\lV \bX-\bX_\init \rV_{H^5_r(\Sigma)}\leq \epsilon_1$. For each $a$ in $\Sigma$ and $z=(z_1,z_2)$ in $\{|z_1|,|z_2|<r\}$ let us define $\tilde \bP_\grad(\bU,\bB,\bX)(a;z)$ to be the $\C^2$-valued solution to the system \eqref{PressureGradSystem} in which we replace $\bU(a)$ by $\bU(a+z)$, $\bB(a)$ by $\bB(a+z)$, and $X(\theta)$ by $\tilde X(\theta;z)$. We have existence and uniqueness of such a solution as a consequence of Lax-Milgram (with a similar justification as in the comments given below \eqref{InteriorPressureGradSystem} on the existence and uniqueness of $\bP^-_\grad$). One can verify that for each fixed $a$ in $\Sigma$, $\tilde\bP_\grad(\bU,\bB,\bX)(a;z)$ is analytic with respect to $z$. In fact, one finds that we can define an analytic extension of $\bP_\grad(\bU,\bB,\bX)(a)=\tilde\bP_\grad(\bU,\bB,\bX)(a;0)$ from $a$ in $\Sigma$ to the complexified domain $\Sigma_r$, satisfying the following:
\begin{align}
&&\bP_\grad(\bU,\bB,\bX)(a+z) &=\tilde \bP_\grad(\bU,\bB,\bX)(a;z) &(a\in \Sigma, \ |z_1|,|z_2|<r).
\end{align}
Moreover, one can check that $\bP_\grad(\bU,\bB,\bX):\Sigma_r\to\C^2$ satisfies the estimate below:
\begin{align}
\lV \bP_\grad(\bU,\bB,\bX)\rV_{H^4_{r}(\Sigma)} 
&\leq C( \lV \bX \rV_{H^5_r(\Sigma)})
\left(
    \lV \bU \rV_{H^4_{r}(\Sigma)}+\lV \bB \rV_{H^4_{r}(\Sigma)}
    + \sup_{|z_1|,|z_2|<r}\lV H(\tilde X(\cdot;z)) \rV_{H^{9/2}(S^1)}
\right), \\
&\leq C'
\left(
    \lV \bU \rV_{H^4_r(\Sigma)}+\lV \bB \rV_{H^4_r(\Sigma)}
    + \lV \bX \rV_{H^5_r(\Sigma)}
\right).
\end{align}

Defining $\mathfrak{F}$ on $\mathpzc H^4_r$ by
\begin{align}
\mathfrak{F}(\bU,\bB,\bX)=\begin{pmatrix}
\bB_\theta + \bP_\grad(\bU,\bB,\bX) \\
\bU_\theta \\
\bU
\end{pmatrix},
\end{align}
and using the Cauchy integral formula together with the above estimate, we are able to verify for any $r\leq R_0$ and $(\bU,\bB,\bX)$ in $\mathpzc H^4_r$ with $\lV (\bU,\bB,\bX)\rV_{\mathpzc H^4_r}\leq M_0$ and $\lV \bX-\bX_\init \rV_{H^5_r(\Sigma)}\leq \epsilon_1$ that
\begin{align}
&&\lV\mathfrak{F}(\bU,\bB,\bX)\rV_{\mathpzc H^4_{r'}}&\leq \frac{C}{r-r'} \lV (\bU,\bB,\bX)\rV_{\mathpzc H^4_r}
&(0<r'<r).
\end{align}
Similarly, one establishes an estimate of the form
\begin{align}
&&\lV\mathfrak{F}(\bU,\bB,\bX)-\mathfrak{F}(\underline{\bU},\underline{\bB},\underline{\bX})\rV_{\mathpzc H^4_{r'}}&\leq \frac{C}{r-r'} \lV (\bU,\bB,\bX)-(\underline{\bU},\underline{\bB},\underline{\bX})\rV_{\mathpzc H^4_r}
&(0<r'<r).
\end{align}
This allows one to apply Theorem 1.1 of \cite{safonov}, which, when taking the system \eqref{finalLagrSystem} together with the analytic initial data $(\bU_\init,\bB_\init,\bX_\init)$, implies the existence and uniqueness of a solution $(\bU,\bB,\bX)(t,a)$ analytic with respect to $t$ and $a$ up to a small positive time. From this, it is not hard to deduce the analyticity of the corresponding solution to ideal MHD in Eulerian coordinates, as claimed in the statement of the proposition. 
\end{proof}

We now combine the previous two propositions together with the analytic obstruction result of Theorem~\ref{nonexist2Dtheorem} to prove the statement of our main theorem.
\begin{theorem}\label{maintheorem2}
There exists a solution to the ideal MHD system $(u_\star,b_\star,p_\star,h_\star,\Gamma_\star)$ (given by Proposition~\ref{mainproposition}) on the time interval $[0,t_\splash]$ which exhibits a splash--squeeze singularity at time $t=t_\splash$ and retains Sobolev regularity over the entire interval of existence $[0,t_\splash]$. Furthermore, there is a time $t_\star>0$ with $t_\star\leq t_\splash$ such that the solution is analytic in $[0,t_\star)$, but the solution fails to be analytic at time $t=t_\star$.
\end{theorem}
\begin{proof}
This follows immediately from Proposition~\ref{mainproposition}, Proposition~\ref{analyticlwp}, and Theorem~\ref{nonexist2Dtheorem}.
\end{proof}
\appendix

\section{Glossary of notation}

\subsubsection*{Eulerian setting}

\begin{description}[leftmargin=3.6em,labelwidth=2.8em]

  \item[$\Omega$]
    Plasma region $\Omega=\Omega(t)$, with free upper boundary $\Gamma(t)$ and fixed floor boundary $\{x_2=0\}$.

  \item[$\Gamma$]
    plasma--vacuum interface $\Gamma=\Gamma(t)$.

  \item[$\cW$]
    Vacuum wall, the union $\cW_1\cup\cW_2$ of the two fixed boundaries $\cW_1$ and $\cW_2$ (see Figure~\ref{OpenedSplash1}).

  \item[$\cV$]
    Vacuum region $\cV=\cV(t)$, bounded by the interface $\Gamma(t)$ and wall $\cW$.

  \item[$\delta_\Gamma$]
    Pinch of $\Gamma$,
    the distance between the nearly touching left and right branches of $\Gamma$.

  \item[$u(t,x)$]
    Velocity field in $\Omega(t)$.

  \item[$b(t,x)$]
    Magnetic field in $\Omega(t)$.

  \item[$h(t,x)$]
    Magnetic field in $\cV(t)$.

  \item[$\omega(t,x)$]
    Vorticity 
    \(
      \omega(t,x) = \nabla^\perp\!\cdot u(t,x)
    \).

  \item[$j(t,x)$]
    Current 
    \(
      j(t,x) = \nabla^\perp\!\cdot b(t,x)
    \).
     
  \item[$\tau(t,x)$]
   Unit tangent pointing left-to-right along $\Gamma(t)$.
   
  \item[$n(t,x)$]
   Unit normal on the interface $\Gamma(t)$ pointing outward from the plasma region $\Omega(t)$.

\end{description}
\subsubsection*{Lagrangian coordinate setting}

\begin{description}[leftmargin=3.6em,labelwidth=2.8em]

  \item[$S^1$]
    The periodic interval $S^1 = [-\pi,\pi]$ with endpoints identified.

  \item[$\Sigma$]
    Lagrangian label domain for the plasma
    \(
      \Sigma = S^1\times[-1,0],
    \)
    with generic label $a = (\theta,\psi)\in\Sigma$.

  \item[$\bX(t,a)$]
    Lagrangian trajectory map
    \[
      \bX : [0,T]\times\Sigma \to \R^2,\qquad
      \partial_t\bX(t,a) = u(t,\bX(t,a)),\quad
      \bX(0,a) = \bX_0(a),
    \]
    where $\bX_0$ is the initial label-to-position map
    (see Section~\ref{maglagvariablessection}).

  \item[$\bU,\bB,\bom,\bj$]
    Interior Lagrangian fields
    \[
      \bU = u \circ \bX,\qquad
      \bB = b \circ \bX,\qquad
      \bom = \omega \circ \bX,\qquad
      \bj = j \circ \bX.
    \]

  \item[$X,U,B,H$]
    Lagrangian traces on the interface $\{\psi=0\}$:
    \[
      X(t,\theta) = \bX(t,\theta,0),\quad
      U(t,\theta) = \bU(t,\theta,0),\quad
      B(t,\theta) = \bB(t,\theta,0),\quad
       H(t,\theta) = h(t,X(t,\theta)).
    \]

  \item[$\bigtau,N$]
    Tangent and normal along $\Gamma$ in Lagrangian variables:
    \[
      \bigtau(t,\theta) = \tau(t,X(t,\theta)) = \frac{X_\theta(t,\theta)}{|X_\theta(t,\theta)|},
      \qquad
      N(t,\theta) = n(t,X(t,\theta)) = \frac{X_\theta^\perp(t,\theta)}{|X_\theta(t,\theta)|}.
    \]

\end{description}

\subsubsection*{Dirichlet-to-Neumann maps}

\begin{description}[leftmargin=3.6em,labelwidth=2.8em]

  \item[$\Omega_\pm$]
    For interface $\Gamma$, $\Omega_-$ is the region below $\Gamma$ and
    $\Omega_+$ is the region above $\Gamma$
    (Definition~\ref{harmonicextdefn}).
    
  \item[$\bigh_\pm$]
    Harmonic extension operators into $\Omega_\pm$ (Definition~\ref{harmonicextdefn}).

  \item[$\bign_\pm$]
    Interior and exterior Dirichlet-to-Neumann operators
    associated to interface $\Gamma$
    (Definition~\ref{DirichToNeumDefs}):
    \[
      \bign_\pm f = \partial_n(\bigh_\pm f) \qquad (x\in\Gamma, \ f:\Gamma\to\R).
    \]

  \item[$\cN_\pm$]
    Lagrangian Dirichlet-to-Neumann operators:
    \[
      \cN_\pm F(\theta)
        = \big(\bign_\pm(F\circ X^{-1})\big)\big(X(\theta)\big) \qquad (\theta\in S^1, \ F:S^1\to\R).
    \]

  \item[$\Nres$]
    Residual Dirichlet-to-Neumann operator $\Nres = \cN_+ + \cN_-$ (a zero-order operator).

\end{description}
\subsubsection*{Coefficients and common zero-order operators}

\begin{description}[leftmargin=3.6em,labelwidth=2.8em]

  \item[$\bm{\sigma}$]
    Div-curl system coefficient $\bm{\sigma}(\psi) = |\bB_0(0,\psi)|^2$ (see \eqref{vortEvoSystem} and \eqref{OriginalLagrDivCurlSys}).

  \item[$\bD$]
    Coefficient matrix (see \eqref{bDfirstRef} and Definition~\ref{DEcommDefns}):
    \[
      \bD
        = \begin{pmatrix}
            1 & 0 \\
            0 & 1 + \dfrac{|H|^2}{|B|^2}
          \end{pmatrix}.
    \]

  \item[$\cH$]
    Hilbert transform associated to interface $\Gamma$; acts on functions on $S^1$ (see Proposition~\ref{bEDefns}).

  \item[$\bE$]
    Projector for $V:S^1\to\R^2$ (see
    \eqref{bEopintro} and Proposition~\ref{bEDefns}):
    \[
      \bE V
        = \begin{pmatrix}
            1 & \cH \\
            0 & 1
          \end{pmatrix}
          \left[\,\col{\bigtau}{N} V\,\right].
    \]

\end{description}

\subsubsection*{State vector and product function spaces}

\begin{description}[leftmargin=3.6em,labelwidth=2.8em]

  \item[$\xi$]
    Full state vector
    $(\dot U^*, \dot B^*, \bom, \bj, X, U)$.
    
  \item[$\xi_\dagger$]
    Truncated state vector $(\dot U^*, \dot B^*, \bom, \bj)$.

  \item[$\sH^s$]
    State space 
    \((H^s(S^1))^2
          \times (H^{s+1/2}(\Sigma))^2
          \times H^{s+2}(S^1)
          \times H^{s+1}(S^1)
    \), with norm
    \[
      \begin{aligned}
      \lV \xi \rV_{\sH^s}
        &= \lV \dot U^* \rV_{H^s(S^1)}
         + \lV \dot B^* \rV_{H^s(S^1)}
         + \lV \bom \rV_{H^{s+1/2}(\Sigma)}
         + \lV \bj \rV_{H^{s+1/2}(\Sigma)} \\
        &\qquad
         + \lV X \rV_{H^{s+2}(S^1)}
         + \lV U \rV_{H^{s+1}(S^1)}.
      \end{aligned}
    \]

  \item[$\sH^s_\dagger$]
    Truncated state space
    \(
      (H^s(S^1))^2
          \times (H^{s+1/2}(\Sigma))^2.
    \)



  \item[$\sH^k_\gp$]
    Subset of $\sH^k$ of glancing-permitted states (see \eqref{sHsgp}).

  \item[$\sH^k_\splash$]
    Admissible splash states:
    \[
      \sH^k_\splash
        = \{
            \xi \in\sH^k :
            \Gamma \ \mbox{has a splash point at } \ p_\splash, \ |p_\splash-(0,2)|<10^{-3},\ \lV X - X_\init \rV_{H^{k+1}(S^1)}\leq r_1
          \},
    \]

\end{description}

\subsubsection*{Classes for fixed-time states}

\begin{flalign*}
& \sB^\bbox(\delta)
  = \big\{
      \xi \in \sH^k_\gp :
      \lV \xi \rV_{\sH^k} \leq M,\ 
      \lV X - X_0 \rV_{H^{k+1}(S^1)} \leq C\delta,\ 
      c\delta \leq \delta_\Gamma \leq C\delta
    \big\}, &&
\\[0.5em]
& \sB^\bbox
  = \bigcup_{\delta \in [0,\delta_0]}\sB^\bbox(\delta) , &&
\\[0.5em]
& \sB^\wbox
  = \bigcup_{\delta \in (0,\delta_0]}\sB^\bbox(\delta) , &&
\end{flalign*}

The classes $\sB^\bbox(\delta,\xi_\splash)$, $\sB^\bbox(\xi_\splash)$, and $\sB^\wbox(\xi_\splash)$
are defined analogously for a generic splash state $\xi_\splash$ (see Definition~\ref{fixedtimeclasses}).

\subsubsection*{Classes of states with opening splashes (backward-in-time approach)}
\begin{flalign*}
& \sB_\dagger
  = \big\{
      \xi_\dagger \in C^0([0,T];\sH^k_\dagger) :
      \sup_{t} \lV \xi_\dagger \rV_{\sH^k_\dagger} \leq M,\ 
      \sup_{t} \lV \xi_\dagger - \xi_{\dagger,0} \rV_{\sH^{k-1}_\dagger} \leq r_0
    \big\}, &&
\\[0.5em]
& \sB
  = \big\{
      \xi \in C^0([0,T];\sH^k) :
      \sup_{t} \lV \xi \rV_{\sH^k} \leq M,\ 
      \xi_\dagger \in \sB_\dagger,\ 
      \xi(0)=\xi_0,\ 
      X_t = U,\ 
      \lV U \rV_{C^1_{t,\theta}} \leq M
    \big\}. &&
\end{flalign*}

\subsubsection*{Classes of states with closing splashes (forward-in-time approach)}
To-splash classes:
\begin{flalign*}
& \sB^\tosplash_\dagger
  = \big\{
      \xi_\dagger \in C^0([-T,T];\sH^k_\dagger) :
      \sup_{t} \lV \xi_\dagger \rV_{\sH^k_\dagger} \leq M,\ 
      \sup_{t} \lV \xi_\dagger - \xi_\dagger(0) \rV_{\sH^{k-1}_\dagger} \leq r_0
    \big\}, &&
\\[0.5em]
& \sB^\tosplash
  = \big\{
      \xi \in C^0([-T,T];\sH^k) :
      \sup_{t} \lV \xi \rV_{\sH^k} \leq M,\ 
      \xi_\dagger \in \sB^\tosplash_\dagger,\ 
      \xi(0)\in\bm{\Xi}_\init,\ 
      X_t = U,\ 
      \lV U \rV_{C^1_{t,\theta}} \leq M
    \big\}. &&
\end{flalign*}
Generalized class:
\begin{flalign*}
& \sB_{gen}
  = \big\{
      \xi \in C^0([-T,T];\sH^k) :
      \sup_{t} \lV \xi \rV_{\sH^k} \leq 1.1 M,\ 
      \lV \xi_\dagger - \xi_\dagger(0) \rV_{\sH^{k-1}_\dagger} \leq 1.1 r_0,\ 
      X_t = U,\ 
      \lV U \rV_{C^1_{t,\theta}} \leq 1.1 M
    \big\}. &&
\end{flalign*}
Start-at-splash classes for splash time $t_\splash$ and splash state $\xi_\splash$:
\begin{flalign*}
& \sB^\skipto(\xi_\splash; t_\splash)
  = \big\{
      t \mapsto \xi(t + t_\splash) :
      \xi \in \sB_{gen},\ 
      \xi(t_\splash) = \xi_\splash
    \big\}, &&
\\[0.5em]
& \sB^\skipto(\xi_\splash)
  = \bigcup_{0 < s \leq T} \sB^\skipto(\xi_\splash; s). &&
\end{flalign*}

\subsubsection*{Auxiliary domains}

\begin{description}[leftmargin=3.6em,labelwidth=2.8em]

  \item[$\Sigma_-$]
    Lower comparison domain $\Sigma_- = \{ (\theta,\psi) : \theta\in S^1,\ \psi\leq 0 \}$; flattened domain for comparing vector fields pulled back from $\Omega_-$.
    
  \item[$\Sigma_+(\delta)$]
    Upper comparison domain
    \(
      \Sigma_+(\delta)
        = \{ (\theta,\psi) : \theta\in S^1,\ 0<\psi<\delta+\cos^2\theta \}
    \); 
    domain corresponding to $\Omega_+$ for vector field comparisons (see \eqref{SigmaPlusDefn}).


  \item[$\Sigma^\circ_+(\delta)$]
    Vacuum comparison domain 
    \(
      \Sigma^\circ_+(\delta)
        = \{
            a\in\Sigma_+(\delta) :
            |a-a_\star|>10^{-2},\ |a-a_\infty|>10^{-2}
          \}
    \)\quad 
    (see Section~\ref{vacuuminterfaceestimatessection} and Proposition~\ref{extensionProp2}).

\end{description}
\subsubsection*{Weighted Sobolev norms}

\begin{description}[leftmargin=3.6em,labelwidth=2.8em]

    \item[$\lV \cdot \rV_{Y}$]
  
for $Y = H^s(S^1,\delta,m)$, $H^s(\Gamma ,\delta,m)$, $H^s(\Sigma_+,\delta,m)$, $H^s(\cV,\delta,m)$, etc. (see Section~\ref{weightedspacessection}).
\end{description}

\subsubsection*{Layer potentials}

\begin{description}[leftmargin=3.6em,labelwidth=2.8em]

    \item[$\Slp,\slp,\Dlp,\dcal{T}$]
    
(see \eqref{newtpotdefn}--\eqref{pottheorylim}).

\end{description}

\bibliographystyle{plain}
\addcontentsline{toc}{section}{Bibliography}
\bibliography{refs}
\end{document}